\documentclass[11pt]{imsart}
\oddsidemargin
-0.0in \topmargin -0.5in
 \usepackage[utf8]{inputenc}
\usepackage{lipsum}
 
\usepackage{latexsym}
\usepackage{amssymb}
\usepackage{amsmath}
\allowdisplaybreaks
\usepackage{color}
\usepackage{bbold}
\usepackage{bm}
\usepackage{bbm}
 \usepackage{graphicx}
 \usepackage{subfigure}
 \usepackage{xcolor}
\usepackage{amsthm}
\usepackage[breaklinks=true]{hyperref}
\usepackage{longtable}
\usepackage{todonotes}
\usepackage{comment}

\startlocaldefs

 0

\newcommand{\bea}{\begin{eqnarray*}}
\newcommand{\eea}{\end{eqnarray*}}
\newcommand{\be}{\begin{eqnarray}}
\newcommand{\ee}{\end{eqnarray}}
\newcommand{\beq}{\begin{equation}}
\newcommand{\eeq}{\end{equation}}
\newcommand{\dn}{\stackrel{\cal D}{\longrightarrow}}

\newcommand{\prob}{\stackrel{\mathbb{P}}{\longrightarrow}}

\newcommand{\Z}{\mathbb{Z}}
\newcommand{\N}{\mathbb{N}}
\newcommand{\R}{\mathbb{R}}

\newcommand{\Cov}{\operatorname{Cov}}
\newcommand{\E}{\mathbb{E}}
\newcommand{\Var}{\operatorname{Var}}

\newcommand{\sinc}{\text{sinc}}

\newcommand{\Landau}{\mathcal{O}}
\newcommand{\cum}{\text{cum}}
\newcommand{\im}{\mathrm{i}}
\newcommand{\e}{\text{e}}
\newcommand*\tageq{\refstepcounter{equation}\tag{\theequation}}

\newtheorem{satz}{Theorem}[section]
\newtheorem{example}[satz]{Example}

\newtheorem{prop}[satz]{Proposition}

\newtheorem{rem}[satz]{Remark}
\newtheorem{assumption}[satz]{Assumption}
\newtheorem{theorem}[satz]{Theorem}

\newtheorem{corollary}[satz]{Corollary}
\newtheorem{lemma}[satz]{Lemma}

\newsavebox{\smlmat}
\savebox{\smlmat}{$\left(\begin{smallmatrix}\cos(\pi/4) & \sin(\pi/4) \\ -\frac{1}{3}\sin(\pi/4) & \frac{1}{3}\cos(\pi/4)\end{smallmatrix}\right)$}

\usepackage[authoryear]{natbib}
\AtBeginDocument{}

\endlocaldefs
\begin{document}

\begin{frontmatter}

\title{Balancing the edge effect and dimension of spectral spatial statistics under irregular sampling with applications to isotropy testing}

\begin{aug}
\author[A]{\fnms{Theresa} \snm{Eckle}\ead[label=e2]{}}
\and
\author[B]{\fnms{Anne} \snm{van Delft}\ead[label=e1,mark]{anne.vandelft@columbia.edu}}
\and
\author[A]{\fnms{Holger} \snm{Dette}\ead[label=e2]{holger.dette@rub.de}}
\address[B]{Department of Statistics, Columbia University, 1255 Amsterdam Avenue, New York, NY 10027, USA. \printead{e1}}

\address[A]{Ruhr-Universit\"at Bochum, Fakult\"at f\"ur Mathematik, 44780 Bochum, Germany \printead{e2}}
\end{aug}

\begin{abstract}
We investigate distributional properties of a class of spectral spatial statistics under irregular sampling of a  random field that is   defined  on $\mathbb{R}^d$, and use this to obtain a test for isotropy. Within this context, edge effects are well-known to create a bias in classical estimators commonly encountered in the analysis of spatial data. This bias increases with dimension $d$ and, for $d>1$, can become non-negligible in the limiting distribution of such statistics to the extent that a  nondegenerate distribution does not exist. We provide a general theory for a class of (integrated) spectral statistics that enables to 1) significantly reduce this bias and  2) that ensures that asymptotically Gaussian limits can be derived for $d \le 3$ for appropriately tapered versions of such statistics. 
We use this to address some crucial gaps in the literature, and demonstrate that tapering with a sufficiently smooth function is necessary to achieve such results. Our findings specifically shed a new light on a recent result in \cite{subbarao2017}. Our theory then is used to propose a novel test for 
isotropy. In contrast to most of the literature, which validates this assumption on a finite number of spatial locations (or a finite number of Fourier frequencies), we develop a test for isotropy  on the full spatial domain by means of its characterization in the frequency domain. 
More precisely, we derive an explicit expression for the minimum $L^2$-distance between the spectral density of the random field and its best approximation by a  spectral density of an isotropic process.  We prove asymptotic normality of an estimator of this quantity  in the mixed increasing domain  framework and use this result to derive an asymptotic level $\alpha$-test.

\end{abstract}


  
\end{frontmatter}

\section{Introduction} 
\label{sec1} 
\renewcommand{\theequation}{\thesection.\arabic{equation}}
\setcounter{equation}{0}

The analysis of spatial data is an important field of statistics, with many applications in various areas. While originally focus has been on random fields $\{Z(\bm{s}): \bm{s} \in \R^d\}$  that are observed on a regular lattice, more recently interest  has shifted in several fields of application to the assumption that the data is observed irregularly. In the irregularly sampling regime, it is assumed that a random field $Z=\{Z(\bm{s}):\bm{s}\in\R^d\}$ is observed at points $\bm{s}_1,\ldots, \bm{s}_n$ 
sampled from a random distribution yielding the observations $Z(\bm{s}_1),\ldots,Z(\bm{s}_n)$    \citep[see e.g.][and the references therein]{matsuda09,bandyopadhyay2017,subbarao2017}. Because this assumption is often more realistic in practice, focus is in this paper on the latter sampling scheme.

A central point of interest in  spatial data analysis is the covariance matrix of the underlying random field.
The estimation of this matrix becomes much easier under additional structural assumptions such as a parametric form, (second order) stationarity or isotropy. 
However, as several authors have pointed out, parametric or nonparametric estimation of covariance matrices in the spatial domain under the irregular sampling scheme is generally computationally intense and nonparametric estimators are generally not guaranteed to be nonnegative definite. There appears some consensus that these drawbacks are better tackled or can be avoided in the frequency domain
\citep[see, for example,][]{subbarao2017,fuentes07,matsuda09}. These papers mainly focus on a Whittle-type likelihood estimation approach but handle the irregular sampling differently. \cite{fuentes07} accounts for irregular sampling indirectly by assuming local smoothness properties, whereas  \cite{matsuda09} explicitly account for having irregular spatial locations. In fact, these authors 
investigate this problem within a more general context in which a definition of a spatial periodogram is proposed that enables  the development of asymptotic theory for frequency domain-based estimators for irregularly spaced data in which the authors specifically derive  asymptotic uncorrelatedness of the spatial DFT at distinct Fourier frequencies. 

However, and just as in the regularly sampled case, situations where the true function of interest is defined on a larger domain (such as an unbounded set) than the available domain from which we can sample in practice, require  a careful consideration of a phenomenon known as \textit{edge effects} \citep{tukey67,griffith83,ripley84,guyon82,cressie93}. 
This is an epithet for the estimation error that accumulates at the boundaries  and that creates a bias that grows with the number of surrounding edges, i.e., it grows with dimension $d$. Indeed, in practice the actual sampling region is bounded which leads to a loss of information about the process outside this region. It is well-known that this bias is  potentially non-disappearing for dimensions $d\ge 2$ for estimators both in the time and frequency domain. Not surprisingly, the edge effect problem has received significant attention in a wide variety of papers, and specifically in the context of (parametric) Whittle estimation \citep[see][and the references therein]{guyon82,dahlhauskuensch87,fuentes04,matsuda09,robinson07,crujeiras10}. A common technique to deal with edge effects is tapering \citep{brillinger81,dahlhaus83,dahlhaus88}, and tapering of an inconsistent estimator of the spectral density (spatial periodogram) has  been indeed considered in \cite{matsuda09}.
With the exception of the latter, 
the literature mentioned above focuses on the regularly sampled case, and the problem in the irregular case has sometimes even been overlooked altogether (see below). 

Our contribution in this paper is two-fold. For one thing, we propose a general theory to derive limiting distributional properties of integrated nonlinear functionals of the spatial spectral density under the irregular sampling regime in case the data is sampled on an unbounded domain. Specifically, for 
a class of estimators for expressions of the form 
\begin{align}
\int_{\R^d} g(\bm{\omega}) f^p(\bm{\omega})\, d\bm{\omega},  \label{intf}
\end{align}
where $g:\R^d\to\R$ is a bounded  function. 
This developed theory turned out to be necessary to arrive at our original intent: the development of a test for isotropy  for irregularly sampled data on an unbounded domain, which provides our second contribution. However, it also turned out to be necessary to provide a more general theory to address the lack of accounting for edge effects in existing work for the case $p=1$. Both goals are intricate and we shall start with a motivation of our test to epxlain the need for a general theory to deal with estimators for functionals of the form \eqref{intf} further. 

As mentioned above, the literature on spatial data usually relies heavily on having additional structural assumptions being satisfied. A key such assumption is  that the stationary random field, say $\{Z(\bm{s}): \bm{s} \in \R^d\}$,  is \textit{isotropic}, which means that 
the covariance function $c(\bm{h}) = \Cov[Z(\bm{s}),Z(\bm{s}+\bm{h})]$ is invariant under rotations, i.e.,
\begin{align} \label{isotropy_via_covariance_fct}
c(\bm{h})=c_0(\|\bm{h}\|)  ~~\text{for all } \bm{h} \in \R^d
\end{align}
for some function  $c_0:\R^+_0 \to\R$. \cite{guan_et_al_2004} demonstrate that erroneously assuming isotropy can have detrimental effects on the outcome of spatial predictions. Therefore, validity of the isotropy assumption  should be carefully verified before performing any further statistical analysis under this structural assumption.
Isotropy can furthermore equivalently be characterized in the frequency domain.  Indeed, a stationary random field $Z$  with existing spectral density, say $f$, is isotropic if and only if $f$ is rotationally symmetric, that is
\begin{align} \label{isotropy_via_spectral}
f(\bm{\omega})=f_0(\|\bm{\omega}\|)
~~\text{for all} ~\bm{\omega} \in \R^d~, 
\end{align}
for some function $f_0:\R^+_0 \to\R$. For reasons mentioned above, several authors proposed  to shift the testing problem to the frequency domain  \citep[see, for example,][]{fuentesreich2010,wellerhoeting16}. However,  testing for rotational invariance of the spectral density function is very challenging, leading several authors to consider weaker concepts than isotropy, and with the exception of \cite{vanhala_et_al_2018}, only study  the regularly sampled case. We refer to Section \ref{sec2} for a more detailed literature review on available isotropy tests.

Unlike any existing literature we are interested in developing a test for the irregularly sampled case that is not restricted to pre-specified directions and thus enables the detection of \textit{all} types of anisotropy. Our approach is based on  an estimator of the $L^2$-distance between $f$ and its best approximation by a function only depending on the norm of its  argument, that is 
\begin{align*} M_d:=\min_{g} \int_{\R^d} \big(f(\bm{\omega})-g(\|\bm{\omega}\|)\big)^2 \,d\bm{\omega},  \tageq \label{eq:Md}
\end{align*}
where the minimum is taken over all functions $g: \R^+_{0 }\to\R$. Similar arguments as in \cite{vandelft18}  show 
 that $M_d=0$,  if and only if isotropy holds.
    We derive an explicit expression of this quantity in terms of integrals of the spectral density, for which we propose an 
 appropriate estimator
 and prove its asymptotic normality (after 
 appropriate scaling  and centering). 
 In fact, this requires one for example to carefully study the properties of  estimators 
 for integrals of the form \eqref{intf} 
with  $p = 2$.  
 While perhaps the idea to use a minimum distance approach for the purpose of hypothesis testing is a natural one,  in the current context of spatial irregularly sampled data on an unbounded domain, 
it  is a highly nontrivial task to develop corresponding inference methods for this measure. 

To this end, we require an asymptotic framework in which the sampling region is allowed to become unbounded in the limit and such that the distances between the sampling locations converges to zero. This assumption is referred to as the \textit{mixed increasing domain framework} (MID) in the literature \citep[see e.g.,][]{bandyo09,subbarao2017}. We make this more precise in Assumption \autoref{assumption_on_sampling_scheme}. It should come as no surprise that we must account for various sources of  bias of which the order is dimension-dependent, including the one stemming from edge effects, as well as from integral approximations. Part of our contribution is to put forward a comprehensive theory that enables 1) to investigate and reduce the bias for higher dimensions in this MID framework, and 2) to derive asymptotically Gaussian limits for spatial spectral estimators within this framework. The theory we propose to do so is  reminiscent of the $L$-function theory approach in \cite{dahlhaus83}, but then for the highly nontrivial case of spatial data suitable to be applied in the MID framework. It is worth mentioning that the tapered spatial periodogram as considered in \cite{matsuda09} will be a special case of this framework. We demonstrate that tapering is necessary to ensure that dimensions $d > 1$ can be considered, and that for $d > 3$ the edge effect problem cannot be controlled in the MID context for any $p \ge 1$. 
It is also important to mention that recently \cite{subbarao2017} investigated a \textit{nontapered} estimator for the functional \eqref{intf} with  $p=1$ in the MID context and proved its asymptotic normality. However, the aforementioned serious bias problem as well as an additional bias caused by the chosen integral approximation using the Fourier frequencies, have been unfortunately overlooked by the author. Indeed, we verified with the author that the 
distributional result  \textit{only} holds true for dimension $d=1$;  
for $d>1$, the bias in  Theorem 4.1 (ii) of \cite{subbarao2017} is of higher order than the standard deviation in author's  Theorem 4.3. Upscaling the test statistic with the appropriate factor therefore 
only provides an asymptotically normal and consistent estimator if $d=1$. The theory and results put forward in our paper will enable to show that suitable nondegenerate test statistics for  functionals of  form \eqref{intf} with $p \ge 1$ can be obtained  for all $d\le 3$ provided the data is tapered and the integral  in \eqref{intf} is approximated appropriately. Our contribution thus immediately addresses the issue in \cite{subbarao2017}, but furthermore makes clear that the constraint $d\leq 3$ is, also in that case, necessary, and that results for $d>  3$ do not appear attainable.

\section{Best approximation by an isotropic spectral density} \label{sec2}
\renewcommand{\theequation}{\thesection.\arabic{equation}}
\setcounter{equation}{0}

In this section, we introduce a   measure for the deviation of a stationary process from isotropy in the spectral domain. We derive an explicit expression for this measure in terms of two integrals of the spectral density. Estimators of these quantites will be defined and analyzed in Section \ref{sec_iso_test}. 

To be precise, let $Z=\{Z(\bm{s}):\bm{s}\in \R^d\}$ denote a real-valued zero mean spatial random field 
with covariance function $
\Cov[Z(\bm{s_1}),Z(\bm{s_2})]= \E  [Z(\bm{s_1})\,\Z(\bm{s_2})]$
for all $\bm{s_1},\bm{s_2}\in\R^d$.
Under the assumption of second order  stationarity this function is only a function of the difference between two measurement locations. It is therefore translationally invariant and can  be denoted by $c(\bm{h})= \Cov[Z(\bm{s}),Z(\bm{s} + \bm{h})]$ for all $\bm{h}=(h_1,\ldots,h_d)^T\in\R^d$.
If furthermore 
$
c(\bm{h})=c_0(\|\bm{h}\|)$
for some function $c_0:\R^+\rightarrow\R$, i.e., if the covariance function is also invariant under rotations, then the process $Z$ is called \textit{isotropic}. Most spatial domain-based tests for isotropy exploit that  rotational invariance of the covariance function in \eqref{isotropy_via_covariance_fct} is equivalent to rotational invariance of the variogram  
$2\gamma(\bm{h}):=\Var[Z(\bm{s}+\bm{h})-Z(\bm{s})]$. We mention specifically
\cite{matheron61,cabana87,baczkowskimardia90, baczkowski90,jonalasinio01} for the regularly sampled case. More recently,  \cite{guan_et_al_2004} proposed a  test for  isotropy at a   finite set of locations that  makes use of the asymptotic joint normality of the sample variogram at different spatial lags. The test  can be applied both to processes on a grid and to processes on $\R^2$ with sampling locations generated from a homogeneous Poisson process. 
A similar testing method is proposed in  \cite{maitysherman12}. Further contributions to the problem of testing isotropy  are \cite{guan_et_al_2007}  and  \cite{bowman13}. 
While the above-mentioned references consider isotropy in the spatial domain, isotropy can furthermore equivalently be characterized in the frequency domain.  Indeed, if existing, the spectral distribution of  a stationary random field $Z$ has a density; 
\begin{align}  \label{det300}
f(\bm{\omega})=\int_{\R^d} \exp(-\im\bm{\omega}^T \textbf{\emph{h}})\, c(\textbf{\emph{h}}) \, d\textbf{\emph{h}}, \quad \bm{\omega}=(\omega_1,\ldots,\omega_d)^\top,
\end{align}
which then forms a Fourier pair with the covariance function. 
In particular, the random field $Z$ is isotropic if and only if $f$ is rotationally invariant, that is if and only if \eqref{isotropy_via_spectral} holds. Testing for rotational invariance of the spectral density function is very challenging, leading several authors to consider weaker concepts than isotropy. 
Exemplary, we mention the work of \cite{scacciamartin02, scacciamartin05,luzimmermann05} who  consider lattice processes in the two-dimensional case. 
Another procedure in the frequency domain which can be used for  testing isotropy without originally being designed for this  purpose  is proposed  by \cite{fuentes05},
who suggests to specify a set of frequencies with the same absolute values. or irregularly spaced observations, approaches to test for isotropy in the frequency domain are scarce.   
An exception is  \cite{vanhala_et_al_2018} in which a general empirical likelihood approach is proposed to assess the form of spatial covariance structures, and illustrate this general method with an application to testing (among others) for isotropy.
However, just like the literature in the regularly spaced setting, this method requires selection of a finite number of frequency vectors in $\R^d$ with equal norm. More recently, \cite{sahoo_et_al_2019} introduce an isotropy test for spherical data, which  uses the fact  that the spherical harmonic coefficients are uncorrelated if and only if the random field is isotropic. 

Unlike available literature, we aim to develop a test for the irregularly sampled case on unbounded domain that is not restricted to pre-specficied directions and thus enables the detection of \textit{all} types of anisotropy. This leads us to consider the  null hypothesis 
\begin{align} \label{null}
H_0: ~~ \exists \text{ a function } f_0:\R^+_0 \to\R 
~\text{ such that }~
f(\bm{\omega})=f_0(\|\bm{\omega}\|) ~~  \forall  \bm{\omega}\in\R^d. 
\end{align}
To avoid the direct estimation of $f$ and $f_{0 } $ we  use  a minimum distance approach and determine an explicit expression of $M_{d}$ as given in \eqref{eq:Md} in terms of 
integrals of the the spectral density $f$.
A similar idea  was used by  \cite{dette2011}  stationarity tesing in time series  and more recently by  \cite{bagchi18} and \cite{vandelft18} 
 for white noise and stationarity  testing in functional time series, respectively.

The following lemma provides an explicit expression of this distance measure $M_d$ in terms of integrals of the spectral density.

\begin{lemma} \label{two_integrals}
Let $\tilde{f}$ denote the spherical coordinate version of the function $f$,  that is $\tilde{f}(r,\theta) = f ( \phi ^{-1} ( (r,\theta)  )) $, where $\phi$ maps 
$\bm{\omega} \in \tau=[0,2\pi) \times [0,\pi)^{d-2}$ 
to its spherical coordinates $(r,\theta)$. Further, let 
$D(r,\theta)$ denote the spherical volume element, then  
\begin{align} \label{M_d=...}
M_d&=D_{1,d}-D_{2,d},
\end{align}
where 
\begin{align} \label{D_1}
D_{1,d} &:=\int_{\R^d}f^2(\bm{\omega}) \, d\bm{\omega} , \\
 \label{D_2}
 D_{2,d} & :=\frac{1}{\int_{\tau} D(1,\theta) d\theta} \int_{0}^{\infty} \Big (\int_{\tau} \tilde{f}(r,\theta) D(1,\theta)\, d\theta\Big  )^2 r^{d-1} dr 
\end{align}
In particular, we have for $d=2$
\begin{align} \label{D_22}
D_{2,2}=\frac{1}{2\pi} \int_{0}^{\infty} \Big  (\int_{\R^2} f(\bm{\omega}) J_0(r\|\bm{\omega}\|)\, d\bm{\omega} \Big  )^2 \, r \, dr,
\end{align}
where $J_0$ is the Bessel function of order $0$ \citep[see e.g.,][]{watson44}.
\end{lemma}

\begin{proof}
Let 
\begin{align*}
g^\ast(r):=\frac{\int_{\tau} \tilde{f}(r,\theta) D(1,\theta)\, d\theta}{\int_{\tau} D(1,\theta)\, d\theta}
\end{align*}
and note that for an arbitrary function $g:\R^+ \to \R$ we have 
\begin{align*}
\int_{\R^d} \big(f(\bm{\omega})-g(\|\bm{\omega}\|)\big)^2 d\bm{\omega} &= 
\int_{0}^{\infty} \left[\int_{\tau} \big(\tilde{f}(r,\theta)-g(r)\big)^2 D(1,\theta) \,d\theta \right] r^{d-1} dr\\
&=\int_{0}^{\infty} \left[\int_{\tau} \big( \tilde{f}(r,\theta)-g^\ast(r)\big)^2 D(1,\theta)\, d\theta \right] r^{d-1} dr\\
&+2 \int_{0}^{\infty} \left[\int_{\tau} \big(\tilde{f}(r,\theta)-g^\ast(r)\big) \big(g^\ast(r)-g(r)\big) D(1,\theta)\, d\theta \right] r^{d-1} dr\\
&+\int_{0}^{\infty} \left[\int_{\tau} \big(g^\ast(r)-g(r)\big)^2 D(1,\theta)\, d\theta \right] r^{d-1} dr.
\end{align*}
By definition of the function $g^\ast$, it is easy to see that the second term equals zero, and therefore 
\begin{align*}
&\int_{\R^d}\big(f(\bm{\omega})-g(\|\bm{\omega}\|)\big)^2 d\bm{\omega}
\geq \int_{0}^{\infty} \left[\int_{\tau} \big( \tilde{f}(r,\theta)-g^\ast(r)\big)^2 D(1,\theta)\, d\theta \right] r^{d-1} dr \\
&\phantom{========}=\int_{\R^d} f^2(\bm{\omega})\, d\bm{\omega} - \frac{1}{\int_{\tau} D(1,\theta)\, d\theta}\int_{0}^{\infty} \left(\int_{\tau} \tilde{f}(r,\theta) D(1,\theta)\, d\theta \right)^2 r^{d-1} dr,
\end{align*}
where there is equality if and only if $g=g^\ast$ almost everywhere. Recalling the definition of $M_d$ in \eqref{eq:Md}, this yields \eqref{M_d=...}. In order to prove \eqref{D_22}, it suffices to show that
\begin{align} \label{to_show_for_D22}
D_{2,2}=2\pi \int_{0}^{\infty} \Big (\int_{0}^{2\pi} c(r \cos \theta, r \sin \theta)\, d\theta \Big )^2 r \, dr,
\end{align}
since it is easy to see that the Bessel function satisfies
\begin{align*}
J_0(r\|\bm{x}\|)=\frac{1}{2\pi} \int_{0}^{2\pi} \exp\big (\im r (\cos \theta , \sin \theta ) \bm{x}\big ) \, d\theta. 
\end{align*}
For a proof of \eqref{to_show_for_D22}, let $\bm{x}\in\R^2$ be arbitrary and define the radially symmetric function
\begin{align*}
k (\bm{x}):=\int_{0}^{2\pi} f(\|\bm{x}\| \cos \theta, \|\bm{x}\| \sin \theta)\, d\theta.
\end{align*}
Note that
\begin{align*}
\int_{\R^2} k^2 (\bm{x}) \, d\bm{x}
=2\pi \int_{0}^{\infty} \Big (\int_{0}^{2\pi} f(r\, \cos \theta, r\, \sin \theta) \, d\theta \Big )^2 r\, dr
\end{align*}
and therefore 
\begin{align*}
D_2
&=\frac{1}{(2\pi)^4} \int_{\R^2} [(\mathcal{F} k)(\bm{\xi})]^2\, d\bm{\xi}= \frac{1}{(2\pi)^4} \int_{0}^{\infty} \int_{0}^{2\pi} [(\mathcal{F} k) (r\, \cos \theta, r\, \sin \theta)]^2 \, r\, d\theta\, dr,
\end{align*}
where we used the Theorem of Plancherel and the fact that $k$ is real-valued and symmetric.
It is therefore sufficient to show
\begin{align} \label{to_show_lemma34}
(\mathcal{F} k) (r\, \cos \theta, r\, \sin \theta)=(2\pi)^2 \int_{0}^{2\pi} c(r\, \cos \phi, r\, \sin \phi)\, d\phi
\end{align}
for arbitrary $r> 0$, $\theta\in[0,2\pi)$.
By inserting polar coordinates, we get
\begin{align*}
(\mathcal{F} k)(r\, \cos \theta, r\, \sin \theta) 
&=\int_{0}^{\infty} \int_{0}^{2\pi} k(s\, \cos \phi, s\, \sin \phi) \exp\big (-\im \, r\, s 
\|  (\cos \theta , \sin \theta )^\top \|^2
\big) s\, d\phi\, ds\\
&=\int_{0}^{\infty} \int_{0}^{2\pi} k(s\, \cos \phi, s\, \sin \phi) \exp(-\im \, r\, s \cos(\theta-\phi))\, s\, d\phi\, ds.
\end{align*}
Substituting $\phi-\theta=\alpha$, the above expression equals
\begin{align*}
\int_{0}^{\infty} \int_{-\theta}^{2\pi - \theta} k(s\, \cos(\alpha+\theta), s\, \sin(\alpha+\theta)) \exp(-\im \, r\, s\, \cos \alpha)\, s\, d\alpha\, ds,
\end{align*}
such that due to the radial symmetry of the function $k$ and the $2\pi$-periodicity of the trigonometric finctions   
\begin{align*}
(\mathcal{F} k)(r\, \cos \theta, r\, \sin \theta)=\int_{0}^{\infty} k(s\, \cos \theta, s\, \sin \theta) \Big (\int_{0}^{2\pi} \exp(-\im \, r\, s\, \cos \alpha) \,  d\alpha\Big ) s\, ds.
\end{align*}
On the other hand, by definition of the covariance function, we get
\begin{align*}
(2\pi)^2 \int_{0}^{2\pi} c(r\, \cos \phi, r\, \sin \phi)\, d\phi &= \int_{0}^{2\pi}  \int_{\R^2} f(\bm{\omega}) \exp\big (\im \, r
(\cos \theta , \sin \theta )  \bm{\omega}\big ) \, d\bm{\omega} \, d\phi\\
&=\int_{0}^{\infty} \int_{0}^{2\pi} f(s \cos \theta, s \sin \theta) \Big (\int_{0}^{2\pi} \exp(\im \, r s \cos(\phi-\theta))\, d\phi\Big) s\, d\theta ds.
\end{align*}
Again using the index shift $\theta-\phi=\alpha$ and exploiting the $2\pi$-periodicity of cosinus gives
\begin{align*}
\int_{0}^{2\pi} \exp(\im \, r\, s\, \cos(\phi-\theta))\, d\phi = \int_{0}^{2\pi} \exp(-\im \, r\, s\, \cos \alpha)\, d\alpha,
\end{align*}
and therefore we also have
\begin{align*}
(2\pi)^2 \int_{0}^{2\pi} c(r \cos \phi, r \sin \phi)\, d\phi &= \int_{0}^{\infty} \Big (\int_{0}^{2\pi} f(s \cos \theta, s \sin \theta)\, d\theta \Big )\Big ( \int_{0}^{2\pi} \exp(-\im \,r  s \cos \alpha)\,  d\alpha\Big )s\, ds\\
&=\int_{0}^{\infty} k (s\, \cos \theta, s\, \sin \theta) \Big (\int_{0}^{2\pi} \exp(-\im \, r\, s\, \cos \alpha)\,  d\alpha\Big ) s\, ds.
\end{align*}
This yields \eqref{to_show_lemma34} and therefore completes the claim of the lemma.
\end{proof}

It follows from Lemma \ref{two_integrals}  that an estimator $\hat{M}_d$ of the population distance ${M}_d$ can be obtained by providing estimators of $ D_{1,d}$ and $D_{2,d}$. Observe that this involves estimating integrals of the form \eqref{intf}. Specifically, in order to derive a weak convergence result (Theorem \autoref{asymptotic_normality_M}) of a scaled version of $\hat{M}_d-{M}_d$, we must provide a wholesome approach to deal with edge effects in the (higher order) dependence structure of integrated spectral estimators within an appropriate asymptotic framework.

\section{Distributional properties of spectral estimators within the MID framework}
\label{sec3}
\renewcommand{\theequation}{\thesection.\arabic{equation}}
\setcounter{equation}{0}

In this section, we introduce a general theory to analyze spectral estimators within the MID framework under irregular sampling. We start by providing details on the main ingredient to the spatial spectral estimators, and explain its pecularities within the MID framework (defined below), which we then subsequently address by introducing a suitable so-called $L$-function theory that enables the derivation of distributional properties of such estimators.
Proofs of all results can be found in the Appendix.

Let  $Z(\bm{s}_1),\ldots,Z(\bm{s}_n)$ 
denote $n$ observations from the process $Z$
at randomly distributed 
locations $\{\bm{s}_j\}_{j=1}^n$  on a hypercube of length $\lambda:=\lambda(n)>0$. As mentioned in the introduction, we require  the  sampling region to become unbounded in the limit, while at the same time, we require that the distance between the sampling locations converges  to zero asymptotically. More specifically, we shall work under the following assumption.

\begin{assumption} 
{\rm (Sampling scheme and asymptotic framework) ~\ \label{assumption_on_sampling_scheme}
\begin{enumerate}
\item[(i)] For some $\lambda>0$, the random locations $\{\bm{s}_j\}$, $j=1,\ldots,n$, are independent uniformly distributed on the interval $[-\lambda/2,\lambda/2]^d$.
\item[(ii)] It holds that $\lambda\rightarrow \infty$ as $n\rightarrow \infty$, and $\lambda^d/n \rightarrow 0$. 
\end{enumerate}
}
\end{assumption}

Part (ii) is generally known as \textit{mixed increasing domain asymptotics} (MID). An important building block in the development of suitable estimates of the form \eqref{intf} is the \textit{spatial periodogram}, which is defined by 
\begin{align} \label{period}
I_{n,\lambda,d}(\bm{\omega}):=\Big\vert\frac{\lambda^{d/2}}{(2\pi)^{d/2}n} \sum_{j=1}^n Z(\bm{s}_j) \exp(\im \bm{s}_j^\top  \bm{\omega})\Big\vert^2, \qquad \bm{\omega}\in\R^d.
\end{align}
The latter can be shown to yield an asymptotically unbiased but inconsistent estimator of $f(\bm{\omega})$. Furthermore, 
 periodogram ordinates at different \textit{Fourier frequencies}
\begin{align} \label{Four_freq}
\bm{\omega}_{\bm{k},\lambda}:=\big(\frac{2\pi k_1}{\lambda},\ldots,\frac{2\pi k_d}{\lambda}\big)~,~~ \bm{k} =( k_1 ,\ldots,  k_d )^\top \in \mathbb{Z}^d
\end{align}  
are asymptotically uncorrelated as $\lambda \to \infty$ \citep[see][]{bandyopadhyay2017}. 
Thus, 
intuitively  the quantity $ D_{1,d}$ in \eqref{D_1} can be estimated by the statistic
\begin{align} \label{ad_hoc_1}
\Big (\frac{2\pi}{\lambda}\Big )^d \sum_{k_1,\ldots,k_d=-a}^a I_{n,\lambda,d}^p(\omega_{k_1,\lambda},\ldots,\omega_{k_d,\lambda}) 
\end{align}
 with $p=2$, 
provided that     $\lambda\to\infty$, $a\to\infty$ and $a/\lambda\to\infty$ as $n \to \infty$. 
However, some care is necessary in this argument due to our definition of the spectral density in \eqref{det300}.  Indeed, it can be shown that the estimator \eqref{ad_hoc_1} has to be multiplied with the factor  $\frac{1}{p}(2\pi)^{pd/2}$ to become asymptotically unbiased for the quantity $D_{1,d}$. Moreover, its standard deviation can be shown to be  of order $\Landau(1/\lambda^{d/2})$ and thus converges to $0$ as $\lambda \to \infty$ \citep[see][for a similar statement for the statistic \eqref{ad_hoc_1} with $p=1$]{subbarao2017}. A slightly more involved  argument yields that a consistent estimator of the quantity \eqref{D_22} is given by \\
\begin{align} \label{D_22_ad-hoc}
(2\pi)^{2d/2}\frac{1}{\lambda}\sum_{r=1}^a \omega_{r,\lambda} \Big[
\Big (\frac{2\pi}{\lambda}\Big )^2 \sum_{k_1,k_2=-a}^a  I_{n,\lambda,2}(\omega_{k_1,\lambda},\omega_{k_2,\lambda})\, J_0\big
(\omega_{r,\lambda}\, \| (\omega_{k_1,\lambda},
\omega_{k_2,\lambda})^\top \| \big)\Big]^2,
\end{align}

However, spatial periodogram-based estimators such as the above suffer in a detrimental way from the occurrence of edge effects. After scaling with the appropriate factor that ensures the statistic of choice is non-degenerate, the resulting bias becomes non-negligible for spatial periodogram-based estimators for dimension $d > 1$.

Another important source of bias for \textit{integrated} estimators of the above form stems from the Riemann approximation of the integral. In particular, using the Fourier frequencies in \eqref{Four_freq} corresponds to a Riemann sum approximation that 
 evaluates the function at the left endpoints of intervals. However, 
a better integral approximation is obtained using the midpoints of the intervals. More precisely, we will demonstrate in Section   \ref{sec_iso_test} that choosing the shifted Fourier frequencies
\begin{align} \label{shifted_Four_freq}
\tilde{\bm{\omega}}_{\bm{k},\lambda}:=\Big (\frac{2\pi k_1}{\lambda}+\frac{\pi}{\lambda},\ldots,\frac{2\pi k_d}{\lambda}+\frac{\pi}{\lambda}\Big )
\end{align}
will reduce the Riemann approximation error with an order that is crucial to ensure derivation of $\lambda^{d/2}$-consistent estimators for $d\ge 2$.

\subsection{Reducing edge effect bias}

It is well-known that tapering the data can help to reduce the bias stemming from edge effects \citep[see, for example,][]{dahlhauskuensch87,matsuda09}.  
More specifically, instead of the original data we consider building our statistics using the tapered data 
$$
h(\bm{s}_1/\lambda) Z(\bm{s}_1), \ldots, h(\bm{s}_n/\lambda) Z(\bm{s}_n),
$$ 
where $h:\R^d \rightarrow \R$ is a taper function, which we assume satisfies the following regularity conditions.
\begin{assumption} (Taper function)  \label{ass_on_h}
Let $\bm{s}=(s_1, \ldots , s_d)^\top $. We assume that
\vspace*{-5pt}
\begin{enumerate}
\item[(i)] $
h(\bm{s})=\prod_{i=1}^d h_i(s_i), $ 
where the  functions $h_1, \ldots , h_d:\mathbb{R}\to \mathbb{R}^+ $ are symmetric with compact support on $[-1/2,1/2]$;
\item[(ii)] $h_1, \ldots , h_d:\mathbb{R}\to \mathbb{R}^+ $ are twice differentiable on $\R$.
\end{enumerate}
\end{assumption}

If $h_i$ is the indicator function of the interval $[-\frac{1}{2},\frac{1}{2}]$, then Assumption \ref{ass_on_h}(i) gives the rectangular kernel, which will be  denoted by $h^{\text{rect}}$ in the following discussion.  
The rectangular window does not satisfy the regularity conditions given in Assumption \ref{ass_on_h}(ii). The latter are however satisfied by the important class of \emph{cosine windows} \label{cosine_windows} \citep[see, for example,][]{harris78}. Here, for $\alpha\in\N_{0}$, we define the cosine window
\begin{align} \label{cosine_window}
h_i(s_i) =
\begin{cases}
\cos^{\alpha}\left(\pi s_i\right), \qquad &\text{if } s_i\in [-\frac{1}{2},\frac{1}{2}] ,\\
0, \qquad &\text{else}
\end{cases}
\end{align}
($i=1, \ldots , d$) and denote the taper  
by $h_{\text{cos}^{\alpha}}(\bm{s}) 
= \prod_{i=1}^d \cos^{\alpha}\left(\pi s_i\right) I_{[-\frac{1}{2},\frac{1}{2}]}(s_i).
$
For $\alpha=0$ we obtain  the rectangular window. It is easy to see that for $i=1,\ldots,d$ and $\alpha\geq 1$, the function $h_{\text{cos}^{\alpha}}(s_i)$ is $(\alpha-1)$-times continuously differentiable on $\R$.  Therefore, Assumption \ref{ass_on_h} is satisfied for $h_{\text{cos}^{\alpha}} $ if $\alpha\geq 3$.

For a given taper $h$, define the \emph{tapered discrete Fourier transform (DFT)}
\begin{align}
\label{dft} 
J_{n,\lambda,d,h}(\bm{\omega}):=\frac{\lambda^{d/2}}{(2\pi)^{d/2}n\, H_{d,h}(\bm{0})^{\frac{1}{2}}}\sum_{j=1}^n h\left(\frac{\bm{s}_j}{\lambda}\right) Z(\bm{s}_j) \exp(\im \bm{s}_j^T \bm{\omega}),
\end{align}
and 
the \emph{tapered spatial periodogram}
\begin{align} \label{periodogram}
I_{n,\lambda,d,h}(\bm{\omega})&:=\left|J_{n,\lambda,d,h}(\bm{\omega})\right|^2
\end{align} 
where the normalizing constant is given by
\begin{align} \label{H(m)}
H_{d,h}(\bm{m}):=\int_{\left[-\frac{1}{2},\frac{1}{2}\right]^d} h^2(\bm{s}) \exp(-2\pi \im \bm{s}^T \bm{m})\, d\bm{s}, \quad \bm{m}=(m_1,\ldots,m_d)^T\in\Z^d~. 
\end{align}
Tapering the data with a twice differentiable taper function $h$ reduces the bias considerably. To see this, denote the frequency window of $h$ by
\begin{align} \label{Four_trafo_of_h}
B_{\lambda,d,h}(\bm{u}):= \int_{[-\lambda/2,\lambda/2]^d} h\left(\frac{\bm{s}}{\lambda}\right) \exp(-\im \bm{s}^T \bm{u})\, d\bm{s}.
\end{align}
Then it can be shown that  \citep[see equation \eqref{eq:cumJ2} below or ][]{matsuda09}
\begin{align}
\E\left[I_{n,\lambda,d,h}(\bm{\omega}_{\bm{k},\lambda})\right]&= \frac{n-1}{ (2\pi \lambda )^d   H_{d,h}(\bm{0})\,n} \int_{\R^d} f(\bm{u}-\bm{\omega}_{\bm{k},\lambda}) B_{\lambda,d,h}(\bm{u})^2\, d\bm{u} +\Landau \Big (\frac{\lambda^d}{n} \Big ). \tageq \label{eq:exptapspat}
\end{align}
 Above expression makes clear that the bias of the (tapered) periodogram depends on the tails of the frequency window $B_{\lambda,d,h}$; 
 smoother taper functions yield
 less heavy-tailed frequency windows \citep[see also][]{brillinger81}. Note that \eqref{eq:exptapspat} for the untapered spatial periodogram is simply given by
\begin{align*} 
\E\left[I_{n,\lambda,d, h^{\text{rect}}}(\bm{\omega}_{\bm{k},\lambda})\right]&=\frac{\lambda^d}{(2\pi)^d} \int_{\R^d} f\left(\bm{u}-\bm{\omega}_{\bm{k},\lambda}\right) \sinc^2\Big (\frac{\lambda \bm{u}}{2}\Big )\, d\bm{u}+\Landau \Big (\frac{\lambda^d}{n})+\Landau\Big (\frac{1}{n}\Big )   
\tageq \label{bias_sinc}
\end{align*}
where $\sinc(\bm{x}):=\prod_{i=1}^d \frac{\sin(x_i)}{x_i}$ denotes the sinc function. In addition, 
it holds under appropriate assumptions on the random field $Z$ that
\begin{align*}
\E\left[I_{n,\lambda,d,h}(\bm{\omega}_{\bm{k},\lambda})\right]=\begin{cases}  f(\bm{\omega}_{\bm{k},\lambda}) + \Landau\left(\frac{1}{\lambda}+\frac{\lambda^d}{n}\right), &\qquad \text{if } h=h^{\text{rect}},\\
f(\bm{\omega}_{\bm{k},\lambda}) + \Landau\left(\frac{1}{\lambda^2}+\frac{\lambda^d}{n}\right), &\qquad \text{if } h \text{ satisfies Assumption \ref{ass_on_h}},
\end{cases}\tageq \label{eq:mom1}
\end{align*}
We remark that the first moment of the spatial periodogram as given here is nothing else than the second order cumulant of the spatial DFT  defined in \eqref{dft}; see Example \autoref{ex:period}. Consequently, \eqref{eq:mom1} is a special case of the more general theory that we introduce below to obtain distributional properties.
It follows that after scaling with an order of the standard deviation, $\Landau(1/\lambda^{d/2})$, the bias for the choice of taper  $h=h^{\text{rect}}$ is not asymptotically negligible for $d \ge 2$ and blows up for $d\geq 3$.  As a consequence, such untapered estimators do not lend themselves for the construction of $\lambda^{d/2}$-consistent test statistics for any $d > 1$. Note that a taper that satisfies Assumption \ref{ass_on_h} will reduce the first bias term by an order $\Landau(1/\lambda)$. However, the edge effect bias cannot be reduced further by assuming stricter regularity conditions on $h$ than Assumption \ref{ass_on_h}.

 \subsection{An L-function theory for spectral cumulant spatial statistics}
 \label{sec_taper_proofs}
 In the previous subsection,  we showed that edge effect bias can be reduced by appropriately tapering the data.  In order to derive the distributional properties of spectral estimators, it is important to have a comprehensive setting to analyze the higher order cumulants of the spatial DFT of the tapered data in the MID framework. Specifically, for the statistical test proposed in this paper and to address the issue in \cite{subbarao2017},  a careful investigation of a specific class of integrated tapered spatial spectral estimates in the setting of irregularly sampled data is essential. 
We introduce a general theory here in order to so, and highlight that -- unlike in the time series setting where integrated untapered periodograms can provide suitable building blocks for asymptotically normal test statistics \citep[see e.g.,][]{Whittle53,Taniguchi80,FoxTaqqu86,Deo2000}-- a smooth taper is required to reduce the bias to become asymptotically negligible in statistics of the form \eqref{intf} for dimension $2\le d \le 3$. 
The theory proposed here is reminiscent of that in 
\cite{dahlhaus83, dahlhaus88, dahlhaus97} who introduces a class of so-called $L$\textit{-functions}, which are  used to considerably facilitate derivation of upper bounds of cumulants of time series statistics, and to establish weak convergence results \citep[see also][for a general theory to derive Gaussian limits in the locally stationary functional case] {vandelft18a}. For the current spatial framework with mixed increasing domain asymptotics and irregularly sampled data, we proceed in a similar fashion and define a suitable class of $L$-functions. Namely, a crucial ingredient to establish the distributional properties of estimates of \eqref{intf} are the functions

\begin{align} \label{L-fct}
L_{\lambda}^{(s)}(u):=\begin{cases}
\left(\frac{s}{\e}\right)^s \lambda, \qquad &|u|\leq \frac{\e^s}{\lambda},\\
\frac{\log^s(\lambda |u|)}{|u|}, \qquad &|u|>\frac{\e^s}{\lambda}
\end{cases}, \quad s\in\N_0 
\end{align}
and
\begin{align} \label{ell-fct}
\ell^{(s)}(u):=\frac{1}{\lambda} L_{\lambda}^{(s)}\left(\frac{u}{\lambda}\right)=\begin{cases}
\left(\frac{s}{\e}\right)^s, \qquad & |u|\leq \e^s,\\
\frac{\log^s(|u|)}{|u|}, \qquad & |u|>\e^s.
\end{cases}
\end{align}

Lemma \autoref{bounds_for_B}(i) below shows that $L_{\lambda}^{(0)}$ serves as an upper bound for  the function $B_{\lambda,1,h}$ 
in \eqref{Four_trafo_of_h}  (up to a constant), such that convolutions with these windows  can in turn be bounded by convolutions of $L$-functions. Two important properties of afore-introduced  functions are the following. 

\begin{lemma} \label{properties_of_ell}
For $p,q\in\N_0$, the functions $L_{\lambda}^{(s)}$ and $\ell^{(s)}$ satisfy the following:
\begin{enumerate}
\item[(i)] For $v\in\R$, it holds that
$\int_{\R} L_{\lambda}^{(p)}(u) L_{\lambda}^{(q)}(v-u)\, du \leq C(p,q)\, L_{\lambda}^{(p+q+1)}(v),$ for some constant $C(p,q)>0$ depending on $p$ and $q$ only.
\item[(ii)] For $r\in\Z$, it holds that $\sum_{m=-\infty}^{\infty} \ell^{(p)}(m) \ell^{(q)}(m+r)\leq C(p,q)\, \ell^{(p+q+1)}(r)$, for some constant $C(p,q)>0$ depending on $p$ and $q$ only.
\end{enumerate}
\end{lemma}

For $\bm{u}=(u_1,\ldots,u_d)^T\in\R^d$, we define
\begin{align} \label{ell-fct-multi}
L_{\lambda}^{(s)}(\bm{u}):=\prod_{i=1}^d L_{\lambda}^{(s)}(u_i), \qquad \ell^{(s)}(\bm{u}):=\prod_{i=1}^d \ell^{(s)}(u_i).
\end{align}

Lemma \autoref{bounds_for_B} summarizes important bounds that will be used to derive distributional properties of estimators of the form \eqref{intf} and for other estimators needed for the isotropy tests as developed in the next section. 

\begin{lemma} \label{bounds_for_B}
Let the taper function $h$ fulfill the requirements in Assumption \ref{ass_on_h}, part (i). Define the Fourier frequencies $\bm{\omega}_{\bm{m},\lambda}$ and the frequency window $B_{\lambda,d,h}$ as in \eqref{Four_freq} and \eqref{Four_trafo_of_h}, respectively. Furthermore, let $a\in\N$, $\bm{u}=(u_1,\ldots,u_d)^T\in\R^d$, $\bm{v}=(v_1,\ldots,v_d)^T\in\R^d$, and let $C>0$ denote a generic constant. 
Then, the following holds:
\begin{itemize}
\item[(i)] $|B_{\lambda,d,h}(\bm{u})|\leq C\, L_{\lambda}^{(0)}(\bm{u})$ 
\color{black}{
\item[(ii)] $\frac{1}{\lambda^d} \sum_{\bm{m}=-a}^a \int_{\R^d} \left| B_{\lambda,d,h}(\bm{u}) B_{\lambda,d,h}(\bm{\omega}_{\bm{m},\lambda}-\bm{u})\right|\, d\bm{u} \leq C\, \log^{2d}(a)$
\item[(iii)] $\frac{1}{\lambda^{2d}} \sum_{\bm{m}=-a}^{a} \int_{\R^{2d}} \left|B_{\lambda,d,h}(\bm{u}) B_{\lambda,d,h}(\bm{\omega}_{\bm{m},\lambda}-\bm{u}) B_{\lambda,d,h}(\bm{v}) B_{\lambda,d,h}(\bm{\omega}_{\bm{m},\lambda}-\bm{v})\right|\, d\bm{u} d\bm{v}\leq C$
\item[(iv)] Let $t\geq 2$ and $s\in\{0,\ldots,t-1\}$ be arbitrary. Then, for $c_1,\ldots,c_t\geq 1$ and $\bm{d}_1,\ldots,\bm{d}_t\in\Z^d$, we have
\begin{align*}
&\frac{1}{\lambda^{dt}}\sum_{\bm{m}_1,\ldots,\bm{m}_{t-1}=-a}^{a} \int_{\R^{d(t-s)}} \Big|\prod_{j=1}^{s} B_{\lambda,d,h}\big(\frac{c_j}{\lambda}(\bm{m}_j+\bm{d}_j)\big)
 \prod_{j=s+1}^{t-1} \big[B_{\lambda,d,h}(\bm{x}_j) B_{\lambda,d,h}\big(\bm{x}_j-\frac{c_j}{\lambda}(\bm{m}_j+\bm{d}_j)\big)\big]\\
&\phantom{==============} \times B_{\lambda,d,h}(\bm{x}_t)B_{\lambda,d,h}\big(\bm{x}_t+\frac{c_t}{\lambda}\big(\sum_{k=1}^{t-1} \bm{m}_k +\bm{d}_t\big)\big)\Big|\prod_{j=s+1}^t d\bm{x}_j\leq C(t),
\end{align*}
where $C(t)$ is some constant depending on $t$ only. Here, a product $\prod_{j=b}^c z_j$ with $b>c$ is defined as $1$.
\item[(v)] Under the same assumptions on the quantities $t$, $c_1,\ldots,c_t$ and $\bm{d}_1,\ldots,\bm{d}_t$ as in part (iv), it holds that
\begin{align*}
\frac{1}{\lambda^{dt}} \sum_{\bm{m}_1,\ldots,\bm{m}_{t-1}=-a}^{a} \Big|  \prod_{j=1}^{t-1} B_{\lambda,d,h}\big(\frac{c_j}{\lambda}(\bm{m}_j+\bm{d}_j)\big)\, B_{\lambda,d,h}\big(\frac{c_t}{\lambda}(\sum_{k=1}^{t-1} \bm{m}_k +\bm{d}_t)\big)\Big| \leq C(t),
\end{align*}
for some constant $C(t)$ depending on $t$ only.}
\end{itemize}
\end{lemma}

Note that Lemma \ref{bounds_for_B} does not make any assumptions on the smoothness of the taper functions and also applies to rectangular tapers.
\begin{example}[The expectation of the spatial periodogram] \label{ex:period}
{\rm Before continuing we briefly relate this more general theory to what we presented in the previous subsection, which can be seen as a special case.
Indeed, under Assumption \autoref{ass_on_h}, the law of total cumulance and the Fourier transform of the covariance function yield
\begin{align*}
    &\cum\big(J_{n,\lambda,d,h}(\bm{\omega}),J_{n,\lambda,d,h}(\bm{\omega})\big) 
\\&=\frac{(\lambda)^d}{(2\pi)^d n^2 H_{d,h}(\bm{0})^2} \sum_{j_1, j_2: j_1\neq j_2} \E \Big  [h\left(\frac{\bm{s}_{j_1}}{\lambda}\right) h(\frac{\bm{s}_{j_2}}{\lambda}) \exp\big(\im (\bm{s}_{j_1}-\bm{s}_{j_2})^T {\bm{\omega}}\big)\, \E \big[Z(\bm{s}_{j_1}) Z(\bm{s}_{j_2})  \big|\, \bm{s}_{j_1},\bm{s}_{j_2}\big]\Big ]
\\& =\frac{(\lambda)^d}{(2\pi)^d n^2 H_{d,h}(\bm{0})^2} \sum_{j_1, j_2: j_1\neq j_2} \E \Big[h\left(\frac{\bm{s}_{j_1}}{\lambda}\right) h(\frac{\bm{s}_{j_2}}{\lambda}) \exp\big(\im (\bm{s}_{j_1}-\bm{s}_{j_2})^T {\bm{\omega}}\big)\, c(\bm{s}_{j_1}-\bm{s}_{j_2}) \Big ]
\\& = \frac{(n-1)(\lambda)^d}{ (2\pi)^d n H_{d,h}(\bm{0})^2}\int_{\R^{2d}} \Big( \frac{1}{\lambda^{2d}}  \int_{[-\lambda/2,\lambda/2]^{2d}} h\left(\frac{\bm{s}_{1}}{\lambda}\right) h\left(\frac{\bm{s}_{2}}{\lambda}\right)  \exp\big(\im(\bm{s}_1-\bm{s}_2)^T(\bm{x}+\bm{\omega}) \big)d\bm{s}_1 d\bm{s}_2 \Big)  f(\bm{x}) d \bm{x}
\\& = \frac{(n-1)}{ (2\pi)^d n H_{d,h}(\bm{0})^2 (\lambda)^d}\int_{\R^{d}}  \big\vert B_{\lambda,1,h}(\bm{x}+\bm{\omega})\big\vert^2  f(\bm{x}) d \bm{x}, \tageq \label{eq:cumJ2}
\end{align*}
which gives the first term in \eqref{eq:exptapspat}, up to a change of variables. Note that the second term in \eqref{eq:exptapspat} is a bias that can be suppressed by not allowing for $j_1=j_2$.}
\end{example}
To obtain further distributional properties we rely on the  derived bounds to control the higher order dependence structure. Specifically, a cumulant central limit theorem can be shown to hold. We refer to proofs in Appendix B of an application of these results to estimators for objects of the form \eqref{intf} for $p=2$.

We end this section with a final crucial lemma and a subsequent result for integrated spectral estimators. The lemma summarizes bounds on functionals of the frequency window $B_{\lambda,1,h}$, which  depend heavily on the differentiability properties of the taper function $h$, thereby shedding light ont he importance of usage of a smooth taper as we elaborate on below. 

\begin{lemma} \label{orders_of_B}
Let $m\in\Z$, define $\omega_{m,\lambda}$ and $B_{\lambda,1,h}$ as in \eqref{Four_freq} and \eqref{Four_trafo_of_h}, and let $C>0$ denote some generic constant. If the taper function $h$ satisfies the requirements from Assumption \ref{ass_on_h}, part (i), then we have
\begin{enumerate}
\item[(i)] \, \,
$
\int_{\R} \left|B_{\lambda,1,h}(u) B_{\lambda,1,h}(\omega_{m,\lambda}-u)\right| du \leq C \lambda \big(\mathrm{1}_{|m|\le 1}+\frac{\log(|m|)}{|m|}\mathrm{1}_{|m|>1} \big)
$ 
\end{enumerate}
If $h$ moreover satisfies Assumption \ref{ass_on_h}, part (ii), then 
\begin{enumerate}
\item[(ii)] 
$
\left|\int_{\R} B_{\lambda,1,h}(u) B_{\lambda,1,h}(\omega_{m,\lambda} - u) \, u \, du\right| \leq  \frac{C}{|m|} \mathrm{1}_{m\neq 0}. 
$
\item[(iii)]
$
\int_{\R} \left|B_{\lambda,1,h}(u) B_{\lambda,1,h}(\omega_{m,\lambda}-u)\, u \right|\, du \leq C,
$
\item[(iv)]
$
\int_{\R} \left|B_{\lambda,1,h}(u) B_{\lambda,1,h}(\omega_{m,\lambda}-u) \,u^2 \right| du \leq \frac{C}{\lambda}.
$
\end{enumerate}
\end{lemma}

To make clear why twice differentiability of the taper function is required in for example part (iii), note that for $h=h^{\text{rect}}$ and $m=0$ we would get
\begin{align*}
\int_{\R} |B_{\lambda,1,h}(u) B_{\lambda,1,h}( - u) \,u| \, du = \lambda^2 \int_{\R} 
{\sinc}^2 \left(\frac{\lambda u}{2}\right)\, |u|\,du = 
\infty.
\end{align*}

A crucial consequence is a bound on the convoluted interplay of the functions  $B_{\lambda,d,h}$  with the spatial spectral density $f$. As emphasized in the remark below, this bound shows specifically the importance of tapering \textit{integrated} spectral estimates within the MID framework. 

\begin{prop} \label{Lemma F.2_SSR}
Let $\bm{m}=(m_1,\ldots,m_d)^T\in\Z^d$ and $b_{1,1},b_{1,2},\ldots,b_{d,1},b_{d,2}\in\R^+$ be arbitrary and set 
$b:=\max\{b_{1,1},b_{1,2},\ldots,b_{d,1},b_{d,2}\}.$
Furthermore, denote by $C>0, C_1>0$ and $C_2>0$ some generic constants, with $C$ being  independent of $b$. 
Let $f:\R^d\to\R$ be a positive and symmetric function which is twice differentiable with 
\begin{align*}
\sup_{\bm{\omega}\in\R^d} \left|\frac{\partial f(\bm{\omega})}{\partial \omega_i}\right| \leq C_1, \qquad  \sup_{\bm{\omega}\in\R^d} \left|\frac{\partial^2 f(\bm{\omega})}{\partial \omega_i \, \partial \omega_j}\right| \leq C_2, \qquad \text{and} \qquad \int_{\R^d} \left|\frac{\partial f(\bm{\omega})}{\partial \omega_i}\right| \, d\bm{\omega} < \infty 
\end{align*}
for $i,j=1,\ldots,d$. Then, assuming the taper function $h$ satisfies Assumption \ref{ass_on_h}, part (ii), the term 
\begin{align*}
&R:=\frac{1}{\lambda^d}\big\vert\int_{\R^d} B_{\lambda,d,h}(\bm{u})\, B_{\lambda,d,h}(\bm{\omega}_{\bm{m},\lambda}-\bm{u}) (\int_{-b_{1,1}}^{b_{1,2}} \ldots \int_{-b_{d,1}}^{b_{d,2}} g(\bm{v}) \left[f(\bm{u}-\bm{v})-f(\bm{v})\right]\, d\bm{v}) d\bm{u}\big\vert
\end{align*}
satisfies
\begin{align*}
(i) { \text{ if $g:\mathbb{R}^d \to \mathbb{R}$ is bounded }} 
& R \leq C\,\sup_{\bm{v}\in\R^d} |g(\bm{v})| \times \begin{cases} \frac{ b^d}{\lambda^2}, &\qquad \text{if  } m_1=\ldots=m_d=0,\\ 
\frac{1}{\lambda}+\frac{b^d}{\lambda^2}, &\qquad \text{if  }  \exists\, i\in\{1,\ldots,d\}: m_i\neq 0.  \end{cases} 
\\ (ii) \text{if $g:\R^d \rightarrow \R$ is integrable }
& R \leq C\,\int_{\R^d} |g(\bm{v})|\, d\bm{v} \times  \begin{cases} \frac{1}{\lambda^2}, &\qquad \text{if  } m_1=\ldots=m_d=0,\\
\frac{1}{\lambda}, &\qquad \text{if  }  \exists\, i\in\{1,\ldots,d\}: m_i\neq 0.  \end{cases}
\end{align*}
\end{prop}

\begin{rem} \label{rem_edge_effect}
\rm{
While technical, to appreciate the above result, it is important to compare it with the bound available in the untapered case \cite[see][for a proof of the case $d=1$]{subbarao_suppl}. That is, if $h=h^{\text{rect}}$ we have $
B_{\lambda,d,h}(\bm{u})=\lambda^d\, \prod_{i=1}^d \sinc\big(\frac{\lambda u_i}{2}\big)$, $\bm{u}=(u_1,\ldots,u_d)^T\in\R^d$
and find for 
$m_1=\ldots=m_d=0$,
\begin{align*}
&\phantom{\eqsim i}\frac{1}{\lambda^d}\Big|\int_{\R^d} B_{\lambda,d,h}(\bm{u})^2\Big(\int_{-b}^{b} \ldots \int_{-b}^b g(\bm{v}) \Big[f(\bm{u}-\bm{v})-f(\bm{v})\Big] d\bm{v}\Big) d\bm{u}\Big| 
= 
\Landau\Big(\frac{\log \lambda}{\lambda}\Big),
\end{align*}
which is considerably worse than the bounds in Proposition \ref{Lemma F.2_SSR} for $m_1=\ldots=m_d=0$. Indeed, weak convergence results for integrated spatial periodogram estimators in the MID framework will require tapering with a smooth taper to ensure that the bias is sufficiently small. The interested reader is referred to the proofs of Theorems \ref{expectation_theo} and \ref{expect_theo_sec_int}. 
}
\end{rem}

\section{Testing for isotropy  } \label{sec_iso_test}
\renewcommand{\theequation}{\thesection.\arabic{equation}}
\setcounter{equation}{0}

In this section, we define a  consistent estimate of the distance measure $M_d$ in \eqref{eq:Md} and prove its asymptotic normality. We will estimate the two terms  $D_{1,d}$ and $D_{2,d}$ defined  in Lemma \ref{two_integrals} and establish weak convergence of the corresponding estimates  (after appropriate normalization) to centered normal distributions for $d\leq 3$ ($D_{1,d}$)
and  $d \leq 2$ ($D_{2,d}$), respectively. 
These  results, which will be proved  in the Appendix,  require several technical assumptions, which are stated first. Following  \cite{subbarao2017} we define for 
given  constants $\delta>0$,  $C>0$,  a function   
$$\beta_\delta(\bm{\omega})=\prod_{i=1}^d \beta_\delta(\omega_i)~,  ~~~~
\bm{\omega}=(\omega_1,\ldots,\omega_d)^T\in\R^d ,
$$ 
where
\begin{align} \label{beta_fct}
\beta_\delta(\omega)=\begin{cases}
C, \quad &|\omega|\leq 1,\\
C \left|\omega\right|^{-\delta}, \quad & |\omega|>1
\end{cases}
\end{align}

\begin{assumption} \label{assumption_on_Z}
{\rm The random field 
$Z=\{Z(\bm{s}): \bm{s}\in\R^d\}$ is a second-order stationary, mean-square continuous centered  Gaussian process with  covariance function
$ c(  \bm{h} ) = 
\Cov\left[Z(\bm{s}),Z(\bm{s} + \bm{h} )\right]$ 
such that 
\begin{enumerate}
\item[(i)] $c$ is uniformly bounded and satisfies for some $\varepsilon>0$
\begin{align*} 
c(\bm{h})=\Landau\big (\|\bm{h}\|^{-(2+\varepsilon)}\big ) \quad \text{ as } \|\bm{h}\|\rightarrow \infty.
\end{align*} 
\item[(ii)] the spectral density of  $Z$ fulfills $f(\bm{\omega})\leq \beta_{1+\delta}(\bm{\omega})$ for some $\delta>2$ and is two times differentiable with partial derivatives satisfying 
\begin{align*}
\Big |\frac{\partial f(\bm{\omega})}{\partial \omega_i}\Big |\leq  \beta_{1+\delta}(\bm{\omega}) \qquad \text{and} \qquad \Big |\frac{\partial^2 f(\bm{\omega})}{\partial \omega_i \partial \omega_j}\Big |\leq  \beta_{1+\delta}(\bm{\omega}) 
\end{align*}
for $i,j=1,\ldots,d$. 
\end{enumerate}
}
\end{assumption}
The final assumption in this section concerns the relationship between the  parameters $\lambda$ and  $a$ in \eqref{ad_hoc_1} and \eqref{D_22_ad-hoc}.

\begin{assumption} ~\ \label{assumptions_on_a}
Let $a, \lambda, n\rightarrow \infty$. Then
\begin{itemize}
\item[(i)] $
\frac{a}{\lambda}\rightarrow \infty$
 and $\frac{a^d}{n}=\,\Landau(1)$
\item[(ii)]$ \frac{(\log a)^{2d}}{\lambda}\rightarrow 0$
\item[(iii)] $\frac{a^d}{\lambda^{d+1}}\to 0.$
\end{itemize}
\end{assumption}

\begin{example} \label{ex_matern}  
{\rm 
To illustrate Assumption \ref{assumption_on_Z}, consider the Mat\'{e}rn covariance function  
\begin{align*}
c_{\nu,\ell}(\bm{h})=\frac{2^{1-\nu}}{\Gamma(\nu)} \Big (\frac{\sqrt{2\nu} \|\bm{h}\|}{\ell}\Big)^\nu K_{\nu}\Big (\frac{\sqrt{2\nu}  \|\bm{h}\|}{\ell}\Big ),
\end{align*}
where $\Gamma$ is the Gamma function and $K_\nu$ denotes a  modified Bessel function \citep[see e.g.,][]{rasmussen06}.  
The  spectral density of the Mat\'{e}rn  covariance functions is given by 
\begin{align*}
f_{\nu,\ell}(\bm{\omega})=\frac{2^d \pi^{d/2} \Gamma(\nu+d/2) (2\nu)^\nu}{\Gamma(\nu) \ell^{2\nu}} \Big (\frac{2\nu}{\ell^2} + 4\pi^2 \|\bm{\omega}\|^2\Big )^{-(\nu+d/2)},
\end{align*}
which is  infinitely differentiable and a straightforward calculation shows  that 
Assumption \ref{assumption_on_Z} (ii)  is  satisfied if  $\nu>d$. Details are left to the reader.
}
\end{example}

Recalling the definition of the shifted Fourier frequencies in
\eqref{shifted_Four_freq} we estimate the term $D_{1,d} = \int_{\R^d}f^2(\bm{\omega}) d\bm{\omega} $ by 
\begin{align} \label{est_F}
\hat{D}_{1,d,\lambda,a}:=\frac{(2\pi\lambda)^d}{2n^4 H_{d,h}(\bm{0})^2} \sum_{\bm{k}=-a}^{a-1}  \sum_{(j_1,j_2,j_3,j_4)\in \mathcal{E}} & 
\prod_{k=1}^4  
h\left(\frac{\bm{s}_{j_k}}{\lambda}\right)  Z(\bm{s}_{j_k})\exp\big(\im(\bm{s}_{j_1}-\bm{s}_{j_2}+\bm{s}_{j_3} - \bm{s}_{j_4})^T \tilde{\bm{\omega}}_{\bm{k},\lambda}\big),
\end{align}
where $\sum_{\bm{k}=-a}^{a-1}$ denotes the multiple sum $\sum_{k_1=-a}^{a-1} \ldots \sum_{k_d=-a}^{a-1}$
and 
\begin{align} \label{set_E}
\mathcal{E}:=\{(j_1,j_2,j_3,j_4)\in\{1,\ldots,n\}^4: j_1,j_2,j_3,j_4 \text{ pairwise different}\}.
\end{align}
This estimate is motivated by the fact that that the spatial periodogram $I_{n,\lambda,d,h}$ in \eqref{periodogram}
 is asymptotically unbiased for $f(\bm{\omega})$. The 
 summation of the squares $I_{n,\lambda,d,h}^2 ( \tilde{\bm{\omega}}_{\bm{k},\lambda})$
then  yields a consistent estimate of the integral of the squared spectral density. Note that we evaluate the periodogram at the shifted Fourier frequencies \eqref{shifted_Four_freq} and   we only take the sum over the pairwise different elements in  \eqref{est_F}, which will not change the asymptotic results. Moreover, we emphasize that the 
statistic $\hat{D}_{1,d,\lambda,a}$ is real-valued, which 
follows from the facts  $\tilde{\omega}_{-k,\lambda}=-\tilde{\omega}_{k-1,\lambda}$ and 
\begin{align*}
\sum_{k=-a}^{a-1} \sin(s\, \tilde{\omega}_{k,\lambda}) 
= - \sum_{k=1}^a \sin(s\, \tilde{\omega}_{k-1,\lambda}) + \sum_{k=1}^a \sin(s\, \tilde{\omega}_{k-1,\lambda}) = 0.
\end{align*}
for any $s\in\R$. The following result establishes the asymptotic normality of $\hat{D}_{1,d,\lambda,a}$  (after scaling).

\begin{theorem} \label{corr_first_int}
If $d\leq 3$, Assumption \ref{assumption_on_sampling_scheme},  
 \ref{ass_on_h}, \ref{assumption_on_Z},  
Assumption \ref{assumptions_on_a}(i),(ii) are satisfied and 
the condition
\begin{align}  \label{further_ass_int_1}
\frac{\lambda^{d/2+1+2\delta}}{a^{1+2\delta}}\rightarrow 0.
\end{align}
holds,   then 
\begin{align*}
\frac{\lambda^{d/2}}{\tau_{1,d,\lambda,a}} \big(\hat{D}_{1,d,\lambda,a}-D_{1,d}\big) \dn \mathcal{N}(0,1) \qquad \text{as } a, \lambda, n \to \infty,
\end{align*}
where 
\begin{align}
\label{det3}
\tau_{1,d,\lambda,a}^2:=(2\pi)^d \sum_{\bm{m}=-2a+1}^{2a-1} \Big ( \frac{8\, H_{d,h}(\bm{m})^2}{H_{d,h}(\bm{0})^2}+\frac{2\, H_{d,h}(\bm{m})^4}{H_{d,h}(\bm{0})^4}\Big ) \int_{2\pi \max(-a,-a-\bm{m})/\lambda}^{2\pi\min(a,a-\bm{m})/\lambda}    f^4(\bm{\omega})\, d\bm{\omega}.
\end{align}.
\end{theorem}

Note that the rectangular taper function 
does not satisfy Assumption \ref{ass_on_h}. In fact, for this taper  Theorem \ref{corr_first_int} only holds  for dimension $d=1$ as the  bias of the test statistic $\hat{D}_{1,d,\lambda,a}$ is of order $\Landau \big ( \frac{\log \lambda}{\lambda}  + \frac{1}{n} + (\frac{\lambda}{a})^{1+2\delta} \big )$,  see Remark \ref{rem_edge_effect} and the proof of Theorem \ref{expectation_theo}  and Proposition \ref{approx_D1} in the Appendix.  As $\tau_{1,d,\lambda,a}^2$ is of order $\Landau (1)$
(see Proposition \ref{asymptotic_variances} below)
this is of higher order than the order of the  standard deviation $\lambda^{d/2} $ for $d\geq 2$.

We now continue proving a similar statement for an
estimate of the quantity $D_{2,2}$ in Lemma \ref{two_integrals}. To be  precise, we consider  
\begin{align} \label{est_F_2}
\hat{D}_{2,2,\lambda,a} := \frac{1}{\lambda} \sum_{r=0}^{a -1} \tilde{\omega}_{r,\lambda} &\Big[\left(\frac{2\pi}{\lambda}\right)^4 \sum_{\bm{k}=-a}^{a-1} \sum_{\bm{\ell}=-a}^{a-1}  \frac{\lambda^4}{n^4 H_{2,h}(\bm{0})^2} \sum_{(j_1,j_2,j_3,j_4)\in\mathcal{E}} h\left(\frac{\bm{s}_{j_1}}{\lambda}\right)h\left(\frac{\bm{s}_{j_2}}{\lambda}\right) h\left(\frac{\bm{s}_{j_3}}{\lambda}\right)  h\left(\frac{\bm{s}_{j_4}}{\lambda}\right)\nonumber\\
&\phantom{=} 
\times Z(\bm{s}_{j_1}) Z(\bm{s}_{j_2}) Z(\bm{s}_{j_3}) Z(\bm{s}_{j_4})\exp\big(\im(\bm{s}_{j_1}-\bm{s}_{j_2})^T \tilde{\bm{\omega}}_{\bm{k},\lambda}\big) \exp\big(\im(\bm{s}_{j_3}-\bm{s}_{j_4})^T \tilde{\bm{\omega}}_{\bm{\ell},\lambda}\big)\nonumber\\
&\phantom{=} \times
J_0(\tilde{\omega}_{r,\lambda} \|\tilde{\bm{\omega}}_{\bm{k},\lambda}\|) J_0(\tilde{\omega}_{r,\lambda} \|\tilde{\bm{\omega}}_{\bm{\ell},\lambda}\|) \Big],
\end{align}
where the set $\mathcal{E}$ is defined in \eqref{set_E}.
Note that $\hat{D}_{2,2,\lambda,a}$ is real valued, because we have for arbitrary $\bm{s}\in\R^2$ and $r\in\R$ 
(note we write $\tilde{\bm{\omega}}_{\bm{k},\lambda}=\tilde{\bm{\omega}}_{\bm{k}}$)
\begin{align*}
 \sum_{\bm{k}=-a}^{a-1} \exp(\im \bm{s}^T \tilde{\bm{\omega}}_{\bm{k,\lambda}}) J_0(r\|\tilde{\bm{\omega}}_{\bm{k},\lambda}\|) 
& =\sum_{k_1,k_2=-a}^{a-1} \big[\cos(s_1\, \tilde{\omega}_{k_1,\lambda}) \cos(s_2\, \tilde{\omega}_{k_2,\lambda})  \\
&  ~~~~~~~~~~~~~~~~~~~~~~~~~~~~~
- \sin(s_1\, \tilde{\omega}_{k_1,\lambda}) \sin(s_2\, \tilde{\omega}_{k_2,\lambda})\big] J_0(r\|\tilde{\bm{\omega}}_{\bm{k},\lambda}\|)\\
&+\im \sum_{k_1,k_2=-a}^{a-1} \big[\cos(s_1 \, \tilde{\omega}_{k_1,\lambda}) \sin(s_2\, \tilde{\omega}_{k_2,\lambda}) \\
&  ~~~~~~~~~~~~~~~~~~~~~~~~~~~~~
+ \sin(s_1 \, \tilde{\omega}_{k_1,\lambda}) \cos(s_2\, \tilde{\omega}_{k_2,\lambda}) \big]  J_0(r\|\tilde{\bm{\omega}}_{\bm{k},\lambda}\|),
\end{align*}
and using $\tilde{\omega}_{-k}=-\tilde{\omega}_{k-1}$, it is easy to see that the imaginary part vanishes.

\begin{theorem} \label{corr_sec_int}
If  $d=2$ and Assumption \ref{ass_on_h}, \ref{assumption_on_Z},  \ref{assumption_on_sampling_scheme} and \ref{assumptions_on_a}  and the condition 
\begin{align} \label{further_assumption_2}
\frac{\lambda^{3+2\varepsilon}}{a^{2+2\varepsilon}}+\frac{\lambda^{2\delta-1}}{a^{2\delta-2}} + \frac{\lambda^{1+\delta}}{a^\delta}=o(1)
\end{align}
are satisfied, we have
\begin{align*}
\frac{\lambda}{\tau_{2,2,\lambda,a}} \big(\hat{D}_{2,2,\lambda,a}-D_{2,2}\big) \dn \mathcal{N}(0,1) \qquad \text{as } a, \lambda, n \to \infty, 
\end{align*}
where 
\begin{align} \nonumber 
\tau_{2,2,\lambda,a}^2&:=8 \sum_{\bm{m}=-2a+1}^{2a-1} \frac{H_{2,h}(\bm{m})^2}{H_{2,h}(\bm{0})^2} \int_{2\pi \max(-a,-a-\bm{m})/\lambda}^{2\pi\min(a,a-\bm{m})/\lambda} f^2(\bm{z})\\
&\phantom{=========} \Big(\int_{0}^{2\pi a /\lambda} \Big[\int_{[-2\pi a/\lambda,2\pi a/\lambda]^2} f(\bm{x})  J_0(r\|\bm{x}\|) J_0(r \|\bm{z}\|)\, d\bm{x}\Big] \, r\, dr\Big)^2  \, d\bm{z}.
\label{det6}
\end{align}.
\end{theorem}

The main result of this section refers to the estimator 
\begin{align}
    \label{hol1}
\hat{M}_{2,\lambda,a}:=\hat{D}_{1,2,\lambda,a}-\hat{D}_{2,2,\lambda,a}
\end{align}
of $M_2=D_{1,2}-D_{2,2}$ and shows  shows that  a normalized version of $\hat{M}_{2,\lambda,a} - M_2$
converges weakly to a standard normal distribution.

\begin{theorem} \label{asymptotic_normality_M}
If the  assumptions of  Theorem \ref{asymptotic_normality_secint} are satisfied and conditions \eqref{further_ass_int_1} and \eqref{further_assumption_2} hold for $d=2$, then
\begin{align*}
\frac{\lambda}{\tau_{\lambda,a}}\big (\hat{M}_{2,\lambda,a}-M_2\big )\dn \mathcal{N}(0,1) \qquad \text{as } a,\lambda,n\rightarrow\infty,
\end{align*}
where the normalizing factor is given by
\begin{align} \label{det1}
\tau_{\lambda,a}^2:=\tau_{1,2,\lambda,a}^2+\tau_{2,2,\lambda,a}^2-2\, \kappa_{1,2,\lambda,a} ~
\end{align}
$\tau_{1,2,\lambda,a}^2$ and $\tau_{2,2,\lambda,a}^2$ are defined in \eqref{det3} and \eqref{det6}, respectively, and 
\begin{align} \label{covariance_of_M}
\kappa_{1,2,\lambda,a}&:=16\pi \sum_{\bm{m}=-2a+1}^{2a-1} \frac{H_{2,h}(\bm{m})^2}{H_{2,h}(\bm{0})^2} \int_{0}^{2\pi a/\lambda} r \Big[\int_{[-2\pi a/\lambda,2\pi a/\lambda]^2} f(\bm{y}) J_0(r\|\bm{y}\|)\, d\bm{y}\nonumber\\
&\phantom{========}  \times  \int_{2\pi\max(-a,-a-\bm{m})/\lambda}^{2\pi\min(a,a-\bm{m})/\lambda} f(\bm{x})^3 J_0(r\|\bm{x}\|) \,d\bm{x} \Big] \, dr.
\end{align}
\end{theorem}

We conclude this section with a statement regarding the asymptotic behaviour of the normalizing factors 
$\tau_{1,2,\lambda,a}^2$, $\tau_{2,2,\lambda,a}^2$, $\kappa_{1,2,\lambda,a}$  and 
a representation of the limit of $\tau^2_{\lambda , a} $ under the null hypothesis of isotropy, which is subsequently used to define a consistent asymptotic level $\alpha$ test for the hypothesis  \eqref{null}.

\begin{prop} \label{asymptotic_variances}
If  Assumption \ref{ass_on_h} and \ref{assumption_on_Z} are satisfied 
and   $a/\lambda\to\infty$ as $\lambda,a\to\infty$, then 
\begin{align}
\label{(i)}
\tau_1^2 &:=\lim_{\lambda,a\to\infty} \tau_{1,2,\lambda,a}^2 = (2\pi)^2 \sum_{\bm{m}=-\infty}^{\infty} \Big ( \frac{8\, H_{2,h}(\bm{m})^2}{H_{2,h}(\bm{0})^2}+\frac{2\, H_{2,h}(\bm{m})^4}{H_{2,h}(\bm{0})^4}\Big ) \int_{\R^2}    f^4(\bm{\omega})\, d\bm{\omega}, \\
\label{(ii)}
\tau_2^2 &:=\lim_{\lambda,a\to\infty} \tau_{2,2,\lambda,a}^2 = 8 \sum_{\bm{m}=-\infty}^{\infty} \frac{H_{2,h}(\bm{m})^2}{H_{2,h}(\bm{0})^2} \int_{\R^2} f^2(\bm{z})\\
\nonumber 
&\phantom{=========}  \times \Big(\int_{0}^{\infty} \Big[\int_{\R^2} f(\bm{x})  J_0(r\|\bm{x}\|)  J_0(r \|\bm{z}\|)\, d\bm{x}\Big] \, r\, dr\Big)^2  \, d\bm{z}, \\
\label{(iii)}
\kappa_{1,2}&:=\lim_{\lambda,a\to\infty} \kappa_{1,2,\lambda,a} = 16\pi \sum_{\bm{m}=-\infty}^{\infty} \frac{H_{2,h}(\bm{m})^2}{H_{2,h}(\bm{0})^2} \int_{0}^{\infty} r \Big[\int_{\R^2} f(\bm{y}) J_0(r\|\bm{y}\|)\, d\bm{y}\\
\nonumber 
&\phantom{======================} \times 
\int_{\R^2} f(\bm{x})^3 J_0(r\|\bm{x}\|) \,d\bm{x} \Big] \, dr.
\end{align}
In particular, 
\begin{align} \label{asymptotic_variance}
 \tau^2:= \lim_{\lambda,a\to\infty}  \tau_{\lambda,a}^2 =
\tau_1^2 + \tau_2^2 - 2\kappa_{1,2},
\end{align}
and this expression simplifies  to 
\begin{align} \label{det7}
\tau_{\text{H}_0}^2= 2\,  (2\pi)^2 \sum_{\bm{m}=-\infty}^{\infty} \frac{H_{2,h}(\bm{m})^4}{H_{2,h}(\bm{0})^4} \int_{\R^2} f^4(\bm{\omega})\, d\bm{\omega}. 
\end{align}
under the null hypothesis of isotropy.
\end{prop}

By Theorem \ref{asymptotic_normality_M} and Proposition \ref{asymptotic_variances} we have 
$
\frac{\lambda}{\tau} (\hat{M}_{2,\lambda,a}-M_2 ) \dn \mathcal{N}(0,1) 
$
as $ a,\lambda,n\to\infty,$ where $\tau^2 $ is defined in 
\eqref{asymptotic_variance}, which simplifies to the expression $\tau_{\text{H}_0}^2$ in \eqref{det7} under the null hypothesis of isotropy. Therefore,  an asymptotic level $\alpha$ test for the hypotheses of isotropy ($M_2=0$) is obatained by  by rejecting the null hypothesis  for large  
values of the statistic 
$
\frac{\lambda}{\hat{\tau}_{\text{H}_0,\lambda,a}} \,\hat{M}_{2,\lambda,a}~,
$
where 
$\hat{\tau}^2_{\text{H}_0,\lambda,a} $ is an appropriate estimator of the asymptotic variance  $\tau_{\text{H}_0}^2$ under the null hypothesis of isotropy.  The final result of this section defines such an estimator and establishes its consistency.
For its formulation recall the definition 
$I_{n,\lambda,d,h}(\bm{\omega})$ of  the tapered spatial periodogram in  \eqref{periodogram}. 

\begin{theorem} \label{estimator_variance_H0}
Let Assumption \ref{ass_on_h}, Assumption \ref{assumption_on_Z}, Assumption \ref{assumption_on_sampling_scheme}, and Assumption \ref{assumptions_on_a}, part (i), hold true. Then, it follows that
\begin{align} \label{det8}
\hat{F}_{\lambda,a}:= \frac{1}{24} \Big (\frac{2\pi}{\lambda}\Big )^2 \sum_{\bm{k}=-a}^{a-1} I_{n,\lambda,2,h}(\tilde{\bm{\omega}}_{\bm{k},\lambda})^4 \prob \int_{\R^2} f^4(\bm{\omega})\, d\bm{\omega} \qquad \text{as } a,\lambda,n\to\infty.
\end{align}
In particular, under the null hypothesis of isotropy, the statistic
\begin{align} \label{var_estimate}
\hat{\tau}_{\text{H}_0,\lambda,a}^2 = 2\, (2\pi)^2 \hat{F}_{\lambda,a} \sum_{\bm{m}=-a}^{a-1} \frac{ H_{2,h}(\bm{m})^4}{H_{2,h}(\bm{0})^4}
\end{align}
is a consistent estimator of ${\tau}_{\text{H}_0}^2$.
\end{theorem}

The above results give rise to an asymptotic level $\alpha$-test for isotropy: We reject the null hypothesis whenever
\begin{align}
\label{test_stat_finite_sample_prop}
\frac{\lambda}{\hat{\tau}_{\text{H}_0,\lambda,a}} \,\hat{M}_{2,\lambda,a} > z_{1-\alpha},
\end{align}
where  $z_{1-\alpha}$ is  the $(1-\alpha)$-quantile of the standard normal distribution.

\section{Finite sample properties} \label{sec5}
\renewcommand{\theequation}{\thesection.\arabic{equation}}
\setcounter{equation}{0}

In this section, we present a  simulation study in which we investigate the finite sample performance of the test \eqref{test_stat_finite_sample_prop}. We start by introducing  slight modifications  of the estimators $
\hat{D}_{1,2,\lambda,a}$, $
\hat{D}_{2,2,\lambda,a}$, and $\hat{\tau}_{\text{H}_0,\lambda,a}^2$ in \eqref{est_F}, \eqref{est_F_2}, and \eqref{var_estimate}, respectively, which are easier to implement, and show that these modifications do  not change the asymptotic properties of the resulting test statistics.
Proofs of these results can be found in the Appendix.

For the statistic $\hat{D}_{1,2,\lambda,a}$, 
a computationally more tractable and asymptotically equivalent statistic is obtained by 
replacing the set  
$\mathcal{E}$ 
for the summation in \eqref{est_F} by the set  
\begin{align*}
\tilde{\mathcal{E}}:=\{(j_1,j_2,j_3,j_4)\in\{1,\ldots,n\}^4: j_1 \neq j_2, j_1 \neq j_4, j_2 \neq j_3, j_3\neq j_4\},
\end{align*}
which gives 
\begin{align*}
\hat{D}_{1,\lambda,a}^{\text{eff}}&:=\frac{(2\pi\lambda)^2}{2n^4 H_{2,h}(\bm{0})^2} \sum_{\bm{k}=-a}^{a-1} \sum_{(j_1,j_2,j_3,j_4)\in\tilde{\mathcal{E}}} h\Big (\frac{\bm{s}_{j_1}}{\lambda}\Big) \ldots h\Big (\frac{\bm{s}_{j_4}}{\lambda}\Big ) Z(\bm{s}_{j_1}) \ldots Z(\bm{s}_{j_4}) 
 \\ &\phantom{=================} \times
\exp\big(\im (\bm{s}_{j_1}-\bm{s}_{j_2}+\bm{s}_{j_3}-\bm{s}_{j_4})^T \tilde{\bm{\omega}}_{\bm{k},\lambda}\big).
\end{align*}
The next proposition shows how one can rewrite this estimator in such a way that eases its implementation. 
\begin{prop} \label{prop_umschreiben_D1}
It holds that
\begin{align*}
\hat{D}_{1,\lambda,a}^{\mathrm{eff}}& = \frac{(2\pi\lambda)^2}{2n^4 H_{2,h}(\bm{0})^2} \sum_{\bm{k}=-a}^{a-1} \Big \{  \Big (\Big |\sum_{j=1}^n h\Big (\frac{\bm{s}_j}{\lambda}\Big ) Z(\bm{s}_j) \exp(\im \bm{s}_j^T \tilde{\bm{\omega}}_{\bm{k},\lambda})\Big |^2 - \sum_{j=1}^n h^2\big(\frac{\bm{s}_j}{\lambda}\big) Z^2(\bm{s}_j)\Big )^2 \\
&\phantom{====}- \sum_{j_1=1}^n h^2\Big (\frac{\bm{s}_{j_1}}{\lambda}\Big ) Z^2(\bm{s}_{j_1}) \Big | \sum_{j\neq j_1} h\big(\frac{\bm{s}_j}{\lambda}\big) Z(\bm{s}_j) \exp(\im \bm{s}_j^T \tilde{\bm{\omega}}_{\bm{k},\lambda}) \Big |^2\\
&\phantom{====}-\sum_{j_1=1}^n h^2\Big (\frac{\bm{s}_{j_1}}{\lambda}\Big ) Z^2(\bm{s}_{j_1}) \Big (\Big |\sum_{j\neq j_1} h\big(\frac{\bm{s}_j}{\lambda}\big) Z(\bm{s}_j)\exp(\im \bm{s}_j^T \tilde{\bm{\omega}}_{\bm{k},\lambda})\Big |^2 - \sum_{j\neq j_1} h^2\big(\frac{\bm{s}_j}{\lambda}\big) Z^2(\bm{s}_j)\Big ) \Big \} .
\end{align*}
\end{prop}

 The following result then justifies using the estimator $\hat{D}_{1,\lambda,a}^{\text{eff}}$ instead of $D_{1,2,\lambda,a}$.

\begin{prop} \label{prop_eff_D1}
Under the conditions of Theorem \ref{corr_first_int}  (with $d=2$), we have
\begin{align*}
\frac{\lambda}{\tau_{1,2,\lambda,a}} \big (\hat{D}_{1,\lambda,a}^{\mathrm{eff}}- D_{1,2 } \big ) \dn \mathcal{N}(0,1) \qquad \text{as } a,\lambda,n\rightarrow \infty,
\end{align*}
where $D_{1,2}$ 
 and $
\tau_{1,2,\lambda,a}$ are  defined in
\eqref{D_1} and  \eqref{det3}, respcetively.
\end{prop}

Similarly, we can obtain a  computationally more efficient estimator of $\hat{D}_{2,\lambda,a}$ by replacing it with 
\begin{align*}
\hat{D}_{2,\lambda,a}^{\text{eff}}&:=\frac{1}{\lambda} \sum_{r=0}^{a-1} \tilde{\omega}_{r,\lambda} \Big[\Big (\frac{2\pi}{\lambda}\Big )^4 \sum_{\bm{k}=-a}^{a-1} \sum_{\bm{\ell}=-a}^{a-1} \frac{\lambda^4}{n^4 H_{2,h}(\bm{0})^2} \sum_{(j_1,j_2,j_3,j_4)\in \tilde{\tilde{\mathcal{E}}}} h\Big (\frac{\bm{s}_{j_1}}{\lambda}\Big )\ldots h\Big (\frac{\bm{s}_{j_4}}{\lambda}\Big ) Z(\bm{s}_{j_1}) \ldots Z(\bm{s}_{j_4})  \\
& ~~~~~~~~~~~~~~~  \times  \exp\big(\im (\bm{s}_{j_1}-\bm{s}_{j_2})^T \tilde{\bm{\omega}}_{\bm{k},\lambda}\big) \exp\big(\im(\bm{s}_{j_3}-\bm{s}_{j_4})^T \tilde{\bm{\omega}}_{\bm{\ell},\lambda}\big) J_0(\tilde{\omega}_{r,\lambda} \|\tilde{\bm{\omega}}_{\bm{k},\lambda}\|) J_0(\tilde{\omega}_{r,\lambda} \|\tilde{\bm{\omega}}_{\bm{\ell},\lambda}\|) \Big],
\end{align*}
where
\begin{align*}
\tilde{\tilde{\mathcal{E}}}:= \{(j_1,j_2,j_3,j_4)\in\{1,\ldots,n\}^4\, : \, j_1\neq j_2, j_3\neq j_4\}.
\end{align*}
In this case, it is easily seen that an efficient implementation simply requires one to note that we can write 
\begin{align}
\hat{D}_{2,\lambda,a}^{\text{eff}}&=
\label{alternative_Darstellung_D2_hat}
 \frac{(2\pi)^4}{\lambda} \sum_{r=0}^{a-1} \tilde{\omega}_{r,\lambda} \, \hat{c_0}^2(\tilde{\omega}_{r,\lambda}),
\end{align}
where 
\begin{align}
\hat{c_0} (\tilde{\omega}_{r,\lambda}) & = \sum_{\bm{k}=-a}^{a-1} \frac{1}{n^2 H_{2,h}(\bm{0})} \sum_{\substack{j_1,j_2=1 \\ j_1\neq j_2}}^n h\Big (\frac{\bm{s}_{j_1}}{\lambda}\Big ) Z(\bm{s}_{j_1})\, h\Big (\frac{\bm{s}_{j_2}}{\lambda}\Big ) Z(\bm{s}_{j_2}) \\
&   ~~~~~~~~~~~~~~~~~~~~~~~~~~~~~~~~~~~~~
\times \exp\big(\im (\bm{s}_{j_1}-\bm{s}_{j_2})^T \tilde{\bm{\omega}}_{\bm{k},\lambda}\big) J_0(\tilde{\omega}_{r,\lambda} \|\tilde{\bm{\omega}}_{\bm{k},\lambda}\|) ~. 
\label{det100}
\end{align}
 The next result establishes that $\hat{D}_{2,\lambda,a}^{\mathrm{eff}}$ has the same limiting distributional properties as  $\hat{D}_{2,\lambda,a}$, and thus that we can use the former to replace the latter.
\begin{prop} \label{prop_eff_D2}
Under the assumptions of Theorem \ref{corr_sec_int} we have
\begin{align*}
\frac{\lambda}{\tau_{2,2,\lambda,a}} \big (\hat{D}_{2,\lambda,a}^{\mathrm{eff}}- D_{2,2}\big ) \dn \mathcal{N}(0,1) \qquad \text{as } a,\lambda,n\rightarrow \infty,
\end{align*}
where $
D_{2,2,}$ and $
\tau_{2,2,\lambda,a}$ are defined in \eqref{D_22} and \eqref{det6}, respectively.
\end{prop}

Finally, to define a computationally efficient version of the test statistic \eqref{test_stat_finite_sample_prop},
we note that,
by definition of the tapered periodogram $I_{n,\lambda,2,h}(\bm{\omega})$ in \eqref{periodogram}, the variance estimator in \eqref{var_estimate} can be represented as 
\begin{align*}
\hat{\tau}_{\text{H}_0,\lambda,a}^2&=\frac{(2\pi)^4}{12\lambda^2} \sum_{\bm{m}=-a}^{a-1} \frac{H_{2,h}(\bm{m})^4}{H_{2,h}(\bm{0})^4} \sum_{\bm{k}=-a}^{a-1} \Big (\frac{\lambda^2}{n^2 H_{2,h}(\bm{0})} \Big|\sum_{j=1}^n h\big(\frac{\bm{s}_j}{\lambda}\big) Z(\bm{s}_j) \exp(\im \bm{s}_j^T \tilde{\bm{\omega}}_{\bm{k},\lambda})\Big|^2\Big)^4\\
&=\frac{(2\pi)^4 \lambda^6}{12n^8 H_{2,h}(\bm{0})^8} \sum_{\bm{m}=-a}^{a-1} H_{2,h}(\bm{m})^4 \sum_{\bm{k}=-a}^{a-1} \Big|\sum_{j=1}^n h\big(\frac{\bm{s}_j}{\lambda}\big) Z(\bm{s}_j) \exp(\im \bm{s}_j^T \tilde{\bm{\omega}}_{\bm{k},\lambda})\Big|^8.
\end{align*}
Observe that, in this case,    constraints on the combination of summands $j_1,\ldots,j_8$ are not necessary since the estimator $\hat{\tau}_{\text{H}_0,\lambda,a}^2$ only needs to be consistent and thus the bias is only required to be  asymptotically negligible. 
However, to improve finite sample performance it is generally   beneficial to implement a bias-corrected version of $\hat{\tau}_{\text{H}_0,\lambda,a}^2$. To this end, we make use of calculations that were already required  for the analysis of  $\hat{D}_{1,\lambda,a}^{\text{eff}}$ noting that  every term in $\hat{\tau}_{\text{H}_0,\lambda,a}^2$ with $j_k=j_l$ for some $k\neq l$ is of lower order $\Landau(\lambda^2/n)$ (see the proof of Theorem \ref{estimator_variance_H0}). Hence,  we can estimate $\tau_{\text{H}_0}^2$ consistently by the bias-corrected version 
\begin{align*}
&(\hat{\tau}_{\text{H}_0,\lambda,a}^{\text{bias-corr}})^2:=\frac{(2\pi)^4 \lambda^6}{12 n^8 H_{2,h}(\bm{0})^8} \sum_{\bm{m}=-a}^{a-1} H_{2,h}^4(\bm{m}) \sum_{\bm{k}=-a}^{a-1} \Big[\sum_{(j_1,\ldots,j_4)\in\tilde{\mathcal{E}}} h\big(\frac{\bm{s}_{j_1}}{\lambda}\big) \ldots h\big(\frac{\bm{s}_{j_4}}{\lambda}\big) Z(\bm{s}_{j_1}) \ldots Z(\bm{s}_{j_4}) \\
&\phantom{=========================} \times \exp\big(\im (\bm{s}_{j_1}-\bm{s}_{j_2}+\bm{s}_{j_3}-\bm{s}_{j_4})^T \tilde{\bm{\omega}}_{\bm{k},\lambda}\big)\Big]^2\\
&=\frac{(2\pi)^4 \lambda^6}{12n^8 H_{2,h}(\bm{0})^8} \sum_{\bm{m}=-a}^{a-1} H_{2,h}^4(\bm{m}) \sum_{\bm{k}=-a}^{a-1} \Big[ \Big(\Big  |\sum_{j=1}^n h\big(\frac{\bm{s}_j}{\lambda}\big) Z(\bm{s}_j) \exp(\im \bm{s}_j^T \tilde{\bm{\omega}}_{\bm{k},\lambda})\Big  |^2 - \sum_{j=1}^n h^2\big(\frac{\bm{s}_j}{\lambda}\big) Z^2(\bm{s}_j)\Big)^2 \\
&\phantom{======}- \sum_{j_1=1}^n h^2\big(\frac{\bm{s}_{j_1}}{\lambda}\big) Z^2(\bm{s}_{j_1}) \Big  | \sum_{j\neq j_1} h\big(\frac{\bm{s}_j}{\lambda}\big) Z(\bm{s}_j) \exp(\im \bm{s}_j^T \tilde{\bm{\omega}}_{\bm{k},\lambda}) \Big |^2\\
&\phantom{======}-\sum_{j_1=1}^n h^2\big(\frac{\bm{s}_{j_1}}{\lambda}\big) Z^2(\bm{s}_{j_1}) \Big(\Big  |\sum_{j\neq j_1} h\big(\frac{\bm{s}_j}{\lambda}\big) Z(\bm{s}_j)\exp(\im \bm{s}_j^T \tilde{\bm{\omega}}_{\bm{k},\lambda})\Big  |^2 - \sum_{j\neq j_1} h^2\big(\frac{\bm{s}_j}{\lambda}\big) Z^2(\bm{s}_j)\Big) \Big]^2.
\end{align*}
Indeed, note that this estimator has reduced bias compared to $\tau_{\text{H}_0}^2$ since we are suppressing effectively the bias terms of highest order by restricting the summands in this estimator to $\tilde{\mathcal{E}}$ (see also the comment below equation \eqref{eq:cumJ2}).

Using similar arguments as given in the proofs of Proposition \ref{prop_eff_D1}  and \ref{prop_eff_D2} we obtain (under the null hypothesis)
\begin{align} \label{modified_test_statistic}
\hat{T}_{2,\lambda,a}^{\rm eff} = 
\frac{\lambda}{\hat{\tau}_{\text{H}_0,\lambda,a}^{\text{bias-corr}}} \big (\hat{D}_{1,\lambda,a}^{\text{eff}}-\hat{D}_{2,\lambda,a}^{\text{eff}}\big )  \dn \mathcal{N}(0,1) \qquad \text{as } a,\lambda,n\rightarrow \infty ~. 
\end{align}
Therefore we propose to reject the null hypothesis, whenever
\begin{align}
\label{hol10}
\hat{T}_{2,\lambda,a}^{\rm eff} > z_{1-\alpha},
\end{align}
where $z_{1-\alpha}$ is the $(1-\alpha)$-quantile of the standard normal distribution. We illustrate the properties of this test by means of simulation study for a Gaussian process with covariance 
\begin{align}
\label{anisotropic_model}
c(\boldsymbol{h})=\text{exp}\big (- 4 \|A_r\boldsymbol{h}\|^2\big )~,  
\end{align}
where  the matrix $A_r$ is given by 
\begin{align*}
A_r=\begin{pmatrix}
1 & 0\\
0 & 1/r
\end{pmatrix}
\begin{pmatrix}
\cos(\pi/4) & \sin(\pi/4) \\
-\sin(\pi/4) & \cos(\pi/4)
\end{pmatrix} ~, 
\end{align*}
and  $r \geq  1$. Note that the case $r=1$ corresponds to the null hypothesis of isotropy, that is $c (\bm{h})=\exp (- 4 {\|\bm{h}\|^2}  )$. An increasing value of $r$ leads to a larger deviation from  isotropy. The locations $\bm{s}_1,\ldots,\bm{s}_n$ in the test statistic \eqref{modified_test_statistic} are sampled from a uniform distribution on the spatial domain $[-\lambda/2,\lambda/2]^2$.
\\
For the  estimators  we consider the
cosine window with  parameter $\alpha=3$ as taper function, which gives $ H_{2,h}(\bm{0})=\left(\frac{5}{16}\right)^2 $.
While the implementation of $D_{1,2}$
is straightforward, the estimation of  $D_{2,2}$ is more intricate.
Note that under the null hypothesis we have 
  We choose $\lambda=30$ and $a=80$ 
for the $\tilde{\bm{\omega}}_{\bm{k},\lambda}$ in \eqref{det100}. However, to improve the approximation of the integral 
we use $\lambda=300$ and $a=800$
for $\tilde{\omega}_{r,\lambda} $
in \eqref{alternative_Darstellung_D2_hat}.
To avoid  an accumulation of noise, the summation  in \eqref{alternative_Darstellung_D2_hat} is truncated, whenever $\hat{c}_0(\tilde{\omega}_{r,\lambda})$ is  smaller than  zero for the first time.
\\
 \begin{table}[t]
      \centering
     \begin{tabular}{c|c|c|c|c|c|c}
       $n$  &  $1000$ & $2000$ & $5000$ & $10000$  \\
        \hline
       $r=1$           & $0.184$  & $0.088$ &  $0.056$   & $0.054$  \\
       \hline
       $r=2$  & $0.180 $  &  $0.134$ & $0.146$ & $0.180$  \\
       \hline
       $r=3$  & $0.276$ & $0.278$ & $0.450$ & $0.498$\\
       \hline
       $r=4$   & $0.362$ &  $0.482$ & $0.606$  &$0.666$ \\
       \hline
     \end{tabular}
     \caption{\it Empirical  rejection probabilities of the test \eqref{hol10} with level $\alpha= 5\%$. The random filed is a Gaussian process with covariance kernel \eqref{anisotropic_model}. 
     The case $r=1$ corresponds to the null hypothesis of isotropy and $r=2,3,4$ to the alternative of an anisotropic model. }
     \label{det101}
     \end{table}
     In Table \ref{det101} we present the simulated rejection probabilities of  the test \eqref{hol10} based on  $500$ simulation runs.
   We observe that for an increasing sample size the test keeps its nominal level $\alpha = 5\%$ (first line in Table \ref{det101}). Note that the required sample sizes for this property are rather large, but such sample sizes are necessary  as we are in fact   testing  the hypotheses of isotropy on the full spatial domain $\R^2$ and do not restrict this investigation to a finite (relatively small) number of spatial locations. The cases   $r=2,3,4$ in Table \ref{det101}  
   represent the alternative  of an anisotropic covariance function. 
We observe that the rejection probabilities are increasing with $r$
(more deviation from isotropy) and  increasing  with the sample size $n$.

\bigskip\noindent
\textbf{Acknowledgements}  
The work of H. Dette and T. Eckle  has been partially supported  by the German Research Foundation (DFG), project number 45723897. Anne van Delft has been partially supported by NSF grant DMS-2311338.


\bibliography{Bibliography}

\newpage

\appendix
\addcontentsline{toc}{section}{Appendices}
\renewcommand{\theequation}{\thesection.\arabic{equation}}
\setcounter{equation}{0}

\section{Proofs of results in Section \ref{sec3}}
For the ease of notation, we omit the indices $\lambda$, $d$ and $h$ for the frequency window and write $B_{\lambda,d,h}(\bm{u})=B(\bm{u})$. Furthermore, we denote $\bm{\omega}_{\bm{m},\lambda}=\bm{\omega_m}$. 
A first simple result in this respect is the convolution of the window with itself. 

\begin{lemma} \label{convolution_of_h}
Let the taper $h:\R^d\to\R^+$ be a bounded function satisfying Assumption \ref{ass_on_h}(i). Furthermore, define $B_{\lambda,d,h}$ as in \eqref{Four_trafo_of_h}.
For arbitrary $\bm{x}\in\R^d$, we have $$\int_{\R^d} B_{\lambda,d,h}(\bm{u}) B_{\lambda,d,h}(\bm{x}-\bm{u})\, d\bm{u}=(2\pi)^d \int_{[-\lambda/2,\lambda/2]^d} h^2\left(\frac{\bm{s}}{\lambda}\right) \exp(-\im \bm{s}^T \bm{x})\, d\bm{s}.$$ 
In particular, for $\bm{m}\in\Z^d$, $$\frac{1}{(2\pi\lambda)^d}\int_{\R^d} B_{\lambda,d,h}(\bm{u}) B_{\lambda,d,h}(\bm{\omega}_{\bm{m},\lambda}-\bm{u})  \, d\bm{u}= H_{d,h}(\bm{\bm{m}}),$$
with $H_{d,h}(\bm{m})$ and $\bm{\omega}_{\bm{m},\lambda}$ defined in \eqref{H(m)} and \eqref{Four_freq}, respectively.
\end{lemma}
\begin{proof}

First consider the case $d=1$ and let $x\in\R$ be arbitrary. Using the convolution theorem for Fourier transforms, we directly obtain
\begin{align*}
\int_{\R} B(u) B(x-u)\, du
&=\lambda^2 \int_{\R} \Big(\int_{\frac{1}{2}}^{\frac{1}{2}} h(t_1) \exp(-\im \lambda t_1 u )\, dt_1\Big) \Big(\int_{\frac{1}{2}}^{\frac{1}{2}} h(t_2) \exp(-\im \lambda t_2 (x-u))\, dt_2\Big) du\\
&=\lambda \int_{\R} (\mathcal{F} h)(z) (\mathcal{F}h)(\lambda x-z) \, dz\\ &=2\pi\lambda \, \mathcal{F}(h^2)(\lambda x)=2\pi \int_{-\lambda/2}^{\lambda/2} h^2\left(\frac{s}{\lambda}\right) \exp(-\im s x)\, ds.
\end{align*}
The statement can easily be generalized to the case where $d>1$ by noting that
\begin{align*}
\int_{\R^d} B(\bm{u}) B(\bm{x}-\bm{u})\, d\bm{u} &= \prod_{i=1}^d \Big(\int_{\R} B(u_i) B(x_i-u_i)\, du_i\Big).
\end{align*}
\end{proof}

\subsection{Proof of Lemma \ref{properties_of_ell}}

Throughout this proof, we denote by $C(\cdot)>0$, $\tilde{C}(\cdot)>0$ and $\tilde{\tilde{{C}}}(\cdot)>0$ some generic and bounded constants which only depend on the argument(s) inside their brackets.
The proof of part (i) works in large parts analogously to the proof of Lemma $2$ in \cite{dahlhaus83}. 
\begin{enumerate}
\item[(i)] 
We assume without loss of generality $p\geq q$. 
Recalling the definition of $L_{\lambda}^{(s)}$ in \eqref{L-fct}, note that 
\begin{align*}
\int_{-\infty}^{\infty} L_{\lambda}^{(s)}(u)^2\, du 
&=\frac{2s^{2s}}{\e^s} \lambda + 2\lambda \int_{\e^s}^{\infty} \frac{\log^{2s}(x)}{x^2}\, dx \leq C(s)\, \lambda.
\end{align*}
Using the Cauchy-Schwarz inequality, it then follows for all $v\in\R$ that
\begin{align*} 
\int_{-\infty}^{\infty} L_{\lambda}^{(p)}(u) L_{\lambda}^{(q)}(v-u)\, du & \leq \sqrt{\int_{-\infty}^{\infty} L_{\lambda}^{(p)}(u)^2\, du \,\int_{-\infty}^{\infty} L_{\lambda}^{(q)}(v-u)^2\, du }\leq C(p,q)\, \lambda.
\end{align*}
This obviously proves the assertion if $|v|\leq \e^{p+q+1}/\lambda$, and it remains to show that
\begin{align*}
\int_{-\infty}^{\infty} L_{\lambda}^{(p)}(u) L_{\lambda}^{(q)}(v-u)\, du & \leq C(p,q) \frac{\log^{p+q+1}(\lambda |v|)}{|v|}
\end{align*}
for all $v\in\R$ with $|v|>\e^{p+q+1}/\lambda$.
We assume without loss of generality that $v>\e^{p+q+1}/\lambda$, since the other case can be proven analogously.
In particular, it then holds that $v>2\e^{p}/\lambda$. We make the decomposition
\begin{align*}
\int_{-\infty}^{\infty} L_{\lambda}^{(p)}(u) L_{\lambda}^{(q)}(v-u)\, du &= I_1+I_2+I_3+I_4+I_5+I_6,
\end{align*}
where
\begin{align*}
I_1:= \int_{0}^{\e^p/\lambda} L_{\lambda}^{(p)}(u)\, L_{\lambda}^{(q)}(v-u)\, du, \qquad I_2:= \int_{\e^p/\lambda}^{v-\e^q/\lambda} L_{\lambda}^{(p)}(u) L_{\lambda}^{(q)}(v-u)\, du,
\end{align*}
\begin{align*}
I_3:=\int_{v-\e^q/\lambda}^{v+\e^q/\lambda} L_{\lambda}^{(p)}(u) L_{\lambda}^{(q)}(v-u)\, du, \qquad I_4:=\int_{v+\e^q/\lambda}^{\infty} L_{\lambda}^{(p)}(u) L_{\lambda}^{(q)}(v-u)\, du,
\end{align*}
\begin{align*}
I_5:=\int_{-\e^p/\lambda}^0 L_{\lambda}^{(p)}(u) L_{\lambda}^{(q)}(v-u)\, du, \qquad \text{and} \qquad I_6:=\int_{-\infty}^{-\e^p/\lambda} L_{\lambda}^{(p)}(u) L_{\lambda}^{(q)}(v-u)\, du.
\end{align*}
Since the function $x \mapsto\log^q(x)/x$ is monotonically decreasing on the interval $[\e^q,\infty)$ and $\lambda v/2 >\e^q$, it holds that
\begin{align*}
I_1=\left(\frac{p}{\e}\right)^p \lambda\int_{0}^{\e^p/\lambda} \frac{\log^{q}(\lambda (v-u))}{v-u}\, du &\leq 2\left(\frac{p}{\e}\right)^p \lambda \int_{0}^{\e^p/\lambda} \frac{\log^q(\lambda v/2)}{v} \, du \leq C(p) \frac{\log^{q}(\lambda v)}{v}.
\end{align*}
Moreover,
\begin{align*}
I_2&=\int_{\e^p/\lambda}^{v-\e^q/\lambda} \frac{\log^p(\lambda u) \log^q(\lambda (v-u))}{u(v-u)}\, du\\ &\leq \log^{p+q}(\lambda v) \int_{\e^p/\lambda}^{v-\e^q/\lambda} \frac{1}{u(v-u)} \, du=\frac{\log^{p+q}(\lambda v)}{v} \left[\log\left(\frac{\lambda v -\e^q}{\e^q}\right)+\log\left(\frac{\lambda v-\e^p}{\e^p}\right)\right]\\
&\phantom{===================}\leq C(p,q) \frac{\log^{p+q+1}(\lambda v)}{v}.
\end{align*}
Concerning the term $I_3$, note that
\begin{align*}
I_3=\left(\frac{q}{\e}\right)^q \, \lambda\, \int_{v-\e^q/\lambda}^{v+\e^q/\lambda} \frac{\log^p(\lambda u)}{u}\, du &\leq \left(\frac{q}{e}\right)^q \lambda \log^p(\lambda v + \e^q) \int_{v-\e^q/\lambda}^{v+\e^q/\lambda} \frac{1}{u}\, du\\
&\leq 2  q^q \log^p(\lambda v + \e^q)  \frac{\lambda}{\lambda v-\e^q}.
\end{align*}
Note that $\lambda v >2\e^p$. Therefore, $\lambda v-\e^q$ is bounded away from $0$ and it follows that
\begin{align*}
\frac{\lambda v}{\lambda v-\e^q}\leq C(q),
\end{align*}
which yields
$I_3 \leq C(p,q) \frac{\log^{p}(\lambda v)}{v}$.
Now consider the term $I_4$.
First let $p=q=0$, in which case
\begin{align} \label{result_for_p=q=0}
I_4=\int_{v+1/\lambda}^{\infty} \frac{1}{u(u-v)}\, du &=\left[\frac{1}{v}(\log(u-v)-\log(u))\right]_{u=v+1/\lambda}^{\infty} \nonumber\\
&=\frac{1}{v}\log\left(\frac{v+1/\lambda}{1/\lambda}\right)=\frac{\log(1+\lambda v)}{v} \leq 2\, \frac{\log(\lambda v)}{v}.
\end{align}
For $p+q\geq 1$, using integration by parts yields
\begin{align*}
I_4&=\int_{v+\e^q/\lambda}^{\infty} \frac{\log^p(\lambda u) \log^q(\lambda (u-v))}{u(u-v)}\, du \leq \int_{v+\e^q/\lambda}^{\infty} \frac{\log^{p+q}(\lambda u)}{u(u-v)}\, du\\
&= \frac{1}{v} \log^{p+q}(\lambda v + \e^q) \log\left(\frac{\lambda v + \e^q}{\e^q}\right) + (p+q) \int_{v+\e^q/\lambda}^{\infty} \frac{\log^{p+q-1}(\lambda u)}{uv} \log\left(\frac{u}{u-v}\right)\, du\\
&\leq C(q) \frac{\log^{p+q+1}(\lambda v)}{v} + (p+q) \int_{v+\e^q/\lambda}^{\infty} \frac{\log^{p+q-1}(\lambda u)}{u(u-v)}\, du.
\end{align*}
For the last step of the above calculation, we used 
\begin{align*}
\log\left(\frac{u/v}{u/v-1}\right)\leq \frac{1}{u/v-1},
\end{align*}
which holds true since $\log(x)\leq x-1$ for all $x>0$.
Applying the result for $p=q=0$ in \eqref{result_for_p=q=0} and induction on $p+q$ yields
\begin{align*}
I_4 \leq C(p,q) \frac{\log^{p+q+1}(\lambda v)}{v}.
\end{align*}
For the term $I_5$, note that since the function $x \mapsto\log^q(x)/x$ is monotonically decreasing on $[e^q,\infty)$, we get
\begin{align*}
I_5=\left(\frac{p}{\e}\right)^p \lambda\, \int_{-\e^p/\lambda}^0 \frac{\log^q(\lambda (v-u))}{v-u}\, du\leq \left(\frac{p}{\e}\right)^p \lambda \int_{-\e^p/\lambda}^{0} \frac{\log^q(\lambda v)}{v}\, du \leq C(p) \frac{\log^q(\lambda v)}{v}.
\end{align*}
Finally, we obtain
\begin{align*}
I_6=-\int_{-\infty}^{-\e^p/\lambda} \frac{\log^p(-\lambda u) \log^q(\lambda(v-u))}{u(v-u)}\, du 
&=\int_{v+\e^p/\lambda}^{\infty} \frac{\log^p(\lambda(u-v))\log^q(\lambda u)}{u(u-v)}\, du.
\end{align*}
Analogously to the term $I_4$, we now get $I_6\leq C(p,q) \frac{\log^{p+q+1}(\lambda v)}{v}$.
$\hfill \Box$\\

\item[(ii)] 
Let $p,q\in\N_0$ and $r\in\Z$ be arbitrary, w.l.o.g. $r\geq 0$. We distinguish the cases $r\leq \e^{p+q+1}$ and $r>\e^{p+q+1}$. If $r\leq \e^{p+q+1}$, we have 
\begin{align*}
\sum_{m=-\infty}^{\infty} \ell^{(p)}(m) \ell^{(q)}(m+r)&\leq C(p,q)+\sum_{\substack{|m|>\e^{p},\\ |m+r|>\e^q}} \ell^{(p)}(m) \ell^{(q)}(m+r)\\
&=C(p,q) + \sum_{m>\max\{\e^p,\e^q-r\}} \frac{\log^p(m)}{m} \frac{\log^q(m+r)}{m+r} \\
&\phantom{===iiiii} + \sum_{\e^q-r < m < -\e^p} \frac{\log^p(-m)}{-m} \frac{\log^q(m+r)}{m+r} \\
&\phantom{===iiiii} + \sum_{m<\min\{-\e^p,-\e^q-r\}} \frac{\log^p(-m)}{-m} \frac{\log^q(-m-r)}{-m-r} \\
&\leq \tilde{C}(p,q)+2\sum_{m=1}^{\infty} \frac{\log^{p+q}(m+\e^{p+q+1})}{m^2} \leq \tilde{\tilde{{C}}}(p,q) \left(\frac{p+q+1}{\e}\right)^{p+q+1}.
\end{align*}
Now let $r>\e^{p+q+1}$. Then, we have
\begin{align*}
\sum_{m=-\infty}^{\infty} \ell^{(p)}(m) \ell^{(q)}(m+r)&=A+B+C+D+E+F,
\end{align*}
where
\begin{align*}
A:=\sum_{|m|\leq \lfloor \e^p\rfloor} \ell^{(p)}(m) \ell^{(q)}(m+r), \qquad B:=\sum_{m=\lceil \e^p\rceil}^r \ell^{(p)}(m) \ell^{(q)}(m+r),
\end{align*}
\begin{align*}
C:=\sum_{m=r+1}^{\infty} \ell^{(p)}(m) \ell^{(q)}(m+r), \qquad
D:=\sum_{m=-\lfloor \e^q\rfloor -r}^{-\lceil \e^p \rceil} \ell^{(p)}(m) \ell^{(q)}(m+r),
\end{align*}
\begin{align*}
E:=\sum_{m=-2r}^{-\lceil \e^q\rceil - r} \ell^{(p)}(m) \ell^{(q)}(m+r), \qquad \text{and} \qquad F:=\sum_{m=-\infty}^{-2r-1} \ell^{(p)}(m) \ell^{(q)}(m+r).
\end{align*}
We consider these six terms separately and start with the term $A$. 
Using that the function $x \mapsto\log^q(x)/x$ is monotonically decreasing on $[\e^q,\infty)$, we obtain
\begin{align*}
A=\sum_{m=-\lfloor \e^p\rfloor}^{\lfloor \e^p\rfloor} \left(\frac{p}{\e}\right)^p \frac{\log^q(m+r)}{m+r}&\leq \left(\frac{p}{\e}\right)^p\left(\sum_{m=0}^{\lfloor \e^p\rfloor}  \frac{\log^q(r)}{r}+\sum_{m=-\lfloor \e^p\rfloor}^{-1} \frac{\log^q(r)}{m+r}\right)\\
&\leq \left(\frac{p}{\e}\right)^p \frac{\log^q(r)}{r}\left(\lfloor \e^p\rfloor+1+\lfloor\e^p\rfloor\cdot \frac{r}{r-\lfloor \e^p\rfloor}\right).
\end{align*}
Since the function $x \mapsto\frac{x}{x-\lfloor \e^p\rfloor}$ is monotonically decreasing on $(\lfloor \e^p\rfloor,\infty)$, it follows that
\begin{align*}
A\leq \left(\frac{p}{\e}\right)^p \frac{\log^q(r)}{r}\left(\lfloor \e^p\rfloor +1+ \lfloor \e^p \rfloor \cdot \frac{\e^{p+q+1}}{\e^{p+q+1}-\lfloor\e^p\rfloor}\right)\leq C(p,q)\frac{\log^{p+q+1}(r)}{r}.
\end{align*}
Now consider the term $B$. 
Again using the monotonicity of the function $x \mapsto\log^q(x)/x$ on $[\e^q,\infty)$,
we have
\begin{align*}
B=\sum_{m=\lceil \e^p\rceil}^r \frac{\log^p(m)}{m} \frac{\log^q(m+r)}{m+r}
&\leq \log^{p}(r) \frac{\log^q(r)}{r} \sum_{m=\lceil \e^p\rceil}^r \frac{1}{m}\lesssim \frac{\log^{p+q+1}(r)}{r}.
\end{align*}
Concerning the term $C$, note that
\begin{align*}
C\leq \sum_{m=r+1}^{\infty} \frac{\log^{p+q}(m+r)}{m(m+r)}
&=\frac{1}{r^2}\sum_{j=1}^{\infty} \frac{\left(\log(r)+\log\left(2+\frac{j}{r}\right)\right)^{p+q}}{\left(1+\frac{j}{r}\right)\left(2+\frac{j}{r}\right)}\\
&\leq \frac{2^{p+q-1}}{r}\left(\frac{\log^{p+q}(r)}{r} \sum_{j=1}^{\infty} \frac{1}{\left(1+\frac{j}{r}\right)\left(2+\frac{j}{r}\right)}+\frac{1}{r}\sum_{j=1}^{\infty} \frac{\log^{p+q}\left(2+\frac{j}{r}\right)}{\left(1+\frac{j}{r}\right)\left(2+\frac{j}{r}\right)}\right)\\
&\leq \frac{1}{r}\left(\tilde{C}(p,q)\log^{p+q}(r)+\tilde{\tilde{C}}(p,q)\right),
\end{align*}
where we used $(x+y)^n\leq 2^{n-1} (x^n+y^n)$ for $x,y\in\R$ and $n\in\N_0$.
Now we consider term $D$
and make the decomposition
$
D=D_1+D_2$,
where
\begin{align*}
D_1:=\sum_{m=-\lfloor \e^q\rfloor -r}^{\lfloor \e^q \rfloor -r} \ell^{(p)}(m) \ell^{(q)}(m+r)\qquad \text{and} \qquad D_2:=\sum_{m=\lfloor \e^q \rfloor-r+1}^{-\lceil \e^p \rceil} \ell^{(p)}(m) \ell^{(q)}(m+r).
\end{align*}
Using similar arguments as above, we obtain
\begin{align*}
D_1 = \left(\frac{q}{\e}\right)^q \sum_{m=r-\lfloor \e^q \rfloor}^{r+\lfloor \e^q \rfloor} \frac{\log^p(m)}{m} &\leq \left(\frac{q}{\e}\right)^q (2\lfloor \e^q \rfloor +1) \frac{\log^p(r-\lfloor \e^q \rfloor)}{r-\lfloor \e^q \rfloor}
&\leq C(q) \frac{\log^p(r)}{r} \frac{\e^{p+q+1}}{\e^{p+q+1}-\lfloor \e^q\rfloor}.
\end{align*}
Moreover,
\begin{align*}
D_2&=\sum_{m=\lfloor \e^p \rfloor}^{r-\lfloor \e^q \rfloor -1} \frac{\log^p(m)}{m} \frac{\log^q(r-m)}{r-m}\leq D_{21}+D_{22},
\end{align*}
where
\begin{align*}
D_{21}:=\sum_{m=\lfloor \e^p \rfloor}^{\lfloor r/2 \rfloor} \frac{\log^p(m)}{m} \frac{\log^q(r-m)}{r-m} \qquad \text{and} \qquad D_{22}:= \sum_{m=\lfloor r/2 \rfloor +1}^{r-\lfloor \e^q \rfloor -1} \frac{\log^p(m)}{m} \frac{\log^q(r-m)}{r-m}.
\end{align*}
Using similar arguments as above, it follows that
\begin{align*}
D_{21} \leq \frac{2 \log^q(r/2)}{r} \sum_{m=\lfloor \e^p \rfloor}^{\lfloor r/2 \rfloor} \frac{\log^p(m)}{m}&\leq \frac{2 \log^q(r/2) \log^p(r)}{r} \sum_{m=\lfloor \e^p \rfloor}^{\lfloor r/2 \rfloor} \frac{1}{m}\lesssim \frac{\log^{p+q+1}(r)}{r}
\end{align*}
and
\begin{align*}
D_{22} &\leq \frac{2\log^p(r)}{r} \sum_{m=\lfloor r/2 \rfloor +1}^{r-\lfloor \e^q \rfloor -1} \frac{\log^q(r-m)}{r-m}=\frac{2 \log^p(r)}{r} \sum_{z=\lfloor \e^q \rfloor +1}^{r-\lfloor r/2 \rfloor -1} \frac{\log^q(z)}{z}\lesssim \frac{\log^{p+q+1}(r)}{r}.
\end{align*}
It remains to deal with the terms $E$ and $F$. Since
\begin{align*}
E=\sum_{m=-2r}^{-\lceil \e^q \rceil -r} \frac{\log^p(-m)}{-m} \frac{\log^q(-m-r)}{-m-r}=\sum_{z=\lceil \e^q \rceil}^{r} \frac{\log^p(z+r)}{z+r} \frac{\log^q(z)}{z},
\end{align*}
the term $E$ can be dealt with analogously as the term $B$. Moreover,
\begin{align*}
F=\sum_{m=-\infty}^{-2r-1} \frac{\log^p(-m)}{-m} \frac{\log^q(-m-r)}{-m-r}=\sum_{z=r+1}^{\infty} \frac{\log^p(z+r)}{z+r} \frac{\log^q(z)}{z},
\end{align*}
so the term $F$ can be dealt with in the same way as the term $C$.
\end{enumerate}

\subsection{Proof of Lemma \ref{bounds_for_B}}

\begin{enumerate}
\item[(i)]
If $d=1$, we directly obtain
\begin{align*}
|B(u)|=\lambda \left| \int_{\frac{1}{2}}^{\frac{1}{2}} h(s) \exp(-\im \lambda s u)\, ds\right|\leq \lambda \int_{\frac{1}{2}}^{\frac{1}{2}} |h(s)|\, ds
\end{align*}
for all $u\in\R$. Furthermore, for $u\neq 0$, integration by parts gives
\begin{align*}
|B(u)|
&=\left|\frac{1}{\im u} \int_{\frac{1}{2}}^{\frac{1}{2}} h'(s) \exp(-\im \lambda s u)\, ds\right| \leq \frac{1}{|u|} \int_{\frac{1}{2}}^{\frac{1}{2}} \left|h'(s)\right|\, ds.
\end{align*}
In particular, we obtain
\begin{align*}
\left|B(u)\right|\leq K \begin{cases}
\lambda, \qquad & |u|\leq \frac{1}{\lambda}\\
\frac{1}{|u|}, \qquad & |u|>\frac{1}{\lambda}
\end{cases}\quad =K \, L_{\lambda}^{(0)}(u),
\end{align*}
where $K:=\max\{\int_{\frac{1}{2}}^{\frac{1}{2}} |h(s)|\, ds, \int_{\frac{1}{2}}^{\frac{1}{2}} \left|h'(s)\right|\, ds\}$.
The cases $d>1$ can be shown analogously by noting that
\begin{align*}
\left|B(\bm{u})\right|&= \prod_{i=1}^d \left| \int_{-\lambda/2}^{\lambda/2} h_i\left(\frac{s_i}{\lambda}\right) \exp(-\im s_i u_i)\, ds_i \right|\lesssim \prod_{i=1}^d L_{\lambda}^{(0)}(u_i)= L_{\lambda}^{(0)}(\bm{u}).
\end{align*}

\item[(ii)] From part (i) and Lemma \ref{properties_of_ell}, part (i), we obtain
\begin{align*}
\frac{1}{\lambda^d} \sum_{m_1,\ldots,m_d=-a}^a \int_{\R^d} \left| B(\bm{u}) B(\bm{\omega_m}-\bm{u})\right|\, d\bm{u}&\lesssim \frac{1}{\lambda^d} \sum_{m_1,\ldots,m_d=-a}^a \int_{\R^d} L_{\lambda}^{(0)}(\bm{u}) L_{\lambda}^{(0)}(\bm{\omega}_{\bm{m}}-\bm{u})\, d\bm{u}\\
&\lesssim \prod_{i=1}^d \left(\frac{1}{\lambda} \sum_{m_i=-a}^{a} L_{\lambda}^{(1)}(\omega_{m_i})\right)\\
&=\prod_{i=1}^d \left(\frac{1}{\e}+2\sum_{m_i=1}^a \frac{\log(2\pi m_i)}{2\pi m_i}\right)\lesssim \log^{2d}(a),
\end{align*}
where for the last step we used
\begin{align*}
\sum_{m_i=1}^a \frac{\log(2\pi m_i)}{m_i}\leq \log(2\pi a) \sum_{m_i=1}^a \frac{1}{m_i}\lesssim \log^2(a).
\end{align*}

\item[(iii)] From part (i) and from Lemma \ref{properties_of_ell}, part (i), we have
\begin{align*}
&\phantom{=i}\int_{\R^{2d}} \left|B(\bm{u}) B(\bm{\omega_m}-\bm{u}) B(\bm{v}) B(\bm{\omega_m}-\bm{v})\right|\, d\bm{u} d\bm{v}\\
&\lesssim \int_{\R^{d}} L_{\lambda}^{(0)}(\bm{u}) L_{\lambda}^{(0)}(\bm{\omega_m}-\bm{u}) \, d\bm{u} \int_{\R^{d}} L_{\lambda}^{(0)}(\bm{v}) L_{\lambda}^{(0)}(\bm{\omega_m}-\bm{v})\, d\bm{v}\lesssim \prod_{i=1}^d \left( L_{\lambda}^{(1)}(\omega_{m_i})\right)^2,
\end{align*}
which yields
\begin{align*}
\frac{1}{\lambda^{2d}} \sum_{m_1,\ldots,m_d=-a}^a \int_{\R^{2d}} \left|B(\bm{u}) B(\bm{\omega_m}-\bm{u}) B(\bm{v}) B(\bm{\omega_m}-\bm{v})\right| d\bm{u} d\bm{v}\lesssim \prod_{i=1}^d \left(\frac{1}{\lambda^2} \sum_{m_i=-a}^a \left(L_{\lambda}^{(1)}(\omega_{m_i})\right)^2 \right).
\end{align*}
Note that
\begin{align*}
\frac{1}{\lambda^2} \sum_{m_i=-a}^a \left(L_{\lambda}^{(1)}(\omega_{m_i})\right)^2
&=\frac{1}{\e^2}+\frac{2}{(2\pi)^2} \sum_{m_i=1}^a \frac{\log^2(2\pi m_i)}{m_i^2},
\end{align*}
and the last expression is bounded by a constant since the integral $\int_1^{\infty} \frac{\log^2(2\pi x)}{x^2} \, dx$ converges. This yields the claim.

\item[(iv)] We will first prove the statement for the case $s=0$. The other cases can be proven in the same way and we will shortly explain how to do so at the end of this proof.
Let $t\geq 2$, $c_1,\ldots,c_t\geq 1$, and $\bm{d}_1,\ldots,\bm{d}_t\in\Z^d$ be arbitrary.
 Using part (i) and Lemma \ref{properties_of_ell}, part (i), it then follows that 
\begin{align*}
&\phantom{=i}\frac{1}{\lambda^{dt}}\sum_{\bm{m}_1,\ldots,\bm{m}_{t-1}=-a}^{a} \int_{\R^{dt}} \Bigg|\prod_{j=1}^{t-1} \left[B(\bm{x}_j) B\left(\bm{x}_j-\frac{c_j(\bm{m}_j+\bm{d}_j)}{\lambda}\right)\right]\\
&\phantom{============iiii} B(\bm{x}_t)B\left(\bm{x}_t+\frac{c_t\left(\sum_{k=1}^{t-1} \bm{m}_k +\bm{d}_t\right)}{\lambda}\right)\Bigg|\prod_{j=1}^t d\bm{x}_j \\
&\lesssim \frac{1}{\lambda^{dt}} \sum_{\bm{m}_1,\ldots,\bm{m}_{t-1}=-a}^{a} \prod_{i=1}^d  \Bigg\{\prod_{j=1}^{t-1}\left[ \int_{\R} L_{\lambda}^{(0)} (x_{ji}) L_{\lambda}^{(0)}\left(x_{ji}-\frac{c_j(m_{ji}+d_{ji})}{\lambda}\right) dx_{ji}\right]\\
&\phantom{===========i==} \int_{\R}  L_{\lambda}^{(0)}(x_{ti}) L_{\lambda}^{(0)}\left(x_{ti}+\frac{c_t \left(\sum_{k=1}^{t-1} m_{ki} + d_{ti}\right)}{\lambda}\right) dx_{ti}\Bigg\}\\
 &\lesssim \sum_{\bm{m}_1,\ldots,\bm{m}_{t-1}=-a}^{a} \prod_{i=1}^d \left\{\prod_{j=1}^{t-1} \ell^{(1)} (c_j(m_{ji}+d_{ji})) \times \ell^{(1)} \left(c_t \left(\sum_{k=1}^{t-1} m_{ki}+d_{ti}\right)\right)\right\},
\end{align*}
where the function $\ell^{(s)}$ for $s\in\N_0$ is defined in \eqref{ell-fct}.
Note that $\ell^{(s)}(v)\leq \ell^{(s)}(w)$ for all $v,w\in\R$ with $|w|\leq |v|$ and $s\in\N_0$ and recall that $\sum_{\bm{m}=-a}^{a} $ denotes the multiple sum $\sum_{m_1=-a}^a \ldots \sum_{m_d=-a}^a$. Since by assumption $c_j\geq 1$ for all $j=1,\ldots,t$, the above term is therefore bounded by
\begin{align*}
\prod_{i=1}^d  \sum_{m_{1i},\ldots,m_{t-1,i}=-a}^a  \left\{ \prod_{j=1}^{t-1}  \ell^{(1)} (m_{ji}+d_{ji})  \times \ell^{(1)} \left(\sum_{k=1}^{t-1} m_{ki}+d_{ti}\right) \right\}.
\end{align*}
Fix $i\in\{1,\ldots,d\}$. We now aim to apply part (ii) of Lemma \ref{properties_of_ell} and write
\begin{align*}
&\phantom{=}\sum_{m_{1i},\ldots,m_{t-1,i}=-a}^a \left\{ \prod_{j=1}^{t-1}  \ell^{(1)} (m_{ji}+d_{ji})  \times \ell^{(1)} \left(\sum_{k=1}^{t-1} m_{ki}+d_{ti}\right)\right\}\\
&=\sum_{m_{1i},\ldots,m_{t-2,i}=-a}^a \left\{ \prod_{j=1}^{t-2} \ell^{(1)} (m_{ji}+d_{ji})  \sum_{m_{t-1,i=-a}}^a \left[ \ell^{(1)}(m_{t-1,i}+d_{t-1,i}) \, \ell^{(1)} \left(\sum_{k=1}^{t-1} m_{ki}+d_{ti}\right)\right]\right\}.
\end{align*}
We make the index shift $m_{t-1,i}+d_{t-1,i}=z_{t-1,i}$ to see that this expression is bounded by 
\begin{align*}
\sum_{m_{1i},\ldots,m_{t-2,i}=-\infty}^{\infty}  \Bigg\{ \prod_{j=1}^{t-2} \ell^{(1)} (m_{ji}+d_{ji})  \sum_{z_{t-1,i}=-\infty}^{\infty}\Bigg[ &\ell^{(1)} (z_{t-1,i})\\
&\ell^{(1)} \left(z_{t-1,i}+\sum_{k=1}^{t-2} m_{ki} + d_{ti}-d_{t-1,i}\right)\Bigg]\Bigg\}.
\end{align*}
By part (ii) of Lemma \ref{properties_of_ell}, we have
\begin{align*}
\sum_{z_{t-1,i}=-\infty}^{\infty} \Bigg[\ell^{(1)} (z_{t-1,i})\, \ell^{(1)} \left(z_{t-1,i}+\sum_{k=1}^{t-2} m_{ki} + d_{ti}-d_{t-1,i}\right)\Bigg] \lesssim \ell^{(3)} \left(\sum_{k=1}^{t-2} m_{ki} + d_{ti} -d_{t-1,i}\right).
\end{align*}
We now apply part (ii) of Lemma \ref{properties_of_ell} $(t-1)$-times and obtain
\begin{align*}
&\phantom{=}\sum_{m_{1i},\ldots,m_{t-2,i}=-\infty} ^{\infty} \Bigg\{ \prod_{j=1}^{t-2} \ell^{(1)} (m_{ji}+d_{ji}) \times \ell^{(3)} \left(\sum_{k=1}^{t-2} m_{ki}+d_{ti}-d_{t-1,i}\right)\Bigg\}\\
&=\sum_{m_{1i},\ldots,m_{t-3,i}=-\infty}^{\infty} \Bigg\{ \prod_{j=1}^{t-3}  \ell^{(1)}(m_{ji}+d_{ji}) \\
&\phantom{==========}\sum_{m_{t-2,i}=-\infty}^{\infty} \left[\ell^{(1)}(m_{t-2,i}+d_{t-2,i})\, \ell^{(3)} \left(\sum_{k=1}^{t-2} m_{ki}+d_{ti}-d_{t-1,i}\right)\right]\Bigg\}\\
&=\sum_{m_{1i},\ldots,m_{t-3,i}=-\infty}^{\infty}\Bigg\{  \prod_{j=1}^{t-3} \ell^{(1)}(m_{ji}+d_{ji}) \\
&\phantom{==========}\sum_{z_{t-2,i}=-\infty}^{\infty} \left[ \ell^{(1)} (z_{t-2,i} )\, \ell^{(3)} \left(z_{t-2,i}+\sum_{k=1}^{t-3} m_{ki} + d_{ti}-\sum_{k=t-2}^{t-1} d_{ki}\right)\right]\Bigg\}\\
&\lesssim \sum_{m_{1i},\ldots,m_{t-3,i}=-\infty}^{\infty}\Bigg\{ \prod_{j=1}^{t-3} \ell^{(1)}(m_{ji}+d_{ji})\times \ell^{(5)} \left(\sum_{k=1}^{t-3} m_{ki} + d_{ti}-\sum_{k=t-2}^{t-1} d_{ki}\right)\Bigg\}\\
&\lesssim \ldots\\
&\lesssim \sum_{m_{1i}=-\infty}^{\infty} \left\{ \ell^{(1)}(m_{1i}+d_{1i})\, \ell^{(2t-3)} \left(m_{1i}+d_{ti}-\sum_{k=2}^{t-1} d_{ki}\right)\right\}\\
&\lesssim \ell^{(2t-1)} \left(d_{ti}-\sum_{k=1}^{t-1} d_{ki}\right)\leq C(t),
\end{align*}
for some bounded constant $C(t)$ which only depends on $t$. Now assume that $0<s\leq t-1$. With the same arguments as given above, it holds that
\begin{align*}
&\phantom{\leq i}\frac{1}{\lambda^{dt}}\sum_{\bm{m}_1,\ldots,\bm{m}_{t-1}=-a}^{a} \int_{\R^{d(t-s)}} \Bigg|\prod_{j=1}^{s} B\left(\frac{c_j(\bm{m}_j+\bm{d}_j)}{\lambda}\right) \prod_{j=s+1}^{t-1} \left[ B(\bm{x}_j) B\left(\bm{x}_j-\frac{c_j(\bm{m}_j+\bm{d}_j)}{\lambda}\right)\right]\\
&\phantom{================} B(\bm{x}_t)B\left(\bm{x}_t+\frac{c_t\left(\sum_{k=1}^{t-1} \bm{m}_k +\bm{d}_t\right)}{\lambda}\right)\Bigg|\prod_{j=s+1}^t d\bm{x}_j\\
 &\lesssim \prod_{i=1}^d \sum_{m_{1i},\ldots,m_{t-1,i}=-a}^a  \Bigg\{\prod_{j=1}^{s} \ell^{(0)}\left(m_{ji}+d_{ji}\right) \times \prod_{j=s+1}^{t-1} \ell^{(1)} \left(m_{ji}+d_{ji}\right)\times \ell^{(1)}\left(\sum_{k=1}^{t-1} m_{ki} + d_{ti} \right)\Bigg\}.
\end{align*}
Now, the proof works exactly as above, since we can again apply part (ii) of Lemma \ref{properties_of_ell}.

\item[(v)] In the same way as in part (iv), we obtain
\begin{align*}
&\phantom{\leq i}\frac{1}{\lambda^{dt}} \sum_{\bm{m}_1,\ldots,\bm{m}_{t-1}=-a}^{a} \Bigg|  \prod_{j=1}^{t-1} B\left(\frac{c_j(\bm{m}_j+\bm{d}_j)}{\lambda}\right) \times B\left(                  \frac{c_t\left(\sum_{k=1}^{t-1} \bm{m}_k +\bm{d}_t\right)}{\lambda}\right)\Bigg|\\                         
&\lesssim \prod_{i=1}^d  \sum_{m_{1i},\ldots,m_{t-1,i}=-\infty}^{\infty}  \Bigg\{\prod_{j=1}^{t-1} \ell^{(0)}\left(m_{ji}+d_{ji}\right) \times  \ell^{(0                                                                                                                                                                                                                                                                                                                                                                                                                    )} \left(\sum_{k=1}^{t-1} m_{ki}+d_{ti}\right) \Bigg\} .
\end{align*}       
Now, the proof works again as in part (iv), iteratively applying part (ii) of Lemma \ref{properties_of_ell}. 

\end{enumerate}

\subsection{Proof of Lemma \ref{orders_of_B}}

\begin{enumerate}
\item[(i)] This follows directly from Lemma \ref{properties_of_ell}, part (i), and Lemma \ref{bounds_for_B}, part (i). 
\item[(ii)]
By partial integration and the convolution theorem, we obtain
\begin{align*}
&\phantom{=i}\left|\int_{\R} B(u) B(\omega_m-u)\, u\, du\right| \\
&=\left|\int_{\R} \Big(\lambda \int_{\frac{1}{2}}^{\frac{1}{2}} h(t_1) \exp(-\im \lambda t_1 u)\, dt_1\Big)\Big(\lambda \int_{\frac{1}{2}}^{\frac{1}{2}} h(t_2) \exp(-\im \lambda t_2(\omega_m-u))\, dt_2\Big) u\, du\right|\\
&=\Big|\int_{\R} \Big(\frac{1}{\im u} \int_{\frac{1}{2}}^{\frac{1}{2}} h'(t_1) \exp(-\im \lambda t_1 u)\, dt_1\Big)\Big(\lambda \int_{\frac{1}{2}}^{\frac{1}{2}} h(t_2) \exp(-\im \lambda t_2(\omega_m-u))\, dt_2\Big) u\, du\Big|\\
&=\Big| \int_{\R} \Big(\int_{\frac{1}{2}}^{\frac{1}{2}} h'(t_1)\, \exp(-\im t_1 z)\, dt_1\Big) \Big(\int_{\frac{1}{2}}^{\frac{1}{2}} h(t_2) \exp(-\im t_2(2\pi m-z))\, dt_2\Big)\, dz\Big|\\
&=\big| (\mathcal{F}h') \circledast (\mathcal{F} h)(2\pi m)\big|=2\pi \left|\mathcal{F}(h'\cdot h)(2\pi m)\right|=2\pi \Big|\int_{\frac{1}{2}}^{\frac{1}{2}} h'(s) h(s) \exp(- 2\pi \im m s) \, ds\Big|.
\end{align*}
Note that for $m=0$, the above expression equals $0$.
%
%
For $m\neq 0$, we apply partial integration once more and obtain
\begin{align*}
\left|\int_{\frac{1}{2}}^{\frac{1}{2}} h'(s) h(s) \exp(-2\pi \im m s)\, ds\right|= \frac{1}{2\pi |m|} \left| \int_{\frac{1}{2}}^{\frac{1}{2}} \left[h''(s) h(s) + (h'(s))^2 \right] \exp(-2\pi \im m s) \, ds \right|.
\end{align*}
This yields the claim.

\item[(iii)] 
Using the same calculation steps as above, we obtain
\begin{align*}
\int_{\R} \left|B(u) B(\omega_m-u) \, u\right| du
=S_1+S_2+S_3,
\end{align*}
where
\begin{align*}
S_1:=\int_{-\infty}^{2\pi m-1} \Big|\int_{\frac{1}{2}}^{\frac{1}{2}} h'(s_1) \exp(-\im s_1 z)\, ds_1 \int_{\frac{1}{2}}^{\frac{1}{2}} h(s_2) \exp(-\im s_2 (2\pi m-z))\, ds_2\Big|\, dz,
\end{align*}
\begin{align*}
S_2:= \int_{2\pi m-1}^{2\pi m+1} \Big|\int_{\frac{1}{2}}^{\frac{1}{2}} h'(s_1) \exp(-\im s_1 z)\, ds_1 \int_{\frac{1}{2}}^{\frac{1}{2}} h(s_2) \exp(-\im s_2 (2\pi m-z))\, ds_2\Big|\, dz,
\end{align*}
and
\begin{align*}
S_3:=\int_{2\pi m+1}^{\infty} \Big|\int_{\frac{1}{2}}^{\frac{1}{2}} h'(s_1) \exp(-\im s_1 z)\, ds_1 \int_{\frac{1}{2}}^{\frac{1}{2}} h(s_2) \exp(-\im s_2 (2\pi m-z))\, ds_2\Big|\, dz.
\end{align*}
The term $S_2$ is obviously bounded by a generic constant. For the first and the third term, we perform integration by parts on the second integral twice (using $h(\frac{1}{2})=h(\frac{1}{2})=h^\prime(\frac{1}{2})=h^\prime(\frac{1}{2})=0$) to obtain
\begin{align*}
S_1 &= \int_{-\infty}^{2\pi m-1} \frac{1}{(2\pi m -z)^2} \Big|\int_{\frac{1}{2}}^{\frac{1}{2}} h'(s_1) \exp(-\im s_1 z)\, ds_1 \\
&\phantom{=========} \int_{\frac{1}{2}}^{\frac{1}{2}} h''(s_2) \exp(-\im s_2 (2\pi m-z))\, ds_2 \Big|\, dz \lesssim \int_{-\infty}^{2\pi m-1} \frac{1}{(2\pi m-z)^2}\, dz  \leq C,
\end{align*}
and in the same way $S_3\leq C$ for some bounded constant $C>0$. This yields the claim.

\item[(iv)] Again using integration by parts, note that the expression of interest equals
\begin{align*}
&\phantom{=i} \lambda^2 \int_{\R} \Big|u^2 \int_{\frac{1}{2}}^{\frac{1}{2}} h(s_1) \exp(-\im \lambda s_1 u)\, ds_1 \int_{\frac{1}{2}}^{\frac{1}{2}} h(s_2) \exp(-\im \lambda s_2 (\omega_m-u)) \,ds_2 \Big|\,  du\\
&= \int_{\R} \Big| \int_{\frac{1}{2}}^{\frac{1}{2}} h''(s_1) \exp(-\im \lambda s_1 u)\, ds_1 \int_{\frac{1}{2}}^{\frac{1}{2}} h(s_2) \exp(-\im s_2 (2\pi m - \lambda u)) \, ds_2 \Big| \,du\\
&=\frac{1}{\lambda} \int_{\R} \Big| \int_{\frac{1}{2}}^{\frac{1}{2}} h''(s_1) \exp(-\im s_1 z)\, ds_1 \int_{\frac{1}{2}}^{\frac{1}{2}} h(s_2) \exp(-\im s_2 (2\pi m - z)) \,ds_2 \Big| \,dz.
\end{align*}
The claim now follows with the same argument as in part (iii) (performing integration by parts on the second integral twice). 
\end{enumerate}

\subsection{Proof of Proposition \ref{Lemma F.2_SSR}}

We make use of the Taylor expansion
\begin{align*}
f(\bm{u}-\bm{v})-f(\bm{v})=f(\bm{v}-\bm{u})-f(\bm{v})=-\sum_{j=1}^d u_j \frac{\partial f}{\partial v_j}(\bm{v}) + \frac{1}{2}\sum_{j,k=1}^d u_j u_k \frac{\partial^2 f}{\partial v_j \partial v_k} (\bm{\xi}),
\end{align*}
where $\bm{\xi}=\bm{v}-\theta \bm{u}$ for some $\theta\in [0,1]$. 
Therefore, 
\begin{align} \label{formula_last_result_taper}
&\phantom{\leq i}\frac{1}{\lambda^d}\left|\int_{\R^d} B(\bm{u}) B(\bm{\omega_m}-\bm{u}) \left(\int_{-b_{1,1}}^{b_{1,2}} \ldots \int_{-b_{d,1}}^{b_{d,2}} g(\bm{v}) \left[f(\bm{u}-\bm{v})-f(\bm{v})\right]\, d\bm{v}\right) d\bm{u} \right|\nonumber\\
&\lesssim \frac{1}{\lambda^d}\sum_{j=1}^d \left(\int_{-b}^b \ldots \int_{-b}^b \left|g(\bm{v})\,\frac{\partial f}{\partial v_j}(\bm{v})\right| d\bm{v}\right) \left|\int_{\R^d} B(\bm{u}) B(\bm{\omega_m}-\bm{u}) \, u_j\, d\bm{u}\right|\nonumber\\
&+\frac{1}{2\lambda^d}\sum_{j,k=1}^d \int_{-b}^b \ldots \int_{-b}^b \left|g(\bm{v})\right| \left( \int_{\R^d} \left|\frac{\partial^2 f}{\partial v_j \partial v_k} (\bm{\xi})\, u_j u_k\, B(\bm{u}) B(\bm{\omega_m}-\bm{u})\right| d\bm{u}\right) d\bm{v}\nonumber\\
&\lesssim \begin{cases}
\sup_{\bm{v}\in \R^d} |g(\bm{v})|\Bigg(\frac{1}{\lambda^d}  \sum_{j=1}^d  \left|\int_{\R^d} B(\bm{u}) B(\bm{\omega_m}-\bm{u})\, u_j\, d\bm{u}\right|\\
\phantom{============}+\frac{b^d}{\lambda^d} \left( \sum_{j,k=1}^d \int_{\R^d} \left|u_j u_k\, B(\bm{u}) B(\bm{\omega_m}-\bm{u})\right| d\bm{u}\right)\Bigg), & \text{ if $g$ bounded},\\
\int_{\R^d} |g(\bm{v})|\, d\bm{v} \Bigg(\frac{1}{\lambda^d}  \sum_{j=1}^d  \left|\int_{\R^d} B(\bm{u}) B(\bm{\omega_m}-\bm{u})\, u_j\, d\bm{u}\right|\\
\phantom{===========}+\frac{1}{\lambda^d} \left( \sum_{j,k=1}^d \int_{\R^d} \left|u_j u_k\, B(\bm{u}) B(\bm{\omega_m}-\bm{u})\right| d\bm{u}\right)\Bigg), & \text{ if $g$ abs. integr.}
\end{cases}
\end{align}
We consider the terms
\begin{align} \label{second_term}
\int_{\R^d} B(\bm{u}) B(\bm{\omega_m}-\bm{u})\, u_j\, d\bm{u} \qquad \text{and} \qquad \int_{\R^d} \left|u_j u_k\, B(\bm{u}) B(\bm{\omega_m}-\bm{u})\right| d\bm{u}
\end{align}
separately and start with the first one. Assume w.l.o.g. that $j=1$ and note that by Lemma \ref{orders_of_B}
\begin{align*} 
\left|\int_{\R^d} B(\bm{u}) B(\bm{\omega_m}-\bm{u}) \, u_1 \, d\bm{u}\right|
&\lesssim \lambda^{d-1} \left|\int_{\R} B(u_1) B(\omega_{m_1}-u_1)\, u_1\, du_1\right|\nonumber\\
&\lesssim \begin{cases}
0, \qquad &m_1=0,\\
\lambda^{d-1}, \qquad &m_1\neq 0.
\end{cases}
\end{align*}
Concerning the second term in \eqref{second_term}, we distinguish the cases $j=k$ and $j\neq k$. First, let $j\neq k$ and assume w.l.o.g. $j=1$, $k=2$. Then, according to Lemma \ref{orders_of_B},
\begin{align*} 
&\phantom{=i}\int_{\R^d} \left|u_1 u_2\, B(\bm{u}) B(\bm{\omega_m}-\bm{u})\right| d\bm{u}\nonumber\\
&=\prod_{i=1}^2 \left(\int_{\R} \left|u_i\, B(u_i) B(\omega_{m_i}-u_i)\right| du_i\right) 
\prod_{i=3}^d \left(\int_{\R} \left|B(u_i) B(\omega_{m_i}-u_i)\right|\, du_i\right)\lesssim \lambda^{d-2}.
\end{align*}

We now consider the case $j=k$ and assume w.l.o.g. that $j=k=1$. Then, Lemma \ref{orders_of_B} gives
\begin{align*}
\int_{\R^d} \left|u_1^2\, B(\bm{u}) B(\bm{\omega_m}-\bm{u})\right| d\bm{u}&\lesssim \lambda^{d-1} \left(\int_{\R} \left|u_1^2 \, B(u_1) B(\omega_{m_1}-u_1)\right|\, du_1\right) \lesssim \lambda^{d-2}.
\end{align*}
Plugging the bounds for the terms in \eqref{second_term} into equation \eqref{formula_last_result_taper} then yields the claim.

\section{Proof of Theorem \ref{corr_first_int}} 

\setcounter{equation}{0}

The assertion of Theorem \ref{corr_first_int} is obtained in several steps. We begin with statement regarding the first and second moment of the statistic $\hat{D}_{1,d,\lambda,a}$
.
\begin{theorem} \label{expectation_theo}
If $d\in\N$ be arbitrary and Assumption \ref{ass_on_h}, Assumption \ref{assumption_on_Z}, Assumption \ref{assumption_on_sampling_scheme}, and Assumption \ref{assumptions_on_a}, part (i) are satisfied, then
\begin{align*}
\E\big[\hat{D}_{1,d,\lambda,a}\big] &=
D_{1,d,\lambda,a} + \Landau\left(\frac{1}{\lambda^2}+\frac{1}{\lambda^d}+\frac{1}{n}\right) \\
\lambda^d\Var\big[\hat{D}_{1,d,\lambda,a}\big]&=\tau_{1,d,\lambda,a}^2+ \Landau\left(\frac{(\log a)^{2d}}{\lambda}+\frac{1}{(\log\lambda)^3}+\frac{\lambda^d}{n}\right) 
\end{align*}
as $a,\lambda,n\rightarrow\infty$,
where and  $\tau_{1,d,\lambda,a}$ is  defined in  \eqref{det3} and  
\begin{align}
 \label{D_1,d,lambda,a}
D_{1,d,\lambda,a} &:=\int_{[-2\pi a/\lambda,2\pi a/\lambda]^d}f^2(\bm{\omega}) \, d\bm{\omega}.
\end{align}
\end{theorem}
The next  statement shows the  asymptotic normality of $\hat  D_{1,d,\lambda,a} - D_{1,d,\lambda,a} $ after appropriate scaling.
\begin{theorem} \label{asymptotic_normality}
If $d\leq 3$ and  Assumption \ref{ass_on_h}, \ref{assumption_on_Z},  \ref{assumption_on_sampling_scheme}, and Assumption \ref{assumptions_on_a}(i),(ii) are satisfied,  we have
\begin{align*}
\frac{\lambda^{d/2}}{\tau_{1,d,\lambda,a}} \left(\hat{D}_{1,d,\lambda,a}- D_{1,d,\lambda,a}\right) \dn \mathcal{N}(0,1) \qquad \text{as } a,\lambda,n\rightarrow \infty,
\end{align*}
where $D_{1,d,\lambda,a}$  and  $\tau_{1,d,\lambda,a}$ are  defined in \eqref{D_1,d,lambda,a} and \eqref{det3}, respectively.
\end{theorem}
The proof is completed investigating the error of the approximation  of $D_{1,d}$ by $D_{1,d,\lambda,a}$.
\begin{prop} \label{approx_D1}
Assume that $f(\bm{\omega})\leq \beta_{1+\delta}(\bm{\omega})$ for some $\delta>0$ and that
 $a/\lambda\to\infty$, then 
\begin{align*}
|D_{1,d}-D_{1,d,\lambda,a}|=\Landau\left(\frac{\lambda^{1+2\delta}}{a^{1+2\delta}}\right) \qquad \text{as } \lambda,a\to\infty.
\end{align*}
\end{prop}
Assumption \ref{ass_on_h} and integration by parts yield
\begin{align*}
|H_{d,h}(\bm{m})| = \prod_{i=1}^{d} |H_{1,h}(m_i)| \lesssim \frac{1}{\prod_{i=1}^d m_i^2}
\end{align*}
for $m_1,\ldots,m_d\neq 0$, and it is obvious that $H_{d,h}(\bm{0})$ must be bounded away from $0$.  Consequently, we have
 $$0<\inf_{\lambda,a} \tau_{1,d,\lambda,a}^2 \leq \sup_{\lambda,a} \tau_{1,d,\lambda,a}^2 < \infty.
 $$ 
The assertion of Theorem  \ref{corr_first_int} now follows from   the decomposition 
\begin{align*}
&\phantom{==}\frac{\lambda^{d/2}}{\tau_{1,d,\lambda,a}} \left(\hat{D}_{1,d,\lambda,a}- D_{1,d}\right)= 
\frac{\lambda^{d/2}}{\tau_{1,d,\lambda,a}} \left(\hat{D}_{1,d,\lambda,a}-D_{1,d,\lambda,a}\right)+\frac{\lambda^{d/2}}{\tau_{1,d,\lambda,a}} \left(D_{1,d,\lambda,a}- D_{1,d}\right).
\end{align*}
By Theorem \ref{asymptotic_normality}, the first term on the right hand side converges weakly  to a standard normal distribution. By Proposition \ref{approx_D1} and \eqref{further_ass_int_1}, the 
second term is of order $o_P(1)$.  \hfill $\Box$ 

\bigskip 

In the remaining part of this section we will prove Theorem  \ref{expectation_theo},
\ref{asymptotic_normality}
and Proposition  \ref{approx_D1}.
For the sake of simplicity, we write $\tilde{\bm{\omega}}_{\bm{k},\lambda}=\tilde{\bm{\omega}}_{\bm{k}}$, $\bm{\omega}_{\bm{k},\lambda}=\bm{\omega}_{\bm{k}}$, $H_{d,h}(\bm{m})=H(\bm{m})$ and $B_{\lambda,d,h}(\bm{u})=B(\bm{u})$. 

\subsection{Proof of Theorem \ref{expectation_theo} } 

\paragraph{Asymptotic  bias:}
By the tower property, it holds that
\begin{align*}
\E\big[\hat{D}_{1,d,\lambda,a}\big]& = 
\frac{(2\pi\lambda)^d}{2n^4 H(\bm{0})^2} \sum_{\bm{k}=-a}^{a-1} \sum_{(j_1,j_2,j_3,j_4)\in \mathcal{E}} \E \bigg[h\left(\frac{\bm{s}_{j_1}}{\lambda}\right) h\left(\frac{\bm{s}_{j_2}}{\lambda}\right) h\left(\frac{\bm{s}_{j_3}}{\lambda}\right) h\left(\frac{\bm{s}_{j_4}}{\lambda}\right) \\
&\phantom{====}\exp\big(\im (\bm{s}_{j_1}-\bm{s}_{j_2}+\bm{s}_{j_3} - \bm{s}_{j_4})^T \tilde{\bm{\omega}}_{\bm{k}}\big)\, \E \big[Z(\bm{s}_{j_1}) Z(\bm{s}_{j_2}) Z(\bm{s}_{j_3}) Z(\bm{s}_{j_4}) \big|\, \bm{s}_{j_1},\bm{s}_{j_2},\bm{s}_{j_3},\bm{s}_{j_4}\big]\bigg].
\end{align*}
Since $\{Z(\bm{s}):\bm{s}\in\R^d\}$ is assumed to be a Gaussian process with mean zero, we have
\begin{align*}
\E \big[Z(\bm{s}_{j_1}) Z(\bm{s}_{j_2}) Z(\bm{s}_{j_3}) Z(\bm{s}_{j_4}) \big|\, \bm{s}_{j_1},\bm{s}_{j_2},\bm{s}_{j_3},\bm{s}_{j_4}\big]&=c(\bm{s}_{j_1}-\bm{s}_{j_2})\, c(\bm{s}_{j_3}-\bm{s}_{j_4})+c(\bm{s}_{j_1}-\bm{s}_{j_3})\, c(\bm{s}_{j_2}-\bm{s}_{j_4})\\&+c(\bm{s}_{j_1}-\bm{s}_{j_4})\, c(\bm{s}_{j_2}-\bm{s}_{j_3}),
\end{align*}
and since the locations $\bm{s}_j$ are independently sampled, this yields
\begin{align*}
\E\big[\hat{D}_{1,d,\lambda,a}\big]& = E_1 + E_2 + E_3,
\end{align*}
where
\begin{align*} 
E_1& = \frac{(2\pi\lambda)^d}{2n^4 H(\bm{0})^2} \sum_{\bm{k}=-a}^{a-1} \sum_{(j_1,j_2,j_3,j_4)\in \mathcal{E}} \bigg(  \E\bigg[h\left(\frac{\bm{s}_{j_1}}{\lambda}\right) h\left(\frac{\bm{s}_{j_2}}{\lambda}\right) c(\bm{s}_{j_1}-\bm{s}_{j_2})\exp\big(\im(\bm{s}_{j_1}-\bm{s}_{j_2})^T \tilde{\bm{\omega}}_{\bm{k}}\big)\bigg]\nonumber\\
&\phantom{===============iii=}  \E\bigg[h\left(\frac{\bm{s}_{j_3}}{\lambda}\right) h\left(\frac{\bm{s}_{j_4}}{\lambda}\right) c(\bm{s}_{j_3}-\bm{s}_{j_4})\exp\big(\im(\bm{s}_{j_3}-\bm{s}_{j_4})^T \tilde{\bm{\omega}}_{\bm{k}}\big)\bigg]\bigg),
\end{align*}
\begin{align} \label{E_2}
E_2 &=\frac{(2\pi\lambda)^d}{2n^4 H(\bm{0})^2} \sum_{\bm{k}=-a}^{a-1} \sum_{(j_1,j_2,j_3,j_4)\in \mathcal{E}} \bigg( \E\bigg[h\left(\frac{\bm{s}_{j_1}}{\lambda}\right) h\left(\frac{\bm{s}_{j_3}}{\lambda}\right) c(\bm{s}_{j_1}-\bm{s}_{j_3})\exp\big(\im(\bm{s}_{j_1}+\bm{s}_{j_3})^T \tilde{\bm{\omega}}_{\bm{k}}\big)\bigg]\nonumber\\
&\phantom{===============iii=}  \E\bigg[h\left(\frac{\bm{s}_{j_2}}{\lambda}\right) h\left(\frac{\bm{s}_{j_4}}{\lambda}\right) c(\bm{s}_{j_2}-\bm{s}_{j_4})\exp\big(\im(-\bm{s}_{j_2}-\bm{s}_{j_4})^T \tilde{\bm{\omega}}_{\bm{k}}\big)\bigg] \bigg),
\end{align}
and
\begin{align*} 
E_3 & = \frac{(2\pi\lambda)^d}{2n^4 H(\bm{0})^2} \sum_{\bm{k}=-a}^{a-1} \sum_{(j_1,j_2,j_3,j_4)\in \mathcal{E}} \bigg( \E\bigg[h\left(\frac{\bm{s}_{j_1}}{\lambda}\right) h\left(\frac{\bm{s}_{j_4}}{\lambda}\right) c(\bm{s}_{j_1}-\bm{s}_{j_4})\exp\big(\im(\bm{s}_{j_1}-\bm{s}_{j_4})^T \tilde{\bm{\omega}}_{\bm{k}}\big)\bigg]\nonumber\\
&\phantom{===========iii=====}  \E\bigg[h\left(\frac{\bm{s}_{j_2}}{\lambda}\right) h\left(\frac{\bm{s}_{j_3}}{\lambda}\right) c(\bm{s}_{j_2}-\bm{s}_{j_3})\exp\big(\im(\bm{s}_{j_3}-\bm{s}_{j_2})^T \tilde{\bm{\omega}}_{\bm{k}}\big)\bigg]\bigg).
\end{align*}
We now consider these three terms separately. $E_1$ and $E_3$ can be dealt with analogously and we thus only consider the term $E_1$. To this end, we define the quantity
\begin{align} \label{c_q,n}
c_{q,n}:=\frac{n(n-1)\ldots (n-4q+1)}{n^{4q}}
\end{align}
and obtain
\begin{align*}
E_1&\overset{\text{(1)}}{=} \frac{(2\pi\lambda)^d c_{1,n}}{2 H(\bm{0})^2}\, \sum_{\bm{k}=-a}^{a-1} \E\left[h\left(\frac{\bm{s}_{1}}{\lambda}\right) h\left(\frac{\bm{s}_{2}}{\lambda}\right)  c(\bm{s}_1-\bm{s}_2) \exp\big(\im (\bm{s}_1-\bm{s}_2)^T \tilde{\bm{\omega}}_{\bm{k}}\big)\right]\\
&\phantom{===iiiiiiiiiiii==} \E\left[h\left(\frac{\bm{s}_{3}}{\lambda}\right) h\left(\frac{\bm{s}_{4}}{\lambda}\right)  c(\bm{s}_3-\bm{s}_4) \exp\big(\im (\bm{s}_3-\bm{s}_4)^T \tilde{\bm{\omega}}_{\bm{k}}\big)\right]\\
&\overset{\text{(2)}}{=} \frac{(2\pi\lambda)^d c_{1,n}}{2 H(\bm{0})^2} \sum_{\bm{k}=-a}^{a-1} \frac{1}{\lambda^{4d}} \int_{[-\lambda/2,\lambda/2]^{4d}} h\left(\frac{\bm{s}_{1}}{\lambda}\right) h\left(\frac{\bm{s}_{2}}{\lambda}\right) h\left(\frac{\bm{s}_{3}}{\lambda}\right) h\left(\frac{\bm{s}_{4}}{\lambda}\right)\\
&\phantom{===========}  c(\bm{s}_1-\bm{s}_2) c(\bm{s}_3-\bm{s}_4) \exp\big(\im(\bm{s}_1-\bm{s}_2+\bm{s}_3-\bm{s}_4)^T \tilde{\bm{\omega}}_{\bm{k}}\big) \, d\bm{s}_1 d \bm{s}_2 d\bm{s}_3 d\bm{s}_4 \\
&\overset{\text{(3)}}{=} \frac{(2\pi\lambda)^d c_{1,n}}{2 (2\pi)^{2d} H(\bm{0})^2} \sum_{\bm{k}=-a}^{a-1} \int_{\R^{2d}} \bigg( \frac{1}{\lambda^{4d}}  \int_{[-\lambda/2,\lambda/2]^{4d}} h\left(\frac{\bm{s}_{1}}{\lambda}\right) h\left(\frac{\bm{s}_{2}}{\lambda}\right) h\left(\frac{\bm{s}_{3}}{\lambda}\right) h\left(\frac{\bm{s}_{4}}{\lambda}\right) \exp\big(\im(\bm{s}_1-\bm{s}_2)^T \bm{x}\big) \\
&\phantom{========} \exp\big(\im (\bm{s}_3-\bm{s}_4)^T \bm{y}\big) 
\exp\big(\im (\bm{s}_1-\bm{s}_2+\bm{s}_3-\bm{s}_4)^T \tilde{\bm{\omega}}_{\bm{k}}\big)\, d\bm{s}_1 d\bm{s}_2 d\bm{s}_3 d\bm{s}_4 
\bigg)  f(\bm{x}) f(\bm{y}) \, d\bm{x} d\bm{y}\\
&=\frac{(2\pi\lambda)^d c_{1,n}}{2 (2\pi)^{2d} H(\bm{0})^2} \sum_{\bm{k}=-a}^{a-1} \int_{\R^{2d}}\bigg( \frac{1}{\lambda^{4d}} \int_{[-\lambda/2,\lambda/2]^{4d}}  h\left(\frac{\bm{s}_{1}}{\lambda}\right) h\left(\frac{\bm{s}_{2}}{\lambda}\right) h\left(\frac{\bm{s}_{3}}{\lambda}\right) h\left(\frac{\bm{s}_{4}}{\lambda}\right) \exp\big(\im \bm{s}_1^T(\bm{x}+\tilde{\bm{\omega}}_{\bm{k}})\big)\\
&\phantom{==}  \exp\big(-\im \bm{s}_2^T(\bm{x}+\tilde{\bm{\omega}}_{\bm{k}})\big)\exp\big(\im \bm{s}_3^T(\bm{y}+\tilde{\bm{\omega}}_{\bm{k}})\big)\exp\big(-\im \bm{s}_4^T(\bm{y}+\tilde{\bm{\omega}}_{\bm{k}})\big) \, d\bm{s}_1 d \bm{s}_2 d\bm{s}_3 d\bm{s}_4 \bigg)  f(\bm{x}) f(\bm{y})\, d\bm{x} d\bm{y}.
\end{align*}
For (1), we used that the locations $\{\bm{s}_j\}$ are independent identically distributed and that $(j_1,\ldots,j_4)\in\mathcal{E}$; in (2), we exploited that the locations are uniformly distributed on $[-\lambda/2,\lambda/2]^d$, and in (3) we inserted the definition of the covariance function.
Note that 
\begin{align*}
\int_{[-\lambda/2,\lambda/2]^{2d}} h\left(\frac{\bm{s}_{1}}{\lambda}\right) h\left(\frac{\bm{s}_{2}}{\lambda}\right) \exp\big(\im \bm{s}_1^T(\bm{x}+\tilde{\bm{\omega}}_{\bm{k}})\big)\exp\big(-\im \bm{s}_2^T(\bm{x}+\tilde{\bm{\omega}}_{\bm{k}})\big) \, d\bm{s}_1 d\bm{s}_2&=\left|B(\bm{x}+\tilde{\bm{\omega}}_{\bm{k}})\right|^2,
\end{align*}
with $B$ being the frequency window of $h$, see \eqref{Four_trafo_of_h}. Under Assumption \ref{ass_on_h},  $B$ is a real function, which yields
\begin{align} \label{calc_E1}
E_1&=\frac{c_{1,n}}{2 (2\pi)^{d} H(\bm{0})^2 \lambda^{3d}} \sum_{\bm{k}=-a}^{a-1} 
\int_{\R^{2d}} B(\bm{x}+\tilde{\bm{\omega}}_{\bm{k}})^2 B(\bm{y}+\tilde{\bm{\omega}}_{\bm{k}})^2 f(\bm{x}) f(\bm{y})\, d\bm{x} d\bm{y}\nonumber\\
&=\frac{c_{1,n}}{2(2\pi\lambda)^{2d} H(\bm{0})^2} \int_{\R^{2d}} B(\bm{u})^2 B(\bm{v})^2 \left(\left(\frac{2\pi}{\lambda}\right)^d \sum_{\bm{k}=-a}^{a-1} f(\bm{u}-\tilde{\bm{\omega}}_{\bm{k}}) f(\bm{v}-\tilde{\bm{\omega}}_{\bm{k}})\right) \, d\bm{u} d\bm{v}.
\end{align}
By Lemma  \ref{Riemann_sum}, the last expression equals 
\begin{align*}
\frac{c_{1,n}}{2(2\pi\lambda)^{2d} H(\bm{0})^2}\int_{\R^{2d}} B(\bm{u})^2 B(\bm{v})^2 \left(\int_{[-2\pi a/\lambda,2\pi a/\lambda]^d} f(\bm{u}-\bm{\omega}) f(\bm{v}-\bm{\omega})\, d\bm{\omega} \right) d\bm{u} d\bm{v} + \Landau\left(\frac{1}{\lambda^2}\right)
\end{align*}
as $a,\lambda\rightarrow\infty$.
Then, for the first term we can write
\begin{align*}
&\phantom{=i}\frac{c_{1,n}}{2(2\pi\lambda)^{2d} H(\bm{0})^2} \int_{\R^{2d}} B(\bm{u})^2 B(\bm{v})^2 \left(\int_{[-2\pi a/\lambda,2\pi a/\lambda]^d} f^2(\bm{\omega})\, d\bm{\omega} \right) d\bm{u} d\bm{v}\\
&+\frac{c_{1,n}}{2(2\pi\lambda)^{2d} H(\bm{0})^2} \int_{\R^{2d}} B(\bm{u})^2 B(\bm{v})^2 \Bigg(\int_{[-2\pi a/\lambda,2\pi a/\lambda]^d} f(\bm{u}-\bm{\omega}) f(\bm{v}-\bm{\omega}) \, d\bm{\omega}\\
&\phantom{===========================} - \int_{[-2\pi a/\lambda,2\pi a/\lambda]^d} f^2(\bm{\omega}) \, d\bm{\omega} \Bigg) d\bm{u} d\bm{v}
\end{align*}
Lemma  \ref{convolution_of_h} yields
\begin{align*}
E_1&=\frac{1}{2} \int_{[-2\pi a/\lambda,2\pi a/\lambda]^d} f^2(\bm{\omega})\, d\bm{\omega}\\
& +\frac{1}{2(2\pi\lambda)^{2d} H(\bm{0})^2} \int_{\R^{2d}} B(\bm{u})^2 B(\bm{v})^2 \Bigg(\int_{[-2\pi a/\lambda,2\pi a/\lambda]^d} f(\bm{u}-\bm{\omega}) f(\bm{v}-\bm{\omega}) \, d\bm{\omega}\\
&\phantom{===========================} -  \int_{[-2\pi a/\lambda,2\pi a/\lambda]^d} f^2(\bm{\omega}) \, d\bm{\omega} \Bigg) d\bm{u} d\bm{v}\\
&+\Landau\left(\frac{1}{\lambda^2}\right) + \Landau\left(\frac{1}{n}\right). 
\end{align*}
We now show that the second of these terms is sufficiently small. To this end, we make the expansion
\begin{align*}
&\phantom{=i}\frac{1}{\lambda^{2d}}\int_{\R^{2d}} B(\bm{u})^2 B(\bm{v})^2 \left( \int_{[-2\pi a/\lambda,2\pi a/\lambda]^d} f(\bm{u}-\bm{\omega}) f(\bm{v}-\bm{\omega}) \, d\bm{\omega} -  \int_{[-2\pi a/\lambda,2\pi a/\lambda]^d} f^2(\bm{\omega}) \, d\bm{\omega} \right) d\bm{u} d\bm{v}\\
&=\frac{1}{\lambda^d}\int_{\R^{d}} B(\bm{u})^2 \left(\frac{1}{\lambda^d}\int_{\R^d} B(\bm{v})^2 \left(\int_{[-2\pi a/\lambda,2\pi a/\lambda]^d} f(\bm{u}-\bm{\omega}) \left[f(\bm{v}-\bm{\omega})-f(\bm{\omega})\right] d\bm{\omega}\right) d\bm{v}\right) d\bm{u}\\
&+\frac{1}{\lambda^d}\int_{\R^{d}} B(\bm{v})^2 \left(\frac{1}{\lambda^d} \int_{\R^d} B(\bm{u})^2 \left(\int_{[-2\pi a/\lambda,2\pi a/\lambda]^d} f(\bm{\omega}) \left[f(\bm{u}-\bm{\omega})-f(\bm{\omega})\right]d\bm{\omega}\right) d\bm{u} \right) d\bm{v}.
\end{align*}
Both terms can be dealt with using Proposition \ref{Lemma F.2_SSR}, part (ii): Since the functions $f(\bm{u}-\bm{\omega})$ and $f(\bm{\omega})$ are absolutely integrable with respect to $\bm{\omega}$, we obtain
\begin{align*}
&\phantom{\lesssim}\sup_{\bm{u}\in\R^d}\Bigg\{\frac{1}{\lambda^d}\left| \int_{\R^d} B(\bm{v})^2 \left[\int_{[-2\pi a/\lambda,2\pi a/\lambda]^d} f(\bm{u}-\bm{\omega}) \left[f(\bm{v}-\bm{\omega})-f(\bm{\omega})\right] d\bm{\omega}\right) d\bm{v}\right|\Bigg\}\\
&\lesssim  \frac{1}{\lambda^2}  \sup_{\bm{u}\in\R^d} \int_{\R^d} f(\bm{u}-\bm{\omega})\, d\bm{\omega} \lesssim \frac{1}{\lambda^2}
\end{align*}
and
\begin{align*}
\frac{1}{\lambda^d} \left|\int_{\R^d} B(\bm{u})^2 \left(\int_{[-2\pi a/\lambda,2\pi a/\lambda]^d} f(\bm{\omega}) \left[f(\bm{u}-\bm{\omega})-f(\bm{\omega})\right]d\bm{\omega}\right) d\bm{u}\right|\lesssim \frac{1}{\lambda^2}.
\end{align*}
Since $\int_{\R^d} B(\bm{u})^2\, d\bm{u}\lesssim \lambda^d$ (see e.g. Lemma \ref{convolution_of_h}), we thus obtain 
\begin{align} \label{E1}
E_1=\frac{1}{2} \int_{[-2\pi a/\lambda,2\pi a/\lambda]^d} f^2(\bm{\omega})\, d\bm{\omega}+\Landau\left(\frac{1}{\lambda^2}\right)+\Landau\left(\frac{1}{n}\right),
\end{align}
and in the same way we also get
\begin{align} \label{E3}
E_3=\frac{1}{2} \int_{[-2\pi a/\lambda,2\pi a/\lambda]^d} f^2(\bm{\omega})\, d\bm{\omega}+\Landau\left(\frac{1}{\lambda^2}\right)+\Landau\left(\frac{1}{n}\right)
\end{align}
as $\lambda,a,n\rightarrow\infty$.
Now we deal with the term $E_2$ in \eqref{E_2}.  With the same arguments as given for the calculation of the term $E_1$, it follows that
\begin{align*}
E_2
&=\frac{c_{1,n}}{2(2\pi)^{d} H(\bm{0})^2 \lambda^{3d}} \sum_{\bm{k}=-a}^{a-1} \int_{\R^{2d}} B(\bm{x}+\tilde{\bm{\omega}}_{\bm{k}}) B(-\bm{x}+\tilde{\bm{\omega}}_{\bm{k}}) B(\bm{y}-\tilde{\bm{\omega}}_{\bm{k}}) B(-\bm{y}-\tilde{\bm{\omega}}_{\bm{k}}) \,f(\bm{x}) f(\bm{y})\, d\bm{x} d\bm{y}.
\end{align*}
Comparing this expression with the one in \eqref{calc_E1}, we see that we cannot make a change of variables in a way that all expressions of the form $\tilde{\bm{\omega}}_{\bm{k}}$ vanish from the $B$-functions. Instead, by a change of variables and recalling that by Assumption \ref{ass_on_h} the function $B$ is even, we obtain that the last expression equals
\begin{align*}
\frac{c_{1,n}}{2(2\pi)^{d} H(\bm{0})^2 \lambda^{3d}} \sum_{\bm{k}=-a}^{a-1} \int_{\R^{2d}} B(\bm{u}) B(2\tilde{\bm{\omega}}_{\bm{k}}-\bm{u}) B(\bm{v}) B(\bm{v}+2\tilde{\bm{\omega}}_{\bm{k}})  \,f(\bm{u}-\tilde{\bm{\omega}}_{\bm{k}}) f(\bm{v}+\tilde{\bm{\omega}}_{\bm{k}})\, d\bm{u} d\bm{v}.
\end{align*}
Observing that $2\tilde{\bm{\omega}}_{\bm{k}}=\bm{\omega}_{2\bm{k}+1}$, where $\bm{\omega_k}$ are the usual Fourier frequencies defined in \eqref{Four_freq}, and using that $f$ is bounded, we get 
\begin{align*}
|E_2|\lesssim \frac{1}{\lambda^{3d}} \sum_{\bm{k}=-a}^{a-1} \int_{\R^{2d}} \big| B(\bm{u}) B(-\bm{u}+\bm{\omega}_{2\bm{k}+1}) B(\bm{v}) B(\bm{v}+\bm{\omega}_{2\bm{k}+1}) \big| d\bm{u} d\bm{v}=\Landau\left(\frac{1}{\lambda^d}\right),
\end{align*}
where we have used Lemma \ref{bounds_for_B}, part (iii), in the last step. Taking this result together with \eqref{E1} and \eqref{E3}, we thus obtain
\begin{align*}
\E\big[\hat{D}_{1,d,\lambda,a}\big]=\int_{ [-2\pi a/\lambda,2\pi a/\lambda]^d} f^2(\bm{\omega}) \, d\bm{\omega} + \Landau\left(\frac{1}{\lambda^2}\right)+\Landau\left(\frac{1}{n}\right)+\Landau\left(\frac{1}{\lambda^d}\right)
\end{align*}
as $\lambda,a,n\to\infty$, which yields the statement of Theorem \ref{expectation_theo}.
$\hfill \Box$\\


\paragraph{Asymptotic variance:}
In order for our calculations to make sense, we assume that $n>8$. Given random variables $X_1,\ldots,X_n$, we denote by $\cum_n(X_1,\ldots,X_n)$ their cumulant. \label{page73} For simplicity, we abandon the subscript $n$ from time to time and simply write $\cum(X_1,\ldots,X_n)$.
Since the estimator $\hat{D}_{1,d,\lambda,a}$ is real, we thus have
\begin{align*}
&\lambda^d \Var\big[\hat{D}_{1,d,\lambda,a}\big]=\frac{(2\pi)^{2d}\lambda^{3d}}{4 n^{8} H(\bm{0})^4} \sum_{\bm{k}_1,\bm{k}_2=-a}^{a-1} \sum_{\underline{j}\in \mathcal{D}} \cum\bigg[h\left(\frac{\bm{s}_{j_1}}{\lambda}\right) h\left(\frac{\bm{s}_{j_2}}{\lambda}\right) h\left(\frac{\bm{s}_{j_3}}{\lambda}\right) h\left(\frac{\bm{s}_{j_4}}{\lambda}\right)\\
&\phantom{===========i====i===}  Z(\bm{s}_{j_1}) Z(\bm{s}_{j_2}) Z(\bm{s}_{j_3}) Z(\bm{s}_{j_4}) \exp\big(\im(\bm{s}_{j_1}-\bm{s}_{j_2}+\bm{s}_{j_3}-\bm{s}_{j_4})^T \tilde{\bm{\omega}}_{\bm{k}_1}\big),\\
&\phantom{i=} h\left(\frac{\bm{s}_{j_5}}{\lambda}\right) h\left(\frac{\bm{s}_{j_6}}{\lambda}\right) h\left(\frac{\bm{s}_{j_7}}{\lambda}\right) h\left(\frac{\bm{s}_{j_8}}{\lambda}\right) Z(\bm{s}_{j_5}) Z(\bm{s}_{j_6}) Z(\bm{s}_{j_7}) Z(\bm{s}_{j_8}) \exp\big(\im(\bm{s}_{j_5}-\bm{s}_{j_6}+\bm{s}_{j_7}-\bm{s}_{j_8})^T \tilde{\bm{\omega}}_{\bm{k}_2}\big)\bigg]\\
&=\frac{(2\pi)^{2d}\lambda^{3d}}{4 n^{8} H(\bm{0})^4} \sum_{\bm{k}_1,\bm{k}_2=-a}^{a-1} \sum_{\underline{j}\in \mathcal{D}} \cum\bigg[Y_{1,1}(\underline{j}) \, Y_{1,2}(\underline{j}) \,Y_{1,3}(\underline{j})\, Y_{1,4}(\underline{j})\, ,\,   Y_{2,1}(\underline{j})\, Y_{2,2}(\underline{j})\, Y_{2,3}(\underline{j}) \,Y_{2,4}(\underline{j})\bigg],
\end{align*}
where \begin{align} \label{set_D}
\mathcal{D}:=\{\underline{j}=(j_1,\ldots,j_8)\in\{1,\ldots,n\}^8\,:\, (j_1,j_2,j_3,j_4)\in\mathcal{E} \text{ and } (j_5,j_6,j_7,j_8)\in\mathcal{E}\},
\end{align} 
$\mathcal{E}$ is defined in \eqref{set_E} and $\bm{k}_t=(k_{t1},\ldots,k_{td})^T$ for $t=1,2$. Furthermore, for $t=1,2$ and $c=1,2,3,4,$ the random variables $Y_{t,c}$ are defined as
\begin{align} \label{rv's_Y_tc}
Y_{t,c}(\underline{j}):=h\left(\frac{\bm{s}_{j_{c+4(t-1)}}}{\lambda}\right) Z(\bm{s}_{j_{c+4(t-1)}}) \exp\big((-1)^{c+1}\, \im\, \bm{s}_{j_{c+4(t-1)}}^T \tilde{\bm{\omega}}_{\bm{k}_t} \big).
\end{align}
We need three important results on cumulants:
\begin{itemize}
\item[A)] First, recall the product theorem for cumulants, which states that for a matrix of random variables $X_{i,j}$, $i=1,\ldots,I$, $j=1,\ldots,J_i$ and random variables
\begin{align*}
Y_i=\prod_{j=1}^{J_i} X_{i,j}, \qquad i=1,\ldots,I,
\end{align*}
we have
\begin{align*}
\cum_{I}\big[Y_1,\ldots,Y_I\big]=\sum_{\bm{\nu}} \cum_{|\nu_1|}\big(X_{i,j}\,:\, (i,j)\in\nu_1\big)\times\ldots\times \cum_{|\nu_G|}\big(X_{i,j}\,:\, (i,j)\in\nu_G\big),
\end{align*}
where the sum is over all indecomposable partitions $\bm{\nu}=\{\nu_1, \ldots,\nu_G\}$ of the table
\begin{align*}
\begin{matrix} 
(1,1) & \ldots & (1,J_1)\\
\vdots & \vdots & \vdots \\
(I,1) &\ldots  & (I,J_I)\, 
\end{matrix}
\end{align*}
 \citep[see][]{brillinger81}.
\item[B)] Secondly, note that for any Gaussian process $\{X(\bm{s}):\bm{s}\in\R^d\}$ with mean $0$ and random locations $\{\bm{s}_j\}$, we must have
\begin{align}  \label{result_B}
\cum_{2q+1} \left[X(\bm{s}_{1}),\ldots,X(\bm{s}_{{2q+1}})\right]=0
\end{align}
for all $q\in\N$. This is due to the fact that by the law of total cumulance, it holds that
\begin{align*} 
&\phantom{=i}\cum_q\left[X(\bm{s}_1),\ldots,X(\bm{s}_q)\right]\\
&=\sum_{\bm{\pi}} \cum_b\left[\cum_{|\pi_1|}\left(X_{\pi_1}\,|\, \bm{s}_1,\ldots,\bm{s}_q\right),\ldots,\cum_{|\pi_b|}\left(X_{\pi_b}\,|\, \bm{s}_1,\ldots,\bm{s}_q\right)\right],
\end{align*}
where the sum is over all partitions $\bm{\pi}=\{\pi_1, \ldots, \pi_b\}$ of $\{1,\ldots,q\}$, and 
\begin{align*}
X_{\pi_j}:=\{X(\bm{s}_i)\, :\, i\in\pi_j\}, \qquad j=1,\ldots,b
\end{align*}
[also see \cite{subbarao_suppl}].
For mean zero Gaussian processes, only the second order cumulant is non-zero , which gives \eqref{result_B}.
\item[C)] Thirdly, note that by the same argument as in B), for any Gaussian process $\{X(\bm{s}):\bm{s}\in\R^d\}$ with mean $0$ and random locations $\{\bm{s}_j\}$, we have
\begin{align*}
\cum_{2q}\left[X(\bm{s}_1),\ldots,X(\bm{s}_{2q})\right]=0
\end{align*}
if more than $q+1$ locations of $\bm{s}_1,\ldots,\bm{s}_{2q}$ are independent.
\end{itemize}

We now define $\mathcal{I}$ to be the set of indecomposable partitions of the table 
\begin{align} \label{table_q=2}
\begin{matrix} 
&(1,1) & (1,2) & (1,3) & (1,4)\\
&(2,1) & (2,2) & (2,3) & (2,4).
\end{matrix}
\end{align}

From result A), we thus obtain
\begin{align*} 
\lambda^d \Var\big[\hat{D}_{1,d,\lambda,a}\big] &= \frac{(2\pi)^{2d}\lambda^{3d}}{4 n^{8} H(\bm{0})^4} \sum_{\bm{k}_1,\bm{k}_2=-a}^{a-1} \sum_{\bm{\nu}=\{\nu_1,\ldots,\nu_G\}\in\mathcal{I}} \sum_{\underline{j}\in \mathcal{D}} \cum_{|\nu_1|}\left[Y_{t,c}(\underline{j}):(t,c)\in\nu_1\right]\times\ldots\\
&\phantom{==============ii=i==}\times\cum_{|\nu_G|}\left[Y_{t,c}(\underline{j}):(t,c)\in\nu_G\right].
\end{align*}
From B), we furthermore observe that we can restrict ourselves to the case where
\begin{align*}
|\nu_g|=2p_g \quad \text{ for some } p_g\in\{1,\ldots,4\}, \, \quad \text{for all } g=1,\ldots,G.
\end{align*} 
C) shows that each set $\nu_g$ corresponds to at most $p_g+1$ independent locations, otherwise the corresponding cumulant expression becomes $0$. Therefore, any tupel $\underline{j}\in\mathcal{D}$ contributing a nonzero term to the above sum can have at most
\begin{align*}
\sum_{g=1}^G (p_g+1)=\frac{1}{2}\sum_{g=1}^G |\nu_g|+G= 4+G
\end{align*}
different elements.
Defining the subsets $\mathcal{D}(i)$ of $\mathcal{D}$ by
\begin{align} \label{set_D(i)}
\mathcal{D}(i):=\{\underline{j}\in\mathcal{D}\, :\, i \text{ elements in } \underline{j} \text{ are different}\},
\end{align}
we can thus write 
\begin{align*}
\lambda^d \Var\big[\hat{D}_{1,d,\lambda,a}\big] &= \frac{(2\pi)^{2d}\lambda^{3d}}{4 n^{8} H(\bm{0})^4} \sum_{\bm{k}_1,\bm{k}_2=-a}^{a-1} \sum_{\bm{\nu}=\{\nu_1,\ldots,\nu_G\}\in\mathcal{I}} \sum_{i=4}^{4+G} \sum_{\underline{j}\in\mathcal{D}(i)} \cum_{|\nu_1|}\left[Y_{t,c}(\underline{j}):(t,c)\in\nu_1\right]\times\ldots\\
&\phantom{==========ii=i=i=ii==i=}\times\cum_{|\nu_G|}\left[Y_{t,c}(\underline{j}):(t,c)\in\nu_G\right]\\
&=\frac{(2\pi)^{2d}\lambda^{3d}}{4 n^{8} H(\bm{0})^4} \sum_{\bm{k}_1,\bm{k}_2=-a}^{a-1} \sum_{\bm{\nu}=\{\nu_1,\ldots,\nu_4\}\in\mathcal{I}} \sum_{\underline{j}\in\mathcal{D}(8)} \cum_2\left[Y_{t,c}(\underline{j}):(t,c)\in\nu_1\right]\times\ldots\\
&\phantom{=======iii===i=i==iii=}\times\cum_2\left[Y_{t,c}(\underline{j}):(t,c)\in\nu_4\right]+\Landau\left(\frac{\lambda^d}{n}\right),
\end{align*}
where for the second step we used that by Corollary \ref{result_for_q=2} below, terms with partitions consisting of $G=4$ groups with $i=8$ different elements are of highest order. \\
It will become clear from the following discussion that the indecomposable partitions with $G=4$ groups can be further divided into two sets $\mathcal{I}_1$ and $\mathcal{I}_2$, where partitions from the set $\mathcal{I}_1$ lead to a higher order cumulant and partitions from the set $\mathcal{I}_2$ yield a lower order cumulant. We will consider two examples for illustration, namely the partitions $\bm{\nu}_1^{\ast}=\{\nu_{1,1}^\ast,\ldots,\nu_{1,4}^\ast\}\in \mathcal{I}_1$ and $\bm{\nu}_2^{\ast}=\{\nu_{2,1}^\ast,\ldots,\nu_{2,4}^\ast\}\in \mathcal{I}_2$ given by 
\begin{align} \label{examp_for_part_1}
\bm{\nu}_1^{\ast}=\Big\{\{(1,1),(1,2)\},\{(2,1),(2,2)\},\{(1,3),(2,4)\},\{(1,4),(2,3)\}\Big\}
\end{align}
and
\begin{align} \label{examp_for_part}
\bm{\nu}_2^{\ast}=\Big\{\{(1,1),(1,3)\},\{(1,2),(2,3)\},\{(1,4),(2,1)\},\{(2,2),(2,4)\}\Big\}.
\end{align}
First, we consider the partition $\bm{\nu}_1^{\ast}$ and recall the definition of the quantity $c_{q,n}$ in \eqref{c_q,n}. 
Then, by using that the locations $\{\bm{s}_j\}$ are independent uniformly distributed on $[-\lambda/2,\lambda/2]^d$ and by inserting the definition of the covariance function, we have
\begin{align} \label{only_expression_k_2-k_1}
&\phantom{i=}\frac{(2\pi)^{2d}\lambda^{3d}}{4 n^{8} H(\bm{0})^4} \sum_{\bm{k}_1,\bm{k}_2=-a}^{a-1}  \sum_{\underline{j}\in\mathcal{D}(8)} \cum_2\left[Y_{t,c}(\underline{j}):(t,c)\in\nu_{1,1}^{\ast}\right]\times\ldots\times\cum_2\left[Y_{t,c}(\underline{j}):(t,c)\in\nu_{1,4}^{\ast}\right]\nonumber\\
&=\frac{(2\pi)^{2d}\lambda^{3d}}{4 n^{8} H(\bm{0})^4} \sum_{\bm{k}_1,\bm{k}_2=-a}^{a-1}  \sum_{\underline{j}\in\mathcal{D}(8)} \cum\left[h\left(\frac{\bm{s}_{j_1}}{\lambda}\right) Z(\bm{s}_{j_1})\exp\big(\im \bm{s}_{j_1}^T \tilde{\bm{\omega}}_{\bm{k}_1}\big),h\left(\frac{\bm{s}_{j_2}}{\lambda}\right) Z(\bm{s}_{j_2}) \exp\big(-\im \bm{s}_{j_2}^T \tilde{\bm{\omega}}_{\bm{k}_1}\big)\right]\nonumber\\
&\phantom{=i=i====ii==iiii} \times\cum\left[h\left(\frac{\bm{s}_{j_5}}{\lambda}\right) Z(\bm{s}_{j_5})\exp\big(\im \bm{s}_{j_5}^T \tilde{\bm{\omega}}_{\bm{k}_2}\big),h\left(\frac{\bm{s}_{j_6}}{\lambda}\right) Z(\bm{s}_{j_6}) \exp\big(-\im \bm{s}_{j_6}^T \tilde{\bm{\omega}}_{\bm{k}_2}\big)\right]\nonumber\\
&\phantom{=i=i====ii==iii i}\times\cum\left[h\left(\frac{\bm{s}_{j_3}}{\lambda}\right) Z(\bm{s}_{j_3})\exp\big(\im \bm{s}_{j_3}^T \tilde{\bm{\omega}}_{\bm{k}_1}\big),h\left(\frac{\bm{s}_{j_8}}{\lambda}\right) Z(\bm{s}_{j_8}) \exp\big(-\im \bm{s}_{j_8}^T \tilde{\bm{\omega}}_{\bm{k}_2}\big)\right]\nonumber\\
&\phantom{=i=i====ii==iiii}\times\cum\left[h\left(\frac{\bm{s}_{j_4}}{\lambda}\right) Z(\bm{s}_{j_4})\exp\big(-\im \bm{s}_{j_4}^T \tilde{\bm{\omega}}_{\bm{k}_1}\big),h\left(\frac{\bm{s}_{j_7}}{\lambda}\right) Z(\bm{s}_{j_7}) \exp\big(\im \bm{s}_{j_7}^T \tilde{\bm{\omega}}_{\bm{k}_2}\big)\right]\nonumber\\
&=\frac{(2\pi)^{2d}\lambda^{3d}}{4 n^{8} H(\bm{0})^4} \sum_{\bm{k}_1,\bm{k}_2=-a}^{a-1}  \sum_{\underline{j}\in\mathcal{D}(8)} \E\left[h\left(\frac{\bm{s}_{j_1}}{\lambda}\right) h\left(\frac{\bm{s}_{j_2}}{\lambda}\right)c(\bm{s}_{j_1}-\bm{s}_{j_2})\exp\big(\im (\bm{s}_{j_1}-\bm{s}_{j_2})^T \tilde{\bm{\omega}}_{\bm{k}_1}\big)\right] \nonumber\\
&\phantom{=i=i=iii=ii==iiii}\times \E\left[h\left(\frac{\bm{s}_{j_5}}{\lambda}\right) h\left(\frac{\bm{s}_{j_6}}{\lambda}\right)c(\bm{s}_{j_5}-\bm{s}_{j_6})\exp\big(\im (\bm{s}_{j_5}-\bm{s}_{j_6})^T \tilde{\bm{\omega}}_{\bm{k}_2}\big)\right] \nonumber\\
&\phantom{=i=i=iii=ii==iiii}\times \E\left[h\left(\frac{\bm{s}_{j_3}}{\lambda}\right) h\left(\frac{\bm{s}_{j_8}}{\lambda}\right)c(\bm{s}_{j_3}-\bm{s}_{j_8})\exp\big(\im (\bm{s}_{j_3}^T  \tilde{\bm{\omega}}_{\bm{k}_1}-\bm{s}_{j_8}^T \tilde{\bm{\omega}}_{\bm{k}_2})\big)\right] \nonumber\\
&\phantom{=i=i==iiiiiii==iiii}\times \E\left[h\left(\frac{\bm{s}_{j_4}}{\lambda}\right) h\left(\frac{\bm{s}_{j_7}}{\lambda}\right)c(\bm{s}_{j_7}-\bm{s}_{j_4})\exp\big(-\im (\bm{s}_{j_4}^T \tilde{\bm{\omega}}_{\bm{k}_1}-\bm{s}_{j_7}^T \tilde{\bm{\omega}}_{\bm{k}_2})\big)\right] \nonumber\\
&=\frac{c_{2,n}}{4(2\pi)^{2d} H(\bm{0})^4 \lambda^{5d}} \sum_{\bm{k}_1,\bm{k}_2=-a}^{a-1} \int_{\R^{4d}} B(\bm{s})^2 B(\bm{t})^2 B(\bm{u}) B(\bm{u}+\tilde{\bm{\omega}}_{\bm{k}_2}-\tilde{\bm{\omega}}_{\bm{k}_1}) B(\bm{v}) B(\bm{v}+\tilde{\bm{\omega}}_{\bm{k}_2}-\tilde{\bm{\omega}}_{\bm{k}_1})\nonumber\\
&\phantom{============i=ii=} f(\bm{s}-\tilde{\bm{\omega}}_{\bm{k}_1}) f(\bm{t}-\tilde{\bm{\omega}}_{\bm{k}_2})f(\bm{u}-\tilde{\bm{\omega}}_{\bm{k}_1})f(\bm{v}-\tilde{\bm{\omega}}_{\bm{k}_1})\,d\bm{s} d\bm{t}d\bm{u}d\bm{v},
\end{align}
where we used the same arguments as in the proof of Theorem \ref{expectation_theo}. 
Note that 
\begin{align*} 
\tilde{\bm{\omega}}_{\bm{k}_2}-\tilde{\bm{\omega}}_{\bm{k}_1}=\bm{\omega}_{\bm{k}_2-\bm{k}_1} \qquad \text{and} \qquad \tilde{\bm{\omega}}_{\bm{k}_1+\bm{k}_2}=\tilde{\bm{\omega}}_{\bm{k}_1}+\bm{\omega}_{\bm{k}_2},
\end{align*}
i.e. we can express the shifted Fourier frequencies with the help of the usual Fourier frequencies in \eqref{Four_freq}.
Therefore, setting $\bm{k}_2-\bm{k}_1=\bm{m}$ and $\bm{\ell}=\bm{k}_1$, we obtain that \eqref{only_expression_k_2-k_1} equals
\begin{align*} 
&\phantom{=}\frac{(2\pi)^d c_{2,n}}{4(2\pi \lambda)^{4d} H(\bm{0})^4 } \int_{\R^{4d}} \sum_{\bm{m}=-2a+1}^{2a-1}  B(\bm{s})^2 B(\bm{t})^2 B(\bm{u}) B(\bm{u}+\bm{\omega_m})B(\bm{v}) B(\bm{v}+\bm{\omega_m})\nonumber\\
&\phantom{===ii=} \bigg[ \left(\frac{2\pi}{\lambda}\right)^d \sum_{\bm{\ell}=\max(-a,-a-\bm{m})}^{\min(a-1,a-1-\bm{m})} f(\bm{s}-\tilde{\bm{\omega}}_{\bm{\ell}}) f(\bm{t}-\bm{\omega_m}-\tilde{\bm{\omega}}_{\bm{\ell}})f(\bm{u}-\tilde{\bm{\omega}}_{\bm{\ell}})f(\bm{v}-\tilde{\bm{\omega}}_{\bm{\ell}})\bigg]\, d\bm{s} d\bm{t} d\bm{u} d\bm{v}\nonumber\\
&=\frac{(2\pi)^d c_{2,n}}{4(2\pi\lambda)^{4d} H(\bm{0})^4} \int_{\R^{4d}} \sum_{\bm{m}=-2a+1}^{2a-1} B(\bm{s})^2 B(\bm{t})^2 B(\bm{u}) B(\bm{u}+\bm{\omega_m})B(\bm{v}) B(\bm{v}+\bm{\omega_m})\nonumber\\
& \bigg[\int_{2\pi\max(-a,-a-\bm{m})/\lambda}^{2\pi\min(a,a-\bm{m})/\lambda} f(\bm{s}-\bm{x}) f(\bm{t}-\bm{\omega_m}-\bm{x}) f(\bm{u}-\bm{x}) f(\bm{v}-\bm{x})\, d\bm{x}\bigg] \, d\bm{s} d\bm{t} d\bm{u} d\bm{v}+\Landau\left(\frac{1}{\lambda^2}\right),
\end{align*}
where we furthermore applied Corollary \ref{cor_Riemann_sum} and Lemma \ref{bounds_for_B} (iii).
Setting
\begin{align*}
G_{\bm{m}}(\bm{s},\bm{t},\bm{u},\bm{v}):=\int_{2\pi\max(-a,-a-\bm{m})/\lambda}^{2\pi\min(a,a-\bm{m})/\lambda} f(\bm{s}-\bm{x}) f(\bm{t}-\bm{\omega_m}-\bm{x}) f(\bm{u}-\bm{x}) f(\bm{v}-\bm{x})\, d\bm{x},
\end{align*}
we can then write the above expression as
\begin{align} \label{eq_21}
&\phantom{+i}\frac{(2\pi)^d c_{2,n}}{4(2\pi\lambda)^{4d} H(\bm{0})^4 } \int_{\R^{4d}} \sum_{\bm{m}=-2a+1}^{2a-1} B(\bm{s})^2 B(\bm{t})^2 B(\bm{u}) B(\bm{u}+\bm{\omega_m})B(\bm{v}) B(\bm{v}+\bm{\omega_m})\nonumber\\
&\phantom{==================} G_{\bm{m}}(\bm{0},\bm{0},\bm{0},\bm{0})\, d\bm{s} d\bm{t} d\bm{u} d\bm{v} \nonumber\\
&+\frac{(2\pi)^d c_{2,n}}{4(2\pi \lambda)^{4d} H(\bm{0})^4} \int_{\R^{4d}} \sum_{\bm{m}=-2a+1}^{2a-1} B(\bm{s})^2 B(\bm{t})^2 B(\bm{u}) B(\bm{u}+\bm{\omega_m})B(\bm{v}) B(\bm{v}+\bm{\omega_m})\nonumber\\
&\phantom{=======ii=i==} \left[G_{\bm{m}}(\bm{s},\bm{t},\bm{u},\bm{v})-G_{\bm{m}}(\bm{0},\bm{0},\bm{0},\bm{0})\right] \, d\bm{s} d\bm{t} d\bm{u} d\bm{v}+\Landau\left(\frac{1}{\lambda^2}\right).
\end{align}
In a next step, we show that the second term in \eqref{eq_21} is of lower order than the first one. Note that for quantities $A, B, C, D, E, F, G, H$ we have
\begin{align*}
ABCD-EFGH=ABC(D-E)+EAB(C-F)+EFA(B-G)+EFG(A-H),
\end{align*}
and therefore the second term in \eqref{eq_21} equals
\begin{align} \label{eq_2}
&\phantom{=i}\frac{(2\pi)^d c_{2,n}}{4(2\pi\lambda)^{3d} H(\bm{0})^3} \int_{\R^{3d}} \sum_{\bm{m}=-2a+1}^{2a-1} \Big[ B(\bm{s})^2 B(\bm{t})^2 B(\bm{u}) B(\bm{u}+\bm{\omega_m})\, D_{\bm{m}}^{(1)}(\bm{s},\bm{t},\bm{u})\, d\bm{s} d\bm{t} d\bm{u}\nonumber\\
&\phantom{=================}+B(\bm{s})^2 B(\bm{t})^2 B(\bm{v}) B(\bm{v}+\bm{\omega_m})\, D_{\bm{m}}^{(2)}(\bm{s},\bm{t})\, d\bm{s} d\bm{t} d\bm{v}\nonumber\\
&\phantom{=================}+B(\bm{s})^2 B(\bm{u}) B(\bm{u}+\bm{\omega_m}) B(\bm{v}) B(\bm{v}+\bm{\omega_m}) \,D_{\bm{m}}^{(3)}(\bm{s})\, d\bm{s} d\bm{u} d\bm{v}\nonumber\\
&\phantom{=================}+B(\bm{t})^2 B(\bm{u}) B(\bm{u}+\bm{\omega_m}) B(\bm{v}) B(\bm{v}+\bm{\omega_m}) \,D_{\bm{m}}^{(4)}\, d\bm{t} d\bm{u} d\bm{v}\Big],
\end{align}
where
\begin{align*}
&D_{\bm{m}}^{(1)}(\bm{s},\bm{t},\bm{u}):=\frac{1}{(2\pi\lambda)^d H(\bm{0})} \int_{\R^d} B(\bm{v}) B(\bm{v}+\bm{\omega_m})\\ 
&\phantom{=======} \bigg[ \int_{2\pi \max(-a,-a-\bm{m})/\lambda}^{2\pi \min(a,a-\bm{m})/\lambda} f(\bm{s}-\bm{x}) f(\bm{t}-(\bm{\omega_m}+\bm{x})) f(\bm{u}-\bm{x}) \big[f(\bm{v}-\bm{x})-f(\bm{x})\big]\, d\bm{x} \bigg]d\bm{v},
\end{align*}
\begin{align*}
&D_{\bm{m}}^{(2)}(\bm{s},\bm{t}):=\frac{1}{(2\pi\lambda)^d H(\bm{0})} \int_{\R^d} B(\bm{u}) B(\bm{u}+\bm{\omega_m})\\
&\phantom{======} \bigg[ \int_{2\pi \max(-a,-a-\bm{m})/\lambda}^{2\pi \min(a,a-\bm{m})/\lambda} f(\bm{x}) f(\bm{s}-\bm{x}) f(\bm{t}-(\bm{\omega_m}+\bm{x}))\big[f(\bm{u}-\bm{x})-f(\bm{x})\big]\, d\bm{x} \bigg]d\bm{u},
\end{align*}
\begin{align*}
&D_{\bm{m}}^{(3)}(\bm{s}):=\frac{1}{(2\pi\lambda)^d H(\bm{0})} \int_{\R^d} B(\bm{t})^2 \\
&\phantom{======} \bigg[ \int_{2\pi \max(-a,-a-\bm{m})/\lambda}^{2\pi \min(a,a-\bm{m})/\lambda} f(\bm{x})^2 f(\bm{s}-\bm{x})\big[f(\bm{t}-(\bm{\omega_m}+\bm{x}))-f(\bm{\omega_m}+\bm{x})\big]\, d\bm{x} \bigg]d\bm{t},
\end{align*}
and
\begin{align*}
&D_{\bm{m}}^{(4)}:=\frac{1}{(2\pi\lambda)^d H(\bm{0})} \int_{\R^d} B(\bm{s})^2 \bigg[ \int_{2\pi \max(-a,-a-\bm{m})/\lambda}^{2\pi \min(a,a-\bm{m})/\lambda} f(\bm{x})^2 f(\bm{\omega_m}+\bm{x})\big[f(\bm{s}-\bm{x})-f(\bm{x})\big]\, d\bm{x} \bigg] d\bm{s}.
\end{align*}
Due to part (ii) of Proposition \ref{Lemma F.2_SSR} and by the absolute integrability of $f$, we have
\begin{align*}
&\phantom{==}\max_{i=1,\ldots,d} \max_{|m_i|\leq 2a-1} \sup_{\bm{s},\bm{t},\bm{u}\in\R^d} |D_{\bm{m}}^{(1)}(\bm{s},\bm{t},\bm{u})|\\
&\lesssim \frac{1}{\lambda} \max_{i=1,\ldots,d} \max_{|m_i|\leq 2a-1} \sup_{\bm{s},\bm{t},\bm{u}\in\R^d} \int_{\R^d} f(\bm{s}-\bm{x}) f(\bm{t}-(\bm{\omega_m}+\bm{x})) f(\bm{u}-\bm{x})\, d\bm{x}\lesssim \frac{1}{\lambda},
\end{align*}
and a similar reasoning can also be applied to the quantities $D_{\bm{m}}^{(2)}(\bm{s},\bm{t})$, $D_{\bm{m}}^{(3)}(\bm{s})$, and $D_{\bm{m}}^{(4)}$. 
Therefore, the absolute value of \eqref{eq_2} is (ignoring constants) upper bounded by
\begin{align*}
&\frac{1}{\lambda}\Bigg(\frac{c_{2,n}}{\lambda^{d}} \int_{\R^d} \sum_{\bm{m}=-2a+1}^{2a-1} |B(\bm{u}) B(\bm{u}+\bm{\omega_m})|\, d\bm{u} \\
&\phantom{=i}+\frac{c_{2,n}}{\lambda^{2d}} \int_{\R^{2d}} \sum_{\bm{m}=-2a+1}^{2a-1} |B(\bm{u}) B(\bm{u}+\bm{\omega_m}) B(\bm{v}) B(\bm{v}+\bm{\omega_m})|\, d\bm{u} d\bm{v} \Bigg)
=\Landau\left(\frac{(\log a)^{2d}}{\lambda}\right),
\end{align*}
where we applied Lemma \ref{bounds_for_B}, parts (ii) and (iii). The expression in \eqref{eq_21} thus equals 
\begin{align*} 
\frac{(2\pi)^d}{4} \sum_{\bm{m}=-2a+1}^{2a-1} \frac{H(\bm{m})^2}{H(\bm{0})^2} \int_{2\pi \max(-a,-a-\bm{m})/\lambda}^{2\pi \min(a,a-\bm{m})/\lambda} f^3(\bm{x}) f(\bm{\omega_m}+\bm{x})\, d\bm{x}+\Landau\left(\frac{(\log a)^{2d}}{\lambda}+\frac{1}{n}\right),
\end{align*}
as $a,\lambda,n\to\infty$, where we furthermore applied Lemma \ref{convolution_of_h}. 
This last expression can be further simplified: We rewrite it as
\begin{align*} 
&\phantom{+i}\frac{(2\pi)^d}{4} \sum_{\bm{m}=-2a+1}^{2a-1} \frac{H(\bm{m})^2}{H(\bm{0})^2} \int_{2\pi \max(-a,-a-\bm{m})/\lambda}^{2\pi \min(a,a-\bm{m})/\lambda} f^4(\bm{x})\, d\bm{x} \\
&+\frac{(2\pi)^d}{4} \sum_{\bm{m}=-2a+1}^{2a-1} \frac{H(\bm{m})^2}{H(\bm{0})^2} \int_{2\pi \max(-a,-a-\bm{m})/\lambda}^{2\pi \min(a,a-\bm{m})/\lambda} f^3(\bm{x}) \left[f(\bm{\omega_m}+\bm{x})-f(\bm{x})\right]\, d\bm{x}+\Landau\left(\frac{1}{n}\right),
\end{align*}
and from Lemma \ref{bound_variance} it follows that the second term is of order $\Landau(1/(\log\lambda)^3)$ as $\lambda\to\infty$. 
All in all, the cumulant expression corresponding to the partition $\bm{\nu}_1^{\ast}$ in \eqref{examp_for_part_1} thus amounts to 
\begin{align} \label{result_for_nu_1^ast}
&\frac{(2\pi)^d}{4} \sum_{\bm{m}=-2a+1}^{2a-1} \frac{H(\bm{m})^2}{H(\bm{0})^2} \int_{2\pi \max(-a,-a-\bm{m})/\lambda}^{2\pi \min(a,a-\bm{m})/\lambda} f^4(\bm{x})\, d\bm{x}+\Landau\left(\frac{(\log a)^{2d}}{\lambda}+\frac{1}{n}+\frac{1}{(\log\lambda)^3}
\right)
\end{align}
as $\lambda,a,n\rightarrow\infty$.

We now consider the partition 
\begin{align*}
\bm{\nu}_2^{\ast}=\Big\{\{(1,1),(1,3)\},\{(1,2),(2,3)\},\{(1,4),(2,1)\},\{(2,2),(2,4)\}\Big\}\in\mathcal{I}_2
\end{align*}
in \eqref{examp_for_part}, before we explain more generally what kind of partitions lie in the sets $\mathcal{I}_1$ and $\mathcal{I}_2$, respectively.
Using the same calculation steps as above, we obtain
\begin{align} \label{expressions_k1,k2,k1-k2}
&\phantom{=i}\Bigg|\frac{(2\pi)^{2d}\lambda^{3d}}{4 n^{8} H(\bm{0})^4} \sum_{\bm{k}_1,\bm{k}_2=-a}^{a-1}  \sum_{\underline{j}\in\mathcal{D}(8)} \cum_2\left[Y_{t,c}(\underline{j}):(t,c)\in\nu_{2,1}^{\ast}\right]\times\ldots\times\cum_2\left[Y_{t,c}(\underline{j}):(t,c)\in\nu_{2,4}^{\ast}\right]\Bigg|\nonumber\\
&=\Bigg|\frac{c_{2,n}}{4(2\pi)^{2d} H(\bm{0})^4 \lambda^{5d}} \sum_{\bm{k}_1,\bm{k}_2=-a}^{a-1} \int_{\R^{4d}} B(\bm{s}) B(2\tilde{\bm{\omega}}_{\bm{k}_1}-\bm{s}) B(\bm{t}) B(\tilde{\bm{\omega}}_{\bm{k}_2}-\tilde{\bm{\omega}}_{\bm{k}_1}-\bm{t}) B(\bm{u}) B(\bm{u}+2\tilde{\bm{\omega}}_{\bm{k}_2})\nonumber\\
&\phantom{======i==} B(\bm{v}) B(\tilde{\bm{\omega}}_{\bm{k}_2}-\tilde{\bm{\omega}}_{\bm{k}_1}-\bm{v}) f(\bm{s}-\tilde{\bm{\omega}}_{\bm{k}_1}) f(\bm{t}+\tilde{\bm{\omega}}_{\bm{k}_1}) f(\bm{u}+\tilde{\bm{\omega}}_{\bm{k}_2}) f(\bm{v}+\tilde{\bm{\omega}}_{\bm{k}_1})\,d\bm{s} d\bm{t}d\bm{u}d\bm{v}\Bigg|\nonumber\\
&\lesssim \frac{c_{2,n}}{\lambda^{5d}} \sum_{\bm{k}_1,\bm{k}_2=-a}^{a-1} \int_{\R^{4d}} \Big| B(\bm{s}) B(\bm{\omega}_{2\bm{k}_1+1}-\bm{s}) B(\bm{t}) B(\bm{\omega}_{\bm{k}_2-\bm{k}_1}-\bm{t}) B(\bm{u}) B(\bm{u}+\bm{\omega}_{2\bm{k}_2+1})\nonumber\\
&\phantom{======iiii===}B(\bm{v}) B(\bm{\omega}_{\bm{k}_2-\bm{k}_1}-\bm{v})\Big|\, d\bm{s} d\bm{t}d\bm{u}d\bm{v},
\end{align}
where we again used $2\tilde{\bm{\omega}}_{\bm{k}_1}=\bm{\omega}_{2\bm{k}_1+1}$
and $\tilde{\bm{\omega}}_{\bm{k}_2}-\tilde{\bm{\omega}}_{\bm{k}_1}=\bm{\omega}_{\bm{k}_2-\bm{k}_1}$. 
We set $2\bm{k}_1=\bm{m}_1$, $2\bm{k}_2=\bm{m}_2$. Then, it follows that $\bm{k}_2-\bm{k}_1=\frac{1}{2}(\bm{m}_2-\bm{m}_1)$ and the above is bounded by
\begin{align} \label{eq_4}
&\frac{c_{2,n}}{\lambda^{5d}} \sum_{\bm{m}_1,\bm{m}_2=-2a}^{2a} \int_{\R^{4d}} \Big| B(\bm{s}) B\left(\frac{2\pi(\bm{m}_1+1)}{\lambda}-\bm{s}\right) B(\bm{t}) B\left(\frac{\pi(\bm{m}_2-\bm{m}_1)}{\lambda}-\bm{t}\right)\nonumber \\
&\phantom{===========} B(\bm{u}) B\left(\bm{u}+\frac{2\pi(\bm{m}_2+1)}{\lambda}\right) B(\bm{v}) B\left(\frac{\pi(\bm{m}_2-\bm{m}_1)}{\lambda}-\bm{v}\right)\Big|\, d\bm{s} d\bm{t}d\bm{u}d\bm{v}.
\end{align}
Note that by Lemma \ref{properties_of_ell}, part (i), and Lemma \ref{bounds_for_B}, part (i), we have
\begin{align*}
\frac{1}{\lambda^d} \int_{\R^d} \Big|B(\bm{v}) B\left(\frac{\pi(\bm{m}_2-\bm{m}_1)}{\lambda}-\bm{v}\right)\Big| \, d\bm{v}&\lesssim \frac{1}{\lambda^d} \int_{\R^d} L_{\lambda}^{(0)}(\bm{v}) L_{\lambda}^{(0)}\left(\frac{\pi(\bm{m}_2-\bm{m}_1)}{\lambda}-\bm{v}\right) \, d\bm{v}\\
&\lesssim \frac{1}{\lambda^d}\, L_{\lambda}^{(1)}\left(\frac{\pi(\bm{m}_2-\bm{m}_1)}{\lambda}\right) = \ell^{(1)}\big(\pi(\bm{m}_2-\bm{m}_1)\big),
\end{align*}
where the functions $L_{\lambda}^{(0)}$, $L_{\lambda}^{(1)}$ and $\ell^{(1)}$ are defined as in \eqref{L-fct}, \eqref{ell-fct} and \eqref{ell-fct-multi}, and the last expression is obviously uniformly bounded. Therefore, ignoring constants, \eqref{eq_4} is bounded by
\begin{align*}
&\frac{c_{2,n}}{\lambda^{4d}} \sum_{\bm{m}_1,\bm{m}_2=-2a}^{2a} \int_{\R^{3d}} \Big| B(\bm{s}) B\left(\frac{2\pi(\bm{m}_1+1)}{\lambda}-\bm{s}\right) B(\bm{t}) B\left(\frac{\pi(\bm{m}_2-\bm{m}_1)}{\lambda}-\bm{t}\right)\\
&\phantom{==iiii===i==} B(\bm{u}) B\left(\bm{u}+\frac{2\pi(\bm{m}_2+1)}{\lambda}\right)
 \Big|\, d\bm{s} d\bm{t}d\bm{u},
\end{align*}
and by Lemma \ref{bounds_for_B}, part (iv) for $t=3$ and $s=0$, this expression is of order $\Landau(1/\lambda^d)$ and thus vanishes.\\

In general, whether a partition $\bm{\nu}=\{\nu_1,\ldots,\nu_4\}\in\mathcal{I}$ belongs to the set $\mathcal{I}_1$ or the set $\mathcal{I}_2$ depends on the rank of the linear transformation $K$ which maps $(\bm{k}_1^T,\bm{k}_2^T)^T$ to the vector of linear combinations of $\bm{k}_1$ and $\bm{k}_2$ inside the frequency windows $B$. 
For the partition $\bm{\nu}_1^{\ast}$, the only combination of $\bm{k}_1$ and $\bm{k}_2$ involved in the $B$-functions is $\bm{k}_2-\bm{k}_1$, see equation \eqref{only_expression_k_2-k_1}. The matrix $K$ satisfying 
\begin{align*}
K \times \begin{pmatrix}
\bm{k}_1^T \\
\bm{k}_2^T
\end{pmatrix} = \begin{pmatrix}
\bm{0}^T\\
\bm{0}^T\\
(\bm{k}_2-\bm{k}_1)^T\\
(\bm{k}_2-\bm{k}_1)^T
\end{pmatrix}
\end{align*} 
has rank $1$, and the respective cumulant is of order $\Landau(1)$ (concerning the notation, also see \cite{subbarao_suppl}). On the other hand, the cumulant expression corresponding to the partition $\bm{\nu}_2^{\ast}$ is of lower order: The $B$-functions contain the expressions $\bm{k}_1$, $\bm{k}_2$, and $\bm{k}_2-\bm{k}_1$, see equation \eqref{expressions_k1,k2,k1-k2}, and the matrix $K$ satisfying 
\begin{align*}
K \times \begin{pmatrix}
\bm{k}_1^T \\
\bm{k}_2^T
\end{pmatrix} = \begin{pmatrix}
\bm{k}_1^T\\
\bm{k}_2^T\\
(\bm{k}_2-\bm{k}_1)^T\\
(\bm{k}_2-\bm{k}_1)^T
\end{pmatrix}
\end{align*}
has rank $2$. The order of the respective cumulant is then $\Landau(1/\lambda^d)$.
In general, we can say that the order of a certain cumulant expression depends on the rank of the matrix $K$:
\begin{itemize}
\item If $\text{rank}(K)=1$, then the respective cumulant is of order $\Landau(1)$,
\item If $\text{rank}(K)=2$, then the respective cumulant is of order $\Landau(1/\lambda^d)$.
\end{itemize}
 The rank of $K$ equals the number of independent restrictions between the rows and columns of Table \eqref{table_q=2}. Here, a restriction can appear both by a set connecting the two rows, or by a set that contains two frequencies of the same sign within one single row.
Since we only consider the set of indecomposable partitions of Table \eqref{table_q=2}, there must always be at least one restriction between the two rows. The orders of the cumulant expression depending on the number of restrictions are also made precise in Lemma \ref{order_depending_on_number_of_restr}, setting $q=2$ and $i=8$. \\

We now need to clarify which partitions lead to a higher order and which ones lead to a lower order cumulant. To this end, we distinguish two cases (similar to \cite{dette2011}): Either there exists exactly one set of the partition which consists of two elements of the first row (and thus there exists another set that contains two elements of the second row) (case A), or in each set there is one element of the first row and one element of the second row (case B). \\
First look at case A. One can see that the number of restrictions is $1$ if and only if only those entries $(t,c_1)$ and $(t,c_2)$ of Table \eqref{table_q=2} with $c_1$ even and $c_2$ odd (and vice versa) are in the same set (case A1). One can easily see that there are $32$ possibilities for partitions of this form (compare to \cite{dette2011}). Otherwise, i.e. if there is at least one row $t\in\{1,2\}$ such that $(t,c_1)$ and $(t,c_2)$ are in the same set for some $c_1,c_2\in\{1,...,4\}$ with both $c_1,c_2$ odd or both $c_1,c_2$ even, the number of restrictions is $2$ (case A2). 
A visualization is given in Table \ref{table_partitions_A}. \\
Now consider case B. One can see that the number of restrictions is $1$ if and only if either both odd components of the first row are in the same set with both odd components or with both even components of the second row, respectively (case B1). As in \cite{dette2011}, there are $8$ such possibilities. Otherwise, the corresponding cumulant is of lower order (case B2). This is visualized in Table \ref{table_partitions_B}. \\
In summary, we thus see that the set $\mathcal{I}_1$ (leading to a higher order cumulant) consists of the partitions of the forms A1 and B1, while the set $\mathcal{I}_2$ (leading to a lower order cumulant) comprises all partitions of the forms A2 and B2.

\begin{small} 
 \begin{table} [!htbp] 
 \begin{center}
\begin{tabular}{*4{c}|}\hline
    \multicolumn{1}{|c}{\textbf{+}}   & \multicolumn{1}{c|}{\textbf{$-$}}   & \multicolumn{1}{c|}{\textbf{\textcolor{green}{+}}}   & \textbf{\textcolor{red}{$-$}}    \\ \cline{1-4}
    \multicolumn{1}{|c}{\textbf{+}}   & \multicolumn{1}{c|}{\textbf{$-$}}   & \multicolumn{1}{c|}{\textbf{\textcolor{red}{+}}}  & \textcolor{green}{\textbf{$-$}}   \\  \cline{1-4} 
  \end{tabular}
\qquad \qquad \qquad  \qquad 
\begin{tabular}{*4{c}|}\hline
    \multicolumn{1}{|c|}{\textbf{\textcolor{green}{+}}}   & \multicolumn{1}{c|}{\textbf{\textcolor{cyan}{$-$}}}   & \multicolumn{1}{c|}{\textbf{\textcolor{green}{+}}}   & \textbf{$-$}    \\ \cline{1-4} \cline{3-3}
    \multicolumn{1}{|c|}{\textbf{+}}   & \multicolumn{1}{c|}{\textbf{\textcolor{red}{$-$}}}   & \multicolumn{1}{c|}{\textbf{\textcolor{cyan}{+}}}  & \textbf{\textcolor{red}{$-$}}   \\  \cline{1-4} 
  \end{tabular}
    \caption{Examples for partitions of case A. Left: Example of case A1. Right: Example of case A2. Note that the partition $\bm{\nu}_1^{\ast}$ in \eqref{examp_for_part_1} corresponds to case A1, and the partition $\bm{\nu}_2^{\ast}$ in \eqref{examp_for_part} corresponds to case A2.} \label{table_partitions_A}
     \end{center} 
  \end{table}
\end{small}

\begin{small} 
 \begin{table} [!htbp] 
 \begin{center}
\begin{tabular}{*4{c}|}\hline
    \multicolumn{1}{|c|}{\textbf{+}}   & \multicolumn{1}{c|}{\textbf{$-$}}   & \multicolumn{1}{c|}{\textbf{+}}   & \textbf{$-$}    \\ 
    \multicolumn{1}{|c|}{\textbf{+}}   & \multicolumn{1}{c|}{\textbf{$-$}}   & \multicolumn{1}{c|}{\textbf{+}}  & \textbf{$-$}   \\  \cline{1-4} 
  \end{tabular}
\qquad \qquad \qquad  \qquad 
\begin{tabular}{*4{c}|}\hline
    \multicolumn{1}{|c|}{\textbf{+}}   & \multicolumn{1}{c|}{\textbf{$-$}}   & \multicolumn{1}{c|}{\textbf{\textcolor{green}{+}}}   & \textbf{\textcolor{red}{$-$}}    \\  \cline{3-3} \cline{4-4}
    \multicolumn{1}{|c|}{\textbf{+}}   & \multicolumn{1}{c|}{\textbf{$-$}}   & \multicolumn{1}{c|}{\textbf{\textcolor{red}{+}}}  & \textbf{\textcolor{green}{$-$}}   \\  \cline{1-4} 
  \end{tabular}
    \caption{Examples for partitions of case B. Left: Example of case B1. Right: Example of case B2.} \label{table_partitions_B}
     \end{center} 
  \end{table} 
\end{small}

All partitions of case A1 can be dealt with in the same way as the partition $\bm{\nu}_1^{\ast}$ in \eqref{examp_for_part_1}. 
For partitions of the form B1, the calculation works very similarly as the one for the partition $\bm{\nu}_1^{\ast}$. However, the weights $H(\bm{m})^2/H(\bm{0})^2$ need to be replaced by $H(\bm{m})^4/H(\bm{0})^4$, as can be seen from the following example. Let
\begin{align*}
\bm{\nu}_{1}^{\ast \ast} = \{\nu_{1,1}^{\ast\ast}, \nu_{1,2}^{\ast\ast}, \nu_{1,3}^{\ast\ast}, \nu_{1,4}^{\ast\ast} \}=\Big\{\{(1,1),(2,2)\},\{(1,2),(2,3)\},\{(1,3),(2,4)\},\{(1,4),(2,1)\}\Big\},
\end{align*}
then (analogously to before) the cumulant expression of interest is given by
\begin{align*}
&\phantom{i=}\frac{(2\pi)^{2d} \lambda^{3d}}{4 n^{8} H(\bm{0})^4} \sum_{\bm{k}_1,\bm{k}_2=-a}^{a-1}  \sum_{\underline{j}\in\mathcal{D}(8)} \cum_2\left[Y_{t,c}(\underline{j}):(t,c)\in\nu_{1,1}^{\ast \ast}\right]\times\ldots\times\cum_2\left[Y_{t,c}(\underline{j}):(t,c)\in\nu_{1,4}^{\ast \ast}\right]\\
&=\frac{(2\pi)^{d} c_{2,n}}{4 (2\pi\lambda)^{4d} H(\bm{0})^4} \int_{\R^{4d}} \sum_{\bm{m}=-2a+1}^{2a+1} B(\bm{s}) B(\bm{s}+\bm{\omega_m}) B(\bm{t}) B(\bm{t}+\bm{\omega_m}) B(\bm{u}) B(\bm{u}+\bm{\omega_m}) B(\bm{v})\\ 
& \phantom{=i} B(\bm{v}+\bm{\omega_m})  \Bigg[\int_{2\pi\max(-a,-a-\bm{m})/\lambda}^{2\pi\min(a,a-\bm{m})/\lambda} f(\bm{s}-\bm{x}) f(\bm{t}-\bm{x}) f(\bm{u}-\bm{x}) f(\bm{v}-\bm{x})\, d\bm{x}\Bigg]  d\bm{s} d\bm{t} d\bm{u} d\bm{v}+\Landau\left(\frac{1}{\lambda^2}\right).
\end{align*}
Using the same arguments as in the calculation for the partition $\bm{\nu}_1^{\ast}$, this expression equals
\begin{align*}
&\frac{(2\pi)^{d}}{4 (2\pi\lambda)^{4d} H(\bm{0})^4} \sum_{\bm{m}=-2a+1}^{2a+1} \int_{\R^{4d}}  B(\bm{s}) B(\bm{s}+\bm{\omega_m}) B(\bm{t}) B(\bm{t}+\bm{\omega_m}) B(\bm{u}) B(\bm{u}+\bm{\omega_m}) B(\bm{v}) B(\bm{v}+\bm{\omega_m})\\
&\phantom{===========} d\bm{s} d\bm{t} d\bm{u} d\bm{v}  \Bigg[\int_{2\pi\max(-a,-a-\bm{m})/\lambda}^{2\pi\min(a,a-\bm{m})/\lambda} f(\bm{x})^4 \, d\bm{x}\Bigg]+\Landau\left(\frac{(\log a)^{2d}}{\lambda}+\frac{1}{n}
\right),
\end{align*}
and by Lemma \ref{convolution_of_h} we obtain that this is equal to
\begin{align*}
\frac{(2\pi)^{d}}{4} \sum_{\bm{m}=-2a+1}^{2a+1} \frac{H(\bm{m})^4}{H(\bm{0})^4} \int_{2\pi\max(-a,-a-\bm{m})/\lambda}^{2\pi\min(a,a-\bm{m})/\lambda} f(\bm{x})^4 \, d\bm{x}+\Landau\left(\frac{(\log a)^{2d}}{\lambda}+\frac{1}{n}\right).
\end{align*}
Now all partitions corresponding to the case B1 can be dealt with in this way. Since there are $32$ possibilities for partitions of case A1 and $8$ possibilities for partitions of case B1, we thus obtain 
\begin{align*}
\lambda^d \Var \big[\hat{D}_{1,d,\lambda,a}\big]&=32 \left( \frac{(2\pi)^d}{4} \sum_{\bm{m}=-2a+1}^{2a-1} \frac{H(\bm{m})^2}{H(\bm{0})^2} \int_{2\pi\max(-a,-a-\bm{m})/\lambda}^{2\pi\min(a,a-\bm{m})/\lambda} f(\bm{x})^4 \, d\bm{x}\right)\\
&+8 \left(\frac{(2\pi)^d}{4} \sum_{\bm{m}=-2a+1}^{2a-1} \frac{H(\bm{m})^4}{H(\bm{0})^4} \int_{2\pi\max(-a,-a-\bm{m})/\lambda}^{2\pi\min(a,a-\bm{m})/\lambda} f(\bm{x})^4 \, d\bm{x}\right)\\
&+\Landau\left(\frac{(\log a)^{2d}}{\lambda}+\frac{1}{(\log\lambda)^3}+\frac{\lambda^d}{n}\right)
\end{align*} 
as $a,\lambda,n\to\infty$.
$\hfill \Box$

%
\subsection{Proof of Theorem \ref{asymptotic_normality}}

Considering Theorem \ref{expectation_theo}, it suffices to show that
\begin{align} \label{cumulant_orders}
\lambda^{dq/2}\cum_q\big(\hat{D}_{1,d,\lambda,a}\big)=\begin{cases}
\mathcal{O}(\lambda^{d/2}), &\quad q=1,\\
\mathcal{O}(1), &\quad q=2,\\
o(1), &\quad q\geq 3.
\end{cases}
\end{align}
Note that even if \eqref{cumulant_orders} is satisfied, asymptotic normality does not hold for all $d\in\N$ but only for $d\leq 3$ since the bias term in Theorem \ref{expectation_theo} is of order $o(\lambda^{-d/2})$ for $d\leq 3$ only. \\
We now fix $q\geq 1$ and denote by $\mathcal{I}(q)$ the set of all indecomposable partitions of the table

\begin{equation} 
\begin{matrix} 
(1,1) & (1,2) & (1,3) & (1,4)\\
\vdots & \vdots & \vdots & \vdots\\
(q,1) & (q,2) & (q,3) & (q,4)\, .
\end{matrix}
\label{table}
\end{equation}
For $t=1,\ldots,q$ and $c=1,\ldots,4$ recall the definition of the random variables
\begin{align*}
Y_{t,c}(\underline{j}):=h\left(\frac{\bm{s}_{j_{c+4(t-1)}}}{\lambda}\right) Z(\bm{s}_{j_{c+4(t-1)}}) \exp\big((-1)^{c+1}\, \im\, \bm{s}_{j_{c+4(t-1)}}^T \tilde{\bm{\omega}}_{\bm{k}_t} \big),
\end{align*}
where $\underline{j}=(j_1,\ldots,j_{4q})$ and $\bm{k}_t=(k_{t1},\ldots,k_{td})^T$.
We now generalize the arguments from the proof of Theorem \ref{expectation_theo}, in particular using the cumulant results A), B), and C).
First, we note that due to result A), we have
\begin{align*} 
\lambda^{dq/2} \cum_q\big(\hat{D}_{1,d,\lambda,a}\big) &= \frac{(2\pi)^{dq}\lambda^{3dq/2}}{2^q n^{4q} H(\bm{0})^{2q}} \sum_{\bm{k}_1,\ldots,\bm{k}_q=-a}^{a-1} \sum_{\underline{j}\in \mathcal{D}(q)} \cum\big[Y_{1,1}(\underline{j}) Y_{1,2}(\underline{j})  Y_{1,3}(\underline{j})  Y_{1,4}(\underline{j}),\\
&\phantom{================ii===}  \ldots,\\
&\phantom{================ii===} Y_{q,1}(\underline{j})  Y_{q,2}(\underline{j})  Y_{q,3}(\underline{j})  Y_{q,4}(\underline{j}) \big]\\
&=\frac{(2\pi)^{dq}\lambda^{3dq/2}}{2^q n^{4q}H(\bm{0})^{2q}} \sum_{\bm{k}_1,\ldots,\bm{k}_q=-a}^{a-1}  \sum_{\bm{\nu}=\{\nu_1,\ldots,\nu_G\}\in\mathcal{I}(q)} \sum_{\underline{j}\in \mathcal{D}(q)} \cum_{|\nu_1|}\big[Y_{t,c}(\underline{j})\, : \, (t,c)\in \nu_1\big]  \\
&\phantom{======iii====ii=ii====i}\times \ldots \times \cum_{|\nu_G|}\big[Y_{t,c}(\underline{j})\, :\, (t,c)\in \nu_G\big],
\end{align*}
where
\begin{align} \label{set_D(q)}
\mathcal{D}(q):= \{\underline{j}=(j_1,\ldots,j_{4q})\in\{1,\ldots,n\}^{4q}\, :  \, & (j_1,j_2,j_3,j_4)\in\mathcal{E}, (j_5,j_6,j_7,j_8)\in\mathcal{E}, \ldots,\nonumber\\
& (j_{4q-3},j_{4q-2},j_{4q-1},j_{4q})\in\mathcal{E}\}
\end{align}
is the set of possible indices $\underline{j}$ for a given row length $q$, and $\mathcal{E}$ is defined in \eqref{set_E}. Due to result B), we can assume that $|\nu_g|=2p_g$, $p_g\in\{1,\ldots,2q\}$, for all $g=1,\ldots,G$. For a given partition $\bm{\nu}=\{\nu_1,\ldots,\nu_G\}\in\mathcal{I}(q)$, result C) yields that any tupel $\underline{j}\in\mathcal{D}(q)$ contributing a nonzero term to the above sum can have at most
\begin{align*} 
\sum_{g=1}^G (p_g+1) = \frac{1}{2}\sum_{g=1}^G |\nu_g| + G =2q+G
\end{align*}
different elements.
For $i=4,\ldots,4q$, we thus define the subsets $\mathcal{D}(q,i)$ of $\mathcal{D}(q)$ by
\begin{align} \label{set_D(q,i)}
\mathcal{D}(q,i):=\{\underline{j}\in \mathcal{D}(q) \, : \, i\text{ elements in } \underline{j} \text{ are different}\}
\end{align}
and obtain
\begin{align} \label{cumulant_general}
&\phantom{=i}\lambda^{dq/2} \cum_q\big(\hat{D}_{1,d,\lambda,a}\big)\nonumber \\
&=\frac{(2\pi)^{dq}\lambda^{3dq/2}}{2^q n^{4q}H(\bm{0})^{2q}} \sum_{\bm{k}_1,\ldots,\bm{k}_q=-a}^{a-1}  \sum_{\bm{\nu}=\{\nu_1,\ldots,\nu_G\}\in\mathcal{I}(q)} \sum_{i=4}^{2q+G} \sum_{\underline{j}\in \mathcal{D}(q,i)} \cum_{|\nu_1|}\big[Y_{t,c}(\underline{j})\, : \, (t,c)\in \nu_1\big] \nonumber\\
&\phantom{==========ii=============} \times \ldots  \times \cum_{|\nu_G|}\big[Y_{t,c}(\underline{j})\, :\, (t,c)\in \nu_G\big].
\end{align}
To determine the order of the above expression, it suffices to consider a fixed partition $\bm{\nu}=\{\nu_1,\ldots,\nu_G\}\in\mathcal{I}(q)$.
We first illustrate the procedure for some partition $\bm{\nu}^{\ast}=\{\nu_1^{\ast},\ldots,\nu_{G^{\ast}}^{\ast}\}\in\mathcal{I}(q)$ satisfying $|\nu_g^{\ast}|=2$ for all $g=1,\ldots,G^{\ast}$, i.e. we have $G^{\ast}=2q$ groups. Before considering general values of $i$, we moreover start with the case of maximally many elements allowed, i.e. $i=2q+2q=4q$. After that, we will consider more general group sizes. For now, the cumulant expression of interest is thus given by
\begin{align} \label{order_for_fixed_partition}
\frac{(2\pi)^{dq}\lambda^{3dq/2}}{2^q n^{4q} H(\bm{0})^{2q}} \sum_{\bm{k}_1,\ldots,\bm{k}_q=-a}^{a-1}  \sum_{\underline{j}\in \mathcal{D}(q,4q)} \cum_2\big[Y_{t,c}(\underline{j})\, : \, (t,c)\in \nu_1^{\ast}\big] \times \ldots \times \cum_2\big[Y_{t,c}(\underline{j})\, :\, (t,c)\in \nu_{2q}^{\ast}\big].
\end{align}
Using the notation from Table \eqref{table}, each set $\nu_g^\ast$ with $|\nu_g^\ast|=2$ is a set of two tuples $(t_1,c_1),(t_2,c_2)$ with $t_1,t_2\in\{1,\ldots,q\}$ and $c_1,c_2\in\{1,\ldots,4\}$.
For the following argument, we use a simpler notation for the elements of Table \eqref{table} by identifying each tuple $(t,c)$ with one of the numbers between $1$ and $4q$, respectively. As an example, we identify the element $(1,1)$ from Table \eqref{table} with the number $1$, the element $(2,3)$ with the number $7$, and the element $(q,4)$ with the number $4q$.  
Using this simplified notation, we can then express each set $\nu_g^\ast$ as a single tuple $(a,b)$ for $a,b\in\{1,\ldots,4q\}$ and write 
\begin{align} \label{simplified_notation}
\nu_g^{\ast}=(\nu_{g,1}^{\ast},\nu_{g,2}^{\ast})
\end{align}
for $g=1,\ldots,2q$. Recall the definition of the quantity $c_{q,n}$ in \eqref{c_q,n} 
and assume without loss of generality that $n>4q$.
In the same way as in the proofs of Theorems \ref{expectation_theo}, we then obtain that the expression in \eqref{order_for_fixed_partition} equals
\begin{align} \label{eq8}
&\phantom{=i}\frac{(2\pi)^{dq}\lambda^{3dq/2} c_{q,n}}{2^q H(\bm{0})^{2q}} \sum_{\bm{k}_1,\ldots,\bm{k}_q=-a}^{a-1} \frac{1}{ \lambda^{4dq}} \int_{[-\lambda/2,\lambda/2]^{4dq}}  \prod_{g=1}^{2q} c(\bm{s}_{\nu_{g,1}^{\ast}}-\bm{s}_{\nu_{g,2}^{\ast}})\nonumber \\
&\phantom{=======} \prod_{t=1}^q \left(\prod_{c=1}^4 h\left(\frac{\bm{s}_{c+4(t-1)}}{\lambda}\right) \exp\big(\im\, \tilde{\bm{\omega}}_{\bm{k}_t}^T (\bm{s}_{4t-3}-\bm{s}_{4t-2}+\bm{s}_{4t-1}-\bm{s}_{4t})\big) \right)\prod_{g=1}^{2q} d \bm{s}_{\nu_{g,1}^{\ast}} \, d \bm{s}_{\nu_{g,2}^{\ast}} \nonumber\\
&\eqsim \lambda^{3dq/2} c_{q,n} \sum_{\bm{k}_1,\ldots,\bm{k}_q=-a}^{a-1} \frac{1}{ \lambda^{4dq}}  \int_{\R^{2dq}} \Bigg[\int_{[-\lambda/2,\lambda/2]^{4dq}} \prod_{g=1}^{2q} f(\bm{x}_g) \exp\big(\im \bm{x}^T_g (\bm{s}_{\nu_{g,1}^{\ast}}-\bm{s}_{\nu_{g,2}^{\ast}})\big)\nonumber\\
&\phantom{====} \prod_{t=1}^q \left(\prod_{c=1}^4 h\left(\frac{\bm{s}_{c+4(t-1)}}{\lambda}\right)\exp\big(\im\, \tilde{\bm{\omega}}_{\bm{k}_t}^T (\bm{s}_{4t-3}-\bm{s}_{4t-2}+\bm{s}_{4t-1}-\bm{s}_{4t})\big)\right) \prod_{g=1}^{2q} d \bm{s}_{\nu_{g,1}^{\ast}} \, d \bm{s}_{\nu_{g,2}^{\ast}} \Bigg] \prod_{g=1}^{2q} d \bm{x}_g\nonumber\\
&=\frac{\lambda^{3dq/2} c_{q,n}}{\lambda^{4dq}} \sum_{\bm{k}_1,\ldots,\bm{k}_q=-a}^{a-1} \int_{\R^{2dq}} \prod_{g=1}^{2q} f(\bm{x}_g) B(\bm{x}_g+\tilde{\bm{\omega}}_{\hat{\bm{k}}_{2g-1}}) B(\bm{x}_g+\tilde{\bm{\omega}}_{\hat{\bm{k}}_{2g}}) \, d \bm{x}_g ,
\end{align}
where the function $B$ is defined in \eqref{Four_trafo_of_h}, and for $j=1,\ldots,4q$, the quantities $\hat{\bm{k}}_j$ are $d$-dimensional vectors taken from the set $$\{\bm{k}_1,-\bm{k}_1-1,\bm{k}_2,-\bm{k}_2-1,\ldots,\bm{k}_q,-\bm{k}_q-1\},$$
which are determined by the partition $\bm{\nu}^{\ast}$ under consideration. \\
\noindent\rule{\textwidth}{1pt}
\begin{example} \label{ex_matrix_k}
Let $q=3$ and $\underline{j}\in\mathcal{D}(q,4q)$. 
For
\begin{align*}
\bm{\nu}=\Big\{\{(1,1),(1,2)\},\{(2,1),(2,2)\},\{(3,1),(3,2)\},\{(1,3),(2,3)\},\{(2,4),(3,4)\},\{(1,4),(3,3)\}\Big\},
\end{align*}
it holds that
\begin{align*}
&\phantom{==}\sum_{\bm{k}_1,\bm{k}_2,\bm{k}_3=-a}^{a-1} \cum_2[Y_{t,c}(\underline{j})\, : \, (t,c)\in \nu_1] \times \ldots \times \cum_2[Y_{t,c}(\underline{j})\, :\, (t,c)\in \nu_6]\\
&= \sum_{\bm{k}_1,\bm{k}_2,\bm{k}_3=-a}^{a-1} \int_{\R^{6d}} \frac{1}{\lambda^{12d}}\, f(\bm{x}_1) \ldots f(\bm{x}_6) B^2(\bm{x}_1+\tilde{\bm{\omega}}_{\bm{k}_1}) B^2(\bm{x}_2+\tilde{\bm{\omega}}_{\bm{k}_2}) B^2(\bm{x}_3+\tilde{\bm{\omega}}_{\bm{k}_3}) B(\bm{x}_4+\tilde{\bm{\omega}}_{\bm{k}_1})\\
&\phantom{=========}  B(\bm{x}_4-\tilde{\bm{\omega}}_{\bm{k}_2}) B(\bm{x}_5+\tilde{\bm{\omega}}_{\bm{k}_2}) B(\bm{x}_5-\tilde{\bm{\omega}}_{\bm{k}_3}) B(\bm{x}_6+\tilde{\bm{\omega}}_{\bm{k}_1}) B(\bm{x}_6+\tilde{\bm{\omega}}_{\bm{k}_3}) \, d\bm{x}_1 \ldots\, d\bm{x}_6.
\end{align*}
Since $-\tilde{\bm{\omega}}_{\bm{k}}=\tilde{\bm{\omega}}_{-\bm{k}-1}$, this expression equals
\begin{align*}
&\sum_{\bm{k}_1,\bm{k}_2,\bm{k}_3=-a}^{a-1} \int_{\R^{6d}} \frac{1}{\lambda^{12d}}\, f(\bm{x}_1) \ldots f(\bm{x}_6) B^2(\bm{x}_1+\tilde{\bm{\omega}}_{\bm{k}_1}) B^2(\bm{x}_2+\tilde{\bm{\omega}}_{\bm{k}_2}) B^2(\bm{x}_3+\tilde{\bm{\omega}}_{\bm{k}_3}) B(\bm{x}_4+\tilde{\bm{\omega}}_{\bm{k}_1}) \\
&\phantom{====}  B(\bm{x}_4+\tilde{\bm{\omega}}_{-\bm{k}_2-1}) B(\bm{x}_5+\tilde{\bm{\omega}}_{\bm{k}_2}) B(\bm{x}_5+\tilde{\bm{\omega}}_{-\bm{k}_3-1}) B(\bm{x}_6+\tilde{\bm{\omega}}_{\bm{k}_1}) B(\bm{x}_6+\tilde{\bm{\omega}}_{\bm{k}_3}) \, d\bm{x}_1 \ldots\, d\bm{x}_6,
\end{align*}
and we therefore have
\begin{align*}
&\hat{\bm{k}}_1=\bm{k}_1, \qquad \hat{\bm{k}}_2=\bm{k}_1, \qquad \hat{\bm{k}}_3=\bm{k}_2,  \qquad \hat{\bm{k}}_4=\bm{k}_2, \qquad \hat{\bm{k}}_5=\bm{k}_3, \qquad \hat{\bm{k}}_6=\bm{k}_3,\\
&\hat{\bm{k}}_7=\bm{k}_1, \qquad \hat{\bm{k}}_8=-\bm{k}_2-1, \qquad \hat{\bm{k}}_9=\bm{k}_2, \qquad \hat{\bm{k}}_{10}=-\bm{k}_3-1, \qquad \hat{\bm{k}}_{11}=\bm{k}_1, \qquad \hat{\bm{k}}_{12}=\bm{k}_3.
\end{align*}
\end{example}
\noindent\rule{\textwidth}{1pt}

We now perform a change of variables by setting
\begin{align*} 
\bm{x}_g+\tilde{\bm{\omega}}_{\hat{\bm{k}}_{2g-1}}=\bm{u}_g \qquad g=1,\ldots,2q,
\end{align*}
so the absolute value of \eqref{eq8} is less or equal than
\begin{align} \label{jetzt_hier}
&\phantom{=i}\frac{\lambda^{3dq/2}}{\lambda^{4dq}} \sum_{\bm{k}_1,\ldots,\bm{k}_q=-a}^{a-1} \int_{\R^{2dq}} \prod_{g=1}^{2q} \left|f(\bm{u}_g - \tilde{\bm{\omega}}_{\hat{\bm{k}}_{2g-1}}) B(\bm{u}_g) B(\bm{u}_g+\tilde{\bm{\omega}}_{\hat{\bm{k}}_{2g}} - \tilde{\bm{\omega}}_{\hat{\bm{k}}_{2g-1}})\right| d\bm{u}_g \nonumber\\
&=\frac{\lambda^{3dq/2}}{\lambda^{4dq}} \sum_{\bm{k}_1,\ldots,\bm{k}_q=-a}^{a-1} \int_{\R^{2dq}} \prod_{g=1}^{2q} \left|f(\bm{u}_g - \tilde{\bm{\omega}}_{\hat{\bm{k}}_{2g-1}}) B(\bm{u}_g) B(\bm{u}_g+\bm{\omega}_{\hat{\bm{k}}_{2g}-\hat{\bm{k}}_{2g-1}})\right| d\bm{u}_g.
\end{align}
We now denote by $K\in\{-2,-1,0,1,2\}^{2q\times q}$ and $R\in\{-1,0,1\}^{2q \times d}$ the matrices satisfying the equation
\begin{align} \label{rank}
K \times \begin{pmatrix}
\bm{k}^T_1 \\
\bm{k}^T_2 \\
\vdots\\
\bm{k}^T_q
\end{pmatrix}+R=\begin{pmatrix}
(\hat{\bm{k}}_{2}-\hat{\bm{k}}_1)^T\\
(\hat{\bm{k}}_{4}-\hat{\bm{k}}_3)^T\\
\vdots\\
(\hat{\bm{k}}_{4q}-\hat{\bm{k}}_{4q-1})^T
\end{pmatrix}.
\end{align}
\noindent\rule{\textwidth}{1pt}
\begin{example} \label{ex_matrices_K_R}
Using the same partition as in Example \ref{ex_matrix_k}, we obtain 
\begin{align*}
\begin{pmatrix}
(\hat{\bm{k}}_{2}-\hat{\bm{k}}_1)^T\\
(\hat{\bm{k}}_{4}-\hat{\bm{k}}_3)^T\\
(\hat{\bm{k}}_{6}-\hat{\bm{k}}_5)^T\\
(\hat{\bm{k}}_{8}-\hat{\bm{k}}_7)^T\\
(\hat{\bm{k}}_{10}-\hat{\bm{k}}_9)^T\\
(\hat{\bm{k}}_{12}-\hat{\bm{k}}_{11})^T
\end{pmatrix}=\begin{pmatrix}
\bm{0}^T \\
\bm{0}^T \\
\bm{0}^T \\
(-\bm{k}_2-\bm{k}_1-1)^T \\
(-\bm{k}_3-\bm{k}_2-1)^T \\
(\bm{k}_3-\bm{k}_1)^T 
\end{pmatrix}
\end{align*}
and therefore
\begin{align*}
K=\begin{pmatrix}
0 & 0 & 0 \\
0 & 0 & 0 \\
0 & 0 & 0 \\
-1 & -1 & 0 \\
0 & -1 & -1 \\
-1 & 0 & 1
\end{pmatrix}, \qquad  \qquad R=\begin{pmatrix}
0 & \ldots & 0 \\
0 & \ldots & 0 \\
0 & \ldots & 0 \\
-1 & \ldots & -1 \\
-1 & \ldots & -1 \\
0 & \ldots & 0
\end{pmatrix}.
\end{align*}
\end{example}
\noindent\rule{\textwidth}{1pt}
We claim that the order of the expression \eqref{jetzt_hier} is determined by the rank of the matrix $K$ in \eqref{rank}.
In order to demonstrate this, let $\text{rank}(K)=r$. We assume without loss of generality that $r\leq q-1$ and that
\begin{align*}
\begin{pmatrix}
(\hat{\bm{k}}_{2}-\hat{\bm{k}}_1)^T\\
(\hat{\bm{k}}_{4}-\hat{\bm{k}}_3)^T\\
\vdots\\
(\hat{\bm{k}}_{4q}-\hat{\bm{k}}_{4q-1})^T
\end{pmatrix}=\begin{pmatrix}
(\bm{k}_2-\bm{k}_1)^T\\
(\bm{k}_3-\bm{k}_2)^T\\
\vdots\\
(\bm{k}_{r+1}-\bm{k}_{r})^T\\
(\bm{k}_{1}-\bm{k}_{r+1})^T\\
(\bm{k}_{r+2}-\bm{k}_{r+2})^T\\
\vdots\\
(\bm{k}_q-\bm{k}_q)^T\\
(\bm{k}_1-\bm{k}_1)^T\\
\vdots\\
(\bm{k}_q-\bm{k}_q)^T
\end{pmatrix},
\end{align*}
i.e. it holds that
\begin{align} \label{special_form_of_K_and_R}
K=\begin{pmatrix}
-1 & 1 & 0 & \ldots & \ldots & 0 & 0 & \ldots & 0 \\
0 & -1 & 1 & 0 & \ldots & 0 & 0 &\ldots & 0\\
\vdots & \vdots & \vdots & \vdots & \vdots &  \vdots & \vdots & \vdots &  \vdots\\
0 & \ldots & \ldots & \ldots & -1 & 1&0 &\ldots & 0 \\
1 &  \ldots & \ldots & \ldots & \ldots & -1&0 &\ldots&  0\\
0 & \ldots & \ldots & \ldots & \ldots & 0&0 &\ldots& 0\\
 \vdots & \vdots & \vdots & \vdots & \vdots &  \vdots& \vdots & \vdots &  \vdots\\
0 & \ldots & \ldots & \ldots & \ldots & 0&0 &\ldots &0
\end{pmatrix} \qquad \text{and} \qquad R=\bm{0},
\end{align}
where the matrix $K$ consists of $2q-(r+1)$ rows and $q-(r+1)$ columns with $0$-entries only and $\bm{0}$ denotes the null matrix.
We then set 
\begin{align*}
\bm{k}_1=\bm{k}_1, \qquad \bm{m}_1=\bm{k}_2-\bm{k}_1, \qquad \bm{m}_2=\bm{k}_3-\bm{k}_2, \qquad \ldots, \qquad \bm{m}_{r}=\bm{k}_{r+1}-\bm{k}_{r}, 
\end{align*}
which obviously yields $\bm{k}_{1}-\bm{k}_{r+1}=-\sum_{j=1}^{r} \bm{m}_j$. Ignoring constants, it then follows that the absolute value of \eqref{jetzt_hier} can be bounded by
\begin{align} \label{eq3}
&\frac{\lambda^{3dq/2}}{\lambda^{4dq}} \int_{\R^{2dq}} \sum_{\bm{m}_1,\ldots,\bm{m}_{r}=-2a+1}^{2a-1} G_{\lambda,a} (\bm{u}_1,\bm{u}_{r+2},\ldots,\bm{u}_{q}) \prod_{g=1}^r \Big| B(\bm{u}_g) B(\bm{u}_g+\bm{\omega}_{\bm{m}_g})\Big|\nonumber\\
&\phantom{===============}\Big|B(\bm{u}_{r+1}) B(\bm{u}_{r+1}-\sum_{j=1}^r \bm{\omega}_{\bm{m}_j})\Big| \prod_{g=r+2}^{2q} B^2(\bm{u}_g) \prod_{g=1}^{2q} d\bm{u}_g,
\end{align}
where 
\begin{align*}
G_{\lambda,a} (\bm{u}_1,\bm{u}_{r+2},\ldots,\bm{u}_{q}):= \sum_{\bm{k}_1,\bm{k}_{r+2},\ldots,\bm{k}_q = -a}^{a-1} 
\prod_{g\in\{1,r+2,\ldots,q\}} f(\bm{u}_g-\tilde{\bm{\omega}}_{\bm{k}_g}).
\end{align*}
Note that $G_{\lambda,a} (\bm{u}_1,\bm{u}_{r+2},\ldots,\bm{u}_{q})=\Landau(\lambda^{d(q-r)})$, since
the expression
$
\frac{2\pi}{\lambda} \sum_{k=-a}^{a-1} f(u-\tilde{\omega}_k) 
$
is uniformly bounded.
Besides, we can use Lemma \ref{bounds_for_B} (iv) for $t=r+1$ and $s=0$ and Lemma \ref{orders_of_B} (i) to see that the expression
\begin{align*}
&\int_{\R^{2dq}} \sum_{\bm{m}_1,\ldots,\bm{m}_{r}=-2a+1}^{2a-1}  \prod_{g=1}^r \Big| B(\bm{u}_g) B(\bm{u}_g+\bm{\omega}_{\bm{m}_g})\Big|\nonumber\Big|B(\bm{u}_{r+1}) B(\bm{u}_{r+1}-\sum_{j=1}^r \bm{\omega}_{\bm{m}_j})\Big| \prod_{g=r+2}^{2q} B^2(\bm{u}_g) \prod_{g=1}^{2q} d\bm{u}_g
\end{align*}
is of order $\Landau(\lambda^{2dq})$.
The expression in \eqref{eq3} is thus of order
\begin{align} \label{order_for_2,2,2_and_maximal_elements}
\Landau\left(\lambda^{d\big(3q/2\,-\,4q \,+ \,(q-r)\, +\, 2q\big)}\right)=\Landau\left(\lambda^{d(q/2-r)}\right).
\end{align}

Recall that this order corresponds to the special case where $|\nu_g^{\ast}|=2$ for all $g=1,\ldots,2q$ and $i=4q$, i.e. the maximally allowed number of different elements in $\underline{j}$ is considered. We now interpret the components of this order: While the value $3dq/2$ corresponds to the prefactor in front of the sums over the cumulant expression, the value $4q$ is the number of integrals over the region $[-\lambda/2,\lambda/2]^d$, see \eqref{eq8}. The term $q-r$ is the difference between the number of sums $q$ and the rank of the matrix $K$, and the term $2q$ is the number of sets $G^{\ast}$. It is important to note that due to Lemma \ref{bounds_for_B}, the $dr$ sums in \eqref{eq3} do not need to be bounded separately. 
Note that for the derivation of the above result, we assumed a special form of the matrices $K$ and $R$, see \eqref{special_form_of_K_and_R}. However, it is easy to see that the use of Lemma \ref{bounds_for_B} and Lemma \ref{orders_of_B} always leads to the order in \eqref{order_for_2,2,2_and_maximal_elements} if $|\nu_g^{\ast}|=2$ for all $g=1,\ldots,2q$ and the number of different elements in $\underline{j}$ is $4q$.\\

In a next step, we will explain how the order in \eqref{order_for_2,2,2_and_maximal_elements} changes for partitions $\bm{\nu}=\{\nu_1,\ldots,\nu_{2q}\}$ with $|\nu_g|=2$ for all $g=1,\ldots,2q$, even when considering an arbitrary number of different elements $i$ in $\underline{j}$. We will proceed in two steps: First, we show that setting $i=4q-1$ leads to an order change of at most $\lambda^d/n$ (which by Assumption \ref{assumption_on_sampling_scheme} goes to $0$). Afterwards, we will explain that for $i\leq 4q-1$, setting two additional elements equal leads to a further order change of at most $a^d/n$ (which by Assumption \ref{assumptions_on_a} is bounded), so the order cannot increase any more. In fact, in most cases the change of order is of the form $\lambda^d/n$ instead of $a^d/n$ as well, but this further distinction will not be necessary for our purposes.\\
Before we state our result in form of a lemma, note that the rank $r$ of the matrix $K$ is completely determined by the number of independent restrictions between the rows and the columns of Table \eqref{table}. More precisely, the occurrence of two frequencies $\tilde{\bm{\omega}}_{\bm{k}_i}$ and $\tilde{\bm{\omega}}_{\bm{k}_j}$ with $i\neq j$ within the same set leads to a rank increase. Moreover, if within one set two frequencies of the same row and sign appear twice, the rank is increased as well.
Note that in both cases the rank is naturally only affected if the particular constraint is independent from the already considered ones. In the following, we will use the broader term \textit{restrictions} instead of $\text{rank}(K)$, since it is more natural in terms of Table \eqref{table}. We thus claim the following:

\begin{lemma} \label{order_depending_on_number_of_restr}
For any partition $\bm{\nu}=\{\nu_1,\ldots,\nu_{2q}\}$ of Table \eqref{table} with $|\nu_g|=2$ for $g=1,\ldots,2q$, it holds that
\begin{align*}
&\phantom{==}\frac{\lambda^{3dq/2}}{n^{4q}} \sum_{\bm{k}_1,\ldots,\bm{k}_q=-a}^{a-1} \sum_{\underline{j}\in\mathcal{D}(q,i)} \cum_2\big[Y_{t,c}(\underline{j}):(t,c)\in\nu_1\big] \times \ldots \times \cum_2\big[Y_{t,c}(\underline{j}):(t,c)\in\nu_{2q}\big]\\
&=\begin{cases}  \Landau(\lambda^{d(q/2-\#\text{restrictions})}) \quad &\text{if } i=4q,\\
\Landau\left(\frac{\lambda^{d(q/2+1-\#\text{restrictions})}}{n}\right) \quad &\text{if } i\leq 4q-1.
\end{cases}
\end{align*}
Here, the term $\#\text{restrictions}$ denotes the number of independent restrictions between the rows and columns of Table \eqref{table}, or equivalently the rank of the matrix $K$ in \eqref{rank}. \label{page100}
\end{lemma}

\begin{proof}
The part for $i=4q$ has already been proven. 
Now let $i= 4q-1$, i.e. we assume that $j_k=j_l$ for exactly one pair $(k,l)\in\{1,\ldots,4q\}^2$ with $k\neq l$. 
Note that while the value $3dq/2$ in front of the sums over the cumulant expression obviously remains the same for all $i\in\{4,\ldots,4q\}$, the value $4q$ in \eqref{order_for_2,2,2_and_maximal_elements} is the number of integrals over the region $[-\lambda/2,\lambda/2]^d$. 
If $i=4q-1$, we do not necessarily obtain $4q$ of these integrals (as in the case where $i=4q$), but more generally we have
\begin{align*}
\# \{\text{number of integrals over }  [-\lambda/2,\lambda/2]^d \}&= \sum_{g=1}^{2q} \big(\#\, \{\text{different indices} \text{ in } \underline{j}\text{ belonging to } \nu_g\} \big)\\
& = \begin{cases}
4q-1,  &\text{if } \exists\, \nu_g=\{(t_1,c_1),(t_2,c_2)\} \text{ such that}\\ 
&j_{c_1+4(t_1-1)}=j_{c_2+4(t_2-1)},\\
4q,  &\text{if } \nexists\, \nu_g=\{(t_1,c_1),(t_2,c_2)\} \text{ such that}\\ 
&j_{c_1+4(t_1-1)}=j_{c_2+4(t_2-1)}.
\end{cases}
\end{align*}

We thus lose one integral over the region $[-\lambda/2,\lambda/2]^d$ if the two equal indices in $\underline{j}$ belong to the same set, leading to an order change of $\lambda^d$. Here, we say that an index $j_k$ in $\underline{j}$ belongs to $\nu_g$ if $k=c+4(t-1)$ for some $(t,c)\in\nu_g$. Moreover, note that $|\mathcal{D}(q,4q-1)|=\Landau(n^{4q-1})$, i.e. setting two elements in $\underline{j}$ equal causes a further order change of $1/n$.
Also recall that
\begin{align*}
q-r = \#\{ \text{sums over } \bm{k}_1,\ldots,\bm{k}_q \}- \text{rank}(K) \qquad \text{and} \qquad 2q = \# \text{groups}.
\end{align*}
It can be seen easily that the number of sums and the rank of the matrix $K$ are not affected when setting two elements in $\underline{j}$ equal. This shows that setting two elements equal leads to an order change of at most $\lambda^d/n$ (there is an order change of $1/n$ only if two elements in $\underline{j}$ are equal that do not belong to the same set).\\
For illustration, we consider again the partition $\bm{\nu}^{\ast}$ characterized by the matrices $K$ and $R$ in \eqref{special_form_of_K_and_R}. This time however, we set $i=4q-1$ and assume without loss of generality that the indices $j_2$ and $j_5$ in $\underline{j}=(j_1,\ldots,j_{4q})$ coincide and both belong to the set $\nu_1^\ast$, i.e. we have $\nu_1^\ast=\{(1,2),(2,1)\}$. Recalling the alternative notation $\nu_g^\ast=(\nu_{g,1}^\ast,\nu_{g,2}^\ast)$ for $g=1,\ldots,2q$ from \eqref{simplified_notation}, we then obtain
\begin{align*}
&\phantom{=i}\frac{\lambda^{3dq/2}}{n^{4q}} \sum_{\bm{k}_1,\ldots,\bm{k}_q=-a}^{a-1} \sum_{\underline{j}\in\mathcal{D}(q,4q-1)} \cum_2[Y_{t,c}(\underline{j}):(t,c)\in\nu_1^{\ast}] \times \ldots \times \cum_2[Y_{t,c}(\underline{j}):(t,c)\in\nu_{2q}^{\ast}]\\
&=\Landau\Big(\frac{\lambda^{3dq/2}}{n} \sum_{\bm{k}_1,\ldots,\bm{k}_q=-a}^{a-1} \frac{1}{\lambda^{d(4q-1)}} \int_{[-\lambda/2,\lambda/2]^{d(4q-1)}} c(\bm{0})\, \prod_{g=2}^{2q} c(\bm{s}_{\nu_{g,1}^{\ast}}-\bm{s}_{\nu_{g,2}^{\ast}}) \\
&\phantom{=============}  h\left(\frac{\bm{s}_1}{\lambda}\right) h\left(\frac{\bm{s}_5}{\lambda}\right) h\left(\frac{\bm{s}_3}{\lambda}\right) h\left(\frac{\bm{s}_4}{\lambda}\right)    \exp\big(\im \tilde{\bm{\omega}}_{\bm{k}_1}^T(\bm{s}_1-\bm{s}_5+\bm{s}_3-\bm{s}_4)\big)\\
&\phantom{========} \prod_{t=2}^q \left(\prod_{c=1}^4 h\left(\frac{\bm{s}_{c+4(t-1)}}{\lambda}\right) \exp\big(\im \tilde{\bm{\omega}}_{\bm{k}_t}^T (\bm{s}_{4t-3}-\bm{s}_{4t-2}+\bm{s}_{4t-1}-\bm{s}_{4t})\big) \right)  \, d\bm{s}_5 \prod_{g=2}^{2q} d\bm{s}_{\nu_{g,1}^{\ast}} d\bm{s}_{\nu_{g,2}^{\ast}} \Big)\\
&=\Landau\Big(\frac{\lambda^{3dq/2}}{n \lambda^{d(4q-1)}} \sum_{\bm{k}_1,\ldots,\bm{k}_q=-a}^{a-1} \int_{\R^{d(2q-1)}} \prod_{g=2}^{2q} f(\bm{x}_g) B(\bm{x}_g+\tilde{\bm{\omega}}_{\hat{\bm{k}}_{2g-1}}) B(\bm{x}_g +\tilde{\bm{\omega}}_{\hat{\bm{k}}_{2g}}) \tilde{B}(\tilde{\bm{\omega}}_{\bm{k}_2}-\tilde{\bm{\omega}}_{\bm{k}_1})\, d\bm{x}_g \Big)\\
&=\Landau\Big(\frac{\lambda^{3dq/2}}{n \lambda^{d(4q-1)}} \sum_{\bm{k}_1,\ldots,\bm{k}_q=-a}^{a-1} \int_{\R^{d(2q-1)}} \prod_{g=2}^{2q} f(\bm{u}_g-\tilde{\bm{\omega}}_{\hat{\bm{k}}_{2g-1}}) B(\bm{u}_g) B(\bm{u}_g + \bm{\omega}_{\hat{\bm{k}}_{2g}-\hat{\bm{k}}_{2g-1}}) \tilde{B}(\bm{\omega}_{\bm{k}_2-\bm{k}_1})\, d\bm{u}_g \Big),
\end{align*}
where  
\begin{align} \label{B_tilde}
\tilde{B}(\bm{u}):=\int_{[-\lambda/2,\lambda/2]^d} h^2\left(\frac{\bm{s}}{\lambda}\right) \exp(-\im \bm{s}^T \bm{u})\, d\bm{s}
\end{align}
is the frequency window of the squared taper function $h^2$. 
Now, by the choice of the matrices $K$ and $R$ (see above), we can again make the index shift 
\begin{align*}
\bm{k}_1=\bm{k}_1, \qquad \bm{m}_1=\bm{k}_2-\bm{k}_1, \qquad \bm{m}_2=\bm{k}_3-\bm{k}_2, \qquad \ldots, \qquad \bm{m}_{r}=\bm{k}_{r+1}-\bm{k}_{r}
\end{align*}
and obtain that the above expression is of order
\begin{align*}
&\Landau\Big(\frac{\lambda^{3dq/2}}{n \lambda^{d(4q-1)}} \sum_{\bm{m}_1,\ldots,\bm{m}_r=-2a+1}^{2a-1} \int_{\R^{d(2q-1)}} H_{\lambda,a}(\bm{u}_{r+2},\ldots,\bm{u}_{q+1}) \, |\tilde{B}(\bm{\omega}_{\bm{m}_1})|\,\prod_{g=2}^r  \left|B(\bm{u}_g) B(\bm{u}_g+\bm{\omega}_{\bm{m}_g})\right|\\
&\phantom{====================}  \left|B(\bm{u}_{r+1}) B(\bm{u}_{r+1}-\sum_{j=1}^r \bm{\omega}_{\bm{m}_j})\right| \prod_{g=r+2}^{2q} B^2(\bm{u}_g) \, \prod_{g=2}^{2q} d\bm{u}_g \Big),
\end{align*}
where 
\begin{align*}
H_{\lambda,a}(\bm{u}_{r+2},\ldots,\bm{u}_{q+1}):=\sum_{\bm{k}_1,\bm{k}_{r+2},\ldots,\bm{k}_q=-a}^{a-1} \left(\prod_{g=r+2}^{q} f(\bm{u}_g-\tilde{\bm{\omega}}_{
\bm{k}_{g}})\right) f(\bm{u}_{q+1}-\tilde{\bm{\omega}}_{\bm{k}_1}).
\end{align*}
As in the case where $i=4q$, we obtain $H_{\lambda,a}(\bm{u}_{r+2},\ldots,\bm{u}_{q+1})=\Landau(\lambda^{d(q-r)})$.
By Lemma \ref{bounds_for_B}, part (iv), for $t=r+1$ and $s=1$, and Lemma \ref{orders_of_B}, part (i), we furthermore have
\begin{align*}
& \sum_{\bm{m}_1,\ldots,\bm{m}_r=-2a+1}^{2a-1} \int_{\R^{d(2q-1)}} |\tilde{B}(\bm{\omega}_{\bm{m}_1})|\, \prod_{g=2}^r  \left|B(\bm{u}_g) B(\bm{u}_g+\bm{\omega}_{\bm{m}_g})\right|\\
&\phantom{==========}  \left|B(\bm{u}_{r+1}) B(\bm{u}_{r+1}-\sum_{j=1}^r \bm{\omega}_{\bm{m}_j})\right| \prod_{g=r+2}^{2q} B^2(\bm{u}_g) \, \prod_{g=2}^{2q} d\bm{u}_g =\Landau(\lambda^{2dq})
\end{align*}
(since the proof of Lemma \ref{bounds_for_B} also works when replacing the frequency window of $h$ by the frequency window of $h^2$).
The whole expression is thus of order
\begin{align*}
\Landau\left(\frac{\lambda^{d\big(3q/2-(4q-1)+(q-r)+2q\big)}}{n}\right)= \Landau\left(\frac{\lambda^{d(q/2-r+1)}}{n}\right).
\end{align*}
Setting $i=4q-1$ thus leads to an order change of $\lambda^d/n$.\\
We now explain why setting any two additional elements in $\underline{j}$ equal leads to an order change of at most $a^d/n$. While for most partitions the arguments given above (resulting in an additional factor of $\lambda^d/n$) hold in the case of $i< 4q-1$ as well, there can also be situations where the additional factor is of order $a^d/n$.
This happens if a certain row of Table \eqref{table} does not yield an integral over $\R^d$ and therefore no spectral density any more. For illustration, consider the partition 
$$\bm{\nu}=\Big\{\{(1,1),(2,1)\},\{(1,2),(2,2)\},\{(1,3),(2,3)\},\{(1,4),(2,4)\}\Big\}$$
 (i.e. $q=2$) and the cases $i=5$, $i=4$, respectively. Recall the definition of the frequency window $\tilde{B}$ in \eqref{B_tilde}. For $i=5$, since $|\mathcal{D}(2,5)|=\Landau(n^5)$, we thus obtain 
\begin{align*}
&\phantom{==}\frac{\lambda^{3d}}{n^8}\sum_{\bm{k}_1,\bm{k}_2=-a}^{a-1} \sum_{\underline{j}\in\mathcal{D}(2,5)} \cum_2[Y_{t,c}(\underline{j})\, : \, (t,c)\in \nu_1] \times \ldots \times \cum_2[Y_{t,c}(\underline{j})\, :\, (t,c)\in \nu_4]\\
&=\Landau\Bigg(\frac{\lambda^{3d}}{n^3 \lambda^{5d}} \sum_{\bm{k}_1,\bm{k}_2=-a}^{a-1} \int_{[-\lambda/2,\lambda/2]^{5d}} c(\bm{0})^3\, c(\bm{s}_4-\bm{s}_5) \, h^2\left(\frac{\bm{s}_1}{\lambda}\right) h^2\left(\frac{\bm{s}_2}{\lambda}\right) h^2\left(\frac{\bm{s}_3}{\lambda}\right) h\left(\frac{\bm{s}_4}{\lambda}\right) h\left(\frac{\bm{s}_5}{\lambda}\right)\\
&\phantom{==========} \exp\big(\im \tilde{\bm{\omega}}_{\bm{k}_1}^T (\bm{s}_1-\bm{s}_2+\bm{s}_3-\bm{s}_4)\big) \exp\big(\im \tilde{\bm{\omega}}_{\bm{k}_2}^T (\bm{s}_1-\bm{s}_2+\bm{s}_3-\bm{s}_5)\big) \, d\bm{s}_1 d\bm{s}_2 d\bm{s}_3 d\bm{s}_4 d\bm{s}_5 \Bigg)\\
&=\Landau\Bigg(\frac{\lambda^{3d-5d}}{n^3} \int_{\R^d} \sum_{\bm{k}_1,\bm{k}_2=-a}^{a-1} \tilde{B}(\tilde{\bm{\omega}}_{\bm{k}_1}+\tilde{\bm{\omega}}_{\bm{k}_2})^3 B(\bm{x}-\tilde{\bm{\omega}}_{\bm{k}_1}) B(\bm{x}+\tilde{\bm{\omega}}_{\bm{k}_2}) 
f(\bm{x}) \, d\bm{x} \Bigg)\\
&=\Landau\Bigg(\frac{\lambda^{3d-5d}}{n^3} \int_{\R^d} \sum_{\bm{k}_1,\bm{k}_2 =-a}^{a-1} \tilde{B}(\bm{\omega}_{\bm{k}_1+\bm{k}_2+1})^3 B(\bm{u}) B(\bm{u}+\bm{\omega}_{\bm{k}_1+\bm{k}_2+1}) f(\bm{u}+\tilde{\bm{\omega}}_{\bm{k}_1}) \, d\bm{u} \Bigg)\\
&=\Landau\Bigg(\frac{\lambda^{3d-5d}}{n^3} \int_{\R^d}  \sum_{\bm{m}=-2a}^{2a-2} \tilde{B}(\bm{\omega}_{\bm{m}+1})^3 B(\bm{u}) B(\bm{u}+\bm{\omega}_{\bm{m}+1}) \sum_{\bm{k}_1=\max(-a,\bm{m}-a+1)}^{\min(a-1,\bm{m}+a)} f(\bm{u}+\tilde{\bm{\omega}}_{\bm{k}_1}) \, d\bm{u} \Bigg)\\
&=\Landau\left(\frac{\lambda^{3d-5d+4d+d}}{n^3}\right)=\Landau\left(\frac{\lambda^{3d}}{n^3}\right),
\end{align*}
where we again applied Lemma \ref{bounds_for_B}, part (iv), for $t=2$ and $s=1$. Moreover, we used the fact that $\tilde{B}(\bm{u})=\Landau(\lambda^{d})$. 
For $i=4$, since $|\mathcal{D}(2,4)|=\Landau(n^4)$, we have
\begin{align*}
&\phantom{==}\frac{\lambda^{3d}}{n^8}\sum_{\bm{k}_1,\bm{k}_2=-a}^{a-1} \sum_{\underline{j}\in\mathcal{D}(2,4)} \cum_2[Y_{t,c}(\underline{j})\, : \, (t,c)\in \nu_1] \times \ldots \times \cum_2[Y_{t,c}(\underline{j})\, :\, (t,c)\in \nu_4]\\
&=\Landau\Bigg(\frac{\lambda^{3d}}{n^4 \lambda^{4d}} \sum_{\bm{k}_1,\bm{k}_2=-a}^{a-1} \int_{[-\lambda/2,\lambda/2]^{4d}} c(\bm{0})^4 \, h^2\left(\frac{\bm{s}_1}{\lambda}\right) h^2\left(\frac{\bm{s}_2}{\lambda}\right) h^2\left(\frac{\bm{s}_3}{\lambda}\right) h^2\left(\frac{\bm{s}_4}{\lambda}\right) \\
&\phantom{==========} \exp\big(\im \tilde{\bm{\omega}}_{\bm{k}_1}^T (\bm{s}_1-\bm{s}_2+\bm{s}_3-\bm{s}_4)\big) \exp\big(\im \tilde{\bm{\omega}}_{\bm{k}_2}^T (\bm{s}_1-\bm{s}_2+\bm{s}_3-\bm{s}_4)\big) \, d\bm{s}_1 d\bm{s}_2 d\bm{s}_3 d\bm{s}_4 \Bigg)\\
&=\Landau\Bigg(\frac{\lambda^{3d-4d}}{n^4} \sum_{\bm{k}_1,\bm{k}_2=-a}^{a-1} \tilde{B}(\tilde{\bm{\omega}}_{\bm{k}_1}+\tilde{\bm{\omega}}_{\bm{k}_2})^4 \Bigg)=\Landau\Bigg(\frac{\lambda^{3d-4d}}{n^4}  \sum_{\bm{m}=-2a}^{2a-2} \tilde{B}(\bm{\omega}_{\bm{m}+1})^4 \sum_{\bm{k}_1=\max(-a,\bm{m}-a+1)}^{\min(a-1,\bm{m}+a)} 1 \Bigg)\\
&=\Landau\left(\frac{\lambda^{3d-4d+4d} \, a^d}{n^4}\right)=\Landau\left(\frac{\lambda^{3d}}{n^3} \times \frac{a^d}{n}\right).
\end{align*}
Here, we used Lemma \ref{bounds_for_B}, part (v), for $t=2$, and $\tilde{B}(\bm{u})=\Landau(\lambda^{d})$. 
This example illustrates that an order change of $a^d/n$ is possible as well when setting two additional elements equal. Other cases can be treated similarly and the discussion is omitted for the sake of brevity. No order change of higher order than $a^d/n$ can appear when decreasing the value of $i$ by $1$.
\end{proof}

For an indecomposable partition $\bm{\nu}=\{\nu_1,\ldots,\nu_{2q}\}\in\mathcal{I}(q)$, we must have at least $q-1$ restrictions in frequency direction. Therefore, we obtain the following result, which specifies the order of the cumulant expression for an indecomposable partition $\bm{\nu}$ with $G=2q$ groups and $i\leq 4q$ different elements:

\begin{corollary} \label{G=2q}
Assume that $\bm{\nu}=\{\nu_1,\ldots,\nu_{2q}\}\in\mathcal{I}(q)$ is an indecomposable partition of Table \eqref{table} with $|\nu_g|=2$ for $g=1,\ldots,2q$. Then, we have
\begin{align*}
&\phantom{==}\frac{\lambda^{3dq/2}}{n^{4q}} \sum_{\bm{k}_1,\ldots,\bm{k}_q=-a}^{a-1} \sum_{\underline{j}\in\mathcal{D}(q,i)} \cum_2\big[Y_{t,c}(\underline{j}):(t,c)\in\nu_1\big] \times \ldots \times \cum_2\big[Y_{t,c}(\underline{j}):(t,c)\in\nu_{2q}\big]\\
&=\begin{cases}  \Landau(\lambda^{d(1-q/2)}) \quad &\text{if } i=4q,\\
\Landau\left(\frac{\lambda^{d(2-q/2)}}{n}\right) \quad &\text{if } i\leq 4q-1.
\end{cases}
\end{align*}
In particular, for $q=2$ we obtain an order of $\Landau(1)$ in the case of maximally many different elements (i.e. $i=8$), while for $i<8$, we obtain an order of $\Landau(\lambda^d/n)$. For $q\geq 3$, the expression is always of order $o(1)$. 
\end{corollary}

We now consider a partition $\bm{\nu}=\{\nu_1,\ldots,\nu_G\}$ of Table \eqref{table} with $|\nu_g|>2$ for some $g\in\{1,\ldots,G\}$. Recall that due to the law of total cumulance and the Gaussianity of the locations, each term $\cum_{|\nu_g|}[Y_{t,c}(\underline{j})\, : \, (t,c)\in \nu_g]$ appearing in \eqref{cumulant_general} must be the sum of cumulants of covariances conditioned on the locations. For a set $\nu_g$ with $|\nu_g|=2$, this gives a single expectation of a covariance (as can be seen in our above calculations). For $|\nu_g|>2$ however, we obtain a sum of cumulants of covariances over all subpartitions of $\nu_g$ that are of size $2$. It is then easy to see that the order of the expression
\begin{align*}
\sum_{\bm{k}_1,\ldots,\bm{k}_q=-a}^{a-1} \cum_{|\nu_1|}\big[Y_{t,c}(\underline{j})\, : \, (t,c)\in \nu_1\big] \times \ldots \times \cum_{|\nu_G|}\big[Y_{t,c}(\underline{j})\, : \, (t,c)\in \nu_G\big] 
\end{align*}
is again determined by the number of different elements in $\underline{j}$, the number of restrictions between the variables $\bm{k}_1,\ldots,\bm{k}_q$ in Table \eqref{table}, and the number of groups $G$. Moreover, considering the maximal number of different elements in $\underline{j}$ allowed leads to the highest order, since setting two additional elements equal leads to an order change of at most $a^d/n$. 
Considering the calculations for the case where $|\nu_g|=2$ for all $g=1,\ldots,2q$, we thus obtain 
\begin{align*}
&\phantom{=i}\frac{\lambda^{3dq/2}}{n^{4q}} \sum_{\bm{k}_1,\ldots,\bm{k}_q=-a}^{a-1} \sum_{i=4}^{2q+G} \sum_{\underline{j}\in\mathcal{D}(q,i)} \cum_{|\nu_1|}\big[Y_{t,c}(\underline{j})\, : \, (t,c)\in \nu_1\big] \times \ldots \times \cum_{|\nu_G|}\big[Y_{t,c}(\underline{j})\, : \, (t,c)\in \nu_G\big] \\
&=\Landau\left(\frac{\lambda^{d\big(3q/2-(2q+G)+(q-\#\text{restrictions})+G\big)}}{n^{4q-(2q+G)}}\right)
\end{align*}
[compare to \eqref{order_for_2,2,2_and_maximal_elements}, which is the order of the cumulant expression for a partition with $G=2q$ groups and maximally many different elements $i=4q$]. We have thus shown the following result, which provides the order of the cumulant expression corresponding to an arbitrary partition $\bm{\nu}$ with $G$ groups:

\begin{prop} \label{general_order}
For a fixed partition $\bm{\nu}=\{\nu_1,\ldots,\nu_G\}$ of Table \eqref{table}, we have
\begin{align*}
&\phantom{==}\frac{\lambda^{3dq/2}}{n^{4q}} \sum_{\bm{k}_1,\ldots,\bm{k}_q=-a}^{a-1} \sum_{i=4}^{2q+G} \sum_{\underline{j}\in \mathcal{D}(q,i)} \cum_{|\nu_1|}\big[Y_{t,c}(\underline{j})\, : \, (t,c)\in \nu_1\big] \times \ldots \times \cum_{|\nu_G|}\big[Y_{t,c}(\underline{j})\, :\, (t,c)\in \nu_G\big]\\
&=\mathcal{O}\left(\frac{\lambda^{d(q/2-\#\text{restrictions})}}{n^{2q-G}}\right),
\end{align*}
where the term $\#\text{restrictions}$ denotes the number of independent constraints between the rows and columns of Table \eqref{table}. 
\end{prop}


In order to show \eqref{cumulant_orders}, we need the following lemma from \cite{vandelft18}:

\begin{lemma} \label{number_of_rest}
Let $G=q+s$ for some $s\in\{1,\ldots,q\}$. Then, only partitions with at least $s-1$ restrictions in frequency direction are indecomposable.
\end{lemma}

With the help of Lemma \ref{number_of_rest}, we can prove the following.

\begin{prop} \label{less_or_equal}
Let $\bm{\nu}=\{\nu_1,\ldots,\nu_G\}\in\mathcal{I}(q)$ be an arbitrary indecomposable partition of Table \eqref{table}. Then, we have
\begin{align*}
\frac{\lambda^{d(q/2-\#\text{restrictions})}}{n^{2q-G}} = \begin{cases} 
\Landau\left( \lambda^{-dq/2} \left(\frac{\lambda^d}{n}\right)^q\right), &\qquad \text{if } G\leq q,\\
\Landau\left( \left(\frac{\lambda^{d}}{n}\right)^{q-s} \lambda^{d-dq/2}\right), &\qquad \text{if } G=q+s \quad \text{for some } s\in\{1,\ldots,q\}.
\end{cases}
\end{align*}
\end{prop}

\begin{proof}
We start with the case $G\leq q$. Then, we immediately obtain
\begin{align*}
\frac{\lambda^{d(q/2-\#\text{restrictions})}}{n^{2q-G}} \leq \frac{\lambda^{dq/2}}{n^q} = \lambda^{-dq/2} \left(\frac{\lambda^d}{n}\right)^q.
\end{align*} 
Now assume that $G=q+s$ for some $s\in \{1,\ldots,q\}$. By Lemma \ref{number_of_rest} and the indecomposability assumption, it holds that
\begin{align*}
\frac{q}{2}-\#\text{restrictions} \leq \frac{q}{2}-s+1.
\end{align*}
This yields the claim, since
\begin{align*}
\frac{\lambda^{d(q/2-\#\text{restrictions})}}{n^{2q-G}} \leq \frac{\lambda^{dq/2-ds+d}}{n^{q-s}} =  \left(\frac{\lambda^{d}}{n}\right)^{q-s} \lambda^{d-dq/2}. 
\end{align*}
\end{proof}

Proposition \ref{less_or_equal} directly gives the following result.

\begin{corollary} \label{corollary_proof}
For arbitrary $q\geq 1$, it holds that
\begin{align*}
\max_{\bm{\nu}=\{\nu_1,\ldots,\nu_G\}\in\mathcal{I}(q)}\frac{\lambda^{d(q/2-\#\text{restrictions})}}{n^{2q-G}} = 
\Landau\left(\lambda^{d-dq/2}\right). 
\end{align*}
\end{corollary}

The above corollary proves the statement in \eqref{cumulant_orders}, since
\begin{align*} 
\lambda^{dq/2} \cum_q\big(\hat{D}_{1,d,\lambda,a}\big)=\mathcal{O}\left(\max_{\bm{\nu}=\{\nu_1,\ldots,\nu_G\}\in\mathcal{I}(q)}\frac{\lambda^{d(q/2-\#\text{restrictions})}}{n^{2q-G}}\right).
\end{align*}
This yields the claim in Theorem \ref{asymptotic_normality}. 
\phantom{==========================}\qed

\vspace{1cm}
For the case $q=2$, a more precise result concerning the rate is needed in order to prove the result for the variance in Theorem \ref{expectation_theo}. More precisely, we state the following:

\begin{corollary} \label{result_for_q=2}
Let $q=2$ and let $\bm{\nu}=\{\nu_1,\ldots,\nu_G\}\in\mathcal{I}(2)$ be arbitrary. We then have 
\begin{align*}
&\phantom{==}\frac{\lambda^{3d}}{n^{8}} \sum_{\bm{k}_1,\bm{k}_2=-a}^{a-1} \sum_{\underline{j}\in\mathcal{D}(2,i)} \cum_{|\nu_1|}\big[Y_{t,c}(\underline{j}):(t,c)\in\nu_1\big] \times \ldots \times \cum_{|\nu_G|}\big[Y_{t,c}(\underline{j}):(t,c)\in\nu_{G}\big]\\
&=\begin{cases}
\Landau(1), &\qquad \text{if } G=4 \text{ and } i=8,\\
\Landau\left(\frac{\lambda^d}{n}\right), &\qquad \text{if } G=4 \text{ and } i<8, \text{ or if } G<4.
\end{cases}
\end{align*}
\end{corollary}

\begin{proof}
The statements for the cases $G=4$, $i=8$ and $G=4$, $i<8$ directly follow from Corollary \ref{G=2q}, while the result in the case $G<4$ follows from Proposition \ref{less_or_equal} for $s<q$. 
\end{proof}

\subsection{Proof of Proposition \ref{approx_D1}}

Note that
\begin{align*}
|D_{1,d}-D_{1,d,\lambda,a}| = \int_{([-2\pi a/\lambda,2\pi a/\lambda]^d)^c} f^2(\bm{\omega})\, d\bm{\omega},
\end{align*}
where $A^c$ denotes the complement of a set $A$. \label{page_compl} 
If $\bm{\omega}=(\omega_1,\ldots,\omega_d)^T$, we thus integrate over an area where $\omega_i \notin [-2\pi a/\lambda,2\pi a/\lambda]$ for at least one $i\in\{1,\ldots,d\}$. Since by assumption $f(\bm{\omega})\leq \prod_{i=1}^d \beta_{1+\delta}(\omega_i)$, this yields
\begin{align*}
|D_{1,d}-D_{1,d,\lambda,a}| = \Landau\left(\int_{2\pi a/\lambda}^{\infty} \beta_{1+\delta}^2(\omega)\, d\omega \right) = \Landau\left(\int_{2\pi a/\lambda}^{\infty} \frac{1}{\omega^{2+2\delta}}\, d\omega \right) = \Landau\left(\frac{\lambda^{1+2\delta}}{a^{1+2\delta}}\right).
\end{align*}
\qed

\section{Proof of Theorem \ref{corr_sec_int}}
\setcounter{equation}{0}

The result is proved in three steps. 
 We begin investigating the  expectation and variance of $\hat{D}_{2,2,\lambda,a} $. 
\begin{theorem} \label{expect_theo_sec_int}
If $d=2$ and  Assumption \ref{ass_on_h},  \ref{assumption_on_Z}, \ref{assumption_on_sampling_scheme} and  \ref{assumptions_on_a}(i)  are satisfied, then 
\begin{align*}
\E\big[\hat{D}_{2,2,\lambda,a}\big] &=  D_{2,2,\lambda,a} + \Landau\Bigg(\frac{a^2}{\lambda^4}+\frac{a^{4}}{\lambda^{8}}+\frac{a^{6}}{\lambda^{10}}+\frac{a^{3}}{\lambda^{6}}+\frac{1}{n}\Bigg) \qquad \text{as } a,\lambda,n\rightarrow \infty. \\
\lambda^2 \Var\big[\hat{D}_{2,2,\lambda,a}\big]&=\tau_{2,2,\lambda,a}^2+\Landau\Bigg(\frac{a^{6}}{\lambda^{10}} + \frac{a^{4}}{\lambda^{6}}\left[1+\frac{(\log a)^4}{\lambda^2}+\frac{a}{\lambda^{2}}+\frac{a^{4}}{\lambda^{6}}\right]\\
&\phantom{====}+\left[\frac{(\log a)^4}{\lambda}+\frac{a^{2}}{\lambda^{4}}\right] \left[1+\frac{a}{\lambda^{2}}+\frac{a^{4}}{\lambda^{6}}\right]^2 + \frac{a\, (\log \lambda)^2}{\lambda^{2}}+ \frac{1}{(\log \lambda)^3}+\frac{\lambda^2}{n}\Bigg)
\end{align*}
as $a,\lambda,n\to\infty$, where
\begin{align}
\label{D_22lambda,a}
D_{2,2,\lambda,a} &=  \frac{1}{2\pi} \int_{0}^{2\pi a /\lambda} \left(\int_{[-2\pi a/\lambda,2\pi a/\lambda]^2} f(\bm{x}) J_0(r\|\bm{x}\|)\, d\bm{x} \right)^2 \, r \, dr 
\end{align}
and $\tau_{2,2,\lambda,a}^2$ is defined in \eqref{det6}.
\end{theorem}
The second step consists in the proof of the asymptotic normality of a scaled vsersion of the statistic $\hat{D}_{2,2,\lambda,a}- D_{2,2,\lambda,a}$.
\begin{theorem} \label{asymptotic_normality_secint}
If  $d=2$ and Assumption \ref{ass_on_h}, \ref{assumption_on_Z},  \ref{assumption_on_sampling_scheme} and \ref{assumptions_on_a} are satisfied, we have 
\begin{align*}
\frac{\lambda}{\tau_{2,2,\lambda,a}} \left(\hat{D}_{2,2,\lambda,a}- D_{2,2,\lambda,a}\right) \dn \mathcal{N}(0,1) \qquad \text{as } a,\lambda,n\rightarrow \infty,
\end{align*}
where $D_{2,2,\lambda,a}$ and  $\tau_{2,2,\lambda,a}$ are  defined in \eqref{D_22lambda,a} and  \eqref{det6}, respectively.
\end{theorem}
Note that the asymptotic result from Theorem \ref{asymptotic_normality_secint} does not hold true for the rectangular taper function $h=h^{\text{rect}}$, since in this case the bias of the test statistic $\hat{D}_{2,2,\lambda,a}$ is at least of order $\Landau(\log \lambda/\lambda + 1/n)$ (see the proof of Theorem \ref{expect_theo_sec_int} and Remark \ref{rem_edge_effect}). In particular, the rectangular taper does not satisfy Assumption \ref{ass_on_h}. 
We now derive the approximation rate of the expression $D_{2,2}$ by its truncated version $D_{2,2,\lambda,a}$.
\begin{prop} \label{approx_D2}
Assume that $f(\bm{\omega})\leq \beta_{1+\delta}(\bm{\omega})$ for some $\delta>0$,  that  $c$ is uniformly bounded and satisfies 
\begin{align*}
c(\bm{h}) = \Landau\left(\|\bm{h}\|^{-(2+\varepsilon)}\right) \qquad \text{as } \|\bm{h}\|\rightarrow \infty,
\end{align*}
for some $\varepsilon> 0$.  If  $a/\lambda\to\infty$, then we have
\begin{align*}
|D_{2,2}-D_{2,2,\lambda,a}|=
\Landau\left(\frac{\lambda^{2+2\varepsilon}}{a^{2+2\varepsilon}}+\frac{\lambda^{2\delta-2}}{a^{2\delta-2}} + \frac{\lambda^\delta}{a^\delta}\right) \qquad \text{as } a,\lambda\to\infty.
\end{align*}
\end{prop}
The  assertion of  Theorem \ref{corr_sec_int} now follows, because  condition \eqref{further_assumption_2} allows us to  replace
$D_{2,2,\lambda,a}$ by the quantity of interest  $D_{2,2}$ in Theorem \ref{asymptotic_normality_secint}.
Note that both Assumption \ref{assumptions_on_a} for $d=2$ and the requirement in \eqref{further_assumption_2} can be satisfied simultaneously since Assumption \ref{assumption_on_Z}, (ii), requires $\delta>2$. Under this condition, there exist sequences $a\to\infty$ and $\lambda\to\infty$ with $a/\lambda\to\infty$ such that \eqref{further_assumption_2} and the requirement $a^2/\lambda^3=o(1)$ from Assumption \ref{assumptions_on_a} are both fulfilled.

\vspace{0.5cm}

\subsection{Proof of Theorem \ref{expect_theo_sec_int}}

\paragraph{Asymptotic bias:} 
Similarly to the proof of Theorem \ref{expectation_theo}, we obtain 
\begin{align*}
\E\big[\hat{D}_{2,\lambda,a}\big]&=E_1+E_2+E_3,
\end{align*}
where 
\begin{align*}
E_1&:=\frac{1}{\lambda} \sum_{r=0}^{a-1} \tilde{\omega}_r \Bigg(\left(\frac{2\pi}{\lambda}\right)^4 \sum_{\bm{k}=-a}^{a-1} \sum_{\bm{\ell}=-a}^{a-1}  \frac{\lambda^4}{n^4 H(\bm{0})^2} \sum_{(j_1,\ldots,j_4)\in\mathcal{E}} \E\bigg[h\left(\frac{\bm{s}_{j_1}}{\lambda}\right)h\left(\frac{\bm{s}_{j_2}}{\lambda}\right) h\left(\frac{\bm{s}_{j_3}}{\lambda}\right)  h\left(\frac{\bm{s}_{j_4}}{\lambda}\right)\\
&c(\bm{s}_{j_1}-\bm{s}_{j_2})\, c(\bm{s}_{j_3}-\bm{s}_{j_4})\exp\big(\im(\bm{s}_{j_1}-\bm{s}_{j_2})^T \tilde{\bm{\omega}}_{\bm{k}}\big) \exp\big(\im(\bm{s}_{j_3}-\bm{s}_{j_4})^T \tilde{\bm{\omega}}_{\bm{\ell}}\big) \bigg]  J_0(\tilde{\omega}_r \|\tilde{\bm{\omega}}_{\bm{k}}\|) J_0(\tilde{\omega}_r \|\tilde{\bm{\omega}}_{\bm{\ell}}\|) \Bigg),
\end{align*}
\begin{align} \label{E_2_secint}
E_2&:=\frac{1}{\lambda} \sum_{r=0}^{a-1} \tilde{\omega}_r \Bigg(\left(\frac{2\pi}{\lambda}\right)^4 \sum_{\bm{k}=-a}^{a-1} \sum_{\bm{\ell}=-a}^{a-1}  \frac{\lambda^4}{n^4 H(\bm{0})^2} \sum_{(j_1,\ldots,j_4)\in\mathcal{E}} \E\bigg[h\left(\frac{\bm{s}_{j_1}}{\lambda}\right)h\left(\frac{\bm{s}_{j_2}}{\lambda}\right) h\left(\frac{\bm{s}_{j_3}}{\lambda}\right)  h\left(\frac{\bm{s}_{j_4}}{\lambda}\right)\nonumber\\
&c(\bm{s}_{j_1}-\bm{s}_{j_3})\, c(\bm{s}_{j_2}-\bm{s}_{j_4})\exp\big(\im(\bm{s}_{j_1}-\bm{s}_{j_2})^T \tilde{\bm{\omega}}_{\bm{k}}\big) \exp\big(\im(\bm{s}_{j_3}-\bm{s}_{j_4})^T \tilde{\bm{\omega}}_{\bm{\ell}}\big) \bigg]  J_0(\tilde{\omega}_r \|\tilde{\bm{\omega}}_{\bm{k}}\|) J_0(\tilde{\omega}_r \|\tilde{\bm{\omega}}_{\bm{\ell}}\|) \Bigg),
\end{align}
and
\begin{align} \label{E_3_secint}
E_3&:=\frac{1}{\lambda} \sum_{r=0}^{a-1} \tilde{\omega}_r \Bigg(\left(\frac{2\pi}{\lambda}\right)^4 \sum_{\bm{k}=-a}^{a-1} \sum_{\bm{\ell}=-a}^{a-1}  \frac{\lambda^4}{n^4 H(\bm{0})^2} \sum_{(j_1,\ldots,j_4)\in\mathcal{E}} \E\bigg[h\left(\frac{\bm{s}_{j_1}}{\lambda}\right)h\left(\frac{\bm{s}_{j_2}}{\lambda}\right) h\left(\frac{\bm{s}_{j_3}}{\lambda}\right)  h\left(\frac{\bm{s}_{j_4}}{\lambda}\right)\nonumber\\
&c(\bm{s}_{j_1}-\bm{s}_{j_4})\, c(\bm{s}_{j_2}-\bm{s}_{j_3})\exp\big(\im(\bm{s}_{j_1}-\bm{s}_{j_2})^T \tilde{\bm{\omega}}_{\bm{k}}\big) \exp\big(\im(\bm{s}_{j_3}-\bm{s}_{j_4})^T \tilde{\bm{\omega}}_{\bm{\ell}}\big) \bigg]  J_0(\tilde{\omega}_r \|\tilde{\bm{\omega}}_{\bm{k}}\|) J_0(\tilde{\omega}_r \|\tilde{\bm{\omega}}_{\bm{\ell}}\|) \Bigg).
\end{align}
Using similar arguments as in the proof of Theorem \ref{expectation_theo}, we obtain
\begin{align*}
&\phantom{i=}\frac{\lambda^4}{n^4 H(\bm{0})^2} \sum_{(j_1,\ldots,j_4)\in\mathcal{E}} \E\bigg[h\left(\frac{\bm{s}_{j_1}}{\lambda}\right)h\left(\frac{\bm{s}_{j_2}}{\lambda}\right) h\left(\frac{\bm{s}_{j_3}}{\lambda}\right)  h\left(\frac{\bm{s}_{j_4}}{\lambda}\right)c(\bm{s}_{j_1}-\bm{s}_{j_2})\, c(\bm{s}_{j_3}-\bm{s}_{j_4})\\
&\phantom{==========================} \exp\big(\im (\bm{s}_{j_1}-\bm{s}_{j_2})^T \tilde{\bm{\omega}}_{\bm{k}}\big) \exp\big(\im (\bm{s}_{j_3}-\bm{s}_{j_4})^T \tilde{\bm{\omega}}_{\bm{\ell}}\big)\bigg]\\
&=\frac{c_{1,n}}{(2\pi\lambda)^4 H(\bm{0})^2} \int_{\R^4} B(\bm{u})^2 B(\bm{v})^2 f(\bm{u}-\tilde{\bm{\omega}}_{\bm{k}}) f(\bm{v}-\tilde{\bm{\omega}}_{\bm{\ell}}) \, d\bm{u} d\bm{v},
\end{align*}
which yields
\begin{align*}
E_1 
&=\frac{c_{1,n}}{2\pi(2\pi\lambda)^4 H(\bm{0})^2} \int_{\R^4} B(\bm{u})^2 B(\bm{v})^2 \Bigg[ \left(\frac{2\pi}{\lambda}\right)^4 \sum_{\bm{k}=-a}^{a-1} \sum_{\bm{\ell}=-a}^{a-1} f(\bm{u}-\tilde{\bm{\omega}}_{\bm{k}}) f(\bm{v}-\tilde{\bm{\omega}}_{\bm{\ell}}) \\
&\phantom{==================} \int_{0}^{2\pi a /\lambda} r\,  J_0(r \|\tilde{\bm{\omega}}_{\bm{k}}\|) J_0( r \|\tilde{\bm{\omega}}_{\bm{\ell}}\|)\, dr\Bigg] d\bm{u} d\bm{v} + F_1(a,\lambda)
\end{align*}
for
\begin{align*}
F_1(a,\lambda)&:= \frac{c_{1,n}}{2\pi (2\pi \lambda)^4 H(\bm{0})^2} \int_{\R^4} B(\bm{u})^2 B(\bm{v})^2 \Bigg[ \left(\frac{2\pi}{\lambda}\right)^4 \sum_{\bm{k}=-a}^{a-1} \sum_{\bm{\ell}=-a}^{a-1} f(\bm{u}-\tilde{\bm{\omega}}_{\bm{k}}) f(\bm{v}-\tilde{\bm{\omega}}_{\bm{\ell}})\\
&\Bigg\{\frac{2\pi}{\lambda} \sum_{r=0}^{a-1} \tilde{\omega}_r \,J_0(\tilde{\omega}_r \|\tilde{\bm{\omega}}_{\bm{k}}\|) J_0(\tilde{\omega}_r \|\tilde{\bm{\omega}}_{\bm{\ell}}\|) - \int_{0}^{2\pi a /\lambda} r\, J_0(r\|\tilde{\bm{\omega}}_{\bm{k}}\|) J_0(r\|\tilde{\bm{\omega}}_{\bm{\ell}}\|)\, dr  \Bigg\} \Bigg]\, d\bm{u} d\bm{v}.
\end{align*}
We furthermore obtain
\begin{align*} 
E_1&=\frac{c_{1,n}}{2\pi(2\pi\lambda)^4 H(\bm{0})^2} \int_{\R^4} B(\bm{u})^2 B(\bm{v})^2 \Bigg[ \int_{[-2\pi a/\lambda,2\pi a/\lambda]^4} f(\bm{u}-\bm{x}) f(\bm{v}-\bm{y}) \nonumber\\
&\phantom{========} \Bigg(\int_{0}^{2\pi a/\lambda} r\,  J_0(r \|\bm{x}\|) J_0( r \|\bm{y}\|)\, dr\Bigg) d\bm{x}  d\bm{y}\Bigg] d\bm{u} d\bm{v}+F_1(a,\lambda)+ F_2(a,\lambda),
\end{align*}
where
\begin{align} \label{det2}
F_2(a,\lambda)=F_{21}(a,\lambda)+F_{22}(a,\lambda),
\end{align}
\begin{align*}
F_{21}(a,\lambda)&:=\frac{c_{1,n}}{2\pi(2\pi\lambda)^4 H(\bm{0})^2} \int_{\R^4} B(\bm{u})^2 B(\bm{v})^2 \Bigg(\int_{0}^{2\pi a /\lambda} r\, \Bigg\{ \left(\frac{2\pi}{\lambda}\right)^2 \sum_{\bm{k}=-a}^{a-1} f(\bm{u}-\tilde{\bm{\omega}}_{\bm{k}})\, J_0(r \|\tilde{\bm{\omega}}_{\bm{k}}\|)\nonumber\\
& \Bigg[ \left(\frac{2\pi}{\lambda}\right)^2 \sum_{\bm{\ell}=-a}^{a-1} f(\bm{v}-\tilde{\bm{\omega}}_{\bm{\ell}}) J_0(r \|\tilde{\bm{\omega}}_{\bm{\ell}}\|)-\int_{[-2\pi a/\lambda,2\pi a/\lambda]^2} f(\bm{v}-\bm{y}) J_0(r \|\bm{y}\|) \, d\bm{y}\Bigg]\Bigg\}\,dr \Bigg)d\bm{u}  d\bm{v},
\end{align*}
and
\begin{align*}
F_{22}(a,\lambda)&:=\frac{c_{1,n}}{2\pi(2\pi\lambda)^4 H(\bm{0})^2} \int_{\R^4} B(\bm{u})^2 B(\bm{v})^2 \Bigg(\int_{0}^{2\pi a /\lambda} r\, \Bigg\{ \int_{[-2\pi a/\lambda,2\pi a/\lambda]^2} f(\bm{v}-\bm{y}) J_0(r \|\bm{y}\|)\, d\bm{y} \nonumber\\
& \Bigg[\left(\frac{2\pi}{\lambda}\right)^2 \sum_{\bm{k}=-a}^{a-1} f(\bm{u}-\tilde{\bm{\omega}}_{\bm{k}}) J_0(r \|\tilde{\bm{\omega}}_{\bm{k}}\|)-\int_{[-2\pi a/\lambda,2\pi a/\lambda]^2} f(\bm{u}-\bm{x}) J_0(r \|\bm{x}\|)\, d\bm{x}\Bigg]\Bigg\}\,dr \Bigg)d\bm{u}  d\bm{v}.
\end{align*}
As a third approximation step, we write
\begin{align*} 
E_1&=\frac{c_{1,n}}{2\pi(2\pi\lambda)^4 H(\bm{0})^2} \left(\int_{\R^4} B(\bm{u})^2 B(\bm{v})^2\,d\bm{u} d\bm{v}\right)  \Bigg[  \int_{[-2\pi a/\lambda,2\pi a/\lambda]^4} f(\bm{x}) f(\bm{y}) \nonumber\\
&\phantom{==}\Bigg( \int_{0}^{2\pi a /\lambda} r\,  J_0(r \|\bm{x}\|) J_0( r \|\bm{y}\|)\, dr\Bigg) \,d\bm{x}  d\bm{y}\Bigg]  + F_1(a,\lambda)+ F_2(a,\lambda) +F_3(a,\lambda),
\end{align*}
where
\begin{align*}
F_3(a,\lambda)&:=\frac{c_{1,n}}{2\pi(2\pi\lambda)^4 H(\bm{0})^2} \int_{\R^4} B(\bm{u})^2 B(\bm{v})^2 \Bigg[ \int_{[-2\pi a/\lambda,2\pi a/\lambda]^4} \big[f(\bm{u}-\bm{x}) f(\bm{v}-\bm{y})-f(\bm{x}) f(\bm{y})\big] \nonumber\\
&\phantom{==}\Bigg( \int_{0}^{2\pi a /\lambda} r\,  J_0(r \|\bm{x}\|) J_0( r \|\bm{y}\|)\, dr\Bigg) \,d\bm{x}  d\bm{y}\Bigg] d\bm{u} d\bm{v}.
\end{align*}
Using Lemma \ref{convolution_of_h} and Lemma \ref{D_2_finite} for $\delta>2$, we thus obtain
\begin{align} \label{formula_for_E1}
E_1=D_{2,2,\lambda,a} + \sum_{i=1}^3 F_i(a,\lambda)+ \Landau\left(\frac{1}{n}\right)
\end{align}
as $\lambda,a,n\to\infty$.
{{We start with the term $F_1(a,\lambda)$ 
and obtain from Lemma \ref{Riemann_radius} and Lemma \ref{orders_of_B} (i)
\begin{align*}
|F_1(a,\lambda)|
&\lesssim F_{11}(a,\lambda)+F_{12}(a,\lambda)
\end{align*}
for $a,\lambda$ sufficiently large, where
\begin{align*}
F_{11}(a,\lambda):=\frac{ a^{2}}{\lambda^{6}} \int_{\R^2} B(\bm{u})^2 \left[ \left(\frac{2\pi}{\lambda}\right)^2 \sum_{\bm{k}=-a}^{a-1} f(\bm{u}-\tilde{\bm{\omega}}_{\bm{k}})\, \big(\tilde{\omega}_{k_1}^2+\tilde{\omega}_{k_2}^2\big) \right]  d\bm{u} 
\end{align*}
and
\begin{align*}
F_{12}(a,\lambda):=\frac{a}{\lambda^{5}} \int_{\R^2} B(\bm{u})^2 \left[\left(\frac{2\pi}{\lambda}\right)^2 \sum_{\bm{k}=-a}^{a-1} f(\bm{u}-\tilde{\bm{\omega}}_{\bm{k}}) \big(|\tilde{\omega}_{k_1}|+|\tilde{\omega}_{k_2}|\big)\right] d\bm{u}.
\end{align*}
Lemma \ref{f_times_x^2} (i) directly yields
$ F_{11}(a,\lambda)=\Landau(\frac{a^2}{\lambda^4}[1+\frac{a^2}{\lambda^4}])$.
Moreover, we have
\begin{align*}
F_{12}(a,\lambda)\leq 2\sup_{|k|\leq a} \left|\tilde{\omega}_{k}\right| \left( \frac{a}{\lambda^{5}}  \int_{\R^2} B(\bm{u})^2 \left[\left(\frac{2\pi}{\lambda}\right)^2 \sum_{\bm{k}=-a}^{a-1} f(\bm{u}-\tilde{\bm{\omega}}_{\bm{k}}) \right] d\bm{u}\right)= \Landau\left(\frac{a^{2}}{\lambda^{4}}\right),
\end{align*}
which yields
\begin{align*} 
F_1(a,\lambda)=\Landau\left(\frac{a^{2}}{\lambda^{4}} \left[1+\frac{a^2}{\lambda^4}\right] + \frac{a^{2}}{\lambda^{4}}\right)=
\Landau\left(\frac{a^{2}}{\lambda^{4}}+\frac{a^4}{\lambda^8}\right).
\end{align*}
We now consider the quantity $F_2(a,\lambda)$ and observe the decomposition \eqref{det2}. For the first term we obtain by definition of the Bessel function that
\begin{align*} 
|F_{21}(a,\lambda)| 
&\lesssim \frac{\overline{F}_{21}(a,\lambda)}{\lambda^4} \int_{\R^4} B(\bm{u})^2  B(\bm{v})^2 \left[\int_{\|\bm{x}\|\leq 2\pi a/\lambda} \Bigg| \left(\frac{2\pi}{\lambda}\right)^2 \sum_{\bm{k}=-a}^{a-1} f(\bm{u}-\tilde{\bm{\omega}}_{\bm{k}}) \exp\big(\im  \bm{x}^T \tilde{\bm{\omega}}_{\bm{k}}\big) \Bigg| d\bm{x}\right] d\bm{u} d\bm{v},
\end{align*}
where
\begin{align*}
\overline{F}_{21}(a,\lambda)&:=\sup_{\|\bm{x}\|\leq 2\pi a/\lambda}  \sup_{\bm{v}\in\R^2} \Bigg|\left(\frac{2\pi}{\lambda}\right)^2 \sum_{\bm{\ell}=-a}^{a-1} f(\bm{v}-\tilde{\bm{\omega}}_{\bm{\ell}}) J_0(\|\bm{x}\| \|\tilde{\bm{\omega}}_{\bm{\ell}}\|)\\
&\phantom{=================}- \int_{[-2\pi a/\lambda,2\pi a/\lambda]^2} f(\bm{v}-\bm{y}) J_0(\|\bm{x}\| \|\bm{y}\|) \,d\bm{y}\, \Bigg|.
\end{align*}
By Lemma \ref{Riemann_f_sec_int}, we get
\begin{align*}
\overline{F}_{21}(a,\lambda)
\lesssim \sup_{\|\bm{x}\|\leq 2\pi a/\lambda}  \frac{1+\|\bm{x}\|_1+\|\bm{x}\|^2}{\lambda^2} = \Landau\left(\frac{a^{2}}{\lambda^{4}}\right),
\end{align*}
and furthermore using Lemma \ref{cov_fct_integrable} (i) for $\delta>2$, we thus obtain 
\begin{align*}
F_{21}(a,\lambda)=\Landau\left(\frac{a^{2}}{\lambda^{4}} \left[1+\frac{a}{\lambda^{2}}+\frac{a^{4}}{\lambda^{6}}\right]\right).
\end{align*}
The same arguments (with the only difference that we now apply part (ii) of Lemma \ref{cov_fct_integrable} instead of part (i))
yield
$F_{22}(a,\lambda)=\Landau(\frac{a^{2}}{\lambda^{4}} [1+\frac{a}{\lambda^{2}}])$,
which gives
\begin{align*} 
F_2(a,\lambda) =\Landau\left(\frac{a^{2}}{\lambda^{4}}+\frac{a^{3}}{\lambda^{6}}+ \frac{a^{6}}{\lambda^{10}}\right).
\end{align*}

Finally, we consider the term $F_3(a,\lambda)$ and obtain
\begin{align*}
&\phantom{=i}|F_3(a,\lambda)|\\
&\lesssim \overline{F}_3(a,\lambda)  \Bigg\{ \frac{1}{\lambda^2} \int_{\R^2} B(\bm{u})^2 \Big[ \int_{0}^{2\pi a/\lambda} \int_{0}^{2\pi} r \Big| \int_{[-2\pi a/\lambda,2\pi a/\lambda]^2} f(\bm{u}-\bm{x}) \exp\Big(\im \, r \begin{pmatrix}
\cos \theta \\
\sin \theta
\end{pmatrix}^T \bm{x} \Big) \, d\bm{x}\, \Big| d\theta dr\Big] d\bm{u}\\
&\phantom{=====}+\Big( \frac{1}{\lambda^2} \int_{\R^2} B(\bm{v})^2\, d\bm{v}\Big) \Big[ \int_{0}^{2\pi a/\lambda} \int_{0}^{2\pi} r \Big| \int_{[-2\pi a/\lambda,2\pi a/\lambda]^2} f(\bm{y}) \exp\Big(\im \, r \begin{pmatrix}
\cos \theta \\
\sin \theta
\end{pmatrix}^T \bm{y} \Big) \, d\bm{y}\, \Big|\, d\theta dr\Big]\Bigg\},
\end{align*}
where
\begin{align*}
\overline{F}_3(a,\lambda):=\sup_{r>0}\left|\frac{1}{\lambda^2} \int_{\R^2} B(\bm{u})	^2 \Bigg[ \int_{[-2\pi a/\lambda,2\pi a/\lambda]^2} \big[f(\bm{u}-\bm{x})-f(\bm{x})\big] \, J_0(r\|\bm{x}\|) \, d\bm{x}\Bigg] d\bm{u}\right|.
\end{align*}
By Proposition \ref{Lemma F.2_SSR}, part (i), we must have
\begin{align*}
\overline{F}_3(a,\lambda)&\leq \sup_{\bm{y}\in\R^2}\left| \frac{1}{\lambda^2} \int_{\R^2} B(\bm{u})	^2 \Bigg[ \int_{[-2\pi a/\lambda,2\pi a/\lambda]^2} \big[f(\bm{u}-\bm{x})-f(\bm{x})\big] \exp(\im \bm{y}^T \bm{x}) \, d\bm{x}\Bigg] d\bm{u}\right|=\Landau\left( \frac{a^2}{\lambda^4}\right).
\end{align*}
Furthermore using Lemma \ref{cov_fct_integrable} (ii) for $\delta>2$ and Corollary \ref{cor_of_D2_finite}, we thus obtain
\begin{align*}
F_3(a,\lambda)=\Landau\left(\frac{a^2}{\lambda^4}\left[1+\frac{a}{\lambda^{2}}\right]\right)=
\Landau\left(\frac{a^2}{\lambda^4}+\frac{a^{3}}{\lambda^{6}}\right).
\end{align*}
}}
Considering the expression in \eqref{formula_for_E1}, we thus obtain
\begin{align*}
E_1
= D_{2,2,\lambda,a} + \Landau\Bigg(\frac{a^2}{\lambda^4}+\frac{a^{4}}{\lambda^{8}}+\frac{a^{3}}{\lambda^{6}}+\frac{a^{6}}{\lambda^{10}}+\frac{1}{n}\Bigg)
\end{align*}
as $a,\lambda,n\to\infty$. 
We now show that the terms $E_2$ and $E_3$ in \eqref{E_2_secint} and \eqref{E_3_secint} are sufficiently small. Concerning the term $E_2$, note that the same arguments as above show
\begin{align*}
&\phantom{=i}\frac{\lambda^4}{n^4 H(\bm{0})^2} \sum_{(j_1,\ldots,j_4)\in\mathcal{E}} \E\bigg[h\left(\frac{\bm{s}_{j_1}}{\lambda}\right)h\left(\frac{\bm{s}_{j_2}}{\lambda}\right) h\left(\frac{\bm{s}_{j_3}}{\lambda}\right)  h\left(\frac{\bm{s}_{j_4}}{\lambda}\right)\\
&\phantom{===============} c(\bm{s}_{j_1}-\bm{s}_{j_3})\, c(\bm{s}_{j_2}-\bm{s}_{j_4})
\exp\big(\im(\bm{s}_{j_1}-\bm{s}_{j_2})^T \tilde{\bm{\omega}}_{\bm{k}}\big) \exp\big(\im(\bm{s}_{j_3}-\bm{s}_{j_4})^T \tilde{\bm{\omega}}_{\bm{\ell}}\big) \bigg]\\
&=\frac{c_{1,n}}{(2\pi\lambda)^4 H(\bm{0})^2} \int_{\R^4} B(\bm{u}) B(\bm{u}-(\tilde{\bm{\omega}}_{\bm{k}}+\tilde{\bm{\omega}}_{\bm{\ell}})) B(\bm{v}) B(\bm{v}-(\tilde{\bm{\omega}}_{\bm{k}}+\tilde{\bm{\omega}}_{\bm{\ell}})) \, f(\bm{u}-\tilde{\bm{\omega}}_{\bm{k}}) f(\bm{v}-\tilde{\bm{\omega}}_{\bm{\ell}})\, d\bm{u} d\bm{v}.
\end{align*}
Therefore, using the identity $\tilde{\bm{\omega}}_{\bm{k}}+\tilde{\bm{\omega}}_{\bm{\ell}}=\bm{\omega}_{\bm{k}+\bm{\ell}+1}$ and setting $\bm{k}+\bm{\ell}=\bm{m}$ yields
\begin{align*}
E_2 
&=\frac{c_{1,n}}{2\pi(2\pi\lambda)^4 H(\bm{0})^2} \sum_{\bm{m}=-2a}^{2a-2} \int_{\R^4} B(\bm{u}) B(\bm{u}-\bm{\omega}_{\bm{m}+1}) B(\bm{v}) B(\bm{v}-\bm{\omega}_{\bm{m}+1}) \left(\frac{2\pi}{\lambda}\right)^2\\
& \Bigg[\frac{2\pi}{\lambda} \sum_{r=0}^{a-1} \tilde{\omega}_r \left(\frac{2\pi}{\lambda}\right)^2 \sum_{\bm{k}=\max(-a,\bm{m}-a+1)}^{\min(a-1,\bm{m}+a)} f(\bm{u}-\tilde{\bm{\omega}}_{\bm{k}}) f(\bm{v}-\tilde{\bm{\omega}}_{\bm{m}-\bm{k}}) J_0(\tilde{\omega}_r \|\tilde{\bm{\omega}}_{\bm{k}}\|) J_0(\tilde{\omega}_r \|\tilde{\bm{\omega}}_{\bm{m}-\bm{k}}\|)\Bigg] d\bm{u}  d\bm{v}.
\end{align*}
Since the expression in brackets is of order $\Landau( a^{2}/\lambda^{2})$, Lemma \ref{bounds_for_B} (iii) yields $E_2=\Landau(a^{2}/\lambda^{4})$,
and a similar argument also gives
$
E_3 =\Landau(a^2/\lambda^4)$.
Combining these arguments we finally obtain
\begin{align*}
\E\big[\hat{D}_{2,\lambda,a}\big]&=   D_{2,2,\lambda,a} + \Landau\Big(\frac{a^2}{\lambda^4}+\frac{a^{4}}{\lambda^{8}}+\frac{a^{3}}{\lambda^{6}}+\frac{a^{6}}{\lambda^{10}}+\frac{1}{n}\Big)
\end{align*}
as $a,\lambda,n\to\infty$.

\paragraph{Asymptotic variance:}
For $t=1,2$ and $c=1,2,3,4$, we define the random variables 
\begin{align*}
Y_{t,c}(\underline{j}):=\begin{cases}
h\left(\frac{\bm{s}_{j_{c+4(t-1)}}}{\lambda}\right) Z(\bm{s}_{j_{c+4(t-1)}}) \exp\big((-1)^{c+1}\, \im\, \bm{s}_{j_{c+4(t-1)}}^T \tilde{\bm{\omega}}_{\bm{k}_t}\big), \quad &\text{if } c=1,2,\\
h\left(\frac{\bm{s}_{j_{c+4(t-1)}}}{\lambda}\right) Z(\bm{s}_{j_{c+4(t-1)}}) \exp\big((-1)^{c+1}\, \im\, \bm{s}_{j_{c+4(t-1)}}^T \tilde{\bm{\omega}}_{\bm{\ell}_t}\big), \quad &\text{if }  c=3,4.
\end{cases}
\end{align*}
We now proceed similarly as in the proof of Theorem \ref{expectation_theo} and obtain
\begin{align} \label{variance_general_form}
\lambda^2 \Var\big[\hat{D}_{2,\lambda,a}\big]&= \sum_{r_1,r_2=0}^{a-1} \tilde{\omega}_{r_1} \tilde{\omega}_{r_2}\, \frac{(2\pi)^8}{n^8 H(\bm{0})^4} \sum_{\bm{k}_1,\bm{\ell}_1,\bm{k}_2,\bm{\ell}_2 =-a}^{a-1}\Bigg[\prod_{t=1}^2 \big[ J_0(\tilde{\omega}_{r_t} \|\tilde{\bm{\omega}}_{\bm{k}_t}\|) J_0(\tilde{\omega}_{r_t} \|\tilde{\bm{\omega}}_{\bm{\ell}_t}\|)\big]\nonumber \\
&\sum_{\bm{\nu}=\{\nu_1,\ldots,\nu_4\}\in\mathcal{I}} \sum_{\underline{j}\in\mathcal{D}(8)}\cum_2\big[Y_{t,c}(\underline{j}):(t,c)\in\nu_1\big]\times \ldots \times \cum_2\big[Y_{t,c}(\underline{j}):(t,c)\in\nu_4\big] \Bigg]\nonumber\\
&+\Landau\left(\frac{\lambda^2}{n}+\frac{1}{\lambda^2}\right),
\end{align}
where $\mathcal{I}$ is the set of indecomposable partitions of Table \eqref{table_q=2}, $\underline{j}=(j_1,\ldots,j_8)$, and the set $\mathcal{D}(8)$ is defined in \eqref{set_D(i)} for $i=8$. 
We also used that the overall order is determined by partitions of size $2$ that have maximally many different elements (see Corollary \ref{order_q=2_sec_int_for_variance} below).
We first consider the indecomposable partition
\begin{align*}
\bm{\nu}_{1}^{\ast}=\{\nu_{1,1}^{\ast},\nu_{1,2}^{\ast},\nu_{1,3}^{\ast},\nu_{1,4}^{\ast}\}=\Big\{\{(1,1),(1,2)\},\{(2,1),(2,2)\},\{(1,3),(2,4)\},\{(1,4),(2,3)\}\Big\}
\end{align*}
and recall the definition of $c_{q,n}$ in \eqref{c_q,n}.
Using similar arguments as in the proof of Theorem
\ref{expectation_theo}, 
the cumulant expression corresponding to the partition $\bm{\nu}_1^\ast$ then amounts to
\begin{align} \label{connection_l_only}
&\frac{ c_{2,n}}{(2\pi\lambda)^8 H(\bm{0})^4} \left(\frac{2\pi}{\lambda}\right)^8  \sum_{\bm{k}_1,\bm{\ell}_1,\bm{k}_2,\bm{\ell}_2 =-a}^{a-1} \int_{\R^8} B(\bm{s})^2 B(\bm{t})^2 B(\bm{u})B(\bm{u}+\bm{\omega}_{\bm{\ell}_2-\bm{\ell}_1})\nonumber\\ 
&\phantom{==========}B(\bm{v})B(\bm{v}-\bm{\omega}_{\bm{\ell}_2-\bm{\ell}_1})  f(\bm{s}-\tilde{\bm{\omega}}_{\bm{k}_1}) f(\bm{t}-\tilde{\bm{\omega}}_{\bm{k}_2}) f(\bm{u}-\tilde{\bm{\omega}}_{\bm{\ell}_1}) f(\bm{v}-\tilde{\bm{\omega}}_{\bm{\ell}_2})\nonumber\\
&\phantom{==========}\Bigg[ \sum_{r_1,r_2=0}^{a-1} \tilde{\omega}_{r_1} \tilde{\omega}_{r_2} \prod_{t=1}^2 \big[ J_0(\tilde{\omega}_{r_t} \|\tilde{\bm{\omega}}_{\bm{k}_t}\|) J_0(\tilde{\omega}_{r_t} \|\tilde{\bm{\omega}}_{\bm{\ell}_t}\|)\big]\Bigg] d\bm{s} d\bm{t} d\bm{u} d\bm{v},
\end{align}
and making the index shift $\bm{\ell}_2-\bm{\ell}_1=\bm{m}$, $\bm{\ell}_1=\bm{\ell}$, this expression equals
\begin{align*}
&\phantom{\leq i}\frac{c_{2,n}}{(2\pi \lambda)^8 H(\bm{0})^4} \left(\frac{2\pi}{\lambda}\right)^8 \sum_{\bm{m}=-2a+1}^{2a-1} \int_{\R^8} B(\bm{s})^2 B(\bm{t})^2 B(\bm{u})B(\bm{u}+\bm{\omega}_{\bm{m}}) B(\bm{v})B(\bm{v}-\bm{\omega}_{\bm{m}})\nonumber\\
&\phantom{====} \sum_{\bm{k}_1,\bm{k}_2=-a}^{a-1} \sum_{\bm{\ell}=\max(-a,-a-\bm{m})}^{\min(a-1,a-1-\bm{m})} f(\bm{s}-\tilde{\bm{\omega}}_{\bm{k}_1}) f(\bm{t}-\tilde{\bm{\omega}}_{\bm{k}_2}) f(\bm{u}-\tilde{\bm{\omega}}_{\bm{\ell}}) f(\bm{v}-\tilde{\bm{\omega}}_{\bm{m}+\bm{\ell}})\nonumber\\
& \phantom{====}\Bigg[\sum_{r_1,r_2=0}^{a-1} \tilde{\omega}_{r_1} \tilde{\omega}_{r_2} \prod_{t=1}^2 \big[ J_0(\tilde{\omega}_{r_t} \|\tilde{\bm{\omega}}_{\bm{k}_t}\|) J_0(\tilde{\omega}_{r_t} \|\tilde{\bm{\omega}}_{\bm{\ell}_t}\|)\big]\Bigg] d\bm{s} d\bm{t} d\bm{u} d\bm{v}.
\end{align*}
Tedious calculations now show that the last expression can be rewritten as
\begin{align} \label{variance_with_G_i}
&\sum_{\bm{m}=-2a+1}^{2a-1} \frac{H(\bm{m})^2}{H(\bm{0})^2}  \int_{2\pi\max(-a,-a-\bm{m})/\lambda}^{2\pi\min(a,a-\bm{m})/\lambda} f(\bm{z})^2\nonumber\\
&\phantom{===============} \Bigg( \int_{0}^{2\pi a/\lambda} r \Bigg[\int_{[-2\pi a/\lambda,2\pi a/\lambda]^2} f(\bm{x}) J_0(r \|\bm{x}\|)  J_0(r \|\bm{z}\|) \, d\bm{x}\Bigg]  dr\Bigg)^2  d\bm{z}\nonumber\\
&\phantom{==} + G_1(a,\lambda) + G_2(a,\lambda)+G_3(a,\lambda)+G_4(a,\lambda)+\Landau\left(\frac{1}{n}\right)
\end{align}
for $a,\lambda,n\to\infty$,
where
\begin{align*}
G_1(a,\lambda)&:=\frac{c_{2,n}}{(2\pi \lambda)^8 H(\bm{0})^4} \sum_{\bm{m}=-2a+1}^{2a-1} \int_{\R^8} B(\bm{s})^2 B(\bm{t})^2 B(\bm{u})B(\bm{u}+\bm{\omega}_{\bm{m}}) B(\bm{v})B(\bm{v}-\bm{\omega}_{\bm{m}})\nonumber\\
& \left(\frac{2\pi}{\lambda}\right)^6 \sum_{\bm{k}_1,\bm{k}_2=-a}^{a-1} \sum_{\bm{\ell}=\max(-a,-a-\bm{m})}^{\min(a-1,a-1-\bm{m})} f(\bm{s}-\tilde{\bm{\omega}}_{\bm{k}_1}) f(\bm{t}-\tilde{\bm{\omega}}_{\bm{k}_2}) f(\bm{u}-\tilde{\bm{\omega}}_{\bm{\ell}}) f(\bm{v}-\tilde{\bm{\omega}}_{\bm{m}+\bm{\ell}})\nonumber\\
& \Bigg[\left(\frac{2\pi}{\lambda}\right)^2 \sum_{r_1,r_2=0}^{a-1} \tilde{\omega}_{r_1} \tilde{\omega}_{r_2} J_0(\tilde{\omega}_{r_1} \|\tilde{\bm{\omega}}_{\bm{k}_1}\|) J_0(\tilde{\omega}_{r_1} \|\tilde{\bm{\omega}}_{\bm{\ell}}\|)J_0(\tilde{\omega}_{r_2} \|\tilde{\bm{\omega}}_{\bm{k}_2}\|) J_0(\tilde{\omega}_{r_2} \|\tilde{\bm{\omega}}_{\bm{m}+\bm{\ell}}\|)\\
&-\int_{[0,2\pi a /\lambda]^2} r_1\, r_2 \, J_0(r_1 \|\tilde{\bm{\omega}}_{\bm{k}_1}\|) J_0(r_1 \|\tilde{\bm{\omega}}_{\bm{\ell}}\|) J_0(r_2 \|\tilde{\bm{\omega}}_{\bm{k}_2}\|) J_0(r_2 \|\tilde{\bm{\omega}}_{\bm{m}+\bm{\ell}}\|)\, dr_1 dr_2 \Bigg] d\bm{s} d\bm{t} d\bm{u} d\bm{v},
\end{align*}
\begin{align*}
&G_2(a,\lambda):= \frac{c_{2,n}}{(2\pi\lambda)^8 H(\bm{0})^4} \sum_{\bm{m}=-2a+1}^{2a-1} \int_{\R^8} B(\bm{s})^2 B(\bm{t})^2 B(\bm{u}) B(\bm{u}+\bm{\omega_m}) B(\bm{v}) B(\bm{v}-\bm{\omega_m})\\
&\phantom{==}\Bigg(\int_{[0,2\pi a /\lambda]^2} r_1\, r_2 \Bigg[ \left(\frac{2\pi}{\lambda}\right)^6 \sum_{\bm{k}_1,\bm{k}_2=-a}^{a-1} \sum_{\bm{\ell}=\max(-a,-a-\bm{m})}^{\min(a-1,a-1-\bm{m})} f(\bm{s}-\tilde{\bm{\omega}}_{\bm{k}_1}) f(\bm{t}-\tilde{\bm{\omega}}_{\bm{k}_2}) f(\bm{u}-\tilde{\bm{\omega}}_{\bm{\ell}}) f(\bm{v}-\tilde{\bm{\omega}}_{\bm{m}+\bm{\ell}})\\
&\phantom{===========================}  J_0(r_1 \|\tilde{\bm{\omega}}_{\bm{k}_1}\|) J_0(r_1 \|\tilde{\bm{\omega}}_{\bm{\ell}}\|) J_0(r_2 \|\tilde{\bm{\omega}}_{\bm{k}_2}\|) J_0(r_2 \|\tilde{\bm{\omega}}_{\bm{m}+\bm{\ell}}\|)  \\
&\phantom{========} -\int_{[-2\pi a/\lambda,2\pi a/\lambda]^4} \Bigg\{ \int_{2\pi\max(-a,-a-\bm{m})/\lambda}^{2\pi\min(a,a-\bm{m})/\lambda} f(\bm{s}-\bm{x}) f(\bm{t}-\bm{y}) f(\bm{u}-\bm{z}) f(\bm{v}-(\bm{\omega_m}+\bm{z}))\\
&\phantom{=====================}  J_0(r_1 \|\bm{x}\|) J_0(r_1 \|\bm{z}\|) J_0(r_2 \|\bm{y}\|) J_0(r_2 \|\bm{\omega_m}+\bm{z}\|)\, d\bm{z}\Bigg\} d\bm{y} d\bm{x} \Bigg]\\
&\phantom{====}\, dr_1 dr_2 \Bigg)\, d\bm{s} d\bm{t} d\bm{u} d\bm{v},
\end{align*}
\begin{align*}
&G_3(a,\lambda):=\frac{c_{2,n}}{(2\pi \lambda)^8 H(\bm{0})^4} \sum_{\bm{m}=-2a+1}^{2a-1} \int_{\R^8} B(\bm{s})^2 B(\bm{t})^2 B(\bm{u})B(\bm{u}+\bm{\omega}_{\bm{m}}) B(\bm{v})B(\bm{v}-\bm{\omega}_{\bm{m}})\nonumber\\
& \Bigg(\int_{[0,2\pi a /\lambda]^2} r_1\, r_2 \, \Bigg[\int_{[-2\pi a/\lambda,2\pi a/\lambda]^4} \Bigg\{ \int_{2\pi\max(-a,-a-\bm{m})/\lambda}^{2\pi\min(a,a-\bm{m})/\lambda}\nonumber\\
&\phantom{=====i===}\Big[f(\bm{s}-\bm{x}) f(\bm{t}-\bm{y}) f(\bm{u}-\bm{z}) f(\bm{v}-(\bm{\omega_m}+\bm{z}))-f(\bm{x}) f(\bm{y}) f(\bm{z}) f(\bm{\omega_m}+\bm{z})\Big] \\
&\phantom{====i====} J_0(r_1 \|\bm{x}\|) J_0(r_1 \|\bm{z}\|) J_0(r_2 \|\bm{y}\|) J_0(r_2 \|\bm{\omega_m}+\bm{z}\|)\, d\bm{z}\Bigg\} d\bm{y} d\bm{x} \Bigg] dr_1 dr_2 \Bigg) d\bm{s} d\bm{t} d\bm{u} d\bm{v},\nonumber
\end{align*}
and
\begin{align*}
&G_4(a,\lambda):=c_{2,n} \sum_{\bm{m}=-2a+1}^{2a-1} \frac{H(\bm{m})^2}{H(\bm{0})^2}  \int_{[0,2\pi a /\lambda]^2} r_1\, r_2 \Bigg[\int_{[-2\pi a/\lambda,2\pi a/\lambda]^4} f(\bm{x}) f(\bm{y}) J_0(r_1 \|\bm{x}\|) J_0(r_2 \|\bm{y}\|)\, d\bm{x} d\bm{y} \nonumber\\
&\phantom{===} \int_{2\pi\max(-a,-a-\bm{m})/\lambda}^{2\pi\min(a,a-\bm{m})/\lambda} f(\bm{z})  J_0(r_1 \|\bm{z}\|) \Big\{f(\bm{\omega_m}+\bm{z}) J_0(r_2\|\bm{\omega_m}+\bm{z}\|)-f(\bm{z}) J_0(r_2 \|\bm{z}\|)\Big\} d\bm{z} \Bigg]  dr_1 dr_2.
\end{align*}
In order to obtain the cumulant expression corresponding to the partition $\bm{\nu}_1^{\ast}$, it remains to calculate the order of the remainders $G_1(a,\lambda),\ldots,G_4(a,\lambda)$ in \eqref{variance_with_G_i}.
{{We start with $G_1$ and obtain
\begin{align*}
&\phantom{=i}\Bigg|\left(\frac{2\pi}{\lambda}\right)^2  \sum_{r_1,r_2=0}^{a-1} \tilde{\omega}_{r_1}\, \tilde{\omega}_{r_2}\, J_0(\tilde{\omega}_{r_1} \|\tilde{\bm{\omega}}_{\bm{k}_1}\|) \, J_0(\tilde{\omega}_{r_1} \|\tilde{\bm{\omega}}_{\bm{\ell}}\|)\, J_0(\tilde{\omega}_{r_2} \|\tilde{\bm{\omega}}_{\bm{k}_2}\|)\,  J_0(\tilde{\omega}_{r_2} \|\tilde{\bm{\omega}}_{\bm{m}+\bm{\ell}}\|)\\
&\phantom{====}-\int_{[0,2\pi a /\lambda]^2} r_1 \, r_2 \, J_0(r_1\|\tilde{\bm{\omega}}_{\bm{k}_1}\|) \, J_0(r_1 \|\tilde{\bm{\omega}}_{\bm{\ell}}\|)\, J_0(r_2 \|\tilde{\bm{\omega}}_{\bm{k}_2}\|)\,  J_0(r_2\|\tilde{\bm{\omega}}_{\bm{m}+\bm{\ell}}\|) \, dr_1 dr_2\,\Bigg|\\
&\lesssim \frac{a^{2}}{\lambda^{2}} \times \frac{a^{2}}{\lambda^{4}} \big(\|\tilde{\bm{\omega}}_{\bm{k}_1}\|^2 +\|\tilde{\bm{\omega}}_{\bm{k}_2}\|^2 +\|\tilde{\bm{\omega}}_{\bm{\ell}}\|^2 +\|\tilde{\bm{\omega}}_{\bm{m}+\bm{\ell}}\|^2\big)\\
&\phantom{============}+\frac{a^{2}}{\lambda^{2}} \times \frac{a}{\lambda^{3}} \big(\|\tilde{\bm{\omega}}_{\bm{k}_1}\|_1 +\|\tilde{\bm{\omega}}_{\bm{k}_2}\|_1 +\|\tilde{\bm{\omega}}_{\bm{\ell}}\|_1 +\|\tilde{\bm{\omega}}_{\bm{m}+\bm{\ell}}\|_1\big),
\end{align*}
where we used the triangle inequality and Lemma \ref{Riemann_radius}. Since $\frac{1}{\lambda^2} \sum_{\bm{k}=-a}^{a-1} f(\bm{s}-\tilde{\bm{\omega}}_{\bm{k}})$ is uniformly bounded and $\sup_{|k|\leq a} |\tilde{\omega}_k|=\Landau(a/\lambda)$, Lemma \ref{bounds_for_B} (iii) and Lemma \ref{orders_of_B} (i) give
\begin{align*}
&|G_1(a,\lambda)|
\lesssim \frac{a^{4}}{\lambda^{6}}  \times \frac{1}{\lambda^2} \int_{\R^2} B(\bm{s})^2  \left(\frac{2\pi}{\lambda}\right)^2 \sum_{\bm{k}=-a}^{a-1} f(\bm{s}-\tilde{\bm{\omega}}_{\bm{k}}) \, \tilde{\omega}_{k_1}^2 \, d\bm{s}\\
&+\frac{a^{4}}{\lambda^{6}}  \times \frac{1}{\lambda^4} \sum_{\bm{m}=-2a+1}^{2a-1} \int_{\R^4} |B(\bm{u}) B(\bm{u}+\bm{\omega_m}) B(\bm{v}) B(\bm{v}-\bm{\omega_m})| \left(\frac{2\pi}{\lambda}\right)^2 \sum_{\bm{\ell}=-a}^{a-1} f(\bm{u}-\tilde{\bm{\omega}}_{\bm{\ell}}) \, \tilde{\omega}_{\ell_{1}}^2\, d\bm{u} d\bm{v}+ \Landau\left(\frac{a^4}{\lambda^6}\right)\\
&=\Landau\left(\frac{a^{4}}{\lambda^{6}} \left[1+\frac{(\log a)^4}{\lambda^2}\right]+\frac{a^{6}}{\lambda^{10}} \right),
\end{align*}
as $a,\lambda\to\infty$, where we used Lemma \ref{f_times_x^2} and the assumption $\delta>2$ for the second step. 
We now deal with the term $G_2(a,\lambda)$ in \eqref{variance_with_G_i}. 
An application of the triangle inequality yields
\begin{align*}
|G_2(a,\lambda)|\lesssim G_{21}+G_{22}+G_{23},
\end{align*} 
where 
\begin{align*}
G_{21}&:=\frac{1}{\lambda^8} \sum_{\bm{m}=-2a+1}^{2a-1} \int_{\R^8} B(\bm{s})^2 B(\bm{t})^2 |B(\bm{u}) B(\bm{u}+\bm{\omega_m}) B(\bm{v}) B(\bm{v}-\bm{\omega_m})|\\
& \phantom{===} \Bigg(\int_{0}^{2\pi a /\lambda} \int_{0}^{2\pi a /\lambda} r_1\, r_2\, \Bigg| \left(\frac{2\pi}{\lambda}\right)^4 \sum_{\bm{k}_1,\bm{k}_2 = -a}^{a-1} f(\bm{s}-\tilde{\bm{\omega}}_{\bm{k}_1}) f(\bm{t}-\tilde{\bm{\omega}}_{\bm{k}_2}) J_0(r_1 \|\tilde{\bm{\omega}}_{\bm{k}_1}\|) J_0(r_2 \|\tilde{\bm{\omega}}_{\bm{k}_2}\|) \Bigg|\\
&\phantom{=======}\Bigg| \left(\frac{2\pi}{\lambda}\right)^2 \sum_{\bm{\ell}=\max(-a,-a-\bm{m})}^{\min(a-1,a-1-\bm{m})} f(\bm{u}-\tilde{\bm{\omega}}_{\bm{\ell}}) f(\bm{v}-\tilde{\bm{\omega}}_{\bm{m}+\bm{\ell}}) J_0(r_1 \|\tilde{\bm{\omega}}_{\bm{\ell}}\|) J_0(r_2 \|\tilde{\bm{\omega}}_{\bm{m}+\bm{\ell}}\|)\\
&\phantom{=======} - \int_{2\pi\max(-a,-a-\bm{m})/\lambda}^{2\pi\min(a,a-\bm{m})/\lambda} f(\bm{u}-\bm{z}) f(\bm{v}-(\bm{\omega}_{\bm{m}}+\bm{z}))  J_0(r_1 \|\bm{z}\|) J_0(r_2 \|\bm{\bm{\omega_m}+\bm{z}}\|) \, d\bm{z} \, \Bigg|\\
&\phantom{===}\, dr_1  dr_2\Bigg) d\bm{s} d\bm{t} d\bm{u} d\bm{v},
\end{align*}
\begin{align*}
G_{22}&:=\frac{1}{\lambda^8} \sum_{\bm{m}=-2a+1}^{2a-1} \int_{\R^8} B(\bm{s})^2 B(\bm{t})^2 |B(\bm{u}) B(\bm{u}+\bm{\omega_m}) B(\bm{v}) B(\bm{v}-\bm{\omega_m})|\\
&\phantom{===} \Bigg(\int_{0}^{2\pi a /\lambda} \int_{0}^{2\pi a /\lambda} r_1\, r_2\, \Bigg|\int_{2\pi\max(-a,-a-\bm{m})/\lambda}^{2\pi\min(a,a-\bm{m})/\lambda} f(\bm{u}-\bm{z}) f(\bm{v}-(\bm{\omega}_{\bm{m}}+\bm{z})) \\
&\phantom{=========ii==} J_0(r_1 \|\bm{z}\|) J_0(r_2 \|\bm{\bm{\omega_m}+\bm{z}}\|) \, d\bm{z} \, \left(\frac{2\pi}{\lambda}\right)^2 \sum_{\bm{k}_1=-a}^{a-1}  f(\bm{s}-\tilde{\bm{\omega}}_{\bm{k}_1})  J_0(r_1 \|\tilde{\bm{\omega}}_{\bm{k}_1}\|) \Bigg|\\
&\phantom{=======} \Bigg| \left(\frac{2\pi}{\lambda}\right)^2 \sum_{\bm{k}_2=-a}^{a-1} f(\bm{t}-\tilde{\bm{\omega}}_{\bm{k}_2}) J_0(r_2 \|\tilde{\bm{\omega}}_{\bm{k}_2}\|) - \int_{[-2\pi a/\lambda,2\pi a/\lambda]^2} f(\bm{t}-\bm{y})\, J_0(r_2 \|\bm{y}\|)\, d\bm{y} \Bigg|\\
&\phantom{===}\, dr_1 dr_2 \Bigg)  d\bm{s} d\bm{t} d\bm{u} d\bm{v},
\end{align*}
and
\begin{align*}
G_{23}&:=\frac{1}{\lambda^8} \sum_{\bm{m}=-2a+1}^{2a-1} \int_{\R^8} B(\bm{s})^2 B(\bm{t})^2 |B(\bm{u}) B(\bm{u}+\bm{\omega_m}) B(\bm{v}) B(\bm{v}-\bm{\omega_m})|\\
& \phantom{===} \Bigg(\int_{0}^{2\pi a /\lambda} \int_{0}^{2\pi a /\lambda} r_1\, r_2\, \Bigg| \int_{[-2\pi a/\lambda,2\pi a/\lambda]^2} f(\bm{t}-\bm{y})\, J_0(r_2 \|\bm{y}\|)\, d\bm{y}\\
&\phantom{=====ii===} \int_{2\pi\max(-a,-a-\bm{m})/\lambda}^{2\pi\min(a,a-\bm{m})/\lambda} f(\bm{u}-\bm{z}) f(\bm{v}-(\bm{\omega}_{\bm{m}}+\bm{z}))  J_0(r_1 \|\bm{z}\|) J_0(r_2 \|\bm{\bm{\omega_m}+\bm{z}}\|) \, d\bm{z} \Bigg|\\
&\phantom{========} \Bigg|\left(\frac{2\pi}{\lambda}\right)^2\sum_{\bm{k}_1=-a}^{a-1}  f(\bm{s}-\tilde{\bm{\omega}}_{\bm{k}_1})  J_0(r_1 \|\tilde{\bm{\omega}}_{\bm{k}_1}\|)  - \int_{[-2\pi a/\lambda,2\pi a/\lambda]^2} f(\bm{s}-\bm{x}) J_0(r_1 \|\bm{x}\|)\, d\bm{x} \Bigg| \\
&\phantom{===}\, dr_1 dr_2 \Bigg) d\bm{s} d\bm{t} d\bm{u} d\bm{v}.
\end{align*}
We start with the term $G_{21}$ and note that 
\begin{align*}
G_{21} 
&\lesssim \frac{1}{\lambda^8} \sum_{\bm{m}=-2a+1}^{2a-1} \Bigg(\int_{\R^8} B(\bm{s})^2 B(\bm{t})^2 |B(\bm{u}) B(\bm{u}+\bm{\omega_m}) B(\bm{v}) B(\bm{v}-\bm{\omega_m})|\\
&  \int_{\|\bm{x}\|\leq 2\pi a/\lambda} \int_{\|\bm{y}\|\leq 2\pi a/\lambda} \Bigg|\left(\frac{2\pi}{\lambda}\right)^4 \sum_{\bm{k}_1,\bm{k}_2=-a}^{a-1} f(\bm{s}-\tilde{\bm{\omega}}_{\bm{k}_1}) f(\bm{t}-\tilde{\bm{\omega}}_{\bm{k}_2}) \exp\big(\im  \bm{x}^T \tilde{\bm{\omega}}_{\bm{k}_1}\big)\, \exp\big(\im  \bm{y}^T \tilde{\bm{\omega}}_{\bm{k}_2}\big)\Bigg| d\bm{x} d\bm{y} \\
&\phantom{=========} d\bm{s} d\bm{t} d\bm{u} d\bm{v}\Bigg) \times \overline{G_{21}},
\end{align*}
where
\begin{align*}
\overline{G_{21}}&:=\sup_{\substack{|m_1|\leq 2a-1\\|m_2|\leq 2a-1}}\sup_{\substack{\bm{u}\in\R^2\\ \bm{v}\in\R^2}} \sup_{\substack{\|\bm{x}\|\leq 2\pi a/\lambda\\ \|\bm{y}\|\leq 2\pi a/\lambda}} \Bigg| \left(\frac{2\pi}{\lambda}\right)^2 \sum_{\bm{\ell}=\max(-a,-a-\bm{m})}^{\min(a-1,a-1-\bm{m})} f(\bm{u}-\tilde{\bm{\omega}}_{\bm{\ell}}) f(\bm{v}-(\bm{\omega}_{\bm{m}}+\tilde{\bm{\omega}}_{\bm{\ell}})) \\
&\phantom{=============================} J_0(\|\bm{x}\|  \|\tilde{\bm{\omega}}_{\bm{\ell}}\|) J_0(\|\bm{y}\| \|\bm{\omega}_{\bm{m}}+\tilde{\bm{\omega}}_{\bm{\ell}}\|)\\
& \phantom{=================}-\int_{2\pi\max(-a,-a-\bm{m})/\lambda}^{2\pi\min(a,a-\bm{m})/\lambda} f(\bm{u}-\bm{z}) f(\bm{v}-(\bm{\omega}_{\bm{m}}+\bm{z})) \\
&\phantom{=============================}J_0(\|\bm{x}\| \|\bm{z}\|) J_0(\|\bm{y}\|\|\bm{\omega}_{\bm{m}}+\bm{z}\|) \, d\bm{z}\,\Bigg|. 
\end{align*}
Using the definition of the Bessel function and analogous arguments as in the proof of Lemma \ref{Riemann_f_sec_int}, we obtain
\begin{align*}
\overline{G_{21}}&
\lesssim \sup_{\substack{\|\bm{x}'\|\leq 4\pi a/\lambda}} \frac{1+\|\bm{x}'\|_1+\|\bm{x}'\|^2}{\lambda^2} \lesssim \frac{a^{2}}{\lambda^{4}}.
\end{align*}

Therefore, 
\begin{align*}
G_{21} 
&\lesssim\frac{a^{2}}{\lambda^{4}} \Bigg(\frac{1}{\lambda^2} \int_{\R^2} B(\bm{s})^2 \Bigg[\int_{\|\bm{x}\|\leq 2\pi a/\lambda} \Bigg|\left(\frac{2\pi}{\lambda}\right)^2 \sum_{\bm{k}=-a}^{a-1} f(\bm{s}-\tilde{\bm{\omega}}_{\bm{k}}) \exp\big(\im \bm{x}^T \tilde{\bm{\omega}}_{\bm{k}}\big)\,\Bigg|\, d\bm{x} \Bigg] d\bm{s} \Bigg)^2\\
&=\Landau\left(\frac{a^{2}}{\lambda^{4}} \Bigg[1+\frac{a}{\lambda^{2}}+\frac{a^{4}}{\lambda^{6}}\Bigg]^2\right),
\end{align*}
where we used Lemma \ref{bounds_for_B}, part (iii), Lemma \ref{cov_fct_integrable}, part (i), and the assumption $\delta>2$. 
Using similar arguments as above, we get
\begin{align*}
G_{22} 
&\lesssim \frac{a^{2}}{\lambda^{2}} \times \frac{a^{2}}{\lambda^{4}} \times \frac{1}{\lambda^2} \int_{\R^2} B(\bm{s})^2 \Bigg(\int_{\|\bm{x}\|\leq 2\pi a/\lambda} \Bigg| \left(\frac{2\pi}{\lambda}\right)^2 \sum_{\bm{k}_1=-a}^{a-1} f(\bm{s}-\tilde{\bm{\omega}}_{\bm{k}_1}) \exp\big(\im \bm{x}^T \tilde{\bm{\omega}}_{\bm{k}_1}\big) \Bigg| \, d\bm{x} \Bigg)\, d\bm{s}\\
&=\Landau\left(\frac{a^{4}}{\lambda^{6}} \Bigg[1+\frac{a}{\lambda^{2}}+\frac{a^{4}}{\lambda^{6}}\Bigg]\right),
\end{align*}
and similarly, using Lemma \ref{cov_fct_integrable}, part (ii) instead of part (i), 
$G_{23}=\Landau(\frac{a^{4}}{\lambda^{6}} [1+\frac{a}{\lambda^{2}}]).$
Therefore, it holds that
\begin{align*}
G_2(a,\lambda)=\Landau\left(\frac{a^{2}}{\lambda^{4}}\left[1+\frac{a}{\lambda^{2}}+\frac{a^{4}}{\lambda^{6}}\right]^2+\frac{a^{4}}{\lambda^{6}}  \left[1+\frac{a}{\lambda^{2}}+\frac{a^{4}}{\lambda^{6}}\right]\right)
\end{align*}
as $a,\lambda\to\infty$. 
Concerning the term $G_3(a,\lambda)$ in \eqref{variance_with_G_i}, note that 
\begin{align} \label{eq_3_var_sec_int}
G_3(a,\lambda)&=
\frac{c_{8,n}}{(2\pi\lambda)^6 H(\bm{0})^3} \sum_{\bm{m}=-2a+1}^{2a-1} \Bigg[\int_{\R^6} B(\bm{s})^2 B(\bm{t})^2 B(\bm{u}) B(\bm{u}+\bm{\omega}_{\bm{m}}) D_{\bm{m}}^{(1)}(\bm{s},\bm{t},\bm{u})\, d\bm{s}d\bm{t} d\bm{u}\nonumber\\
&\phantom{==============}+\int_{\R^6} B(\bm{s})^2 B(\bm{t})^2 B(\bm{v}) B(\bm{v}-\bm{\omega}_{\bm{m}}) D_{\bm{m}}^{(2)}(\bm{s},\bm{t})\, d\bm{s}d\bm{t} d\bm{v}\nonumber\\
&\phantom{==============}+\int_{\R^6} B(\bm{s})^2 B(\bm{u}) B(\bm{u}+\bm{\omega}_{\bm{m}}) B(\bm{v}) B(\bm{v}-\bm{\omega}_{\bm{m}}) D_{\bm{m}}^{(3)}(\bm{s})\, d\bm{s}d\bm{u} d\bm{v}\nonumber\\
&\phantom{==============}+\int_{\R^6} B(\bm{t})^2 B(\bm{u}) B(\bm{u}+\bm{\omega}_{\bm{m}}) B(\bm{v}) B(\bm{v}-\bm{\omega}_{\bm{m}}) D_{\bm{m}}^{(4)}\, d\bm{t}d\bm{u} d\bm{v}\Bigg],
\end{align}
where
\begin{align*}
D_{\bm{m}}^{(1)}(\bm{s},\bm{t},\bm{u})&:= \frac{1}{(2\pi\lambda)^2 H(\bm{0})} \int_{\R^2} B(\bm{v}) B(\bm{v}-\bm{\omega}_{\bm{m}}) \Bigg(\int_{[0,2\pi a /\lambda]^2} r_1\, r_2 \Bigg[\int_{[-2\pi a/\lambda,2\pi a/\lambda]^4} \\
&\phantom{:=\times}\Bigg\{\int_{2\pi\max(-a,-a-\bm{m})/\lambda}^{2\pi\min(a,a-\bm{m})/\lambda} f(\bm{s}-\bm{x}) f(\bm{t}-\bm{y})f(\bm{u}-\bm{z})
 \big[f(\bm{v}-\bm{\omega}_{\bm{m}}-\bm{z})-f(\bm{\omega}_{\bm{m}}+\bm{z})\big]\\
 &\phantom{:=\times }   J_0(r_1\|\bm{x}\|) \, J_0(r_1 \|\bm{z}\|)\, J_0(r_2 \|\bm{y}\|)\,  J_0(r_2\|\bm{\omega_m}+\bm{z}\|) \, d\bm{z}\Bigg\}\, d\bm{y} d\bm{x} \Bigg]\, dr_1 dr_2\Bigg)\, d\bm{v},
\end{align*}
\begin{align*}
D_{\bm{m}}^{(2)}(\bm{s},\bm{t})&:= \frac{1}{(2\pi\lambda)^2 H(\bm{0})} \int_{\R^2} B(\bm{u}) B(\bm{u}+\bm{\omega}_{\bm{m}}) \Bigg(\int_{[0,2\pi a /\lambda]^2} r_1\, r_2 \Bigg[\int_{[-2\pi a/\lambda,2\pi a/\lambda]^4} \\
&\phantom{:=\times}\Bigg\{ \int_{2\pi\max(-a,-a-\bm{m})/\lambda}^{2\pi\min(a,a-\bm{m})/\lambda} f(\bm{\omega}_{\bm{m}}+\bm{z}) f(\bm{s}-\bm{x})f(\bm{t}-\bm{y}) \big[f(\bm{u}-\bm{z})-f(\bm{z})\big]\\
&\phantom{:=\times}    J_0(r_1\|\bm{x}\|) \, J_0(r_1 \|\bm{z}\|)\, J_0(r_2 \|\bm{y}\|)\,  J_0(r_2\|\bm{\omega_m}+\bm{z}\|) \, d\bm{z}\Bigg\} \,d\bm{y} d\bm{x} \Bigg]\, dr_1 dr_2 \Bigg)  d\bm{u},
\end{align*}
\begin{align*}
D_{\bm{m}}^{(3)}(\bm{s})&:= \frac{1}{(2\pi\lambda)^2 H(\bm{0})} \int_{\R^2} B(\bm{t})^2 \Bigg(\int_{[0,2\pi a /\lambda]^2} r_1\, r_2 \Bigg[\int_{[-2\pi a/\lambda,2\pi a/\lambda]^4} \Bigg\{\int_{2\pi\max(-a,-a-\bm{m})/\lambda}^{2\pi\min(a,a-\bm{m})/\lambda}\\
&\phantom{:=\times}f(\bm{\omega}_{\bm{m}}+\bm{z}) f(\bm{z})f(\bm{s}-\bm{x}) \big[f(\bm{t}-\bm{y})-f(\bm{y})\big]\\
&\phantom{:=\times}    J_0(r_1\|\bm{x}\|) \, J_0(r_1 \|\bm{z}\|)\, J_0(r_2 \|\bm{y}\|)\,  J_0(r_2\|\bm{\omega_m}+\bm{z}\|) \, d\bm{z}\Bigg\}\, d\bm{y} d\bm{x} \Bigg] \,dr_1 dr_2 \Bigg) d\bm{t},
\end{align*}
\begin{align*}
D_{\bm{m}}^{(4)}&:= \frac{1}{(2\pi\lambda)^2 H(\bm{0})} \int_{\R^2} B(\bm{s})^2 \Bigg(\int_{[0,2\pi a /\lambda]^2} r_1\, r_2 \Bigg[\int_{[-2\pi a/\lambda,2\pi a/\lambda]^4}\Bigg\{ \int_{2\pi\max(-a,-a-\bm{m})/\lambda}^{2\pi\min(a,a-\bm{m})/\lambda}\\
&\phantom{:=\times} f(\bm{\omega}_{\bm{m}}+\bm{z}) f(\bm{z})f(\bm{y}) \big[f(\bm{s}-\bm{x})-f(\bm{x})\big]\\
&\phantom{:=\times}   J_0(r_1\|\bm{x}\|) \, J_0(r_1 \|\bm{z}\|)\, J_0(r_2 \|\bm{y}\|)\,  J_0(r_2\|\bm{\omega_m}+\bm{z}\|) \, d\bm{z}\Bigg\}\, d\bm{y} d\bm{x} \Bigg]\, dr_1 dr_2\Bigg)  d\bm{s}.
\end{align*}
It obviously holds that
\begin{align*}
&\left|D_{\bm{m}}^{(1)}(\bm{s},\bm{t},\bm{u})\right|\leq E_{\bm{m}}^{(1)}(\bm{u})\\
&\phantom{=====}\times \left(\int_{[0,2\pi a/\lambda]^2} r_1 r_2 \,\Bigg|\int_{[-2\pi a/\lambda,2\pi a/\lambda]^4} f(\bm{s}-\bm{x}) f(\bm{t}-\bm{y}) J_0(r_1 \|\bm{x}\|)  J_0(r_2 \|\bm{y}\|)\, d\bm{x} d\bm{y}\Bigg|\, dr_1 dr_2\right),
\end{align*}
where
\begin{align*}
E_{\bm{m}}^{(1)}(\bm{u})&:= \sup_{r_1,r_2>0} \Bigg|\frac{1}{(2\pi\lambda)^2 H(\bm{0})} \int_{\R^2} B(\bm{v}) B(\bm{v}-\bm{\omega}_{\bm{m}})
\Bigg[ \int_{2\pi\max(-a,-a-\bm{m})/\lambda}^{2\pi\min(a,a-\bm{m})/\lambda}
f(\bm{u}-\bm{z}) \, J_0(r_1 \|\bm{z}\|)\\
&\phantom{==============} J_0(r_2 \|\bm{\omega_m}+\bm{z}\|)\big[f(\bm{v}-\bm{\omega}_{\bm{m}}-\bm{z})-f(\bm{\omega}_{\bm{m}}+\bm{z})\big] d\bm{z} \Bigg] d\bm{v}\Bigg|.
\end{align*}
By Proposition \ref{Lemma F.2_SSR}, part (ii), we have
\begin{align*}
&\phantom{==i}\sup_{\substack{|m_1|\leq 2a-1\\ |m_2|\leq 2a-1}} \sup_{\bm{u}\in\R^d} E_{\bm{m}}^{(1)}(\bm{u})
\lesssim\frac{1}{\lambda},
\end{align*}
and therefore
\begin{align*}
&\phantom{\leq ii}\frac{1}{\lambda^6} \sum_{\bm{m}=-2a+1}^{2a-1} \int_{\R^6} B(\bm{s})^2 B(\bm{t})^2 |B(\bm{u}) B(\bm{u}+\bm{\omega_m})\, D_{\bm{m}}^{(1)}(\bm{s},\bm{t},\bm{u})|\, d\bm{s} d\bm{t} d\bm{u}
&=\Landau\left(\frac{(\log a)^4}{\lambda}
\left[1+\frac{a}{\lambda^{2}}\right]^2\right),
\end{align*}
where we furthermore used Lemma \ref{bounds_for_B}, part (ii), Lemma \ref{cov_fct_integrable}, part (ii), and the assumption $\delta>2$. 
 In the same way, we obtain
\begin{align*}
\frac{1}{\lambda^6} \sum_{\bm{m}=-2a+1}^{2a-1} \int_{\R^6} B(\bm{s})^2 B(\bm{t})^2 |B(\bm{v}) B(\bm{v}-\bm{\omega_m})\, D_{\bm{m}}^{(2)}(\bm{s},\bm{t})|\, d\bm{s} d\bm{t} d\bm{v}= \Landau\left(\frac{(\log a)^4}{\lambda}  \left[1+\frac{a}{\lambda^{2}}\right]^2\right).
\end{align*}
Moreover, Proposition \ref{Lemma F.2_SSR} (i) yields
\begin{align*}
|D_{\bm{m}}^{(3)}(\bm{s})|&\lesssim \frac{a^2}{\lambda^4} \int_{[0,2\pi a/\lambda]^2} r_1\, r_2 \Bigg|\int_{[-2\pi a/\lambda,2\pi a/\lambda]^2} \int_{2\pi\max(-a,-a-\bm{m})/\lambda}^{2\pi\min(a,a-\bm{m})/\lambda} f(\bm{\omega_m}+\bm{z}) f(\bm{z}) f(\bm{s}-\bm{x})\\
& \phantom{=============} J_0(r_1\|\bm{x}\|) J_0(r_1\|\bm{z}\|) J_0(r_2 \|\bm{\omega_m}+\bm{z}\|)\, d\bm{z} d\bm{x}\Bigg| \, dr_1 dr_2,
\end{align*}
such that Lemma \ref{bounds_for_B} (iii) and Lemma \ref{cov_fct_integrable} (ii) for $\delta>2$ give
\begin{align*}
&\phantom{\leq =} \frac{1}{\lambda^6} \sum_{\bm{m}=-2a+1}^{2a-1} \int_{\R^6} B(\bm{s})^2 |B(\bm{u}) B(\bm{u}+\bm{\omega_m}) B(\bm{v}) B(\bm{v}-\bm{\omega_m}) \, D_{\bm{m}}^{(3)}(\bm{s})|\, d\bm{s} d\bm{u} d\bm{v}\\
&=\Landau\left(\frac{a^{2}}{\lambda^{4}} \times \frac{a^2}{\lambda^2}\times \left[1+\frac{a}{\lambda^{2}}\right]\right) = \Landau\left(\frac{a^{4}}{\lambda^{6}}  \left[1+\frac{a}{\lambda^{2}}\right]\right).
\end{align*}
Similarly, it follows from Corollary \ref{cor_of_D2_finite} that
\begin{align*}
&\frac{1}{\lambda^6} \sum_{\bm{m}=-2a+1}^{2a-1} \int_{\R^6} B(\bm{t})^2 |B(\bm{u}) B(\bm{u}+\bm{\omega_m}) B(\bm{v}) B(\bm{v}-\bm{\omega_m}) \, D_{\bm{m}}^{(4)}|\, d\bm{t} d\bm{u} d\bm{v}
= \Landau\left(\frac{a^{4}}{\lambda^{6}}\right),
\end{align*}
and therefore
\begin{align*}
G_3(a,\lambda)=\Landau\left(\frac{(\log a)^4}{\lambda}  \left[1+\frac{a}{\lambda^{2}}\right]^2+\frac{a^{4}}{\lambda^{6}}  \left[1+\frac{a}{\lambda^{2}}\right]\right).
\end{align*}
We finally deal with the error term $G_4(a,\lambda)$ in \eqref{variance_with_G_i} and note that
\begin{align*}
|G_4(a,\lambda)|&
\lesssim \sum_{\bm{m}=-2a+1}^{2a-1} H(\bm{m})^2 \left(\int_{\|\bm{y}\|\leq 2\pi a/\lambda} \Bigg| \int_{[-2\pi a/\lambda,2\pi a/\lambda]^2} f(\bm{x}) \exp\big(\im  \bm{y}^T \bm{x}\big) \, d\bm{x}\, \Bigg|\, d\bm{y} \right)^2 \\
&\times \sup_{\|\bm{x}\|\leq 2\pi a/\lambda} \Bigg(\int_{[-2\pi a/\lambda,2\pi a/\lambda]^2} f(\bm{z}) \Big|f(\bm{z}+\bm{\omega_m})\exp\big(\im \bm{x}^T (\bm{\omega_m}+\bm{z})\big) - f(\bm{z})  \exp\big(\im \bm{x}^T \bm{z}\big)\Big|\, d\bm{z}\Bigg)\\
&=\Landau\left(\frac{a\, (\log \lambda)^2}{\lambda^{2}} + \frac{1}{(\log \lambda)^3}\right),
\end{align*}
where we used Corollary \ref{cor_of_D2_finite} for $\delta> 2$ and Lemma \ref{bound_variance_secint}.}}
Observing \eqref{variance_with_G_i} and the orders of the terms $G_i(a,\lambda)$ for $i=1,\ldots,4$, the cumulant expression corresponding to the partition $\bm{\nu}_1^{\ast}$ thus equals
\begin{align*}
&\sum_{\bm{m}=-2a+1}^{2a-1} \frac{H(\bm{m})^2}{H(\bm{0})^2}  \int_{2\pi\max(-a,-a-\bm{m})/\lambda}^{2\pi\min(a,a-\bm{m})/\lambda} f(\bm{z})^2\nonumber\\
&\phantom{===============} \Bigg( \int_{0}^{2\pi a/\lambda} r \Bigg[\int_{[-2\pi a/\lambda,2\pi a/\lambda]^2} f(\bm{x}) J_0(r \|\bm{x}\|)  J_0(r \|\bm{z}\|) \, d\bm{x}\Bigg] \, dr\Bigg)^2 \, d\bm{z}\nonumber\\
&\phantom{=} + \Landau\Bigg(\frac{a^{6}}{\lambda^{10}} + \frac{a^{4}}{\lambda^{6}}\left[1+\frac{(\log a)^4}{\lambda^2}+\frac{a}{\lambda^{2}}+\frac{a^{4}}{\lambda^{6}}\right]\\
&\phantom{===========}+\left[\frac{(\log a)^4}{\lambda}+\frac{a^{2}}{\lambda^{4}}\right] \left[1+\frac{a}{\lambda^{2}}+\frac{a^{4}}{\lambda^{6}}\right]^2 + \frac{a\, (\log \lambda)^2}{\lambda^{2}}+ \frac{1}{(\log \lambda)^3}+\frac{1}{n}\Bigg)
\end{align*}
as $a,\lambda,n\to\infty$.
Recall that in order to calculate the variance of the estimator $\hat{D}_{2,\lambda,a}$, we need to consider all indecomposable partitions of Table \eqref{table_q=2} consisting of $4$ groups with $8$ different elements [see \eqref{variance_general_form}].
As in the proof of Theorem \ref{expectation_theo}, the order corresponding to a certain partition depends on the rank of the linear transformation mapping $(\bm{k}_1^T,\bm{k}_2^T,\bm{\ell}_1^T,\bm{\ell}_2^T)^T$ to the vector of linear combinations of the $\bm{k}'s$ and $\bm{\ell}'s$ inside the $B$-functions. This rank equals the number of independent restrictions between the rows and columns of Table \eqref{table_q=2}. Here, a restriction can appear both by a set connecting the two rows, or by a set that connects the first or second column with the third or fourth column. For the partition $\bm{\nu}_1^{\ast}$ from above, there was 1 restriction between the variables $\bm{\ell}_1$ and $\bm{\ell}_2$ [corresponding to a restriction between the first and the second row of Table \eqref{table_q=2}], since there was one combination of $\bm{\ell}_2-\bm{\ell}_1$ inside the $B$-functions only [see \eqref{connection_l_only}]. The order of the respective cumulant expression was then $\Landau(1)$. Note that there are $8$ indecomposable partitions of Table \eqref{table_q=2} which evoke exactly $1$ restriction, namely 
\begin{enumerate}
\item two partitions restricting $\bm{\ell}_1$ and $\bm{\ell}_2$
\item two partitions restricting $\bm{k}_1$ and $\bm{k}_2$
\item two partitions restricting $\bm{\ell}_1$ and $\bm{k}_2$
\item two partitions restricting $\bm{k}_1$ and $\bm{\ell}_2$.
\end{enumerate}
A straightforward calculation shows that all of these $8$ partitions yield the same cumulant expression as the partition $\bm{\nu}_1^{\ast}$ in \eqref{variance_with_G_i}. 
For partitions with $2$ or more restrictions between the variables $\bm{k}_1$, $\bm{k}_2$, $\bm{\ell}_1$, and $\bm{\ell}_2$ however, we obtain an expression of lower order (note that partitions with $0$ restrictions are not considered since they cannot be indecomposable).
For illustration, consider the indecomposable partition
\begin{align*}
\bm{\nu}_2^{\ast}=\{\nu_{2,1}^\ast,\nu_{2,2}^\ast,\nu_{2,3}^\ast,\nu_{2,4}^\ast\}=\Big\{\{(1,1),(1,3)\},\{(1,2),(2,1)\},\{(1,4),(2,3)\},\{(2,2),(2,4)\}\Big\},
\end{align*}
                                                                                                                                 which evokes $3$ independent restrictions in Table \eqref{table_q=2}.
                                                                                                                                We obtain
\begin{align} \label{rank_3}
&\phantom{=i}\frac{1}{n^8} \sum_{\underline{j}\in\mathcal{D}(8)} \cum_2\big[Y_{t,c}(\underline{j}):(t,c)\in\nu_{2,1}^{\ast}\big]\times \ldots \times \cum_2\big[Y_{t,c}(\underline{j}):(t,c)\in\nu_{2,4}^{\ast}\big]\nonumber \\
&=\frac{c_{8,n}}{(2\pi)^8 \lambda^{16}} \int_{\R^8} B(\bm{s}) B(\bm{s}+\bm{\omega}_{\bm{k}_1+\bm{\ell}_1+1}) B(\bm{t}) B(\bm{t}+\bm{\omega}_{\bm{k}_2+\bm{\ell}_2+1}) B(\bm{u}) B(\bm{u}+\bm{\omega}_{\bm{k}_2-\bm{k}_1}) B(\bm{v}) B(\bm{v}+\bm{\omega}_{\bm{\ell}_2-\bm{\ell}_1})\nonumber\\
&\phantom{=========} f(\bm{s}+\tilde{\bm{\omega}}_{\bm{\ell}_1})  f(\bm{t}+\tilde{\bm{\omega}}_{\bm{\ell}_2}) f(\bm{u}-\tilde{\bm{\omega}}_{\bm{k}_1}) f(\bm{v}-\tilde{\bm{\omega}}_{\bm{\ell}_1})\, d\bm{s} d\bm{t} d\bm{u} d\bm{v}.
\end{align}                      
                                                                                                                                 
                                                                                                                               The matrix $K$ satisfying
\begin{align*}
K \times \begin{pmatrix}
\bm{k}_1^T\\
\bm{k}_2^T\\
\bm{\ell}_1^T\\
\bm{\ell}_2^T\\
\end{pmatrix} = \begin{pmatrix}
(\bm{k}_1+\bm{\ell}_1)^T\\
(\bm{k}_2+\bm{\ell}_2)^T\\
(\bm{k}_2-\bm{k}_1)^T\\
(\bm{\ell}_2-\bm{\ell}_1)^T\\
\end{pmatrix}
\end{align*}                                                                                                                               
has rank $3$,
                                                                                                                             and we set $\bm{k}_1+\bm{\ell}_1+1=\bm{m}_1$, $\bm{k}_2+\bm{\ell}_2+1=\bm{m}_2$, $\bm{k}_2-\bm{k}_1=\bm{m}_3$. This gives $\bm{\ell}_2-\bm{\ell}_1=\bm{m}_2-\bm{m}_1-\bm{m}_3$ and we furthermore set $\bm{\ell}=\bm{\ell}_2$.
                                                                                                                                 Then,
                                                                                                                                  \begin{align*}
                                                                                                                         &\phantom{=}\Bigg|\sum_{r_1,r_2=0}^{a-1} \tilde{\omega}_{r_1} \tilde{\omega}_{r_2}\, \frac{(2\pi)^8}{n^8 H(\bm{0})^4} \sum_{\bm{k}_1,\bm{\ell}_1,\bm{k}_2,\bm{\ell}_2 =-a}^{a-1} \Bigg[\prod_{t=1}^2 \big[ J_0(\tilde{\omega}_{r_t} \|\tilde{\bm{\omega}}_{\bm{k}_t}\|) J_0(\tilde{\omega}_{r_t} \|\tilde{\bm{\omega}}_{\bm{\ell}_t}\|)\big]\\
&\phantom{=================} \sum_{\underline{j}\in\mathcal{D}(8)}\cum_2\big[Y_{t,c}(\underline{j}):(t,c)\in\nu_{2,1}^{\ast}\big] \ldots \cum_2\big[Y_{t,c}(\underline{j}):(t,c)\in\nu_{2,4}^{\ast}\big] \Bigg]\Bigg|\\
&\lesssim \left(\frac{a}{\lambda}\right)^{4} \frac{1}{\lambda^8} \sum_{\bm{m}_1,\bm{m}_2,\bm{m}_3=-3a}^{3a} \int_{\R^8} \Bigg\{\big|B(\bm{s}) B(\bm{s}+\bm{\omega}_{\bm{m}_1}) B(\bm{t}) B(\bm{t}+\bm{\omega}_{\bm{m}_2}) B(\bm{u}) B(\bm{u}+\bm{\omega}_{\bm{m}_3}) \\
&\phantom{=============i==} B(\bm{v}) B(\bm{v}+\bm{\omega}_{\bm{m}_2-\bm{m}_1-\bm{m}_3})\big| \, \Bigg( \frac{1}{\lambda^6} \sum_{\bm{\ell}=-a}^{a-1} f(\bm{t}+\tilde{\bm{\omega}}_{\bm{\ell}})\Bigg)\Bigg\}\, d\bm{s} d\bm{t} d\bm{u} d\bm{v}\\
&=\Landau\left(\frac{a^{4}}{\lambda^{4}} \cdot \frac{1}{\lambda^4}\right) = \Landau\left(\frac{a^4}{\lambda^8}\right),
\end{align*}
where we used Lemma \ref{bounds_for_B}, part (iv) for $d=2$, $t=4$ and $s=0$. 
In general, we obtain an order of at most $\Landau(a^{4}/\lambda^{6})$ for partitions with $2$ restrictions and an order of $\Landau(a^{4}/\lambda^{8})$ in the case of $3$ restrictions (also see Lemma \ref{lemma1}). 
This finally gives 
                                                                                                                                 \begin{align*}
\lambda^2 \Var\big[\hat{D}_{2,\lambda,a}\big]&= 8 \sum_{\bm{m}=-2a+1}^{2a-1} \frac{H(\bm{m})^2}{H(\bm{0})^2}  \int_{2\pi\max(-a,-a-\bm{m})/\lambda}^{2\pi\min(a,a-\bm{m})/\lambda} f(\bm{z})^2\nonumber\\
&\phantom{=========}  \Bigg( \int_{0}^{2\pi a /\lambda} r \Bigg[\int_{[-2\pi a/\lambda,2\pi a/\lambda]^2} f(\bm{x}) J_0(r \|\bm{x}\|)  J_0(r \|\bm{z}\|) \, d\bm{x}\Bigg] \, dr\Bigg)^2 \, d\bm{z}\\
&\phantom{=}+\Landau\Bigg(\frac{a^{6}}{\lambda^{10}} + \frac{a^{4}}{\lambda^{6}}\left[1+\frac{(\log a)^4}{\lambda^2}+\frac{a}{\lambda^{2}}+\frac{a^{4}}{\lambda^{6}}\right]\\
&\phantom{=========}+\left[\frac{(\log a)^4}{\lambda}+\frac{a^{2}}{\lambda^{4}}\right] \left[1+\frac{a}{\lambda^{2}}+\frac{a^{4}}{\lambda^{6}}\right]^2 + \frac{a\, (\log \lambda)^2}{\lambda^{2}}+ \frac{1}{(\log \lambda)^3}+\frac{\lambda^2}{n}\Bigg)
\end{align*}
                                                                                                                               as $a,\lambda,n\to\infty$.

\subsection{Proof of Theorem \ref{asymptotic_normality_secint}}

From Theorem \ref{expect_theo_sec_int}, it directly follows that
\begin{align*}
\lambda\, \E\left[\hat{D}_{2,\lambda,a}-D_{2,2,\lambda,a}\right] \rightarrow 0 \qquad \text{as } \lambda,a,n\rightarrow \infty,
\end{align*}
since by Assumption \ref{assumptions_on_a}, the bias term in Theorem \ref{expect_theo_sec_int} is of order $o(1/\lambda)$. Moreover, under Assumption \ref{assumptions_on_a}, Theorem \ref{expect_theo_sec_int} yields 
\begin{align*}
\frac{\lambda^2 \Var[\hat{D}_{2,\lambda,a}]}{\tau_{2,\lambda,a}^2} \rightarrow 1
\end{align*}
as $\lambda,a,n\to\infty$. In order to prove Theorem \ref{asymptotic_normality_secint}, it thus suffices to show that the higher order cumulants of the estimate $\hat{D}_{2,\lambda,a}$ converge to $0$. 
More precisely, we will prove
\begin{align} \label{to_show_sec_int}
\lambda^q\, \cum_q\big(\hat{D}_{2,\lambda,a}\big) = \begin{cases}
\Landau(1), &\qquad q=2,\\
o(1), &\qquad q\geq 3.
\end{cases}
\end{align}
To this end, let $q\geq 2$ be arbitrary and assume $n>4q$. Analogously to the proofs of Theorem \ref{asymptotic_normality} and Theorem \ref{expect_theo_sec_int}, we obtain
\begin{align} \label{cumulant_formula_general}
&\phantom{i=i}\lambda^q\, \cum_q\big(\hat{D}_{2,\lambda,a} \big)\nonumber\\
&= \sum_{r_1,\ldots,r_q=0}^{a-1} \tilde{\omega}_{r_1} \ldots \tilde{\omega}_{r_q}\, \frac{(2\pi)^{4q}}{n^{4q} H(\bm{0})^{2q}} \sum_{\bm{k}_1,\ldots,\bm{k}_q =-a}^{a-1} \sum_{\bm{\ell}_1,\ldots,\bm{\ell}_q=-a}^{a-1} \Bigg[\prod_{t=1}^q  \left[ J_0(\tilde{\omega}_{r_t} \|\tilde{\bm{\omega}}_{\bm{k}_t}\|)\, J_0(\tilde{\omega}_{r_t} \|\tilde{\bm{\omega}}_{\bm{\ell}_t}\|)\right]\nonumber\\
&\phantom{===}\sum_{\bm{\nu}=\{\nu_1,\ldots,\nu_G\}\in\mathcal{I}(q)} \sum_{i=4}^{2q+G} \sum_{\underline{j}\in\mathcal{D}(q,i)}\cum_{|\nu_1|}\big[Y_{t,c}(\underline{j}):(t,c)\in\nu_1\big]\times \ldots \times \cum_{|\nu_G|}\big[Y_{t,c}(\underline{j}):(t,c)\in\nu_G\big] \Bigg],
\end{align}
where the set $\mathcal{D}(q,i)$ is defined in \eqref{set_D(q,i)}, $\mathcal{I}(q)$ is the set of indecomposable partitions of Table \eqref{table}, and $G$ denotes the number of groups in the partition $\bm{\nu}$. Moreover, for $t=1,\ldots,q$ and $c=1,\ldots,4$, recall the definition of the random variables
\begin{align*}
Y_{t,c}(\underline{j}):=\begin{cases}
h\left(\frac{\bm{s}_{j_{c+4(t-1)}}}{\lambda}\right) Z(\bm{s}_{j_{c+4(t-1)}}) \exp\big((-1)^{c+1} \im \bm{s}_{j_{c+4(t-1)}}^T \tilde{\bm{\omega}}_{\bm{k}_t}\big), \qquad &\text{if } c=1,2\\
h\left(\frac{\bm{s}_{j_{c+4(t-1)}}}{\lambda}\right) Z(\bm{s}_{j_{c+4(t-1)}}) \exp\big((-1)^{c+1} \im \bm{s}_{j_{c+4(t-1)}}^T \tilde{\bm{\omega}}_{\bm{\ell}_t}\big), \qquad &\text{if } c=3,4.
\end{cases}
\end{align*}
We start with an indecomposable partition consisting of $G=2q$ groups and maximally many different elements in $\underline{j}$, i.e. $i=4q$. To this end, let $\bm{\nu}^{\ast}=\{\nu_1^{\ast},\ldots,\nu_{2q}^{\ast}\}\in\mathcal{I}(q)$ with $|\nu_g^\ast|=2$ for $g=1,\ldots,2q$ be arbitrary. As in the proof of Theorem \ref{asymptotic_normality}, we identify each tuple $(t,c)$ in Table \eqref{table} with one of the numbers between $1$ and $4q$, such that each set $\nu_g^\ast$ can be expressed as a single tuple $(a,b)$ for $a,b\in\{1,\ldots,4q\}$. We then write $\nu_g^{\ast}=(\nu_{g,1}^{\ast},\nu_{g,2}^{\ast})$ for all $g=1,\ldots,2q$, and define the quantity $c_{q,n}$ as in \eqref{c_q,n}. 
Then, following the same steps as in the proof of Theorem \ref{asymptotic_normality}, we obtain
\begin{align} \label{cum_int_2}
&\phantom{= i} \sum_{r_1,\ldots,r_q=0}^{a -1} \tilde{\omega}_{r_1} \ldots \tilde{\omega}_{r_q}\, \frac{1}{n^{4q}} \sum_{\bm{k}_1,\ldots,\bm{k}_q =-a}^{a-1} \sum_{\bm{\ell}_1,\ldots,\bm{\ell}_q=-a}^{a-1} \Bigg[\prod_{t=1}^q \left[J_0(\tilde{\omega}_{r_t} \|\tilde{\bm{\omega}}_{\bm{k}_t}\|)\, J_0(\tilde{\omega}_{r_t} \|\tilde{\bm{\omega}}_{\bm{\ell}_t}\|)\right]\nonumber\\
&\phantom{=======} \sum_{\underline{j}\in\mathcal{D}(q,4q)}\cum_2\big[Y_{t,c}(\underline{j}):(t,c)\in\nu_1^{\ast}\big]\times \ldots \times \cum_2\big[Y_{t,c}(\underline{j}):(t,c)\in\nu_{2q}^{\ast}\big] \Bigg]\nonumber\\
&\eqsim c_{q,n} \,  \sum_{r_1,\ldots,r_q=0}^{a-1} \tilde{\omega}_{r_1} \ldots \tilde{\omega}_{r_q}\sum_{\bm{k}_1,\ldots,\bm{k}_q =-a}^{a-1} \sum_{\bm{\ell}_1,\ldots,\bm{\ell}_q=-a}^{a-1} \Bigg[\prod_{t=1}^q \left[ J_0(\tilde{\omega}_{r_t} \|\tilde{\bm{\omega}}_{\bm{k}_t}\|)\, J_0(\tilde{\omega}_{r_t} \|\tilde{\bm{\omega}}_{\bm{\ell}_t}\|)\right] \nonumber\\
&\phantom{================} \frac{1}{\lambda^{8q}} \int_{\R^{4q}} \prod_{g=1}^{2q} f(\bm{u}_g-\tilde{\bm{\omega}}_{\hat{\bm{h}}_{2g-1}}) B(\bm{u}_g) B(\bm{u}_g+\bm{\omega}_{\hat{\bm{h}}_{2g}-\hat{\bm{h}}_{2g-1}}) \, d\bm{u}_g \, \Bigg],
\end{align}
where the frequency window $B$ is defined in \eqref{Four_trafo_of_h}, and for $j=1,\ldots,4q$, the $2$-dimensional vectors $\hat{\bm{h}}_j$ are taken from the set
\begin{align*}
\{\bm{k}_1,-\bm{k}_1-1,\bm{k}_2,-\bm{k}_2-1,\ldots,\bm{k}_q,-\bm{k}_q-1,\bm{\ell}_1,-\bm{\ell}_1-1,\ldots,\bm{\ell}_q,-\bm{\ell}_q-1\}
\end{align*}
and are determined by the partition $\bm{\nu}^{\ast}$ [to see this, note that $-\tilde{\bm{\omega}}_{\bm{k}}=\tilde{\bm{\omega}}_{-\bm{k}-1}$, and recall Example \ref{ex_matrix_k} in the proof of Theorem \ref{asymptotic_normality}]. 
We denote by $K\in\{-1,0,1\}^{2q\times 2q}$ and $R\in\{-1,0,1\}^{2q\times 2}$ the matrices satisfying the equation
\begin{align} \label{matrix_K_sec_int}
K \times \begin{pmatrix}
\bm{k}_1^T \\
\vdots\\
\bm{k}_q^T \\
\bm{\ell}_1^T \\
\vdots\\
\bm{\ell}_q^T
\end{pmatrix} + R = \begin{pmatrix}
(\hat{\bm{h}}_2-\hat{\bm{h}}_1)^T\\
(\hat{\bm{h}}_4-\hat{\bm{h}}_3)^T\\
\vdots\\
(\hat{\bm{h}}_{4q}-\hat{\bm{h}}_{4q-1})^T
\end{pmatrix}
\end{align}
[for illustration, consider Example \ref{ex_matrices_K_R} in the proof of Theorem \ref{asymptotic_normality}].
As in the proof of Theorem \ref{asymptotic_normality}, the order of \eqref{cum_int_2} is determined by the rank of the matrix $K$, which corresponds to the number of independent constraints between the variables $\bm{k}_1,\ldots,\bm{k}_q,\bm{\ell}_1,\ldots,\bm{\ell}_q$ evoked by the partition $\bm{\nu}^\ast$ in Table \eqref{table}. 
In order to demonstrate this, consider the case

\begin{align} \label{partition_example_sec_int}
\begin{pmatrix}
(\hat{\bm{h}}_2-\hat{\bm{h}}_1)^T\\
(\hat{\bm{h}}_4-\hat{\bm{h}}_3)^T\\
\vdots\\
(\hat{\bm{h}}_{4q}-\hat{\bm{h}}_{4q-1})^T
\end{pmatrix} = \begin{pmatrix}
(\bm{k}_1-\bm{k}_1)^T \\
(\bm{k}_2+\bm{k}_3+1)^T\\
(\bm{k}_2+\bm{k}_3+1)^T\\
(\bm{k}_4+\bm{k}_5+1)^T\\
(\bm{k}_4+\bm{k}_5+1)^T\\
\vdots\\
(\bm{k}_{q-1}+\bm{k}_q+1)^T\\
(\bm{k}_{q-1}+\bm{k}_q+1)^T\\
(\bm{\ell}_1+\bm{\ell}_2+1)^T\\
(\bm{\ell}_1+\bm{\ell}_2+1)^T\\
\vdots\\
(\bm{\ell}_{q-2}+\bm{\ell}_{q-1}+1)^T\\
(\bm{\ell}_{q-2}+\bm{\ell}_{q-1}+1)^T\\
(\bm{\ell}_q-\bm{\ell}_q)^T
\end{pmatrix},
\end{align}
if $q$ odd, and 
\begin{align*} 
\begin{pmatrix}
(\hat{\bm{h}}_2-\hat{\bm{h}}_1)^T\\
(\hat{\bm{h}}_4-\hat{\bm{h}}_3)^T\\
\vdots\\
(\hat{\bm{h}}_{4q}-\hat{\bm{h}}_{4q-1})^T
\end{pmatrix} = \begin{pmatrix}
(\bm{k}_1-\bm{k}_1)^T\\
(\bm{k}_2+\bm{k}_3+1)^T\\
(\bm{k}_2+\bm{k}_3+1)^T\\
(\bm{k}_4+\bm{k}_5+1)^T\\
(\bm{k}_4+\bm{k}_5+1)^T\\
\vdots\\
(\bm{k}_{q-2}+\bm{k}_{q-1}+1)^T\\
(\bm{k}_{q-2}+\bm{k}_{q-1}+1)^T\\
(\bm{k}_q-\bm{k}_q)^T\\
(\bm{\ell}_1+\bm{\ell}_2+1)^T\\
(\bm{\ell}_1+\bm{\ell}_2+1)^T\\
\vdots\\
(\bm{\ell}_{q-1}+\bm{\ell}_{q}+1)^T\\
(\bm{\ell}_{q-1}+\bm{\ell}_{q}+1)^T
\end{pmatrix},
\end{align*}
if $q$ even. Note that these particular representations correspond to a matrix $K$ with a rank of $q-1$, which we choose for simplicity and since it naturally makes sense due to indecomposability. 
Assume without loss of generality that $q$ is odd and set
\begin{align*}
&\bm{k}_2+\bm{k}_3=\bm{m}_1, \qquad \bm{k}_4+\bm{k}_5=\bm{m}_2, \qquad \ldots, \qquad \bm{k}_{q-1}+\bm{k}_q=\bm{m}_{(q-1)/2}, \\
&\bm{\ell}_1+\bm{\ell}_2=\bm{m}_{(q-1)/2+1}, \qquad \ldots, \qquad \bm{\ell}_{q-2}+\bm{\ell}_{q-1}=\bm{m}_{q-1}.
\end{align*}
Ignoring constants, the absolute value of the expression in \eqref{cum_int_2} is then bounded by
\begin{align} \label{eq12}
&\phantom{\lesssim \times}\lambda^{q-8q} \int_{\R^{4q}} \sum_{\bm{m}_1,\ldots,\bm{m}_{q-1}=-2a}^{2a-2} \Big| B(\bm{u}_1)^2 B(\bm{u}_{2q})^2 B(\bm{u}_2) B(\bm{u}_2+\bm{\omega}_{\bm{m}_1+1})  B(\bm{u}_3) B(\bm{u}_3+\bm{\omega}_{\bm{m}_1+1})\times \ldots \nonumber \\
&\phantom{============ii=}\times B(\bm{u}_{2q-2}) B(\bm{u}_{2q-2}+\bm{\omega}_{\bm{m}_{q-1}+1})  B(\bm{u}_{2q-1}) B(\bm{u}_{2q-1}+\bm{\omega}_{\bm{m}_{q-1}+1})\Big|\nonumber\\
&\phantom{\lesssim}\times \Bigg( \frac{1}{\lambda^q} \sum_{r_1,\ldots,r_q=0}^{a-1} \tilde{\omega}_{r_1} \ldots \tilde{\omega}_{r_q} \Bigg|\sum_{\bm{k}_1=-a}^{a-1} J_0(\tilde{\omega}_{r_1} \|\tilde{\bm{\omega}}_{\bm{k}_1}\|) f(\bm{u}_1-\tilde{\bm{\omega}}_{\bm{k}_1}) \sum_{\bm{\ell}_q=-a}^{a-1} J_0(\tilde{\omega}_{r_q} \|\tilde{\bm{\omega}}_{\bm{\ell}_q}\|) f(\bm{u}_{2q}-\tilde{\bm{\omega}}_{\bm{\ell}_q}) \Bigg|\Bigg)\nonumber \\
&\phantom{\lesssim}\times \Bigg(\sum_{\substack{\bm{k}_3,\bm{k}_5,\ldots,\bm{k}_{q}=-a\\\bm{\ell}_2,\bm{\ell}_4,\ldots,\bm{\ell}_{q-1}=-a}}^{a-1}  \Big|f(\bm{u}_2+\tilde{\bm{\omega}}_{\bm{k}_3}) f(\bm{u}_3+\tilde{\bm{\omega}}_{\bm{k}_3})\times \ldots \times f(\bm{u}_{2q-2}+\tilde{\bm{\omega}}_{\bm{\ell}_{q-1}}) f(\bm{u}_{2q-1}+\tilde{\bm{\omega}}_{\bm{\ell}_{q-1}})\Big|\Bigg)\nonumber\\
&\phantom{=======} \, d\bm{u}_1 \ldots d\bm{u}_{2q}\nonumber\\
&\lesssim  \lambda^{q-8q+2(2q-(q-1))+4q} \times \Bigg( \frac{1}{\lambda^{4q-4}} \int_{\R^{4q-4}} \sum_{\bm{m}_1,\ldots,\bm{m}_{q-1}=-2a}^{2a-2} \Big|B(\bm{u}_2) B(\bm{u}_2+\bm{\omega}_{\bm{m}_1+1})\nonumber \\
&\phantom{=}\times B(\bm{u}_3) B(\bm{u}_3+\bm{\omega}_{\bm{m}_1+1})\times \ldots \times B(\bm{u}_{2q-2}) B(\bm{u}_{2q-2}+\bm{\omega}_{\bm{m}_{q-1}+1})  B(\bm{u}_{2q-1}) B(\bm{u}_{2q-1}+\bm{\omega}_{\bm{m}_{q-1}+1})\Big|\nonumber\\ 
&\phantom{==================}d\bm{u}_2 \ldots d\bm{u}_{2q-1}\Bigg)\nonumber\\
&\phantom{=}\times \left(\frac{a^{2}}{\lambda^{2}}\right)^{q-2}\times  \left(\frac{1}{\lambda^2} \int_{\R^2} B(\bm{u})^2 \times \frac{1}{\lambda} \sum_{r=0}^{a -1} \tilde{\omega}_r \, \Big|\frac{1}{\lambda^2} \sum_{\bm{k}=-a}^{a-1} J_0(\tilde{\omega}_r \|\tilde{\bm{\omega}}_{\bm{k}}\|) f(\bm{u}-\tilde{\bm{\omega}}_{\bm{k}})\Big|\, d\bm{u}\right)^2.
\end{align}
Due to Lemma \ref{bounds_for_B} (iii)
and Lemma \ref{lemma_for_cumulants} for $\delta>2$, 
the expression in \eqref{eq12} is thus of order 
\begin{align*}
\Landau\left(\lambda^{q-8q+2(2q-(q-1))+4q} \times \left(\frac{a^{2}}{\lambda^{2}}\right)^{q-2}\right).
\end{align*}
We now interpret the components of this order. Since the partition $\bm{\nu}^{\ast}=\{\nu_1^{\ast},\ldots,\nu_{2q}^{\ast}\}$ consists of $G^{\ast}=2q$ groups and since we assumed maximally many different elements in $\underline{j}$ (i.e. $i=4q$), we obtain in the same way as in the proof of Theorem \ref{asymptotic_normality}
\begin{align*}
&8q 
= 2\, \sum_{g=1}^{G^\ast} \#\{\text{different indices in $\underline{j}$ belonging to } \nu_g^\ast\},\\
& 2(2q-(q-1)) = 2(\#\text{\{sums over } \bm{k}_1,\ldots,\bm{k}_q, \bm{\ell}_1,\ldots,\bm{\ell}_q \} - \#\text{restrictions}),\\
&4q = 2G^{\ast}.
\end{align*}
Recall that the number of restrictions equals the rank of the matrix $K$ in \eqref{matrix_K_sec_int} and note that the prefactor $2$ corresponds to the dimension $d=2$. Moreover, the exponent $q-2$ of the expression $a^{2}/\lambda^{2}$ can be interpreted as
\begin{align*}
q-2 = q- \#\{ \text{rows with either $\tilde{\bm{\omega}}_{\bm{k}}$ or $\tilde{\bm{\omega}}_{\bm{\ell}}$ (or both) unrestricted}\},
\end{align*}
since for the partition under consideration two variables (namely $\bm{k}_1$ and $\bm{\ell}_q$) do not evoke a restriction leading to a rank increase of the matrix $K$ in \eqref{matrix_K_sec_int} [compare to \eqref{partition_example_sec_int}]. Taking the above observations together, we can more generally express the order of the expression in \eqref{eq12} as
\begin{align} \label{general_order_q=2_sec_int}
&\Landau\Bigg(\lambda^{q+2\big(-\sum_{g=1}^{G^\ast} \#\{\text{different indices in $\underline{j}$ belonging to } \nu_g^{\ast} \}+(\#\text{sums}-\#\text{restrictions}) + G^\ast\big)}\nonumber\\
&\phantom{============}\times \left(\frac{a^{2}}{\lambda^{2}}\right)^{q- \#\{ \text{rows with either $\tilde{\bm{\omega}}_{\bm{k}}$ or $\tilde{\bm{\omega}}_{\bm{\ell}}$ (or both) unrestricted}\}}\Bigg).
\end{align}

It is easy to see that this order will be the same for all partitions $\bm{\nu}=\{\nu_1,\ldots,\nu_{2q}\}$ and $i=4q$ different elements in $\underline{j}$. In particular, the order in \eqref{general_order_q=2_sec_int} holds for partitions evoking an arbitrary number of restrictions. Note that in the case $G=2q$, $i=4q$ we furthermore have
\begin{align*}
\#\text{groups}-\sum_{g=1}^G \#\{\text{different indices in $\underline{j}$ belonging to } \nu_g\}+\#\text{sums}=0.
\end{align*}
Therefore, we obtain the following result, which specifies the order of the cumulant expression corresponding to an arbitrary (not necessarily indecomposable) partition $\bm{\nu}$ with $G=2q$ groups and $i=4q$ different elements:

\begin{lemma} \label{lemma1}
For $q\geq 2$ and any partition $\bm{\nu}=\{\nu_1,\ldots,\nu_{2q}\}$ of Table \eqref{table}, we have
\begin{align*}
&\frac{1}{n^{4q}} \sum_{r_1,\ldots,r_q=0}^{a-1} \tilde{\omega}_{r_1} \ldots \tilde{\omega}_{r_q}\,  \sum_{\bm{k}_1,\ldots,\bm{k}_q =-a}^{a-1} \sum_{\bm{\ell}_1,\ldots,\bm{\ell}_q=-a}^{a-1} \Bigg[\prod_{t=1}^q \left[ J_0(\tilde{\omega}_{r_t} \|\tilde{\bm{\omega}}_{\bm{k}_t}\|)\, J_0(\tilde{\omega}_{r_t} \|\tilde{\bm{\omega}}_{\bm{\ell}_t}\|)\right] \nonumber\\
&\phantom{=========} \sum_{\underline{j}\in\mathcal{D}(q,4q)}\cum_2\big[Y_{t,c}(\underline{j}):(t,c)\in\nu_1\big]\times \ldots \times \cum_2\big[Y_{t,c}(\underline{j}):(t,c)\in\nu_{2q}\big] \Bigg]\\
&=\Landau\left(\lambda^{q-2(\#\text{restrictions})} \times \left(\frac{a^{2}}{\lambda^{2}}\right)^{q-\#\{ \text{rows with either $\tilde{\bm{\omega}}_{\bm{k}}$ or $\tilde{\bm{\omega}}_{\bm{\ell}}$ (or both) unrestricted}\}}\right),
\end{align*}
where the term $\#\text{restrictions}$ denotes the number of independent restrictions between the rows and columns of Table \eqref{table}, or equivalently the rank of the matrix $K$ in \eqref{matrix_K_sec_int}. 
\end{lemma}

We now aim to further specify the exponent of the term $a^{2}/\lambda^{2}$ in the case of indecomposability. Recall that any indecomposable partition of Table \eqref{table} with $G=2q$ groups must evoke at least $q-1$ restrictions. Furthermore note that in case of indecomposability, there cannot be any row of Table \eqref{table} with both $\tilde{\bm{\omega}}_{\bm{k}}$ and $\tilde{\bm{\omega}}_{\bm{\ell}}$ unrestricted.

\begin{lemma} \label{lemma2}
For $q\geq 2$ and any indecomposable partition $\bm{\nu}=\{\nu_1,\ldots,\nu_{2q}\}\in\mathcal{I}(q)$ evoking exactly $q-1$ restrictions in Table \eqref{table}, we have
\begin{align*}
\#\{ \text{rows with either $\tilde{\bm{\omega}}_{\bm{k}}$ or $\tilde{\bm{\omega}}_{\bm{\ell}}$ unrestricted}\} \geq 2.
\end{align*}
\end{lemma}

\begin{proof}
By assumption, we have $q-1$ restrictions between the variables $\bm{k}_1,\ldots,\bm{k}_q,\bm{\ell}_1,\ldots,\bm{\ell}_q$. In this case, at most $2(q-1)=2q-2$ of the $2q$ variables $\bm{k}_1,\ldots,\bm{k}_q,\bm{\ell}_1,\ldots,\bm{\ell}_q$ can be restricted, i.e. at least $2$ of the $2q$ variables remain unrestricted. By the indecomposability assumption, the $2$ unrestricted variables must be in different rows of Table \eqref{table}, which yields the claim.
\end{proof}

We will use Lemma \ref{lemma1} and \ref{lemma2} to prove the following result, which provides the order of the cumulant expression corresponding to an indecomposable partition $\bm{\nu}$ with $G=2q$ groups and $i\leq 4q$ different elements: 

\begin{prop} \label{order_q=2_sec_int}
For $q\geq 2$ and any indecomposable partition $\bm{\nu}=\{\nu_1,\ldots,\nu_{2q}\}\in\mathcal{I}(q)$ of Table \eqref{table}, it holds that
\begin{align*}
& \frac{1}{n^{4q}} \sum_{r_1,\ldots,r_q=0}^{a-1} \tilde{\omega}_{r_1} \ldots \tilde{\omega}_{r_q}\, \sum_{\bm{k}_1,\ldots,\bm{k}_q =-a}^{a-1} \sum_{\bm{\ell}_1,\ldots,\bm{\ell}_q=-a}^{a-1} \Bigg[\prod_{t=1}^q \left[ J_0(\tilde{\omega}_{r_t} \|\tilde{\bm{\omega}}_{\bm{k}_t}\|)\, J_0(\tilde{\omega}_{r_t} \|\tilde{\bm{\omega}}_{\bm{\ell}_t}\|)\right] \nonumber\\
&\phantom{==========} \sum_{\underline{j}\in\mathcal{D}(q,i)}\cum_2\big[Y_{t,c}(\underline{j}):(t,c)\in\nu_1\big]\times \ldots \times \cum_2\big[Y_{t,c}(\underline{j}):(t,c)\in\nu_{2q}\big] \Bigg]\\
&=\begin{cases}
\Landau\left(\left(\frac{a^{2}}{\lambda^{3}}\right)^{q-2}\right), &\qquad \text{if } i=4q,\\
\Landau\left(\left(\frac{a^{2}}{\lambda^{3}}\right)^{q-2}\times \frac{\lambda^2}{n}\right), &\qquad \text{if } i\leq 4q-1.
\end{cases}
\end{align*}
\end{prop}

\begin{proof}
We first consider the case $i=4q$. Note that any indecomposable partition of Table \eqref{table} must evoke at least $q-1$ restrictions.
If there are exactly $q-1$ restrictions, then we directly obtain from Lemma \ref{lemma1} and Lemma \ref{lemma2} that the order of the expression on the left hand side is given by
\begin{align*}
\Landau\left(\lambda^{2-q} \times \left(\frac{a^{2}}{\lambda^{2}}\right)^{q-2}\right) = \Landau\left(\left(\frac{a^{2}}{\lambda^{3}}\right)^{q-2}\right).
\end{align*}
Furthermore, note that as soon as we introduce one more restriction, it follows from Lemma \ref{lemma1} that there is an order change of at most
\begin{align*}
\frac{1}{\lambda^2} \times \left(\frac{a^{2}}{\lambda^{2}}\right)^2,
\end{align*}
which converges to $0$ by Assumption \ref{assumptions_on_a}.\\ 
Now let $i\leq 4q-1$, then with the same arguments as in the proof of Lemma \ref{order_depending_on_number_of_restr}, we see that decreasing the value $i$ by $1$ leads to an order change of at most $\lambda^2/n$. This yields the claim.
\end{proof}

We now consider a partition $\bm{\nu}=\{\nu_1,\ldots,\nu_G\}$ of Table \eqref{table} with $|\nu_g|>2$ for some $g\in\{1,\ldots,G\}$. With exactly the same arguments as in the proof of Theorem \ref{asymptotic_normality} [i.e. using the law of total cumulance and the Gaussianity of the locations], each term $\cum_{|\nu_g|}[Y_{t,c}(\underline{j})\, : \, (t,c)\in \nu_g]$ in \eqref{cumulant_formula_general} with $|\nu_g|>2$ is the sum of cumulants of covariances over all subpartitions of $\nu_g$ that are of size $2$, conditioned on the locations $\bm{s}_j$. Since the highest order is again determined by the case of maximally many different elements in $\underline{j}$ (i.e. $i=2q+G$), we thus obtain by the same reasoning as in the proof of Theorem \ref{asymptotic_normality}
\begin{align} \label{order_sec_int}
&\phantom{==}\frac{1}{n^{4q}} \sum_{r_1,\ldots,r_q=0}^{a-1} \tilde{\omega}_{r_1} \ldots \tilde{\omega}_{r_q}\, \sum_{\bm{k}_1,\ldots,\bm{k}_q =-a}^{a-1} \sum_{\bm{\ell}_1,\ldots,\bm{\ell}_q=-a}^{a-1} \Bigg[\prod_{t=1}^q \left[ J_0(\tilde{\omega}_{r_t} \|\tilde{\bm{\omega}}_{\bm{k}_t}\|)\, J_0(\tilde{\omega}_{r_t} \|\tilde{\bm{\omega}}_{\bm{\ell}_t}\|)\right] \nonumber\\
&\phantom{==========} \sum_{i=4}^{2q+G} \sum_{\underline{j}\in\mathcal{D}(q,i)}\cum_{|\nu_1|}\big[Y_{t,c}(\underline{j}):(t,c)\in\nu_1\big]\times \ldots \times \cum_{|\nu_G|}\big[Y_{t,c}(\underline{j}):(t,c)\in\nu_{G}\big] \Bigg]\nonumber\\
&=\Landau\bigg(\lambda^{q+2\big(-(2q+G)+(2q-\#\text{restrictions})+G\big)} \nonumber\\
&\phantom{========}\times \left(\frac{a^{2}}{\lambda^{2}}\right)^{q-\#\{\text{rows with either $\tilde{\bm{\omega}}_{\bm{k}}$ or $\tilde{\bm{\omega}}_{\bm{\ell}}$ (or both) unrestricted}\}}\times \frac{1}{n^{4q-(2q+G)}}\bigg)\nonumber\\
&=\Landau\bigg(\frac{\lambda^{q-2(\#\text{restrictions})}}{n^{2q-G}} \times \left(\frac{a^{2}}{\lambda^{2}}\right)^{q-\#\{\text{rows with either $\tilde{\bm{\omega}}_{\bm{k}}$ or $\tilde{\bm{\omega}}_{\bm{\ell}}$ (or both) unrestricted}\}}\bigg)
\end{align}
[compare to \eqref{general_order_q=2_sec_int}, i.e. the order of the cumulant expression for a partition with $G=2q$ groups and $i=4q$]. 
We now specify the order in \eqref{order_sec_int} for the case of indecomposability. 

\begin{prop} \label{prop_sec_int}
For $q\geq 2$ and any indecomposable partition $\bm{\nu}=\{\nu_1,\ldots,\nu_{G}\}\in\mathcal{I}(q)$ of Table \eqref{table}, we have
\begin{align*}
&\frac{1}{n^{4q}} \sum_{r_1,\ldots,r_q=0}^{a-1} \tilde{\omega}_{r_1} \ldots \tilde{\omega}_{r_q}\, \sum_{\bm{k}_1,\ldots,\bm{k}_q =-a}^{a-1} \sum_{\bm{\ell}_1,\ldots,\bm{\ell}_q=-a}^{a-1} \Bigg[\prod_{t=1}^q \left[ J_0(\tilde{\omega}_{r_t} \|\tilde{\bm{\omega}}_{\bm{k}_t}\|)\, J_0(\tilde{\omega}_{r_t} \|\tilde{\bm{\omega}}_{\bm{\ell}_t}\|)\right] \nonumber\\
&\phantom{=====}\sum_{i=4}^{2q+G} \sum_{\underline{j}\in\mathcal{D}(q,i)}\cum_{|\nu_1|}\big[Y_{t,c}(\underline{j}):(t,c)\in\nu_1\big]\times \ldots \times \cum_{|\nu_G|}\big[Y_{t,c}(\underline{j}):(t,c)\in\nu_{G}\big] \Bigg]\\
&=\begin{cases}
\Landau\left(\left(\frac{\lambda^2}{n}\right)^{q-1} \lambda^{2-q}\right), &\qquad \text{if } G=q+1 \text{ and } \#\text{restrictions}=0,\\
\Landau\left(\left(\frac{\lambda^2}{n}\right)^{q-s} \left(\frac{a^{2}}{\lambda^{3}}\right)^{q-2} \left(\frac{\lambda^{2}}{a^{2}}\right)^{q-s}\right), &\qquad \text{if } G=q+s,\, 2\leq s \leq q \text{ and } \#\text{restrictions}=s-1,\\
\Landau\left(\left(\frac{\lambda^2}{n}\right)^{q-s} \left(\frac{a^2}{\lambda^3}\right)^q\right), &\qquad \text{if } G=q+s,\, 1\leq s \leq q \text{ and } \#\text{restrictions}\geq s,\\
\Landau\left(\lambda^{-q}\right), &\qquad \text{if } G\leq q.
\end{cases}
\end{align*}
In particular, we obtain an order of $\Landau(1)$ for partitions $\bm{\nu}=\{\nu_1,\ldots,\nu_{4}\}\in\mathcal{I}(2)$ evoking exactly $1$ restriction in Table \eqref{table_q=2}, while the expression moreover converges to $0$ in all remaining cases. 
\end{prop}

\begin{proof}
We first consider the case where $G=q+s$ for some $s\in\{1,\ldots,q\}$. By the indecomposability assumption and Lemma \ref{number_of_rest}, the number of restrictions in Table \eqref{table} is then lower bounded by $s-1$. 
If there are exactly $s-1$ restrictions, then 
\begin{align} \label{lower_bound_for_unrestricted_rows}
\#\{\text{rows with either $\tilde{\bm{\omega}}_{\bm{k}}$ or $\tilde{\bm{\omega}}_{\bm{\ell}}$ (or both) unrestricted}\} \geq \min\{q-s+2,q\}:
\end{align}
To see this, note that if $q-s+2>q$, then there are no restrictions at all and the left hand side of \eqref{lower_bound_for_unrestricted_rows} equals the number of rows $q$. 
Now, let $q-s+2\leq q$, or equivalently $s>1$. Arguing similarly as in the proof of Proposition 4.1 in \cite{vandelft18}, a typical indecomposable partition with $q+s$ groups and $s-1$ restrictions has one big set covering the first two (or last two) columns and the first $q-s+1$ rows of Table \eqref{table}, and otherwise sets of size $2$ only. Thus, there are no restrictions for either $\tilde{\bm{\omega}}_{\bm{k}}$ or $\tilde{\bm{\omega}}_{\bm{\ell}}$ in $q-s+1$ rows, but $s-1$ restrictions from row $q-s+1$ to row $q$. These $s-1$ restrictions can involve at most $2(s-1)=2s-2$ variables from the set $\{\bm{k}_1,\ldots,\bm{\ell}_q\}$, since they are created by sets of size $2$ only. Therefore, at least $2s-(2s-2)=2$ variables from row $q-s+1$ to row $q$ remain unrestricted. One of them has already been counted in the large set covering $q-s+1$ rows, the other one introduces a further row that contains either unrestricted $\tilde{\bm{\omega}}_{\bm{k}}$ or $\tilde{\bm{\omega}}_{\bm{\ell}}$. The argumentation for other indecomposable partitions with $q+s$ groups evoking exactly $s-1$ restrictions in Table \eqref{table} works similarly, which proves \eqref{lower_bound_for_unrestricted_rows}. The order in \eqref{order_sec_int} in the case of $G=q+s$ groups and $s-1$ restrictions thus amounts to
\begin{align*}
\Landau\left(\frac{\lambda^{q-2(s-1)}}{n^{2q-(q+s)}} \left(\frac{a^{2}}{\lambda^{2}}\right)^{q-\min\{q,q-s+2\}}\right) &= \Landau\left(\left(\frac{\lambda^2}{n}\right)^{q-s} \lambda^{2-q} \left(\frac{a^{2}}{\lambda^{2}}\right)^{\max\{0,s-2\}}\right)\\
&=\begin{cases}
\Landau\left(\left(\frac{\lambda^2}{n}\right)^{q-1} \lambda^{2-q}\right), \quad &\text{if } s=1,\\
\Landau\left(\left(\frac{\lambda^2}{n}\right)^{q-s} \left(\frac{a^{2}}{\lambda^{3}}\right)^{q-2} \left(\frac{a^{2}}{\lambda^{2}}\right)^{s-q}\right), \quad &\text{if } 2\leq s \leq q.
\end{cases}
\end{align*}
If $G=q+s$ for some $s\in\{1,\ldots,q\}$ and $\#\text{restrictions}\geq s$, we directly obtain 
\begin{align*}
&\phantom{=i}\frac{\lambda^{q-2(\#\text{restrictions})}}{n^{2q-G}} \times \left(\frac{a^{2}}{\lambda^{2}}\right)^{q-\#\{\text{rows with either $\tilde{\bm{\omega}}_{\bm{k}}$ or $\tilde{\bm{\omega}}_{\bm{\ell}}$ (or both) unrestricted}\}}\\
&=\Landau\left(\frac{\lambda^{q-2s}}{n^{q-s}} \left(\frac{a^2}{\lambda^2}\right)^q\right) = \Landau\left(\left(\frac{\lambda^2}{n}\right)^{q-s} \left(\frac{a^2}{\lambda^3}\right)^q\right).
\end{align*}
Finally, if $G\leq q$, we can use the rough estimate
\begin{align*}
\frac{\lambda^{q-2(\#\text{restrictions})}}{n^{2q-G}} \times \left(\frac{a^{2}}{\lambda^{2}}\right)^{q-\#\{\text{rows with either $\tilde{\bm{\omega}}_{\bm{k}}$ or $\tilde{\bm{\omega}}_{\bm{\ell}}$ (or both) unrestricted}\}}=\Landau\left(\frac{\lambda^q}{n^q}  \left(\frac{a^2}{\lambda^2}\right)^q\right) = \Landau(\lambda^{-q}),
\end{align*}
where we used $a^2/n=\Landau(1)$ (see Assumption \ref{assumptions_on_a}). This completes the proof of the proposition.
\end{proof}

\begin{corollary} \label{cor_sec_int}
For $q\geq 2$, we have
\begin{align*}
\lambda^q \,\cum_q\big(\hat{D}_{2,\lambda,a}\big) = \Landau\left(\left(\frac{a^{2}}{\lambda^{3}}\right)^{q-2}\right).
\end{align*}
\end{corollary}

\begin{proof}
Note that \eqref{cumulant_formula_general} implies
\begin{align*}
\lambda^q \,\cum_q\big(\hat{D}_{2,\lambda,a}\big) &=\Landau\Bigg(\max_{\bm{\nu}=\{\nu_1,\ldots,\nu_G\}\in\mathcal{I}(q)} \Bigg\{\frac{1}{n^{4q}} \sum_{r_1,\ldots,r_q=0}^{a-1} \tilde{\omega}_{r_1} \ldots \tilde{\omega}_{r_q}\, \sum_{\bm{k}_1,\ldots,\bm{k}_q =-a}^{a-1} \sum_{\bm{\ell}_1,\ldots,\bm{\ell}_q=-a}^{a-1}\\
&\phantom{====} \Bigg[\prod_{t=1}^q \left[ J_0(\tilde{\omega}_{r_t} \|\tilde{\bm{\omega}}_{\bm{k}_t}\|)\, J_0(\tilde{\omega}_{r_t} \|\tilde{\bm{\omega}}_{\bm{\ell}_t}\|)\right] \sum_{i=4}^{2q+G} \sum_{\underline{j}\in\mathcal{D}(q,i)}\nonumber\\
&\phantom{====}\cum_{|\nu_1|}\big[Y_{t,c}(\underline{j}):(t,c)\in\nu_1\big]\times \ldots \times \cum_{|\nu_G|}\big[Y_{t,c}(\underline{j}):(t,c)\in\nu_{G}\big] \Bigg]\Bigg\}\Bigg),
\end{align*}
and recall the result from Proposition \ref{prop_sec_int}. For $1\leq s \leq q$, Assumptions \ref{assumption_on_sampling_scheme} and \ref{assumptions_on_a} give $$\left(\frac{\lambda^2}{n}\right)^{q-s}=\Landau(1) \qquad \text{and} \qquad \left(\frac{\lambda^2}{a^2}\right)^{q-s}=\Landau(1).$$ Moreover, Assumption \ref{assumptions_on_a} yields
\begin{align*}
\lambda^{2-q} = \Landau\left(\left(\frac{a^{2}}{\lambda^{3}}\right)^{q-2}\right)
\end{align*}
for $q\geq 2$.
Finally recalling the assumption $a^2/\lambda^3=o(1)$ yields the claim.
\end{proof}

Corollary \ref{cor_sec_int} yields \eqref{to_show_sec_int} and therefore the claim of Theorem \ref{asymptotic_normality_secint}, since 
\begin{align*}
\left(\frac{a^{2}}{\lambda^{3}}\right)^{q-2} = \begin{cases}
\Landau(1), &\qquad q=2,\\
o(1), &\qquad q>2.
\end{cases}
\end{align*}
\qed 

For the case $q=2$, we need a more precise result in order to prove  the second part of Theorem \ref{expect_theo_sec_int}:

\begin{corollary} \label{order_q=2_sec_int_for_variance}
Let $\bm{\nu}=\{\nu_1,\ldots,\nu_G\}\in\mathcal{I}(2)$ be arbitrary. Then, 
\begin{align*}
&\frac{1}{n^8} \sum_{r_1,r_2=0}^{a-1} \tilde{\omega}_{r_1}  \tilde{\omega}_{r_2}\, \sum_{\bm{k}_1,\bm{k}_2 =-a}^{a-1} \sum_{\bm{\ell}_1,\bm{\ell}_2=-a}^{a-1} \Bigg[\prod_{t=1}^2 \left[ J_0(\tilde{\omega}_{r_t} \|\tilde{\bm{\omega}}_{\bm{k}_t}\|)\, J_0(\tilde{\omega}_{r_t} \|\tilde{\bm{\omega}}_{\bm{\ell}_t}\|)\right] \nonumber\\
&\phantom{======} \sum_{\underline{j}\in\mathcal{D}(2,i)}\cum_{|\nu_1|}\big[Y_{t,c}(\underline{j}):(t,c)\in\nu_1\big]\times \ldots \times \cum_{|\nu_G|}\big[Y_{t,c}(\underline{j}):(t,c)\in\nu_{G}\big] \Bigg]\\
&=\begin{cases}
\Landau(1), &\qquad \text{if } G=4 \text{ and } i=8,\\
\Landau\left(\frac{\lambda^2}{n}\right), &\qquad \text{if } G=4 \text{ and } i<8,\\
\Landau\left(\frac{\lambda^2}{n}\right), &\qquad \text{if } G=3,\\
\Landau\left(\frac{1}{\lambda^2}\right), &\qquad \text{if } G\leq 2.
\end{cases}
\end{align*}
\end{corollary}

\begin{proof}
The orders for the cases $G=4$, $i=8$ and $G=4$, $i<8$ have been derived in Proposition \ref{order_q=2_sec_int}, while the statements for $G<4$ follow from Proposition \ref{prop_sec_int}.
\end{proof}

\subsection{Proof of Proposition \ref{approx_D2}}

Note that 
\begin{align*}
(2\pi)^3 (D_{2,2}-D_{2,2,\lambda,a}) &=I_1+I_2,
\end{align*}
where
\begin{align*}
I_1&:=\int_{2\pi a/\lambda}^{\infty} \left(\int_{0}^{2\pi} \int_{\R^2} f(\bm{x}) \exp\Bigg(\im \, r \begin{pmatrix}
\cos \theta\\
\sin \theta
\end{pmatrix}^T \bm{x} \Bigg) d\bm{x}\, d\theta \right)^2 r \, dr
\end{align*}
and
\begin{align*}
I_2&:=\int_{0}^{2\pi a/\lambda} \Bigg[\left(\int_{0}^{2\pi} \int_{\R^2} f(\bm{x}) \exp\Bigg(\im \, r \begin{pmatrix}
\cos \theta\\
\sin \theta
\end{pmatrix}^T \bm{x} \Bigg) d\bm{x}\, d\theta \right)^2 \\
&\phantom{======}-\left( \int_{0}^{2\pi} \int_{[-2\pi a/\lambda,2\pi a/\lambda]^2} f(\bm{x}) \exp\Bigg(\im \, r \begin{pmatrix}
\cos \theta\\
\sin \theta
\end{pmatrix}^T \bm{x} \Bigg) d\bm{x}\, d\theta \right)^2 \Bigg] \,r \, dr.
\end{align*}
Now, due to the Cauchy-Schwarz inequality, 
\begin{align*}
|I_1|
&\leq 2\pi \int_{2\pi a/\lambda}^{\infty} \int_{0}^{2\pi} \left|\int_{\R^2} f(\bm{x}) \exp\Bigg(\im \, r \begin{pmatrix}
\cos \theta\\
\sin \theta
\end{pmatrix}^T \bm{x} \Bigg) d\bm{x} \right|^2 d\theta\, r\, dr 
\eqsim  \int_{\|\bm{y}\|\geq 2\pi a/\lambda} |c(\bm{y})|^2 \, d\bm{y}.
\end{align*}
Since by assumption $c(\bm{y})=\Landau(\|\bm{y}\|^{-(2+\varepsilon)})$ as $\|\bm{y}\|\rightarrow\infty$, we have 
\begin{align*}
\int_{\|\bm{y}\|\geq 2\pi a/\lambda} \left|c(\bm{y})\right|^2 \, d\bm{y}  
&\lesssim \left(\frac{\lambda}{a}\right)^{2+\varepsilon} \int_{\|\bm{y}\|\geq 2\pi a/\lambda} \|\bm{y}\|^{2+\varepsilon}\, |c(\bm{y})|^2\, d\bm{y}\\
&= \Landau\left(\left(\frac{\lambda}{a}\right)^{2+\varepsilon} \int_{2\pi a/\lambda}^\infty \frac{1}{r^{1+\varepsilon}}\, dr\right)=\Landau\left(\frac{\lambda^{2+2\varepsilon}}{a^{2+2\varepsilon}}\right) 
\end{align*}
as $a,\lambda\to\infty$, which is the order of the term $I_1$.
Using $a^2-b^2=(a+b)(a-b)$, we moreover obtain
\begin{align*}
|I_2| \leq I_{21}(a,\lambda)+I_{22}(a,\lambda),
\end{align*}
where
\begin{align*}
I_{21}(a,\lambda)&:=\int_{0}^{2\pi a/\lambda} \int_{0}^{2\pi} \Bigg| \int_{\R^2} f(\bm{x}) \exp \Bigg(\im \, r \begin{pmatrix}
\cos \theta\\
\sin \theta
\end{pmatrix}^T \bm{x} \Bigg) d\bm{x}\\
 &\phantom{===============}+ \int_{[-2\pi a/\lambda,2\pi a/\lambda]^2} f(\bm{x}) \exp \Bigg(\im \, r \begin{pmatrix}
\cos \theta\\
\sin \theta
\end{pmatrix}^T \bm{x} \Bigg) d\bm{x}\Bigg| \, r \,d\theta \, dr
\end{align*}
and
\begin{align*}
I_{22}(a,\lambda)&:=\sup_{0\leq r \leq 2\pi a/\lambda} \int_{0}^{2\pi} \Bigg| \int_{\R^2} f(\bm{x}) \exp \Bigg(\im \, r \begin{pmatrix}
\cos \theta\\
\sin \theta
\end{pmatrix}^T \bm{x} \Bigg) d\bm{x}\\
&\phantom{============}- \int_{[-2\pi a/\lambda,2\pi a/\lambda]^2} f(\bm{x}) \exp \Bigg(\im \, r \begin{pmatrix}
\cos \theta\\
\sin \theta
\end{pmatrix}^T \bm{x} \Bigg) d\bm{x}\Bigg|\, d\theta.
\end{align*}
First, note that
\begin{align*}
I_{21}(a,\lambda)
&\lesssim  \int_{\|\bm{y}\|\leq 2\pi a/\lambda} \left|c(\bm{y})\right|\, d\bm{y} + \int_{\|\bm{y}\|\leq 2\pi a/\lambda} \left|\int_{[-2\pi a/\lambda,2\pi a/\lambda]^2} f(\bm{x}) \exp \big(\im \bm{y}^T \bm{x} \big) d\bm{x}\right|\,d\bm{y}\\
&=\Landau\left(1+\frac{a^{2-\delta}}{\lambda^{2-\delta}}\right),
\end{align*}
where we used the fact that the covariance function $c$ is absolutely integrable and Corollary \ref{cor_of_D2_finite}. Moreover, we have
\begin{align*}
I_{22}(a,\lambda)
&\lesssim \sup_{\|\bm{y}\|\leq 2\pi a/\lambda} \left|\int_{([-2\pi a/\lambda,2\pi a/\lambda]^2)^c} f(\bm{x}) \exp\big(\im \bm{y}^T \bm{x}\big) \, d\bm{x} \right|\\
&\leq \int_{([-2\pi a/\lambda,2\pi a/\lambda]^2)^c} \beta_{1+\delta}(\bm{x})\, d\bm{x}
=\Landau\left(\int_{2\pi a/\lambda}^{\infty} \beta_{1+\delta}(x)\, dx\right) = \Landau\left(\int_{2\pi a/\lambda}^{\infty} \frac{1}{x^{1+\delta}}\, dx \right) = \Landau\left(\frac{\lambda^\delta}{a^\delta}\right).
\end{align*}
This yields 
\begin{align*}
I_2 = \Landau\left(\frac{\lambda^\delta}{a^\delta}\left[1+\frac{a^{2-\delta}}{\lambda^{2-\delta}}\right]\right) = \Landau\left(\frac{\lambda^\delta}{a^\delta}+\frac{\lambda^{2\delta-2}}{a^{2\delta-2}}\right).
\end{align*}
\qed

\section{Proof of Theorem \ref{asymptotic_normality_M}}
\setcounter{equation}{0}

We need to show the following statements: 
\begin{align} \label{erstens}
\lambda\,\E\big[\hat{M}_{\lambda,a}-M\big] = o(1), 
\end{align}
\begin{align} \label{zweitens}
\frac{\lambda^2}{\tau_{\lambda,a}^2}\Var\big[\hat{M}_{\lambda,a}\big] = 1 + o(1), 
\end{align}
and
\begin{align} \label{drittens}
\lambda^q \, \cum_q\big[\hat{M}_{\lambda,a}\big] = o(1) \qquad \text{for } q \geq 3,
\end{align}
as $\lambda,a,n\rightarrow\infty$ in the sense of Assumption \ref{assumption_on_sampling_scheme} (iii) and Assumption \ref{assumptions_on_a}.
Statement \eqref{erstens} directly follows from the definition $\hat{M}_{\lambda,a}=\hat{D}_{1,\lambda,a}-\hat{D}_{2,\lambda,a}$, Lemma \ref{two_integrals}, and Theorem  \ref{corr_first_int} and  \ref{corr_sec_int}. Concerning statement \eqref{zweitens}, we obtain from Theorems \ref{expectation_theo} and \ref{expect_theo_sec_int} and since the estimators $\hat{D}_{1,\lambda,a}$ and $\hat{D}_{2,\lambda,a}$ are real valued. This follows from the representation 
\begin{align*}
\lambda^2 \, \Var[\hat{M}_{\lambda,a}]&=\lambda^2 \, \Var[\hat{D}_{1,\lambda,a}]+\lambda^2 \, \Var[\hat{D}_{2,\lambda,a}]-2\lambda^2 \Cov[\hat{D}_{1,\lambda,a},\hat{D}_{2,\lambda,a}]\\
&=\tau_{1,\lambda,a}^2+\tau_{2,\lambda,a}^2-2\lambda^2 \cum[\hat{D}_{1,\lambda,a},\hat{D}_{2,\lambda,a}]+o(1).
\end{align*}
It thus suffices to show that
\begin{align} \label{to_show_covariance}
\lambda^2\, \cum[\hat{D}_{1,\lambda,a},\hat{D}_{2,\lambda,a}]&= \tau_{1,2,\lambda,a} + o(1),
\end{align}
where $\tau_{1,2,\lambda,a}$ is defined in \eqref{covariance_of_M}.
Concerning statement \eqref{drittens}, note that by the multilinearity property of cumulants we have
\begin{align*}
\cum_q[\hat{M}_{\lambda,a}]= \cum_q[\hat{D}_{1,\lambda,a}-\hat{D}_{2,\lambda,a}]=\sum_{p=0}^q (-1)^p \, \binom{q}{p} \,\cum_{q-p,p} [\hat{D}_{1,\lambda,a},\hat{D}_{2,\lambda,a}],
\end{align*}
where $\cum_{q-p,p} [\hat{D}_{1,\lambda,a},\hat{D}_{2,\lambda,a}]$ denotes the joint cumulant
\begin{align*}
\cum_q[\underbrace{\hat{D}_{1,\lambda,a},\ldots,\hat{D}_{1,\lambda,a}}_{q-p \text{ times}},\underbrace{\hat{D}_{2,\lambda,a},\ldots,\hat{D}_{2,\lambda,a}}_{p \text{ times}}].
\end{align*}
In order to prove \eqref{drittens}, we thus need to show that
\begin{align} \label{to_show_cum}
\lambda^q \,\cum_{q-p,p} [\hat{D}_{1,\lambda,a},\hat{D}_{2,\lambda,a}] \rightarrow 0 \qquad \text{for } q \geq 3 \text{ and } 0\leq p\leq q.
\end{align}
Note that the cases $p=0$ and $p=q$ have already been treated in the proofs of Theorems \ref{asymptotic_normality} and \ref{asymptotic_normality_secint}. We now show statements \eqref{to_show_covariance} and \eqref{to_show_cum}.\\


\underline{Proof of statement \eqref{to_show_covariance}}\\
Using the same arguments as in the proofs of Theorems \ref{expectation_theo} and \ref{expect_theo_sec_int}, we obtain
\begin{align*}
\lambda^2\, \cum[\hat{D}_{1,\lambda,a},\hat{D}_{2,\lambda,a}]
&=\frac{\pi (2\pi\lambda)^4}{n^8 H(\bm{0})^4} \sum_{\bm{g},\bm{k},\bm{\ell}=-a}^{a-1} \frac{2\pi}{\lambda} \sum_{r=0}^{a-1} \tilde{\omega}_r \, J_0(\tilde{\omega}_r \|\tilde{\bm{\omega}}_{\bm{k}}\|) J_0(\tilde{\omega}_r \|\tilde{\bm{\omega}}_{\bm{\ell}}\|)\\ 
&\phantom{=}\sum_{\bm{\nu}=\{\nu_1,\ldots,\nu_4\}\in\mathcal{I}} \sum_{\underline{j}\in\mathcal{D}(8)} \cum_2\big[Y_{t,c}(\underline{j}): (t,c)\in\nu_1 \big] \times \ldots\\
&\phantom{======iiii=} \times  \cum_2\big[Y_{t,c}(\underline{j}): (t,c)\in\nu_4 \big]+ \Landau\left(\frac{\lambda^2}{n}+\frac{1}{\lambda^2}\right),
\end{align*}
where $\mathcal{I}$ is the set of indecomposable partitions of Table \eqref{table_q=2}, the set $\mathcal{D}(i)$ is defined in \eqref{set_D(i)}, and
\begin{align*}
Y_{t,c}(\underline{j}):=\begin{cases}
h\left(\frac{\bm{s}_{j_{c+4(t-1)}}}{\lambda}\right) Z(\bm{s}_{j_{c+4(t-1)}}) \exp\big((-1)^{c+1}\, \im\, \bm{s}_{j_{c+4(t-1)}}^T \tilde{\bm{\omega}}_{\bm{g}}\big), &\qquad \text{if } t=1,\\
h\left(\frac{\bm{s}_{j_{c+4(t-1)}}}{\lambda}\right) Z(\bm{s}_{j_{c+4(t-1)}}) \exp\big((-1)^{c+1}\, \im\, \bm{s}_{j_{c+4(t-1)}}^T \tilde{\bm{\omega}}_{\bm{k}}\big), &\qquad \text{if } t=2 \text{ and } c=1,2,\\
h\left(\frac{\bm{s}_{j_{c+4(t-1)}}}{\lambda}\right) Z(\bm{s}_{j_{c+4(t-1)}}) \exp\big((-1)^{c+1}\, \im\, \bm{s}_{j_{c+4(t-1)}}^T \tilde{\bm{\omega}}_{\bm{\ell}}\big), &\qquad \text{if } t=2 \text{ and } c=3,4.
\end{cases}
\end{align*}
For the above equality, we used that the order of the cumulant expression is determined by partitions consisting of $G=4$ groups and maximally many different elements in $\underline{j}$, while it is $\Landau(\lambda^2/n+1/\lambda^2)=o(1)$ otherwise [see the proof of statement \eqref{to_show_cum}]. We need to consider each indecomposable partition $\bm{\nu}\in\mathcal{I}$ consisting of $G=4$ groups separately and, for illustration, deal with the partition 
\begin{align*}
\bm{\nu}^\ast=\Big\{\{(1,1),(1,2)\},\{(2,1),(2,2)\},\{(1,3),(2,4)\},\{(1,4),(2,3)\}\Big\}.
\end{align*}
Then, with the same arguments as in the proofs of 
the second part Theorems \ref{expectation_theo} and \ref{expect_theo_sec_int},
\begin{align*}
&\phantom{=i}\frac{1}{n^8} \sum_{\underline{j}\in\mathcal{D}(8)} \cum_2\big[Y_{t,c}(\underline{j}): (t,c)\in\nu_1^\ast \big] \times \ldots \times  \cum_2\big[Y_{t,c}(\underline{j}): (t,c)\in\nu_4^\ast \big]\\
&=\frac{c_{2,n}}{\lambda^{16} (2\pi)^8} \int_{\R^8} B(\bm{s})^2 B(\bm{t})^2 B(\bm{u}) B(\bm{u}+\bm{\omega}_{\bm{\ell}-\bm{g}}) B(\bm{v}) B(\bm{v}+\bm{\omega}_{\bm{g}-\bm{\ell}})\\
&\phantom{=======i=}  f(\bm{s}-\tilde{\bm{\omega}}_{\bm{g}}) f(\bm{t}+\tilde{\bm{\omega}}_{\bm{k}}) f(\bm{u}-\tilde{\bm{\omega}}_{\bm{g}}) f(\bm{v}+\tilde{\bm{\omega}}_{\bm{g}})\, d\bm{s}d\bm{t}d\bm{u}d\bm{v}.
\end{align*}
Therefore, by making the index shift $\bm{\ell}-\bm{g}=\bm{m}$, the cumulant expression corresponding to the partition $\bm{\nu}^\ast$ equals
\begin{align} \label{eq_with_K_i}
&\phantom{=iii}
\frac{c_{2,n}\, \pi}{(2\pi\lambda)^8 H(\bm{0})^4} \sum_{\bm{m}=-2a+1}^{2a-1} \int_{\R^8} \Big[ B(\bm{s})^2 B(\bm{t})^2 B(\bm{u}) B(\bm{u}+\bm{\omega_m}) B(\bm{v}) B(\bm{v}-\bm{\omega_m})\nonumber\\
&\phantom{==============}  \left(\frac{2\pi}{\lambda}\right)^4 \sum_{\bm{k}=-a}^{a-1} \sum_{\bm{g}=\max(-a,-a-\bm{m})}^{\min(a-1,a-1-\bm{m})}  f(\bm{s}-\tilde{\bm{\omega}}_{\bm{g}}) f(\bm{t}+\tilde{\bm{\omega}}_{\bm{k}}) f(\bm{u}-\tilde{\bm{\omega}}_{\bm{g}}) f(\bm{v}+\tilde{\bm{\omega}}_{\bm{g}})\nonumber\\
&\phantom{==============}  \frac{2\pi}{\lambda} \sum_{r=0}^{a-1}\, \tilde{\omega}_r\, J_0(\tilde{\omega}_r \|\tilde{\bm{\omega}}_{\bm{k}}\|) J_0(\tilde{\omega}_r \|\tilde{\bm{\omega}}_{\bm{m}+\bm{g}}\|)\Big] \, d\bm{s}  d\bm{t}d\bm{u}d\bm{v}\nonumber\\
&=\pi \sum_{\bm{m}=-2a+1}^{2a-1} \frac{H(\bm{m})^2}{H(\bm{0})^2} \int_{0}^{2\pi a/\lambda} r \Bigg[\int_{[-2\pi a/\lambda,2\pi a/\lambda]^2} f(\bm{y}) J_0(r\|\bm{y}\|)\, d\bm{y}\nonumber\\
&\phantom{=================}  \int_{2\pi\max(-a,-a-\bm{m})/\lambda}^{2\pi\min(a,a-\bm{m})/\lambda} f(\bm{x})^3 J_0(r\|\bm{x}\|) \,d\bm{x} \Bigg] \, dr + \sum_{i=1}^4 K_i(a,\lambda)+\Landau\left(\frac{1}{n}\right),
\end{align}
where 
\begin{align*}
K_1(a,\lambda)&:=\frac{c_{2,n}\, \pi}{(2\pi\lambda)^8 H(\bm{0})^4} \sum_{\bm{m}=-2a+1}^{2a-1} \int_{\R^8} \bigg[ B(\bm{s})^2 B(\bm{t})^2 B(\bm{u}) B(\bm{u}+\bm{\omega_m}) B(\bm{v}) B(\bm{v}-\bm{\omega_m})\\
&\phantom{=========}  \left(\frac{2\pi}{\lambda}\right)^4 \sum_{\bm{k}=-a}^{a-1} \sum_{\bm{g}=\max(-a,-a-\bm{m})}^{\min(a-1,a-1-\bm{m})}  f(\bm{s}-\tilde{\bm{\omega}}_{\bm{g}}) f(\bm{t}+\tilde{\bm{\omega}}_{\bm{k}}) f(\bm{u}-\tilde{\bm{\omega}}_{\bm{g}}) f(\bm{v}+\tilde{\bm{\omega}}_{\bm{g}})\\
&\phantom{=========}\Bigg(\frac{2\pi}{\lambda} \sum_{r=0}^{a-1} \tilde{\omega}_r J_0(\tilde{\omega}_r \|\tilde{\bm{\omega}}_{\bm{k}}\|) J_0(\tilde{\omega}_r \|\tilde{\bm{\omega}}_{\bm{m}+\bm{g}}\|)\\
&\phantom{================}- \int_{0}^{2\pi a/\lambda} r\, J_0(r \|\tilde{\bm{\omega}}_{\bm{k}}\|) J_0(r \|\tilde{\bm{\omega}}_{\bm{m}+\bm{g}}\|)\, dr\Bigg)\bigg] \, d\bm{s}  d\bm{t}d\bm{u}d\bm{v},
\end{align*}
\begin{align*}
K_2(a,\lambda)&:= \frac{c_{2,n}\, \pi}{(2\pi\lambda)^8 H(\bm{0})^4} \sum_{\bm{m}=-2a+1}^{2a-1} \int_{\R^8}  B(\bm{s})^2 B(\bm{t})^2 B(\bm{u}) B(\bm{u}+\bm{\omega_m}) B(\bm{v}) B(\bm{v}-\bm{\omega_m})\\
& \Bigg(\int_{0}^{2\pi a/\lambda} r \Bigg[\left(\frac{2\pi}{\lambda}\right)^4 \sum_{\bm{k}=-a}^{a-1} \sum_{\bm{g}=\max(-a,-a-\bm{m})}^{\min(a-1,a-1-\bm{m})} f(\bm{s}-\tilde{\bm{\omega}}_{\bm{g}})  f(\bm{t}+\tilde{\bm{\omega}}_{\bm{k}})\\
&\phantom{=====================}  f(\bm{u}-\tilde{\bm{\omega}}_{\bm{g}})  f(\bm{v}+\tilde{\bm{\omega}}_{\bm{g}}) J_0(r\|\tilde{\bm{\omega}}_{\bm{k}}\|) J_0(r\|\tilde{\bm{\omega}}_{\bm{m}+\bm{g}}\|)\\
&\phantom{===}-\int_{[-2\pi a/\lambda,2\pi a/\lambda]^2} \int_{2\pi\max(-a,-a-\bm{m})/\lambda}^{2\pi\min(a,a-\bm{m})/\lambda} f(\bm{s}-\bm{x}) f(\bm{t}+\bm{y})\\
&\phantom{===========} f(\bm{u}-\bm{x}) f(\bm{v}+\bm{x}) J_0(r\|\bm{y}\|) J_0(r\|\bm{\omega_m}+\bm{x}\|) \,d\bm{x} d\bm{y}\Bigg]\,dr \Bigg) \, d\bm{s}d\bm{t}d\bm{u}d\bm{v},
\end{align*}
\begin{align*}
K_3(a,\lambda)&:=\frac{c_{2,n}\, \pi}{(2\pi\lambda)^8 H(\bm{0})^4} \sum_{\bm{m}=-2a+1}^{2a-1} \int_{\R^8}  B(\bm{s})^2 B(\bm{t})^2 B(\bm{u}) B(\bm{u}+\bm{\omega_m}) B(\bm{v}) B(\bm{v}-\bm{\omega_m})\\
&\Bigg(\int_{0}^{2\pi a/\lambda} r \Bigg[\int_{[-2\pi a/\lambda,2\pi a/\lambda]^2} \int_{2\pi\max(-a,-a-\bm{m})/\lambda}^{2\pi\min(a,a-\bm{m})/\lambda} \Big\{f(\bm{s}-\bm{x}) f(\bm{t}+\bm{y}) f(\bm{u}-\bm{x}) f(\bm{v}+\bm{x}) \\
&\phantom{========}-f(\bm{x})^3 f(\bm{y}) \Big\} J_0(r\|\bm{y}\|) J_0(r\|\bm{\omega_m}+\bm{x}\|)\, d\bm{x} d\bm{y} \Bigg] dr \Bigg) d\bm{s} d\bm{t} d\bm{u} d\bm{v},
\end{align*}
and
\begin{align*}
K_4(a,\lambda)&:= c_{2,n}\, \pi \sum_{\bm{m}=-2a+1}^{2a-1} \frac{H(\bm{m})^2}{H(\bm{0})^2} \int_{0}^{2\pi a/\lambda} r \Bigg[ \int_{[-2\pi a/\lambda,2\pi a/\lambda]^2} f(\bm{y}) J_0(r\|\bm{y}\|) d\bm{y}\\
&\phantom{==========} \int_{2\pi\max(-a,-a-\bm{m})/\lambda}^{2\pi\min(a,a-\bm{m})/\lambda} f^3(\bm{x}) \Big\{J_0(r\|\bm{\omega_m}+\bm{x}\|) - J_0(r\|\bm{x}\|) \Big\}\, d\bm{x} \Bigg] \, dr.
\end{align*}
Here, we also used Corollary \ref{cor_of_D2_finite} for $\delta>2$ and Lemma \ref{convolution_of_h}.
{{We now derive the orders of the terms $K_i(a,\lambda)$ for $i=1,\ldots,4$. We start with $K_1(a,\lambda)$ and assume that $a>\lambda\geq 1$. Lemma \ref{Riemann_radius} implies
\begin{align*}
|K_1(a,\lambda)| 
&\lesssim\frac{a^2}{\lambda^4}\times\frac{1}{\lambda^8} \sum_{\bm{m}=-2a+1}^{2a-1} \int_{\R^8} \Bigg(\big| B(\bm{s})^2 B(\bm{t})^2 B(\bm{u}) B(\bm{u}+\bm{\omega_m}) B(\bm{v}) B(\bm{v}-\bm{\omega_m})\big|\\
&\phantom{========} \left(\frac{2\pi}{\lambda}\right)^4 \sum_{\bm{k}=-a}^{a-1} \sum_{\bm{g}=\max(-a,-a-\bm{m})}^{\min(a-1,a-1-\bm{m})}  f(\bm{s}-\tilde{\bm{\omega}}_{\bm{g}}) f(\bm{t}+\tilde{\bm{\omega}}_{\bm{k}}) f(\bm{u}-\tilde{\bm{\omega}}_{\bm{g}}) f(\bm{v}+\tilde{\bm{\omega}}_{\bm{g}})\\
&\phantom{===========0==}\big(\tilde{\omega}_{k_1}^2+\tilde{\omega}_{k_2}^2+\tilde{\omega}_{m_1+g_1}^2+\tilde{\omega}_{m_2+g_2}^2\big)\Bigg) \, d\bm{s} d\bm{t} d\bm{u} d\bm{v}+ \Landau\left(\frac{a^2}{\lambda^4}\right)
\end{align*}
as $a,\lambda\to\infty$, where we moreover used $\sup_{|k|\leq a} |\tilde{\omega}_k|=\Landau(a/\lambda)$, $\frac{1}{\lambda^2} \sum_{\bm{k}=-a}^{a-1} f(\bm{t}+\tilde{\bm{\omega}}_{\bm{k}})=\Landau(1)$, and Lemmas \ref{bounds_for_B} and \ref{orders_of_B}. 
One can easily verify that the first of these two terms can (ignoring constants) be bounded by
\begin{align*}
&\phantom{\lesssim} \frac{a^2}{\lambda^4} \times \frac{1}{\lambda^2} \int_{\R^2} \left( B(\bm{t})^2 \, \left(\frac{2\pi}{\lambda}\right)^2 \sum_{\bm{k}=-a}^{a-1} f(\bm{t}+\tilde{\bm{\omega}}_{\bm{k}}) \big(\tilde{\omega}_{k_1}^2+\tilde{\omega}_{k_2}^2\big) \right) d\bm{t} \\
&+\frac{a^2}{\lambda^4} \times \frac{1}{\lambda^4} \sum_{\bm{m}=-2a+1}^{2a-1} \int_{\R^4} \Bigg( \big|B(\bm{u}) B(\bm{u}-\bm{\omega_m}) B(\bm{v}) B(\bm{v}+\bm{\omega_m})\big|\\
&\phantom{================} \left(\frac{2\pi}{\lambda}\right)^2 \sum_{\bm{g}=\max(-a,-a-\bm{m})}^{\min(a-1,a-1-\bm{m})} f(\bm{u}-\tilde{\bm{\omega}}_{\bm{m}+\bm{g}}) (\tilde{\omega}_{m_1+g_1}^2 + \tilde{\omega}_{m_2+g_2}^2)\Bigg)\, d\bm{u} d\bm{v}.
\end{align*}
Using Lemma \ref{f_times_x^2}, it thus follows that
\begin{align*}
K_1(a,\lambda)=\Landau\left(\frac{a^2}{\lambda^4}\left[1+\frac{a^2}{\lambda^4}+\frac{(\log a)^4}{\lambda^2}\right]\right)=\Landau\left(\frac{a^2}{\lambda^4}\right)
\end{align*}
as $a,\lambda\to\infty$. We now deal with the term $K_2(a,\lambda)$ and it is easy to see that
\begin{align*}
|K_2(a,\lambda)|
&\lesssim \frac{a^2}{\lambda^2} \times \frac{1}{\lambda^8} \sum_{\bm{m}=-2a+1}^{2a-1} \int_{\R^8} \big|B(\bm{s})^2 B(\bm{t})^2 B(\bm{u}) B(\bm{u}+\bm{\omega_m}) B(\bm{v}) B(\bm{v}-\bm{\omega_m})\big|\\
&\Bigg( \sup_{\|\bm{y}\|\leq 2\pi a/\lambda} \Bigg|\left(\frac{2\pi}{\lambda}\right)^2 \sum_{\bm{g}=\max(-a,-a-\bm{m})}^{\min(a-1,a-1-\bm{m})} f(\bm{s}-\tilde{\bm{\omega}}_{\bm{g}}) f(\bm{u}-\tilde{\bm{\omega}}_{\bm{g}}) f(\bm{v}+\tilde{\bm{\omega}}_{\bm{g}}) \exp\big(\im \bm{y}^T (\bm{\omega_m}+\tilde{\bm{\omega}}_{\bm{g}})\big)\\
&\phantom{=}-\int_{2\pi\max(-a,-a-\bm{m})/\lambda}^{2\pi\min(a,a-\bm{m})/\lambda} f(\bm{s}-\bm{x}) f(\bm{u}-\bm{x}) f(\bm{v}+\bm{x}) \exp\big(\im \bm{y}^T (\bm{\omega_m}+\bm{x})\big)\, d\bm{x} \Bigg| \Bigg)\, d\bm{s} d\bm{t} d\bm{u} d\bm{v}\\
&+\frac{a^2}{\lambda^2} \times \frac{1}{\lambda^8} \sum_{\bm{m}=-2a+1}^{2a-1} \int_{\R^8} \big|B(\bm{s})^2 B(\bm{t})^2 B(\bm{u}) B(\bm{u}+\bm{\omega_m}) B(\bm{v}) B(\bm{v}-\bm{\omega_m})\big|\\
&\phantom{\lesssim}\Bigg(\sup_{\|\bm{y}\|\leq 2\pi a/\lambda} \Bigg| \left(\frac{2\pi}{\lambda}\right)^2 \sum_{\bm{k}=-a}^{a-1} f(\bm{t}+\tilde{\bm{\omega}}_{\bm{k}}) \exp(\im \bm{y}^T \tilde{\bm{\omega}}_{\bm{k}}) \\
&\phantom{==============}- \int_{[-2\pi a/\lambda,2\pi a/\lambda]^2} f(\bm{t}+\bm{x}) \exp(\im \bm{y}^T \bm{x})\, d\bm{x}\Bigg|\Bigg)\, d\bm{s} d\bm{t} d\bm{u} d\bm{v}\\
&=\Landau\left(\frac{a^4}{\lambda^6}\right),
\end{align*}
where we used Lemma \ref{Riemann_f_sec_int} for the second step.
Next, note that
\begin{align*}
|K_3(a,\lambda)| &\lesssim \frac{1}{\lambda^6} \sum_{\bm{m}=-2a+1}^{2a-1} \Bigg[\int_{\R^6} B(\bm{s})^2 B(\bm{t})^2 \big|B(\bm{u}) B(\bm{u}+\bm{\omega_m})\big| \big|D_{\bm{m}}^{(1)}(\bm{s},\bm{t},\bm{u})\big|\, d\bm{s} d\bm{t} d\bm{u}\\
&\phantom{\lesssim \frac{1}{\lambda^6} \sum_{\bm{m}=-2a+1}^{2a-1}}+ \int_{\R^6} B(\bm{s})^2 B(\bm{t})^2 \big|B(\bm{v}) B(\bm{v}-\bm{\omega_m})\big| \big|D_{\bm{m}}^{(2)}(\bm{s},\bm{t})\big|\, d\bm{s} d\bm{t} d\bm{v}\\
&\phantom{\lesssim \frac{1}{\lambda^6} \sum_{\bm{m}=-2a+1}^{2a-1}}+ \int_{\R^6} B(\bm{s})^2 \big|B(\bm{u}) B(\bm{u}+\bm{\omega_m}) B(\bm{v}) B(\bm{v}-\bm{\omega_m})\big| \big|D_{\bm{m}}^{(3)}(\bm{s})\big|\, d\bm{s} d\bm{u} d\bm{v}\\
&\phantom{\lesssim \frac{1}{\lambda^6} \sum_{\bm{m}=-2a+1}^{2a-1}}+ \int_{\R^6} B(\bm{t})^2 \big|B(\bm{u}) B(\bm{u}+\bm{\omega_m}) B(\bm{v}) B(\bm{v}-\bm{\omega_m})\big| \big|D_{\bm{m}}^{(4)}\big|\, d\bm{t} d\bm{u} d\bm{v}\Bigg],
\end{align*}
where
\begin{align*}
D_{\bm{m}}^{(1)}(\bm{s},\bm{t},\bm{u})&:=\frac{1}{\lambda^2} \int_{\R^2} B(\bm{v}) B(\bm{v}-\bm{\omega_m}) \Bigg(\int_{0}^{2\pi a/\lambda} r \Bigg[\int_{[-2\pi a/\lambda,2\pi a/\lambda]^2} \int_{2\pi\max(-a,-a-\bm{m})/\lambda}^{2\pi\min(a,a-\bm{m})/\lambda}\\
&\phantom{================} f(\bm{s}-\bm{x}) f(\bm{t}+\bm{y}) f(\bm{u}-\bm{x}) \big[f(\bm{v}+\bm{x})-f(\bm{x})\big]\\
&\phantom{================} J_0(r\|\bm{y}\|) J_0(r\|\bm{\omega_m}+\bm{x}\|)\, d\bm{x} d\bm{y} \Bigg] \,dr \Bigg)\, d\bm{v},
\end{align*}
\begin{align*}
D_{\bm{m}}^{(2)}(\bm{s},\bm{t})&:=\frac{1}{\lambda^2} \int_{\R^2} B(\bm{u}) B(\bm{u}+\bm{\omega_m}) \Bigg(\int_{0}^{2\pi a/\lambda} r \Bigg[\int_{[-2\pi a/\lambda,2\pi a/\lambda]^2} \int_{2\pi\max(-a,-a-\bm{m})/\lambda}^{2\pi\min(a,a-\bm{m})/\lambda}\\
&\phantom{==} f(\bm{x}) f(\bm{s}-\bm{x}) f(\bm{t}+\bm{y}) \big[f(\bm{u}-\bm{x})-f(\bm{x})\big] J_0(r\|\bm{y}\|) J_0(r\|\bm{\omega_m}+\bm{x}\|)\, d\bm{x} d\bm{y} \Bigg] dr \Bigg) d\bm{u},
\end{align*}
\begin{align*}
D_{\bm{m}}^{(3)}(\bm{s})&:=\frac{1}{\lambda^2} \int_{\R^2} B(\bm{t})^2  \Bigg(\int_{0}^{2\pi a/\lambda} r \Bigg[\int_{[-2\pi a/\lambda,2\pi a/\lambda]^2} \int_{2\pi\max(-a,-a-\bm{m})/\lambda}^{2\pi\min(a,a-\bm{m})/\lambda}\\
&\phantom{===} f(\bm{x})^2 f(\bm{s}-\bm{x}) \big[f(\bm{t}+\bm{y})-f(\bm{y})\big] J_0(r\|\bm{y}\|) J_0(r\|\bm{\omega_m}+\bm{x}\|)\, d\bm{x} d\bm{y} \Bigg] dr \Bigg) d\bm{t},
\end{align*}
and 
\begin{align*}
D_{\bm{m}}^{(4)}&:=\frac{1}{\lambda^2} \int_{\R^2} B(\bm{s})^2  \Bigg(\int_{0}^{2\pi a/\lambda} r \Bigg[\int_{[-2\pi a/\lambda,2\pi a/\lambda]^2} \int_{2\pi\max(-a,-a-\bm{m})/\lambda}^{2\pi\min(a,a-\bm{m})/\lambda}\\
&\phantom{===} f(\bm{x})^2 f(\bm{y}) \big[f(\bm{s}-\bm{x})-f(\bm{x})\big] J_0(r\|\bm{y}\|) J_0(r\|\bm{\omega_m}+\bm{x}\|)\, d\bm{x} d\bm{y} \Bigg] dr \Bigg) d\bm{s}.
\end{align*}
Proposition \ref{Lemma F.2_SSR} (ii) and similar arguments as in the proof of the second part of Theorem \ref{expect_theo_sec_int} yield
\begin{align*}
\max_{i=1,2} \max_{|m_i|\leq 2a-1} \sup_{\bm{s},\bm{u}\in\R^2}|D_{\bm{m}}^{(1)}(\bm{s},\bm{t},\bm{u})|
\lesssim \frac{1}{\lambda}  \int_{\|\bm{x}\|\leq 2\pi a/\lambda} \Bigg|\int_{[-2\pi a/\lambda,2\pi a/\lambda]^2} f(\bm{t}+\bm{y}) \exp(\im \bm{x}^T \bm{y})\, d\bm{y} \Bigg|\, d\bm{x},
\end{align*}
and in the same way 
\begin{align*}
\max_{i=1,2} \max_{|m_i|\leq 2a-1} \sup_{\bm{s}\in\R^2}|D_{\bm{m}}^{(2)}(\bm{s},\bm{t})|\lesssim \frac{1}{\lambda}\int_{\|\bm{x}\|\leq 2\pi a/\lambda} \Bigg|\int_{[-2\pi a/\lambda,2\pi a/\lambda]^2} f(\bm{t}+\bm{y}) \exp(\im \bm{x}^T \bm{y})\, d\bm{y} \Bigg|\, d\bm{x}.
\end{align*}
Moreover, using Proposition \ref{Lemma F.2_SSR} (i), we obtain
\begin{align*}
&\phantom{\lesssim i}\max_{i=1,2} \max_{|m_i|\leq 2a-1} \sup_{\bm{s}\in\R^2} |D_{\bm{m}}^{(3)}(\bm{s})|= \Landau\left(\frac{a^4}{\lambda^6}\right),
\end{align*}
and Proposition \ref{Lemma F.2_SSR} (ii) and Corollary \ref{cor_of_D2_finite} for $\delta>2$ yield
\begin{align*}
\max_{i=1,2} \max_{|m_i|\leq 2a-1} |D_{\bm{m}}^{(4)}|
=\Landau\left(\frac{1}{\lambda^2}\right).
\end{align*}
Therefore, we get
\begin{align*}
|K_3(a,\lambda)|
&=\Landau\left(\frac{(\log a)^4}{\lambda}  + \frac{a^4}{\lambda^6}\right)
\end{align*}
as $a,\lambda\to\infty$, where we applied Lemma \ref{cov_fct_integrable}, part (ii), Lemma \ref{bounds_for_B}, parts (ii) and (iii), and Lemma \ref{orders_of_B}, part (i). Furthermore, we used Assumption \ref{assumptions_on_a} and the fact that $\delta>2$.
It remains to deal with the term $K_4(a,\lambda)$. Using Corollary \ref{cor_of_D2_finite} for $\delta>2$ and exactly the same arguments as in the proof of Lemma \ref{bound_variance_secint}, we finally obtain
\begin{align*}
|K_4(a,\lambda)|
&= \Landau\Bigg(  \sum_{\bm{m}=-2a+1}^{2a-1} H(\bm{m})^2  \sup_{\|\bm{z}\|\leq 2\pi a/\lambda} \int_{[-2\pi a/\lambda,2\pi a/\lambda]^2} f(\bm{x})^3 \big|\exp(\im \bm{z}^T (\bm{\omega_m}+\bm{x}))-\exp(\im \bm{z}^T \bm{x})\big|\, d\bm{x}\Bigg)\\
&=\Landau\left(\frac{(\log \lambda)^2\, a}{\lambda^2} + \frac{1}{(\log\lambda)^3}\right)
\end{align*}
as $a,\lambda\to\infty$. }}
Therefore, the cumulant expression corresponding to the partition $\bm{\nu}^\ast$ amounts to
\begin{align*}
&\pi \sum_{\bm{m}=-2a+1}^{2a-1} \frac{H(\bm{m})^2}{H(\bm{0})^2} \int_{0}^{2\pi a/\lambda} r \Bigg[\int_{[-2\pi a/\lambda,2\pi a/\lambda]^2} f(\bm{y}) J_0(r\|\bm{y}\|)\, d\bm{y}\nonumber\\
&\phantom{=} \int_{2\pi\max(-a,-a-\bm{m})/\lambda}^{2\pi\min(a,a-\bm{m})/\lambda} f(\bm{x})^3 J_0(r\|\bm{x}\|) \,d\bm{x} \Bigg] \, dr +\Landau\left(\frac{a^4}{\lambda^6}+\frac{(\log a)^4}{\lambda}+\frac{(\log \lambda)^2\, a}{\lambda^2} + \frac{1}{(\log\lambda)^3}+\frac{1}{n}\right)
\end{align*}
as $a,\lambda\to\infty$, and the error term converges to $0$ by Assumption \ref{assumptions_on_a}.

In order to calculate $\lambda^2\, \cum[\hat{D}_{1,\lambda,a},\hat{D}_{2,\lambda,a}]$, we need to sum over all indecomposable partitions of Table \eqref{table_q=2} and can restrict ourselves to those partitions consisting of $4$ groups and $8$ different elements in $\underline{j}$ (see above). Moreover, the order corresponding to a certain partition depends on the number of independent restrictions in Table \eqref{table_q=2}, and it is maximal for partitions evoking exactly $1$ restriction (see Proposition \ref{prop_mix_cum} below). 
There are exactly $16$ indecomposable partitions with exactly $1$ restriction, namely 
\begin{enumerate}
\item four partitions restricting $\bm{g}$ and $\bm{\ell}$ via $\bm{g}-\bm{\ell}$
\item four partitions restricting $\bm{g}$ and $\bm{\ell}$ via $\bm{g}+\bm{\ell}$
\item four partitions restricting $\bm{g}$ and $\bm{k}$ via $\bm{g}-\bm{k}$
\item four partitions restricting $\bm{g}$ and $\bm{k}$ via $\bm{g}+\bm{k}$.
\end{enumerate}
The partition $\bm{\nu}^\ast$ from above represented the case where $\bm{g}$ and $\bm{\ell}$ are restricted via $\bm{g}-\bm{\ell}$ and it is obvious that the other three partitions restricting $\bm{g}$ and $\bm{\ell}$ via $\bm{g}-\bm{\ell}$ yield the same cumulant expression. In the same way, the partitions restricting $\bm{g}$ and $\bm{k}$ via $\bm{g}-\bm{k}$ yield the same expression since in this case, only the variables $\bm{k}$ and $\bm{\ell}$ are interchanged. Moreover, it is easy to show that the cases 2. and 4. also lead to the same cumulant expression. Since restricting two additional elements leads to an order change of at most $\Landau(a^4/\lambda^6)$ (see Proposition \ref{prop_mix_cum} below), this yields
\begin{align*}
\lambda^2 \cum[\hat{D}_{1,\lambda,a},\hat{D}_{2,\lambda,a}]&= 16\pi \sum_{\bm{m}=-2a+1}^{2a-1} \frac{H(\bm{m})^2}{H(\bm{0})^2} \int_{0}^{2\pi a/\lambda} r \Bigg[\int_{[-2\pi a/\lambda,2\pi a/\lambda]^2} f(\bm{y}) J_0(r\|\bm{y}\|)\, d\bm{y}\nonumber\\
&\phantom{========} \int_{2\pi\max(-a,-a-\bm{m})/\lambda}^{2\pi\min(a,a-\bm{m})/\lambda} f(\bm{x})^3 J_0(r\|\bm{x}\|) \,d\bm{x} \Bigg] \, dr\\
& +\Landau\left(\frac{a^4}{\lambda^6}+\frac{(\log a)^4}{\lambda}+\frac{(\log \lambda)^2\, a}{\lambda^2} + \frac{1}{(\log\lambda)^3}+\frac{\lambda^2}{n}\right),
\end{align*}
which proves \eqref{to_show_covariance}. \\

\underline{Proof of statement \eqref{to_show_cum}}
\\
For simplicity, we write $\hat{D}_1=\hat{D}_{1,\lambda,a}$, $\hat{D}_2=\hat{D}_{2,\lambda,a}$. 
We assume $1\leq p\leq q-1$ and define the sets $\mathcal{D}(q)$ and $\mathcal{D}(q,i)$ as in \eqref{set_D(q)} and \eqref{set_D(q,i)}, respectively. Moreover, let $\mathcal{I}(q)$ denote the set of indecomposable partitions of Table \eqref{table}. We then obtain
\begin{align*}
\lambda^q \,\cum_{q-p,p}[\hat{D}_{1},\hat{D}_{2}]
&\eqsim \lambda^{3q-2p} \sum_{\bm{g}_1,\ldots,\bm{g}_{q-p}=-a}^{a-1} \sum_{\bm{k}_1,\ldots,\bm{k}_p=-a}^{a-1} \sum_{\bm{\ell}_1,\ldots,\bm{\ell}_p=-a}^{a-1} \frac{1}{\lambda^p} \sum_{r_1,\ldots,r_p=0}^{a-1} \tilde{\omega}_{r_1} \ldots\tilde{\omega}_{r_p} \\
&\phantom{========} J_0(\tilde{\omega}_{r_1} \|\tilde{\bm{\omega}}_{\bm{k}_1}\|)J_0(\tilde{\omega}_{r_1} \|\tilde{\bm{\omega}}_{\bm{\ell}_1}\|)\ldots J_0(\tilde{\omega}_{r_p} \|\tilde{\bm{\omega}}_{\bm{k}_p}\|)J_0(\tilde{\omega}_{r_p} \|\tilde{\bm{\omega}}_{\bm{\ell}_p}\|)\\
&\phantom{==} \frac{1}{n^{4q}} \sum_{\substack{\bm{\nu}=\{\nu_1,\ldots,\nu_G\}\\\in\mathcal{I}(q)}} \sum_{i=4}^{2q+G} \sum_{\underline{j}\in\mathcal{D}(q,i)} \cum_{|\nu_1|}[Y_{t,c}(\underline{j}) \, : \, (t,c)\in\nu_1] \times \ldots\\
&\phantom{========ii=====} \times \cum_{|\nu_G|}[Y_{t,c}(\underline{j}) \, : \, (t,c)\in\nu_G],
\end{align*}
where for $t=1,\ldots,q$ and $c=1,\ldots,4$
\begin{align*}
Y_{t,c}(\underline{j}):=\begin{cases}
h\left(\frac{\bm{s}_{j_{c+4(t-1)}}}{\lambda}\right) Z(\bm{s}_{j_{c+4(t-1)}}) \exp\big((-1)^{c+1}\, \im \, \bm{s}_{j_{c+4(t-1)}}^T \tilde{\bm{\omega}}_{\bm{g}_t}\big), &\quad\text{if } t=1,\ldots,q-p,\\
h\left(\frac{\bm{s}_{j_{c+4(t-1)}}}{\lambda}\right) Z(\bm{s}_{j_{c+4(t-1)}}) \exp\big((-1)^{c+1}\, \im \, \bm{s}_{j_{c+4(t-1)}}^T \tilde{\bm{\omega}}_{\bm{k}_{t-q+p}}\big), &\quad\text{if } t=q-p+1,\ldots,q\\
&\phantom{\quad} \text{and } c=1,2,\\
h\left(\frac{\bm{s}_{j_{c+4(t-1)}}}{\lambda}\right) Z(\bm{s}_{j_{c+4(t-1)}}) \exp\big((-1)^{c+1}\, \im \, \bm{s}_{j_{c+4(t-1)}}^T \tilde{\bm{\omega}}_{\bm{\ell}_{t-q+p}}\big), &\quad\text{if } t=q-p+1,\ldots,q \\
&\phantom{\quad} \text{and } c=3,4.
\end{cases}
\end{align*}
As in the proofs of Theorems \ref{asymptotic_normality} and \ref{asymptotic_normality_secint}, we consider an indecomposable partition $\bm{\nu}^\ast=\{\nu_1^\ast,\ldots,\nu_{2q}^\ast\}$ with $i=4q$ different elements in $\underline{j}$ for illustration. Using the same arguments as in the proofs of Theorems \ref{asymptotic_normality} and \ref{asymptotic_normality_secint}, the respective cumulant term then equals
\begin{align*}
&\lambda^{3q-2p} \sum_{\bm{g}_1,\ldots,\bm{g}_{q-p}=-a}^{a-1} \sum_{\substack{\bm{k}_1,\ldots,\bm{k}_p=-a\\\bm{\ell}_1,\ldots,\bm{\ell}_p=-a}}^{a-1} \frac{1}{\lambda^p} \sum_{r_1,\ldots,r_p=0}^{a-1} \tilde{\omega}_{r_1} \ldots\tilde{\omega}_{r_p}  J_0(\tilde{\omega}_{r_1} \|\tilde{\bm{\omega}}_{\bm{k}_1}\|)J_0(\tilde{\omega}_{r_1} \|\tilde{\bm{\omega}}_{\bm{\ell}_1}\|)\times\ldots\\
&\phantom{=========================}\times  J_0(\tilde{\omega}_{r_p} \|\tilde{\bm{\omega}}_{\bm{k}_p}\|)J_0(\tilde{\omega}_{r_p} \|\tilde{\bm{\omega}}_{\bm{\ell}_p}\|)\\
&\phantom{======}\frac{1}{\lambda^{8q}} \int_{\R^{4q}} \prod_{g=1}^{2q} f(\bm{u}_g-\tilde{\bm{\omega}}_{\hat{\bm{h}}_{2g-1}}) B(\bm{u}_g) B(\bm{u}_g+\bm{\omega}_{\hat{\bm{h}}_{2g}-\hat{\bm{h}}_{2g-1}})\, d\bm{u}_g,
\end{align*}
where the $2$-dimensional vectors $\hat{\bm{h}}_j$ are taken from the set
\begin{align*}
\{\bm{k}_1,-\bm{k}_1-1,\ldots,\bm{k}_p,-\bm{k}_p-1,\bm{\ell}_1,-\bm{\ell}_1-1,\ldots,\bm{\ell}_p,-\bm{\ell}_p-1, \bm{g}_1,-\bm{g}_1-1,\ldots,\bm{g}_{q-p},-\bm{g}_{q-p}-1\}
\end{align*}
and are determined by the partition $\bm{\nu}^\ast$ under consideration. Using analogous arguments as for the calculations of $\lambda^q\, \cum_q[\hat{D}_1]$ and $\lambda^q \,\cum_q[\hat{D}_2]$, the order of the above expression is given by
\begin{align*}
&\Landau\Bigg(\lambda^{3q-2p+2\big(-\sum_{g=1}^{G^\ast} \#\{\text{different indices in $\underline{j}$ belonging to } \nu_g^{\ast} \}+(\#\{\text{sums over } \bm{g}_1,\ldots,\bm{\ell}_p\}-\#\text{restrictions}) + G^\ast\big)}\nonumber\\
&\phantom{============}\times \left(\frac{a^{2}}{\lambda^{2}}\right)^{p- \#\{ \text{rows belonging to $\hat{D}_2$ with either $\tilde{\bm{\omega}}_{\bm{k}}$ or $\tilde{\bm{\omega}}_{\bm{\ell}}$ unrestricted}\}}\Bigg)
\end{align*}
(compare to \eqref{general_order_q=2_sec_int} and note that due to indecomposability, there cannot be any row in Table \eqref{table} which belongs to the estimator $\hat{D}_2$ and has both $\tilde{\bm{\omega}}_{\bm{k}}$ and $\tilde{\bm{\omega}}_{\bm{\ell}}$ unrestricted). Note that since we assumed $G^\ast=2q$ and $i=4q$, we have
\begin{align*}
&\phantom{=i}\#\{\text{sums over } \bm{g}_1,\ldots,\bm{\ell}_p\}-\sum_{g=1}^{G^\ast} \#\{\text{different indices in $\underline{j}$ belonging to } \nu_g^{\ast} \}+G^\ast \\
&=q-p+2p -4q+2q = p-q.
\end{align*}
Therefore, the above order simplifies to
\begin{align*}
\Landau\left(\lambda^{q-2(\#\text{restrictions})} \times  \left(\frac{a^{2}}{\lambda^{2}}\right)^{p- \#\{ \text{rows belonging to $\hat{D}_2$ with either $\tilde{\bm{\omega}}_{\bm{k}}$ or $\tilde{\bm{\omega}}_{\bm{\ell}}$ unrestricted}\}}\right)
\end{align*}
(compare to Lemma \ref{lemma1}).
As in the proof of Theorem \ref{asymptotic_normality_secint}, it suffices to consider the case of $q-1$ restrictions (the minimal number of restrictions ensuring indecomposability in the case $G=2q$), since when imposing a further restriction, there is an order change of at most $\Landau(a^4/\lambda^6)=o(1)$.
In the case of $q-1$ restrictions, using the same arguments as in the proof of Theorem \ref{asymptotic_normality_secint}, we obtain 
\begin{align*}
p- \#\{ \text{rows belonging to $\hat{D}_2$ with either $\tilde{\bm{\omega}}_{\bm{k}}$ or $\tilde{\bm{\omega}}_{\bm{\ell}}$ unrestricted}\} \leq q-2:
\end{align*}
If $p\leq q-2$, this follows immediately, while for $p=q-1$, we follow the same argumentation as in the proof of Lemma \ref{lemma2} to see that at least one of the $q-1$ rows of Table \eqref{table} corresponding to the estimator $\hat{D}_2$ must have either $\tilde{\bm{\omega}}_{\bm{k}}$ or $\tilde{\bm{\omega}}_{\bm{\ell}}$ unrestricted (otherwise there would be more than $q-1$ restrictions in total, or the partition would not be indecomposable). Since setting two elements equal evokes an order change of $\lambda^2/n$, for the case of $G=2q$ groups we obtain the same order as in Proposition \ref{order_q=2_sec_int}:

\begin{prop} \label{order_q=2_mixed_cum}
For $q\geq 2$ and any indecomposable partition $\bm{\nu}=\{\nu_1,\ldots,\nu_{2q}\}\in\mathcal{I}(q)$ of Table \eqref{table}, we have
\begin{align*}
&\phantom{=i}\lambda^{3q-2p} \sum_{\bm{g}_1,\ldots,\bm{g}_{q-p}=-a}^{a-1} \sum_{\bm{k}_1,\ldots,\bm{k}_p=-a}^{a-1} \sum_{\bm{\ell}_1,\ldots,\bm{\ell}_p=-a}^{a-1} \Bigg[ \frac{1}{\lambda^p} \sum_{r_1,\ldots,r_p=0}^{a-1} \tilde{\omega}_{r_1} \ldots\tilde{\omega}_{r_p}  J_0(\tilde{\omega}_{r_1} \|\tilde{\bm{\omega}}_{\bm{k}_1}\|)J_0(\tilde{\omega}_{r_1} \|\tilde{\bm{\omega}}_{\bm{\ell}_1}\|)\times\\
&\phantom{===============================} \ldots \times J_0(\tilde{\omega}_{r_p} \|\tilde{\bm{\omega}}_{\bm{k}_p}\|)J_0(\tilde{\omega}_{r_p} \|\tilde{\bm{\omega}}_{\bm{\ell}_p}\|)\\
&\phantom{=iiii} \frac{1}{n^{4q}} \sum_{\underline{j}\in\mathcal{D}(q,i)} \cum_2[Y_{t,c}(\underline{j}) \, : \, (t,c)\in\nu_1] \times \ldots \times \cum_2[Y_{t,c}(\underline{j}) \, : \, (t,c)\in\nu_{2q}]\Bigg]\\
&=\begin{cases}
\Landau\left(\left(\frac{a^{2}}{\lambda^{3}}\right)^{q-2}\right), &\qquad \text{if } i=4q,\\
\Landau\left(\left(\frac{a^{2}}{\lambda^{3}}\right)^{q-2}\times \frac{\lambda^2}{n}\right), &\qquad \text{if } i\leq 4q-1.
\end{cases}
\end{align*}
\end{prop}

If $G<2q$, we obtain analogously to the proof of Theorem \ref{asymptotic_normality_secint} that the order of the respective cumulant expression is given by
\begin{align*}
&\phantom{=i}\Landau\left(\frac{\lambda^{q-2(\#\text{restrictions})}}{n^{2q-G}} \times \left(\frac{a^2}{\lambda^2}\right)^{p-\#\{ \text{rows belonging to $\hat{D}_2$ with either $\tilde{\bm{\omega}}_{\bm{k}}$ or $\tilde{\bm{\omega}}_{\bm{\ell}}$ (or both) unrestricted}\}}\right).
\end{align*}
We now show that the result from Proposition \ref{prop_sec_int} holds again:

\begin{prop} \label{prop_mix_cum}
For $q\geq 2$ and any indecomposable partition $\bm{\nu}=\{\nu_1,\ldots,\nu_{G}\}\in\mathcal{I}(q)$ of Table \eqref{table}, we have
\begin{align*}
&\lambda^{3q-2p} \sum_{\bm{g}_1,\ldots,\bm{g}_{q-p}=-a}^{a-1} \sum_{\bm{k}_1,\ldots,\bm{k}_p=-a}^{a-1} \sum_{\bm{\ell}_1,\ldots,\bm{\ell}_p=-a}^{a-1} \Bigg[ \frac{1}{\lambda^p} \sum_{r_1,\ldots,r_p=0}^{a-1} \tilde{\omega}_{r_1} \ldots\tilde{\omega}_{r_p} J_0(\tilde{\omega}_{r_1} \|\tilde{\bm{\omega}}_{\bm{k}_1}\|)J_0(\tilde{\omega}_{r_1} \|\tilde{\bm{\omega}}_{\bm{\ell}_1}\|)\times\\
&\phantom{=============================}\ldots \times J_0(\tilde{\omega}_{r_p} \|\tilde{\bm{\omega}}_{\bm{k}_p}\|)J_0(\tilde{\omega}_{r_p} \|\tilde{\bm{\omega}}_{\bm{\ell}_p}\|)\\
&
\phantom{===} \frac{1}{n^{4q}}  \sum_{i=4}^{2q+G} \sum_{\underline{j}\in\mathcal{D}(q,i)} \cum_{|\nu_1|}[Y_{t,c}(\underline{j}) \, : \, (t,c)\in\nu_1] \times \ldots \times \cum_{|\nu_G|}[Y_{t,c}(\underline{j}) \, : \, (t,c)\in\nu_G]\Bigg]\\
&=\begin{cases}
\Landau\left(\left(\frac{\lambda^2}{n}\right)^{q-1} \lambda^{2-q}\right), &\qquad \text{if } G=q+1 \text{ and } \#\text{restrictions}=0,\\
\Landau\left(\left(\frac{\lambda^2}{n}\right)^{q-s} \left(\frac{a^{2}}{\lambda^{3}}\right)^{q-2} \left(\frac{\lambda^{2}}{a^{2}}\right)^{q-s}\right), &\qquad \text{if } G=q+s, \, 2\leq s\leq q \text{ and } \#\text{restrictions}=s-1,\\
\Landau\left(\left(\frac{\lambda^2}{n}\right)^{q-s} \left(\frac{a^2}{\lambda^3}\right)^q\right), &\qquad \text{if } G=q+s, \, 1\leq s\leq q \text{ and } \#\text{restrictions}\geq s,\\
\Landau\left(\lambda^{-q}\right), &\qquad \text{if } G\leq q.
\end{cases}
\end{align*}
\end{prop}

\begin{proof}
We apply similar arguments as in the proof of Proposition \ref{prop_sec_int}. 
If $G=q+s$ for some $s\in\{1,\ldots,q\}$, in order to ensure indecomposability, there must be at least $s-1$ restrictions in frequency direction. In the case of exactly $s-1$ restrictions, it holds that
\begin{align} \label{anzahl_rows}
&\phantom{=i}\#\{ \text{rows belonging to } \hat{D}_1\}\nonumber\\
&+\#\{ \text{rows belonging to $\hat{D}_2$ with either $\tilde{\bm{\omega}}_{\bm{k}}$ or $\tilde{\bm{\omega}}_{\bm{\ell}}$ (or both) unrestricted}\} \geq \min\{q-s+2,q\}:
\end{align}
Without loss of generality, let $q-s+2\leq q$, since the result is obvious otherwise. As in the proof of Proposition \ref{prop_sec_int}, a typical indecomposable partition of Table \eqref{table} with a minimal number of restrictions is given by the case of one large set covering the first two (or last two) columns and $q-s+1$ rows, and sets of size $2$ otherwise. It is obvious that the rows covered by the large set either belong to the estimator $\hat{D}_1$ or have $\tilde{\bm{\omega}}_{\bm{k}}$ or $\tilde{\bm{\omega}}_{\bm{\ell}}$ unrestricted. Moreover, with the same arguments as in the proof of Proposition \ref{prop_sec_int}, at least one of the remaining $s-1$ rows must either belong to $\hat{D}_1$ or have $\tilde{\bm{\omega}}_{\bm{k}}$ or $\tilde{\bm{\omega}}_{\bm{\ell}}$ unrestricted, since otherwise there would be more than $s-1$ restrictions in total, or the partition would not be indecomposable. Observing 
\begin{align*}
p=q-\#\{\text{rows belonging to $\hat{D}_1$}\},
\end{align*}
\eqref{anzahl_rows} shows that
\begin{align*}
p-\#\{ \text{rows belonging to $\hat{D}_2$ with either $\tilde{\bm{\omega}}_{\bm{k}}$ or $\tilde{\bm{\omega}}_{\bm{\ell}}$ (or both) unrestricted}\} \leq \max\{0,s-2\}.
\end{align*}
Now, the result for the case $G=q+s$ with $1\leq s\leq q$ and a minimal number of restrictions in Table \eqref{table} follows with the same calculation as in the proof of Proposition \ref{prop_sec_int}. For the remaining cases, we can directly apply the same arguments as in the proof of Proposition \ref{prop_sec_int}. 
\end{proof}

Using the same arguments as in the proof of Corollary \ref{cor_sec_int}, Proposition \ref{prop_mix_cum} implies
\begin{align*}
\lambda^q \, \cum_{q-p,p}[\hat{D}_1,\hat{D}_2]=\Landau\left(\left(\frac{a^2}{\lambda^3}\right)^{q-2}\right) = \begin{cases}
\Landau(1), &\qquad q=2,\\
o(1), &\qquad q\geq 3.
\end{cases}
\end{align*}
Moreover, we see from Proposition \ref{order_q=2_mixed_cum} that if $q=2$, then the cumulant expression for indecomposable partitions consisting of $G=4$ groups with $i<8$ different elements in $\underline{j}$ is of order $\Landau(\lambda^2/n)$, while for $G<4$, Proposition \ref{prop_mix_cum} yields an order of $\Landau(\lambda^2/n+1/\lambda^2)$. 
\qed

\section{Proof of Proposition \ref{asymptotic_variances}}

\setcounter{equation}{0}

 Recall the notation $\tilde{\bm{\omega}}_{\bm{k},\lambda}=\tilde{\bm{\omega}}_{\bm{k}}$, $\bm{\omega}_{\bm{k},\lambda}=\bm{\omega}_{\bm{k}}$, $H_{d,h}(\bm{m})=H_d(\bm{m})$ and $B_{\lambda,d,h}(\bm{u})=B(\bm{u})$. 

{\it Proof of \eqref{(i)}} For the sake of brevity we will only prove the result in the one-dimensional case. The case $d=2$ can be shown exactly by the same arguments. We thus need to show that
\begin{align*}
&\sum_{m=-\infty}^{\infty} \left( \frac{8\, H_1(m)^2}{H_1(0)^2}+\frac{2\, H_1(m)^4}{H_1(0)^4}\right) \int_{-\infty}^{\infty} f^4(\omega) \, d\omega \\
&\phantom{====}- \sum_{m=-2a+1}^{2a-1} \left( \frac{8\, H_1(m)^2}{H_1(0)^2}+\frac{2\, H_1(m)^4}{H_1(0)^4}\right) \int_{2\pi \max(-a,-a-m)/\lambda}^{2\pi\min(a,a-m)/\lambda}    f^4(\omega)\, d\omega = o(1). 
\end{align*}
Using the triangle inequality, the absolute of the left hand side can be bounded by
\begin{align*}
I_1 + I_2,
\end{align*}
where $I_1$ and $I_2$ are defined by
\begin{align*}
I_1:=\sum_{|m|>2a-1} \left( \frac{8\, H_1(m)^2}{H_1(0)^2}+\frac{2\, H_1(m)^4}{H_1(0)^4}\right) \int_{-\infty}^{\infty} f^4(\omega)\, d\omega
\end{align*}
and
\begin{align*}
I_2:=\sum_{m=-2a+1}^{2a-1} \left( \frac{8\, H_1(m)^2}{H_1(0)^2}+\frac{2\, H_1(m)^4}{H_1(0)^4}\right) \Bigg|\int_{-\infty}^{\infty} f^4(\omega)\, d\omega - \int_{2\pi\max(-a,-a-m)/\lambda}^{2\pi\min(a,a-m)/\lambda} f^4(\omega)\, d\omega \Bigg|.
\end{align*}
Since by Assumption \ref{ass_on_h} (ii) it holds that $|H_1(m)|\lesssim 1/m^2$ for $m\neq 0$, we immediately obtain
\begin{align*}
I_1 = \Landau\left(\frac{1}{a^3}\right) \qquad \text{as } a\to\infty.
\end{align*}
For the term $I_2$, we get
\begin{align*}
I_2 &\leq \sum_{|m|\leq a/2} \left( \frac{8\, H_1(m)^2}{H_1(0)^2}+\frac{2\, H_1(m)^4}{H_1(0)^4}\right) \Bigg|\int_{-\infty}^{2\pi\max(-a,-a-m)/\lambda} f^4(\omega)\, d\omega \\
&\phantom{===================}+ \int_{2\pi\min(a,a-m)/\lambda}^{\infty} f^4(\omega)\, d\omega\Bigg|\\
&+ \sum_{|m|>a/2} \left( \frac{8\, H_1(m)^2}{H_1(0)^2}+\frac{2\, H_1(m)^4}{H_1(0)^4}\right) \Bigg|\int_{-\infty}^{\infty} f^4(\omega)\, d\omega - \int_{2\pi\max(-a,-a-m)/\lambda}^{2\pi\min(a,a-m)/\lambda} f^4(\omega)\, d\omega \Bigg|\\
& \lesssim \int_{-\infty}^{-\pi a/\lambda} f^4(\omega)\, d\omega + \int_{\pi a/\lambda}^{\infty} f^4(\omega)\, d\omega + \sum_{|m|>a/2}  \left( \frac{8\, H_1(m)^2}{H_1(0)^2}+\frac{2\, H_1(m)^4}{H_1(0)^4}\right)\\
&\lesssim \int_{\pi a/\lambda}^{\infty} \frac{1}{\omega^{4+4\delta}} \, d\omega + \sum_{|m|>a/2} \left( \frac{8\, H_1(m)^2}{H_1(0)^2}+\frac{2\, H_1(m)^4}{H_1(0)^4}\right)\\
&= \Landau\left(\left(\frac{\lambda}{a}\right)^{3+4\delta}+\frac{1}{a^3}\right) \qquad \text{as } \lambda,a\to\infty,
\end{align*}
where we again used $|H_1(m)|\lesssim 1/m^2$ for $m\neq 0$ and that by Assumption \ref{assumption_on_Z}, $f(\omega)\lesssim |\omega|^{-\delta-1}$ for $|\omega|>1$ and some $\delta>0$. 
Since $\lambda,a\to\infty$ with $a/\lambda\to\infty$, this yields the claim of part (i).
\medskip

{\it Proof of \eqref{(ii)}}  Note that 
\begin{align*}
&\phantom{i=}\sum_{\bm{m}=-\infty}^{\infty} H_2(\bm{m})^2 \int_{\R^2} f^2(\bm{z})\Bigg(\int_{0}^{\infty} \Bigg[\int_{\R^2} f(\bm{x})  J_0(r\|\bm{x}\|) J_0(r \|\bm{z}\|)\, d\bm{x}\Bigg] \, r\, dr\Bigg)^2  \, d\bm{z} 
\\&- \sum_{\bm{m}=-2a+1}^{2a-1} H_2(\bm{m})^2 \int_{2\pi \max(-a,-a-\bm{m})/\lambda}^{2\pi\min(a,a-\bm{m})/\lambda} f^2(\bm{z})\\
&\phantom{=========} \Bigg(\int_{0}^{2\pi a /\lambda} \Bigg[\int_{[-2\pi a/\lambda,2\pi a/\lambda]^2} f(\bm{x})  J_0(r\|\bm{x}\|) J_0(r \|\bm{z}\|)\, d\bm{x}\Bigg] \, r\, dr\Bigg)^2  \, d\bm{z}\\
&=K_1+K_2+K_3+K_4,
\end{align*}
where the quantities $K_j$ are defined by
\begin{align*}
K_1&:=\sum_{\bm{m}=-\infty}^{\infty} H_2(\bm{m})^2 \int_{\R^2} f^2(\bm{z})\Bigg(\int_{0}^{\infty} \Bigg[\int_{\R^2} f(\bm{x})  J_0(r\|\bm{x}\|) J_0(r \|\bm{z}\|)\, d\bm{x}\Bigg] \, r\, dr\Bigg)^2  \, d\bm{z} \\
&-\sum_{\bm{m}=-2a+1}^{2a-1} H_2(\bm{m})^2 \int_{\R^2} f^2(\bm{z})\Bigg(\int_{0}^{\infty} \Bigg[\int_{\R^2} f(\bm{x})  J_0(r\|\bm{x}\|)  J_0(r \|\bm{z}\|)\, d\bm{x}\Bigg] \, r\, dr\Bigg)^2  \, d\bm{z},
\end{align*}
\begin{align*}
K_2&:=\sum_{\bm{m}=-2a+1}^{2a-1} H_2(\bm{m})^2 \int_{\R^2} f^2(\bm{z})\Bigg(\int_{0}^{\infty} \Bigg[\int_{\R^2} f(\bm{x})  J_0(r\|\bm{x}\|) J_0(r \|\bm{z}\|)\, d\bm{x}\Bigg] \, r\, dr\Bigg)^2  \, d\bm{z}\\
&-\sum_{\bm{m}=-2a+1}^{2a-1} H_2(\bm{m})^2 \int_{2\pi\max(-a,-a-\bm{m})/\lambda}^{2\pi\min(a,a-\bm{m})/\lambda} f^2(\bm{z})\\
&\phantom{===================}\Bigg(\int_{0}^{\infty} \Bigg[\int_{\R^2} f(\bm{x})  J_0(r\|\bm{x}\|)  J_0(r \|\bm{z}\|)\, d\bm{x}\Bigg] \, r\, dr\Bigg)^2  \, d\bm{z},
\end{align*}
\begin{align*}
K_3&:=\sum_{\bm{m}=-2a+1}^{2a-1} H_2(\bm{m})^2 \int_{2\pi\max(-a,-a-\bm{m})/\lambda}^{2\pi\min(a,a-\bm{m})/\lambda} f^2(\bm{z})\\
&\phantom{===================}\Bigg(\int_{0}^{\infty} \Bigg[\int_{\R^2} f(\bm{x})  J_0(r\|\bm{x}\|)  J_0(r \|\bm{z}\|)\, d\bm{x}\Bigg] \, r\, dr\Bigg)^2  \, d\bm{z}\\
&-\sum_{\bm{m}=-2a+1}^{2a-1} H_2(\bm{m})^2 \int_{2\pi\max(-a,-a-\bm{m})/\lambda}^{2\pi\min(a,a-\bm{m})/\lambda} f^2(\bm{z}) \\
&\phantom{===================}\Bigg(\int_{0}^{2\pi a/\lambda} \Bigg[\int_{\R^2} f(\bm{x})  J_0(r\|\bm{x}\|)  J_0(r \|\bm{z}\|)\, d\bm{x}\Bigg] \, r\, dr\Bigg)^2  \, d\bm{z} ,
\end{align*}
and
\begin{align*}
K_4&:=\sum_{\bm{m}=-2a+1}^{2a-1} H_2(\bm{m})^2 \int_{2\pi\max(-a,-a-\bm{m})/\lambda}^{2\pi\min(a,a-\bm{m})/\lambda} f^2(\bm{z})\\
&\phantom{=============} \Bigg(\int_{0}^{2\pi a/\lambda} \Bigg[\int_{\R^2} f(\bm{x})  J_0(r\|\bm{x}\|)  J_0(r \|\bm{z}\|)\, d\bm{x}\Bigg] \, r\, dr\Bigg)^2  \, d\bm{z} \\
&-\sum_{\bm{m}=-2a+1}^{2a-1} H_2(\bm{m})^2 \int_{2\pi\max(-a,-a-\bm{m})/\lambda}^{2\pi\min(a,a-\bm{m})/\lambda} f^2(\bm{z}) \\
&\phantom{=============}\Bigg(\int_{0}^{2\pi a/\lambda} \Bigg[\int_{[-2\pi a/\lambda,2\pi a/\lambda]^2} f(\bm{x})  J_0(r\|\bm{x}\|)  J_0(r \|\bm{z}\|)\, d\bm{x}\Bigg] \, r\, dr\Bigg)^2  \, d\bm{z}.
\end{align*}
We now have
\begin{align*}
|K_1| &\lesssim \sum_{i=1}^2 \sum_{\substack{m_j=-\infty\\ j\in\{1,\ldots,i-1\}}}^{\infty} \sum_{\substack{m_k=-2a+1\\ k\in\{i+1,\ldots,2\}}}^{2a-1} \sum_{|m_i|>2a-1} H_2(m_1,m_2)^2 \int_{\R^2} f^2(\bm{z})\, d\bm{z}\\
&\phantom{==================} \Bigg(\int_{0}^{\infty} \Bigg|\int_{\R^2} f(\bm{x}) J_0(r\|\bm{x}\|)\, d\bm{x} \Bigg| \, r\, dr\Bigg)^2.
\end{align*}
By definition of the Bessel function and the covariance function, it holds that
\begin{align} \label{finite}
\int_{0}^{\infty} \Bigg|\int_{\R^2} f(\bm{x}) J_0(r\|\bm{x}\|) d\bm{x} \Bigg|  \, r\, dr \lesssim \int_{\R^2} |c(\bm{x})|\, d\bm{x} \leq C
\end{align}
for some generic constant $C>0$. 
Since 
$|H_1(m)|\lesssim 1/m^2$ for $m\neq 0$, this yields
\begin{align*}
|K_1| \lesssim \sum_{i=1}^2 \sum_{\substack{m_j=-\infty\\ j\in\{1,\ldots,i-1\}}}^{\infty} \sum_{\substack{m_k=-2a+1\\ k\in\{i+1,\ldots,2\}}}^{2a-1} \sum_{|m_i|>2a-1} H_2(m_1,m_2)^2 = \Landau\left(\frac{1}{a^3}\right).
\end{align*}
Concerning the term $K_2$, note that
\begin{align*}
|K_2|&\leq K_{21}+K_{22},
\end{align*}
where $K_{21}$ and $K_{22}$ are defined by
\begin{align*}
K_{21}&:= \sum_{\bm{m}=-2a+1}^{2a-1} H_2(\bm{m})^2 \Bigg|\int_{-\infty}^{\infty} \Bigg\{\int_{-\infty}^{\infty} f^2(\bm{z}) \Bigg(\int_{0}^{\infty} \Bigg[\int_{\R^2} f(\bm{x}) J_0(r\|\bm{x}\|)\, d\bm{x} \Bigg] J_0(r\|\bm{z}\|) \, r\, dr\Bigg)^2 dz_2\\
&\phantom{==}-\int_{2\pi\max(-a,-a-m_2)/\lambda}^{2\pi\min(a,a-m_2)/\lambda} f^2(\bm{z}) \Bigg(\int_{0}^{\infty} \Bigg[\int_{\R^2} f(\bm{x}) J_0(r\|\bm{x}\|)\, d\bm{x} \Bigg] J_0(r\|\bm{z}\|) \, r\, dr\Bigg)^2 dz_2\Bigg\}\, dz_1\, \Bigg|
\end{align*}
and
\begin{align*}
K_{22}&:=\sum_{\bm{m}=-2a+1}^{2a-1} H_2(\bm{m})^2 \Bigg|\int_{2\pi\max(-a,-a-m_2)/\lambda}^{2\pi\min(a,a-m_2)/\lambda} \Bigg\{\int_{-\infty}^{\infty} f^2(\bm{z})\\
&\phantom{====================} \Bigg(\int_{0}^{\infty} \Bigg[\int_{\R^2} f(\bm{x}) J_0(r\|\bm{x}\|) \, d\bm{x} \Bigg] J_0(r\|\bm{z}\|) \, r\, dr\Bigg)^2 dz_1\\
&\phantom{==}-\int_{2\pi\max(-a,-a-m_1)/\lambda}^{2\pi\min(a,a-m_1)/\lambda} f^2(\bm{z}) \Bigg(\int_{0}^{\infty} \Bigg[\int_{\R^2} f(\bm{x}) J_0(r\|\bm{x}\|)\, d\bm{x} \Bigg] J_0(r\|\bm{z}\|) \, r\, dr\Bigg)^2 dz_1\Bigg\} \,dz_2\, \Bigg|.
\end{align*}
For the sake of brevity we only consider the term $K_{21}$, since the second term $K_{22}$ can be treated exactly in the same way.
\eqref{finite} yields
\begin{align*}
K_{21}
&\lesssim \sum_{|m_2|\leq a/2} H_1(m_2)^2 \int_{-\infty}^{\infty} \Bigg[\int_{-\infty}^{2\pi\max(-a,-a-m_2)/\lambda} f^2(\bm{z}) \,dz_2 + \int_{2\pi\min(a,a-m_2)/\lambda}^{\infty} f^2(\bm{z})\, dz_2 \Bigg] \,dz_1\\
&+\sum_{|m_2|>a/2} H_1(m_2)^2 \int_{-\infty}^{\infty} \Bigg[\int_{-\infty}^{2\pi\max(-a,-a-m_2)/\lambda} f^2(\bm{z}) \,dz_2 + \int_{2\pi\min(a,a-m_2)/\lambda}^{\infty} f^2(\bm{z})\, dz_2 \Bigg] \,dz_1\\
&\lesssim  \int_{-\infty}^{\infty} \Bigg[\int_{-\infty}^{-\pi a/\lambda} f^2(\bm{z})\, dz_2 + \int_{\pi a/\lambda}^{\infty} f^2(\bm{z})\, dz_2 \Bigg]\, dz_1 + \sum_{|m_2|>a/2} H_1(m_2)^2\\
&= \Landau\left(\left(\frac{\lambda}{a}\right)^{1+2\delta} + \frac{1}{a^3}\right),
\end{align*}
where we used that by Assumption \ref{assumption_on_Z} $f(\bm{z})\lesssim \beta_{1+\delta}(\bm{z})$ for some $\delta>0$ and $\beta_{\delta}$ defined in \eqref{beta_fct}. 
Using $a^2-b^2=(a-b)(a+b)$ and \eqref{finite}, we furthermore obtain
\begin{align*}
|K_3|& \lesssim \sum_{\bm{m}=-2a+1}^{2a-1} H_2(\bm{m})^2 \int_{2\pi\max(-a,-a-\bm{m})/\lambda}^{2\pi\min(a,a-\bm{m})/\lambda} f^2(\bm{z}) \left(\int_{2\pi a/\lambda}^{\infty} \left| \int_{\R^2} f(\bm{x}) J_0(r\|\bm{x}\|)\, d\bm{x} \right| \, r\, dr\right)\, d\bm{z}\\
&\lesssim \int_{2\pi a/\lambda}^{\infty} \int_{0}^{2\pi} \left| c(r\cos \theta,r \sin \theta) \right| \, r\,d\theta\, dr\lesssim \int_{2\pi a/\lambda}^{\infty}  \frac{1}{r^{1+\varepsilon}}\, dr =  \Landau\left(\left(\frac{\lambda}{a}\right)^{\varepsilon}\right).
\end{align*}
It remains to deal with the term $K_4$ and similarly to above, we obtain
\begin{align*}
|K_4|
&\lesssim \Bigg(\int_{0}^{2\pi a/\lambda} \Bigg| \int_{\R^2} f(\bm{x}) J_0(r\|\bm{x}\|)\, d\bm{x} \Bigg| \, r\, dr \\
&\phantom{==================}+ \int_{0}^{2\pi a/\lambda} \Bigg|\int_{[-2\pi a/\lambda,2\pi a/\lambda]^2} f(\bm{x}) J_0(r\|\bm{x}\|)\, d\bm{x} \Bigg|\, r\, dr\Bigg)\\
&\phantom{====}\times\Bigg( \int_{0}^{2\pi a/\lambda} \Bigg|\int_{([-2\pi a/\lambda,2\pi a/\lambda]^2)^c} f(\bm{x}) J_0(r\|\bm{x}\|)\, d\bm{x}\Bigg|\, r\, dr\Bigg).
\end{align*}
Note that
\begin{align*}
&= \int_{0}^{2\pi a/\lambda} \Bigg|\int_{([-2\pi a/\lambda,2\pi a/\lambda]^2)^c} f(\bm{x}) J_0(r\|\bm{x}\|)\, d\bm{x}\Bigg|\, r\, dr\\
&\lesssim \int_{\|\bm{y}\|\leq 2\pi a/\lambda} \Bigg( \int_{([-2\pi a/\lambda,2\pi a/\lambda]^2)^c} \beta_{1+\delta}(\bm{x})\, d\bm{x} \Bigg) \, d\bm{y}
=\Landau\left(\frac{a^2}{\lambda^2} \left( \int_{2\pi a/\lambda}^{\infty} \frac{1}{x^{1+\delta}}\, dx \right) \right) =\Landau\left(\frac{\lambda^{\delta-2}}{a^{\delta-2}}\right),
\end{align*}
which (in combination with \eqref{finite} and the assumption $\delta>2$) yields
\begin{align*}
K_4 = \Landau\left(\frac{\lambda^{\delta-2}}{a^{\delta-2}}\right)
\end{align*}
as $a,\lambda\to\infty$. Combining this estimate with the estimates for the terms $K_1$, $K_2$ and $K_3$ yields
\begin{align*}
\tau_2^2 = \tau_{2,\lambda,a}^2 + \Landau\left(\frac{1}{a^3}+\left(\frac{\lambda}{a}\right)^{1+2\delta}+\left(\frac{\lambda}{a}\right)^{\varepsilon} + \left(\frac{\lambda}{a}\right)^{\delta-2}\right),
\end{align*}
which proves assertion (ii).
\medskip

{\it Proof of \eqref{(iii)}}  Note that
\begin{align*}
&\phantom{i=}\sum_{\bm{m}=-\infty}^{\infty} H_2(\bm{m})^2 \int_{0}^{\infty} r \Bigg[\int_{\R^2} f(\bm{y}) J_0(r\|\bm{y}\|)\, d\bm{y} \int_{\R^2} f(\bm{x})^3 J_0(r\|\bm{x}\|)\, d\bm{x}\Bigg] \, dr\\
&-\sum_{\bm{m}=-2a+1}^{2a-1} H_2(\bm{m})^2 \int_{0}^{2\pi a/\lambda} r \Bigg[\int_{[-2\pi a/\lambda,2\pi a/\lambda]^2} f(\bm{y}) J_0(r\|\bm{y}\|)\, d\bm{y}\\
&\phantom{========================}\times \int_{2\pi\max(-a,-a-\bm{m})/\lambda}^{2\pi\min(a,a-\bm{m})/\lambda} f(\bm{x})^3 J_0(r\|\bm{x}\|)\, d\bm{x} \Bigg] \, dr\\
&=L_1+L_2+L_3+L_4,
\end{align*}
where the terms $L_j$ are defined by
\begin{align*}
L_1&:=\sum_{\bm{m}=-\infty}^{\infty} H_2(\bm{m})^2 \int_{0}^{\infty} r \Bigg[\int_{\R^2} f(\bm{y}) J_0(r\|\bm{y}\|)\, d\bm{y} \int_{\R^2} f(\bm{x})^3 J_0(r\|\bm{x}\|) \,d\bm{x} \Bigg]\, dr\\
&-\sum_{\bm{m}=-2a+1}^{2a-1} H_2(\bm{m})^2 \int_{0}^{\infty} r \Bigg[\int_{\R^2} f(\bm{y}) J_0(r\|\bm{y}\|) \,d\bm{y} \int_{\R^2} f(\bm{x})^3 J_0(r\|\bm{x}\|)\, d\bm{x} \Bigg]\, dr,
\end{align*}
\begin{align*}
L_2&:=\sum_{\bm{m}=-2a+1}^{2a-1} H_2(\bm{m})^2 \int_{0}^{\infty} r \Bigg[\int_{\R^2} f(\bm{y}) J_0(r\|\bm{y}\|) \,d\bm{y} \int_{\R^2} f(\bm{x})^3 J_0(r\|\bm{x}\|)\, d\bm{x} \Bigg]\, dr\\
&-\sum_{\bm{m}=-2a+1}^{2a-1} H_2(\bm{m})^2 \int_{0}^{2\pi a/\lambda} r \Bigg[\int_{\R^2} f(\bm{y}) J_0(r\|\bm{y}\|)\, d\bm{y} \int_{\R^2} f(\bm{x})^3 J_0(r\|\bm{x}\|)\, d\bm{x} \Bigg]\, dr,
\end{align*}
\begin{align*}
L_3&:=\sum_{\bm{m}=-2a+1}^{2a-1} H_2(\bm{m})^2 \int_{0}^{2\pi a/\lambda} r \Bigg[\int_{\R^2} f(\bm{y}) J_0(r\|\bm{y}\|) \,d\bm{y} \int_{\R^2} f(\bm{x})^3 J_0(r\|\bm{x}\|)\, d\bm{x} \Bigg]\, dr \\
&-\sum_{\bm{m}=-2a+1}^{2a-1} H_2(\bm{m})^2 \int_{0}^{2\pi a/\lambda} r \Bigg[\int_{[-2\pi a/\lambda,2\pi a/\lambda]^2} f(\bm{y}) J_0(r\|\bm{y}\|)\, d\bm{y} \int_{\R^2} f(\bm{x})^3 J_0(r\|\bm{x}\|)\, d\bm{x} \Bigg]\, dr
\end{align*}
and
\begin{align*}
L_4&:=\sum_{\bm{m}=-2a+1}^{2a-1} H_2(\bm{m})^2 \int_{0}^{2\pi a/\lambda} r \Bigg[\int_{[-2\pi a/\lambda,2\pi a/\lambda]^2} f(\bm{y}) J_0(r\|\bm{y}\|) \,d\bm{y} \int_{\R^2} f(\bm{x})^3 J_0(r\|\bm{x}\|)\, d\bm{x} \Bigg]\, dr\\
&-\sum_{\bm{m}=-2a+1}^{2a-1} H_2(\bm{m})^2 \int_{0}^{2\pi a/\lambda} r \Bigg[\int_{[-2\pi a/\lambda,2\pi a/\lambda]^2} f(\bm{y}) J_0(r\|\bm{y}\|) \,d\bm{y}\\
&\phantom{===============} \int_{2\pi\max(-a,-a-\bm{m})/\lambda}^{2\pi\min(a,a-\bm{m})/\lambda} f(\bm{x})^3 J_0(r\|\bm{x}\|)\, d\bm{x} \Bigg]\, dr.
\end{align*}
Using exactly the same arguments as in the proof of part (ii), we get
\begin{align*}
|L_1|
&\lesssim \sum_{i=1}^2 \sum_{\substack{m_j=-\infty\\ j\in\{1,\ldots,i-1\}}}^{\infty} \sum_{\substack{m_k=-2a+1\\ k\in\{i+1,\ldots,2\}}}^{2a-1} \sum_{|m_i|>2a-1} H_2(m_1,m_2)^2=\Landau\left(\frac{1}{a^3}\right).
\end{align*}
Moreover, analogously to the proof of part (ii) it follows that
\begin{align*}
|L_2|&\lesssim \sum_{\bm{m}=-2a+1}^{2a-1} H_2(\bm{m})^2 \int_{2\pi a/\lambda}^{\infty} r \Bigg|\int_{\R^2} f(\bm{y}) J_0(r\|\bm{y}\|)\, d\bm{y} \Bigg|\, dr\\
&\lesssim \int_{2\pi a/\lambda}^{\infty} \int_{0}^{2\pi} |c(r\cos \theta,r\sin \theta)|\, r\, d\theta\,dr = \Landau\left(\left(\frac{\lambda}{a}\right)^{\varepsilon}\right)
\end{align*}
and
\begin{align*}
|L_3|&\lesssim \sum_{\bm{m}=-2a+1}^{2a-1} H_2(\bm{m})^2 \int_{0}^{2\pi a/\lambda} r \left|\int_{([-2\pi a/\lambda,2\pi a/\lambda]^2)^c} f(\bm{y}) J_0(r\|\bm{y}\|)\, d\bm{y} \right|\, dr\\
&\lesssim \int_{0}^{2\pi a/\lambda} r \left(\int_{([-2\pi a/\lambda,2\pi a/\lambda]^2)^c} f(\bm{y})\, d\bm{y}\right)\, dr = \Landau\left(\left(\frac{\lambda}{a}\right)^{\delta-2}\right).
\end{align*}
Finally, again using \eqref{finite},
\begin{align*}
|L_4|
&\lesssim \sup_{0\leq r\leq 2\pi a/\lambda} \sum_{m_2=-2a+1}^{2a-1} H_1(m_2)^2  \Bigg|\int_{-\infty}^{\infty} \Bigg(\int_{-\infty}^{\infty} f(x_1,x_2)^3 J_0(r\|\bm{x}\|) \, dx_2 \\
&\phantom{===================}- \int_{2\pi\max(-a,-a-m_2)/\lambda}^{2\pi\min(a,a-m_2)/\lambda} f(x_1,x_2)^3 J_0(r\|\bm{x}\|) \, dx_2 \Bigg) \, dx_1 \Bigg|\\
&+\sup_{0\leq r\leq 2\pi a/\lambda} \sum_{\bm{m}=-2a+1}^{2a-1} H_2(\bm{m})^2 \Bigg| \int_{2\pi\max(-a,-a-m_2)/\lambda}^{2\pi\min(a,a-m_2)/\lambda} \Bigg(\int_{-\infty}^{\infty} f(x_1,x_2)^3 J_0(r\|\bm{x}\|)\, dx_1 \\
&\phantom{===================}- \int_{2\pi\max(-a,-a-m_1)/\lambda}^{2\pi\min(a,a-m_1)/\lambda} f(x_1,x_2)^3 J_0(r\|\bm{x}\|)\, dx_1 \Bigg)\, dx_2 \Bigg|.
\end{align*}
Both summands can be dealt with analogously and we only consider the first one. It is bounded by
\begin{align*}
&\phantom{=i}\sup_{0\leq r\leq 2\pi a/\lambda}\sum_{|m_2|\leq a/2} H_1(m_2)^2 \Bigg|\int_{-\infty}^{\infty} \Bigg(\int_{-\infty}^{2\pi\max(-a,-a-m_2)/\lambda} f(x_1,x_2)^3 J_0(r\|\bm{x}\|)\, dx_2 \\
&\phantom{================}+ \int_{2\pi\min(a,a-m_2)/\lambda}^{\infty} f(x_1,x_2)^3 J_0(r\|\bm{x}\|)\, dx_2 \Bigg)\, dx_1\Bigg|\\
&+\sup_{0\leq r\leq 2\pi a/\lambda}\sum_{|m_2|> a/2} H_1(m_2)^2 \Bigg|\int_{-\infty}^{\infty} \Bigg(\int_{-\infty}^{2\pi\max(-a,-a-m_2)/\lambda} f(x_1,x_2)^3 J_0(r\|\bm{x}\|)\, dx_2 \\
&\phantom{================}+ \int_{2\pi\min(a,a-m_2)/\lambda}^{\infty} f(x_1,x_2)^3 J_0(r\|\bm{x}\|)\, dx_2 \Bigg)\, dx_1\Bigg|\\
& \lesssim \sup_{0\leq r\leq 2\pi a/\lambda} \Bigg|\int_{-\infty}^{\infty} \Bigg(\int_{-\infty}^{-\pi a/\lambda} f(x_1,x_2)^3 J_0(r\|\bm{x}\|)\, dx_2 + \int_{\pi a/\lambda}^{\infty} f(x_1,x_2)^3 J_0(r\|\bm{x}\|)\, dx_2\Bigg)\, dx_1 \Bigg|\\
&+\sum_{|m_2|>a/2} H_1(m_2)^2=\Landau\left(\left(\frac{\lambda}{a}\right)^{2+3\delta} + \frac{1}{a^3}\right),
\end{align*}
where we again used $f(\bm{x})\lesssim \beta_{1+\delta}(\bm{x})$ for some $\delta>0$. 
We thus obtain
\begin{align*}
\kappa_{1,2}=\tau_{1,2,\lambda,a}+\Landau\left(\frac{1}{a^3}+\left(\frac{\lambda}{a}\right)^{\varepsilon} + \left(\frac{\lambda}{a}\right)^{\delta-2} + \left(\frac{\lambda}{a}\right)^{2+3\delta}\right),
\end{align*}
which proves part \eqref{(iii)} of Proposition \ref{asymptotic_variances}.
\medskip

{\it Proof of \eqref{asymptotic_variance}}
Define the \textit{Hankel transform (of order 0)} of a function $g:\R^{+}\to\R$ as \label{page162}
\begin{align*}
(\mathcal{H}g)(x):=\int_{0}^{\infty} g(r) \, J_0(rx) \,r\, dr.
\end{align*}
After standardization with an appropriate constant, the Hankel transform of a radial-symmetric function equals its Fourier transform [see e.g. \cite{bracewell65}]. For illustration, 
let $f:\R^2\to\R$ be even and radial-symmetric, i.e. assume that $f(\bm{\omega})=f_0(\|\bm{\omega}\|)$ for some function $f_0:\R^{+}\to\R$. In this case, the Fourier transform of $f$ is radial-symmetric as well, i.e. $(\mathcal{F}f)(\bm{x})=(\mathcal{F}f)_0(\|\bm{x}\|)$ for some $(\mathcal{F}f)_0:\R^{+}\to\R$. \
Given any $x\in\R$ and $\alpha\in[0,2\pi)$, the Hankel transform of $f_0$ thus equals
\begin{align*}
(\mathcal{H} f_0)(x) 
&= \frac{1}{2\pi} \int_{0}^{\infty} \int_{0}^{2\pi} f_0(r) \exp(\im r x \cos \theta)\, r\, d\theta\, dr\nonumber\\
&=\frac{1}{2\pi} \int_{0}^{\infty} \int_{0}^{2\pi} f(r \cos \theta, r \sin \theta) \exp(\im r x \cos(\alpha-\theta))\,r\, d\theta\, dr\nonumber\\
&=\frac{1}{2\pi} \int_{\R^2} f(\bm{y}) \exp\Bigg(\im \bm{y}^T \begin{pmatrix}
x \cos \alpha\\
x \sin \alpha
\end{pmatrix}\Bigg)\, d\bm{y}\nonumber\\
&=\frac{1}{2\pi} (\mathcal{F}f)(x\cos \alpha, x\sin \alpha) = \frac{1}{2\pi} (\mathcal{F} f)_0(x).
\end{align*}
Using exactly the same arguments, we obtain
\begin{align*} 
(\mathcal{H} (\mathcal{F} f)_0)(x) = \frac{1}{2\pi} \int_{\R^2} (\mathcal{F}f)(\bm{y}) \exp\Bigg(\im \bm{y}^T \begin{pmatrix}
x \cos \alpha\\
x \sin \alpha
\end{pmatrix}\Bigg) d\bm{y} = 2\pi f(x \cos \alpha, x \sin \alpha) = 2\pi f_0(x).
\end{align*}
Using these equalities, we can now calculate the asymptotic variance $\tau_2^2$ and the asymptotic covariance $\kappa_{1,2}$ under the assumption that $f(\bm{\omega})=f_0(\|\bm{\omega}\|)$ for all $\bm{\omega}\in\R^2$. In this case, it holds that
\begin{align*}
\tau_2^2
&=8 \sum_{\bm{m}=-\infty}^{\infty} \frac{H_2(\bm{m})^2}{H_2(\bm{0})^2} \int_{0}^{\infty} \int_{0}^{2\pi} f^2(s_1 \cos \theta_1, s_1 \sin \theta_1)\\ &\phantom{========}\Bigg(\int_{0}^{\infty} \Bigg[\int_{0}^{\infty} \int_{0}^{2\pi} f(s_2 \cos \theta_2, s_2 \sin \theta_2) J_0(rs_2) \, s_2 \, d\theta_2 \, ds_2 \Bigg] J_0(rs_1) \, r\, dr\Bigg)^2 s_1\, d\theta_1 \, ds_1\\
&=8\,(2\pi)^3  \sum_{\bm{m}=-\infty}^{\infty} \frac{H_2(\bm{m})^2}{H_2(\bm{0})^2} \int_{0}^{\infty} f_0^2(s_1) \Bigg(\int_{0}^{\infty} (\mathcal{H}f_0)(r) J_0(rs_1)\, r\, dr\Bigg)^2 s_1\, ds_1\\
&=8\,(2\pi)^3  \sum_{\bm{m}=-\infty}^{\infty} \frac{H_2(\bm{m})^2}{H_2(\bm{0})^2} \int_{0}^{\infty} f_0^2(s_1) \Bigg(\frac{1}{2\pi}\int_{0}^{\infty} (\mathcal{F}f)_0(r) J_0(rs_1)\, r\, dr\Bigg)^2 s_1 \, ds_1\\
&=8\, (2\pi) \sum_{\bm{m}=-\infty}^{\infty} \frac{H_2(\bm{m})^2}{H_2(\bm{0})^2} \int_{0}^{\infty} f_0^2(s_1) \big[(\mathcal{H} (\mathcal{F}f)_0)(s_1) \big]^2 s_1\, ds_1\\
&=8 \, (2\pi) \sum_{\bm{m}=-\infty}^{\infty} \frac{H_2(\bm{m})^2}{H_2(\bm{0})^2} \int_{0}^{\infty} f_0^2(s_1) \big[ 2\pi f_0(s_1)\big]^2 s_1\, ds_1\\
&=8\, (2\pi)^3 \sum_{\bm{m}=-\infty}^{\infty} \frac{H_2(\bm{m})^2}{H_2(\bm{0})^2} \int_{0}^{\infty} f_0^4(s_1)\, s_1\, ds_1 = 8\, (2\pi)^2 \sum_{\bm{m}=-\infty}^{\infty} \frac{H_2(\bm{m})^2}{H_2(\bm{0})^2} \int_{\R^2} f^4(\bm{\omega})\, d\bm{\omega}
\end{align*}
and
\begin{align*}
\kappa_{1,2} 
&=16\pi \sum_{\bm{m}=-\infty}^{\infty} \frac{H_2(\bm{m})^2}{H_2(\bm{0})^2} \int_{0}^{\infty} r \Bigg[\int_{0}^{\infty} \int_{0}^{2\pi} f_0(s_1) J_0(rs_1)\, s_1\, d\theta_1 \, ds_1\int_{0}^{\infty} \int_{0}^{2\pi} f_0^3(s_2) J_0(rs_2)\, s_2\, d\theta_2\, ds_2\Bigg]\, dr\\
&=16\pi\, (2\pi)^2 \sum_{\bm{m}=-\infty}^{\infty} \frac{H_2(\bm{m})^2}{H_2(\bm{0})^2} \int_{0}^{\infty} r \, (\mathcal{H} f_0)(r)\, (\mathcal{H}f_0^3)(r)\, dr\\
&=16\pi \sum_{\bm{m}=-\infty}^{\infty} \frac{H_2(\bm{m})^2}{H_2(\bm{0})^2}\int_{0}^{\infty} r \, (\mathcal{F}f)_0(r)\, (\mathcal{F}f^3)_0(r)\, dr\\
&= 16\pi \sum_{\bm{m}=-\infty}^{\infty} \frac{H_2(\bm{m})^2}{H_2(\bm{0})^2} \times \frac{1}{2\pi} \int_{\R^2} (\mathcal{F}f)(\bm{\omega}) (\mathcal{F}f^3)(\bm{\omega})\, d\bm{\omega}\\
&=16\pi \sum_{\bm{m}=-\infty}^{\infty} \frac{H_2(\bm{m})^2}{H_2(\bm{0})^2} \times \frac{1}{2\pi} \times (2\pi)^2 \int_{\R^2} f(\bm{\omega}) f^3(\bm{\omega})\, d\bm{\omega}\\
&=8\, (2\pi)^2 \sum_{\bm{m}=-\infty}^{\infty} \frac{H_2(\bm{m})^2}{H_2(\bm{0})^2} \int_{\R^2} f^4(\bm{\omega})\, d\bm{\omega},
\end{align*}
where in the last step we used the theorem of Plancherel. Observing \eqref{asymptotic_variance}, we thus obtain
\begin{align*}
\tau_{\text{H}_0}^2 
&=(2\pi)^2 \sum_{\bm{m}=-\infty}^{\infty} \left( \frac{8\, H_2(\bm{m})^2}{H_2(\bm{0})^2}+\frac{2\, H_2(\bm{m})^4}{H_2(\bm{0})^4}\right) \int_{\R^2}    f^4(\bm{\omega})\, d\bm{\omega}\\
&+8\, (2\pi)^2 \sum_{\bm{m}=-\infty}^{\infty} \frac{H_2(\bm{m})^2}{H_2(\bm{0})^2} \int_{\R^2} f^4(\bm{\omega})\, d\bm{\omega} -2\times 8\, (2\pi)^2 \sum_{\bm{m}=-\infty}^{\infty} \frac{H_2(\bm{m})^2}{H_2(\bm{0})^2} \int_{\R^2} f^4(\bm{\omega})\, d\bm{\omega}.
\end{align*}
This proves the remaining statement \eqref{asymptotic_variance} and completes the proof of Proposition \ref{asymptotic_variances}.
\qed

\section{Proof of Theorem \ref{estimator_variance_H0}}

\setcounter{equation}{0}

It is sufficient to show  the statement \eqref{det8}, then the consistency of the estimate $\hat{\tau}_{\text{H}_0,\lambda,a}^2$
in \eqref{var_estimate} follows from the defintion 
of $
\tau_{\text{H}_0}^2$ in \eqref{det7}.
For a proof of \eqref{det8} we show the statements
\begin{align} \label{to_show_1}
\E[\hat{F}_{\lambda,a}]=\int_{\R^2} f^4(\bm{\omega})\, d\bm{\omega} + o(1)
\end{align}
and
\begin{align} \label{to_show_2}
\Var[\hat{F}_{\lambda,a}]=o(1)
\end{align}
as $a,\lambda,n\to\infty$. 
For the proof of \eqref{to_show_1}, note that
\begin{align*}
\hat{F}_{\lambda,a} &=I_{1,\lambda,a}+I_{2,\lambda,a},
\end{align*}
where 
\begin{align*}
I_{1,\lambda,a}&:= \frac{(2\pi)^2 \lambda^6}{24 n^8 H_2(\bm{0})^4} \sum_{\bm{k}=-a}^{a-1} \sum_{\substack{j_1,\ldots,j_8\in\{1,\ldots,n\}:\\ j_1,\ldots,j_8\, \text{pairwise different}}} h\left(\frac{\bm{s}_{j_1}}{\lambda}\right) \ldots h\left(\frac{\bm{s}_{j_8}}{\lambda}\right) Z(\bm{s}_{j_1})\ldots Z(\bm{s}_{j_8}) \\
&\phantom{================}\times \exp\big(\im(\bm{s}_{j_1}-\bm{s}_{j_2}+\bm{s}_{j_3}-\bm{s}_{j_4}+\bm{s}_{j_5}-\bm{s}_{j_6}+\bm{s}_{j_7}-\bm{s}_{j_8})^T \tilde{\bm{\omega}}_{\bm{k}}\big)
\end{align*}
and
\begin{align*}
I_{2,\lambda,a}&:=\frac{(2\pi)^2 \lambda^6}{24 n^8 H_2(\bm{0})^4} \sum_{\bm{k}=-a}^{a-1} \sum_{\substack{j_1,\ldots,j_8\in\{1,\ldots,n\}:\\j_r=j_t \text{ for some } r,t\in\{1,\ldots,8\}}} h\left(\frac{\bm{s}_{j_1}}{\lambda}\right) \ldots h\left(\frac{\bm{s}_{j_8}}{\lambda}\right) Z(\bm{s}_{j_1})\ldots Z(\bm{s}_{j_8}) \\
&\phantom{================}\times \exp\big(\im(\bm{s}_{j_1}-\bm{s}_{j_2}+\bm{s}_{j_3}-\bm{s}_{j_4}+\bm{s}_{j_5}-\bm{s}_{j_6}+\bm{s}_{j_7}-\bm{s}_{j_8})^T \tilde{\bm{\omega}}_{\bm{k}}\big).
\end{align*}
We will show that
\begin{align} \label{I}
\E[I_{1,\lambda,a}] = \int_{\R^2} f^4(\bm{\omega})\, d\bm{\omega} + \Landau\left(\frac{1}{n}+\frac{1}{\lambda^2}+\left(\frac{\lambda}{a}\right)^{3+4\delta}\right)
\end{align}
and
\begin{align} \label{II}
\E[I_{2,\lambda,a}]=\Landau\left(\frac{\lambda^2}{n}\right),
\end{align}
from which \eqref{to_show_1} follows by Assumptions \ref{assumption_on_sampling_scheme} and \ref{assumptions_on_a}. 
For $t\in\N$ and $c=1,\ldots,8$, we set 
\begin{align} \label{Y_tc_for_variance_estimator}
Y_{t,c}(\underline{j}):=h\left(\frac{\bm{s}_{j_{c+8(t-1)}}}{\lambda}\right) Z(\bm{s}_{j_{c+8(t-1)}}) \exp\big((-1)^{c+1}\, \im\, \bm{s}_{j_{c+8(t-1)}}^T \tilde{\bm{\omega}}_{\bm{k}_t} \big)
\end{align}
and obtain
\begin{align*}
\E[I_{1,\lambda,a}]&=\frac{(2\pi)^2 \lambda^6}{24 n^8 H_2(\bm{0})^4} \sum_{\bm{k}=-a}^{a-1} \sum_{\substack{j_1,\ldots,j_8\in\{1,\ldots,n\}:\\ j_1,\ldots,j_8\, \text{pairwise different}}} \sum_{\bm{\nu}} \cum_2\big[Y_{1,c}(\underline{j})\, :\, (1,c)\in\nu_1\big]\times \ldots \\
&\phantom{======================}\times \cum_2\big[Y_{1,c}(\underline{j})\, :\, (1,c)\in\nu_4\big],
\end{align*}
where the sum is taken over all partitions $\bm{\nu}=\{\nu_1,\ldots,\nu_4\}$ of the table
\begin{align} \label{table_variance}
\begin{matrix} 
&(1,1) & (1,2) & (1,3) & (1,4) & (1,5) & (1,6) & (1,7) & (1,8).
\end{matrix}
\end{align}
We claim that the terms with partitions evoking no restriction at all are of highest order. More precisely, those are the partitions where each of the four sets consists of two frequencies of different signs, respectively. It is easy to see that there are $24$ possibilities for this situation. For illustration, we consider the partition
\begin{align*}
\bm{\nu}=\Big\{\{(1,1),(1,2)\},\{(1,3),(1,4)\},\{(1,5),(1,6)\},\{(1,7),(1,8)\}\Big\}.
\end{align*}
Recalling the definition of the quantity $c_{8,n}$ in \eqref{c_q,n}, the respective summand is then given by
\begin{align*}
&\phantom{=i}\frac{(2\pi)^2 \lambda^6}{24 n^8 H_2(\bm{0})^4} \sum_{\bm{k}=-a}^{a-1} \sum_{\substack{j_1,\ldots,j_8\in\{1,\ldots,n\}:\\ j_1,\ldots,j_8\, \text{pairwise different}}} \E\Bigg[h\left(\frac{\bm{s}_{j_1}}{\lambda}\right) \ldots h\left(\frac{\bm{s}_{j_8}}{\lambda}\right) c(\bm{s}_{j_1}-\bm{s}_{j_2}) c(\bm{s}_{j_3}-\bm{s}_{j_4})\\
&\phantom{====================}\times c(\bm{s}_{j_5}-\bm{s}_{j_6}) c(\bm{s}_{j_7}-\bm{s}_{j_8}) \\
&\phantom{====================}\times \exp\big(\im(\bm{s}_{j_1}-\bm{s}_{j_2}+\bm{s}_{j_3}-\bm{s}_{j_4}+\bm{s}_{j_5}-\bm{s}_{j_6}+\bm{s}_{j_7}-\bm{s}_{j_8})^T \tilde{\bm{\omega}}_{\bm{k}}\big)\Bigg]\\
&=\frac{c_{8,n}}{24(2\pi)^6 H_2(\bm{0})^4 \lambda^{10}} \sum_{\bm{k}=-a}^{a-1} \int_{\R^8} B(\bm{x}_1+\tilde{\bm{\omega}}_{\bm{k}})^2 \ldots B(\bm{x}_4+\tilde{\bm{\omega}}_{\bm{k}})^2 f(\bm{x}_1) \ldots f(\bm{x}_4)\, d\bm{x}_1 \ldots d\bm{x}_4\\
&=\frac{1}{24(2\pi\lambda)^8 H_2(\bm{0})^4} \int_{\R^8} B(\bm{u}_1)^2 \ldots B(\bm{u}_4)^2  \left(\int_{[-2\pi a/\lambda,2\pi a/\lambda]^2} f^4(\bm{\omega}) \, d\bm{\omega}\right)\, d\bm{u}_1 \ldots d\bm{u}_4+ \Landau\left(\frac{1}{n}+\frac{1}{\lambda^2}\right)\\
&=\frac{1}{24}\int_{\R^2} f^4(\bm{\omega})\, d\bm{\omega} + \Landau\left(\frac{1}{n}+\frac{1}{\lambda^2}+\left(\frac{\lambda}{a}\right)^{3+4\delta}\right),
\end{align*}
where we used analogous arguments as in the proof of Theorem \ref{expectation_theo}.
The same calculation works for all $24$ partitions evoking no restriction at all, so in order to prove \eqref{I}, it suffices to show that all terms corresponding to the remaining partitions are of lower order. More precisely, it will become clear that as soon as we introduce a restriction, i.e. if at least one set of the partition contains two frequencies of the same sign, we obtain an order of $\Landau(1/\lambda^2)$. For illustration, consider the partition
\begin{align*}
\bm{\nu}=\Big\{\{(1,1),(1,3)\},\{(1,2),(1,4)\},\{(1,5),(1,6)\},\{(1,7),(1,8)\}\Big\}.
\end{align*}
With the same arguments as above, the respective summand is then given by
\begin{align*}
&\phantom{=i}\frac{(2\pi)^2 \lambda^6}{24 n^8 H_2(\bm{0})^4} \sum_{\bm{k}=-a}^{a-1} \sum_{\substack{j_1,\ldots,j_8\in\{1,\ldots,n\}:\\ j_1,\ldots,j_8\, \text{pairwise different}}} \E\Bigg[h\left(\frac{\bm{s}_{j_1}}{\lambda}\right) \ldots h\left(\frac{\bm{s}_{j_8}}{\lambda}\right) c(\bm{s}_{j_1}-\bm{s}_{j_3}) c(\bm{s}_{j_2}-\bm{s}_{j_4})\\
&\phantom{====================}\times c(\bm{s}_{j_5}-\bm{s}_{j_6}) c(\bm{s}_{j_7}-\bm{s}_{j_8}) \\
&\phantom{====================}\times \exp\big(\im(\bm{s}_{j_1}-\bm{s}_{j_2}+\bm{s}_{j_3}-\bm{s}_{j_4}+\bm{s}_{j_5}-\bm{s}_{j_6}+\bm{s}_{j_7}-\bm{s}_{j_8})^T \tilde{\bm{\omega}}_{\bm{k}}\big)\Bigg]\\
&=\frac{c_{8,n}}{24(2\pi)^6 H_2(\bm{0})^4 \lambda^{10}} \sum_{\bm{k}=-a}^{a-1} \int_{\R^8} B(\bm{u}_1) B(\bm{u}_1-2\tilde{\bm{\omega}}_{\bm{k}}) B(\bm{u}_2) B(\bm{u}_2-2\tilde{\bm{\omega}}_{\bm{k}}) B(\bm{u}_3)^2 B(\bm{u}_4)^2\\
&\phantom{============ii===}  f(\bm{u}_1-\tilde{\bm{\omega}}_{\bm{k}}) \ldots f(\bm{u}_4-\tilde{\bm{\omega}}_{\bm{k}})  \, d\bm{u}_1 \ldots d\bm{u}_4.
\end{align*}
Taking absolutes and using Lemma \ref{bounds_for_B}, this expression is (ignoring constants) bounded by
\begin{align*}
&\phantom{=i}\frac{1}{\lambda^{10}} \sum_{\bm{k}=-a}^{a-1} \int_{\R^8} \big| B(\bm{u}_1) B(\bm{u}_1-\bm{\omega}_{2\bm{k}+1}) B(\bm{u}_2) B(\bm{u}_2-\bm{\omega}_{2\bm{k}+1}) B(\bm{u}_3)^2 B(\bm{u}_4)^2 \big| \, d\bm{u}_1 \ldots d\bm{u}_4 \\
&= \frac{1}{\lambda^2} \times \frac{1}{\lambda^4} \int_{\R^4} B(\bm{u}_3)^2 B(\bm{u}_4)^2\, d\bm{u}_3 d\bm{u}_4\\
&\phantom{===============} \times \frac{1}{\lambda^4} \sum_{\bm{k}=-a}^{a-1} \int_{\R^4} \big|B(\bm{u}_1) B(\bm{u}_1-\bm{\omega}_{2\bm{k}+1}) B(\bm{u}_2) B(\bm{u}_2-\bm{\omega}_{2\bm{k}+1}) \big|\, d\bm{u}_1 d\bm{u}_2\\
&=\Landau\left(\frac{1}{\lambda^2}\right).
\end{align*}
The same calculation works for all other partitions evoking at least one restriction as well and we have thus shown \eqref{I}. 
The proof of \eqref{II} works exactly in the same way as the proof of Lemma \ref{order_depending_on_number_of_restr}, where we have shown for a very similar cumulant expression that setting two or more elements in $\underline{j}$ equal leads to an order change of at most $\lambda^d/n$, where $d$ is the dimension. The details are omitted for the sake of brevity. Therefore, we have proven \eqref{to_show_1} and we will now deal with the variance of the estimator $\hat{F}_{\lambda,a}$. By definition of the random variables $Y_{t,c}(\underline{j})$ in \eqref{Y_tc_for_variance_estimator}, we obtain
\begin{align*}
\lambda^2 \Var\big[\hat{F}_{\lambda,a}\big]&= \frac{(2\pi)^4 \lambda^{14}}{24^2 n^{16} H_2(\bm{0})^8} \sum_{\bm{k}_1,\bm{k}_2 = -a}^{a-1} \sum_{j_1,\ldots,j_{16}=1}^n \cum\bigg[h\left(\frac{\bm{s}_{j_1}}{\lambda}\right) \ldots h\left(\frac{\bm{s}_{j_8}}{\lambda}\right) Z(\bm{s}_{j_1}) \ldots Z(\bm{s}_{j_8})\\
&\phantom{============} \exp\big(\im (\bm{s}_{j_1}-\bm{s}_{j_2}+\bm{s}_{j_3}-\bm{s}_{j_4}+\bm{s}_{j_5}-\bm{s}_{j_6}+\bm{s}_{j_7}-\bm{s}_{j_8})^T \tilde{\bm{\omega}}_{\bm{k}_1}\big),\\
&\phantom{============}  h\left(\frac{\bm{s}_{j_9}}{\lambda}\right) \ldots h\left(\frac{\bm{s}_{j_{16}}}{\lambda}\right) Z(\bm{s}_{j_9}) \ldots Z(\bm{s}_{j_{16}})\\
&\phantom{============} \exp\big(\im (\bm{s}_{j_9}-\bm{s}_{j_{10}}+\bm{s}_{j_{11}}-\bm{s}_{j_{12}}+\bm{s}_{j_{13}}-\bm{s}_{j_{14}}+\bm{s}_{j_{15}}-\bm{s}_{j_{16}})^T \tilde{\bm{\omega}}_{\bm{k}_2}\big)\bigg]\\
&=\frac{(2\pi)^4 \lambda^{14}}{24^2 n^{16} H_2(\bm{0})^8} \sum_{\bm{k}_1,\bm{k}_2 = -a}^{a-1} \sum_{\bm{\nu}=\{\nu_1,\ldots,\nu_G\}\in\mathcal{I}} \sum_{i=1}^{8+G} \sum_{\underline{j}\in\mathcal{D}(i)} \cum_{|\nu_1|}\big[Y_{t,c}(\underline{j}) \, : \, (t,c)\in\nu_1\big] \\
&\phantom{=============ii======}\times \ldots \times \cum_{|\nu_G|}\big[Y_{t,c}(\underline{j}) \, : \, (t,c)\in\nu_G\big],
\end{align*}
where $\mathcal{I}$ is the set of indecomposable partitions of the table 
\begin{align*}
\begin{matrix}
(1,1) & (1,2) & (1,3) & (1,4) & (1,5) & (1,6) & (1,7) & (1,8)\\
(2,1) & (2,2) & (2,3) & (2,4) & (2,5) & (2,6) & (2,7) & (2,8)
\end{matrix}
\end{align*}
and 
\begin{align*}
\mathcal{D}(i):=\{\underline{j}=(j_1,\ldots,j_{16})\in\{1,\ldots,n\}^{16}\, : \, i \text{ elements in } \underline{j} \text{ are different}\}.
\end{align*}
Here, we applied results A), B) and C) from the proof of the second part of  Theorem \ref{expectation_theo}. Furthermore, we used that any tupel $\underline{j}=(j_1,\ldots,j_{16})\in\{1,\ldots,n\}^{16}$ contributing a nonzero term to the above sum can have at most
\begin{align*}
\sum_{g=1}^G \left(\frac{|\nu_g|}{2}+1\right) = 8+G
\end{align*}
different elements.
The proof now works analogously as the proof of the second part of Theorem \ref{expectation_theo}: The above expression is of highest order for partitions consisting of $G=8$ sets and $i=16$ different elements that evoke exactly one restriction, which is necessary to make the partition indecomposable. For illustration, we consider the partition
\begin{align*}
\bm{\nu}=\Big\{&\{(1,1),(1,2)\},\{(1,3),(1,4)\},\{(1,5),(1,6)\},\{(1,7),(2,8)\},\\
&\{(1,8),(2,7)\},\{(2,1),(2,2)\},\{(2,3),(2,4)\},\{(2,5),(2,6)\}\Big\}.
\end{align*}
Recalling the definition of the quantity $c_{16,n}$ in \eqref{c_q,n}, the respective term equals
\begin{align*}
&\phantom{=i}\frac{(2\pi)^4 \lambda^{14} c_{16,n}}{24^2 (2\pi)^{16} H_2(\bm{0})^8} \sum_{\bm{k}_1,\bm{k}_2=-a}^{a-1} \frac{1}{\lambda^{32}} \int_{\R^{16}} B(\bm{x}_1+\tilde{\bm{\omega}}_{\bm{k}_1})^2 B(\bm{x}_2+\tilde{\bm{\omega}}_{\bm{k}_1})^2 B(\bm{x}_3+\tilde{\bm{\omega}}_{\bm{k}_1})^2 B(\bm{x}_4+\tilde{\bm{\omega}}_{\bm{k}_2})^2\\
& \phantom{====================}B(\bm{x}_5+\tilde{\bm{\omega}}_{\bm{k}_2})^2 B(\bm{x}_6+\tilde{\bm{\omega}}_{\bm{k}_2})^2 B(\bm{x}_7+\tilde{\bm{\omega}}_{\bm{k}_1}) B(\bm{x}_7+\tilde{\bm{\omega}}_{\bm{k}_2}) \\
& \phantom{====================}B(\bm{x}_8+\tilde{\bm{\omega}}_{\bm{k}_1}) B(\bm{x}_8+\tilde{\bm{\omega}}_{\bm{k}_2})  f(\bm{x}_1)\ldots f(\bm{x}_8) \, d\bm{x}_1 \ldots d\bm{x}_8\\
&=\frac{c_{16,n}}{24^2 (2\pi)^{14} \lambda^{16} H_2(\bm{0})^8} \sum_{\bm{m}=-2a+1}^{2a-1} \int_{\R^{16}} B(\bm{u}_1)^2 \ldots B(\bm{u}_6)^2 B(\bm{u}_7) B(\bm{u}_7+\bm{\omega}_{\bm{m}}) B(\bm{u}_8) B(\bm{u}_8+\bm{\omega}_{\bm{m}})\\
&\phantom{=======}\Bigg[\left(\frac{2\pi}{\lambda}\right)^2 \sum_{\bm{k}_1 = \max(-a,-a-\bm{m})}^{\min(a-1,a-1-\bm{m})} f(\bm{u}_1-\tilde{\bm{\omega}}_{\bm{k}_1}) f(\bm{u}_2-\tilde{\bm{\omega}}_{\bm{k}_1}) f(\bm{u}_3-\tilde{\bm{\omega}}_{\bm{k}_1}) f(\bm{u}_4-(\tilde{\bm{\omega}}_{\bm{k}_1} + \bm{\omega_m})) \\
&\phantom{=========} f(\bm{u}_5-(\tilde{\bm{\omega}}_{\bm{k}_1} + \bm{\omega_m}))  f(\bm{u}_6-(\tilde{\bm{\omega}}_{\bm{k}_1} + \bm{\omega_m}))  f(\bm{u}_7-\tilde{\bm{\omega}}_{\bm{k}_1}) f(\bm{u}_8-\tilde{\bm{\omega}}_{\bm{k}_1})\Bigg]\, d\bm{u}_1 \ldots d\bm{u}_8,
\end{align*}
and using Lemma \ref{bounds_for_B}, it can immediately be seen that this term is of order $\Landau(1)$. Introducing any further restriction evokes an order change of $1/\lambda^2$, while setting two elements equal leads to an order change of at most $\lambda^2/n$. Moreover, considering any partition consisting of less than $G=8$ groups leads to an order change of $\lambda^2/n$ as well (see Corollary \ref{result_for_q=2}, which can be proven analogously in the present case). This shows that  
\begin{align*}
\Var(\hat{F}_{\lambda,a}) = \Landau\left(\frac{1}{\lambda^2}\right), 
\end{align*}
which proves \eqref{to_show_2} and completes the proof of Theorem \ref{estimator_variance_H0}.
\qed

\section{Proofs of the results in Section \ref{sec5}}

\setcounter{equation}{0}

\subsection{Proof of Proposition \ref{prop_umschreiben_D1}}

First, note that for any quantities $A_1,\ldots,A_n$ 
\begin{align*}
\sum_{(j_1,\ldots,j_4)\in\tilde{\mathcal{E}}} A_{j_1} A_{j_2} A_{j_3} A_{j_4} = \sum_{\substack{j_1\neq j_2\\ j_3\neq j_4}} A_{j_1} A_{j_2} A_{j_3} A_{j_4}  - \sum_{\substack{j_1\neq j_2\\ j_3\neq j_4\\ j_2=j_3}} A_{j_1} A_{j_2} A_{j_3} A_{j_4}  -  \sum_{\substack{j_1\neq j_2\\ j_3\neq j_4\\ j_2\neq j_3 \\ j_1=j_4}} A_{j_1} A_{j_2} A_{j_3} A_{j_4}.
\end{align*}
Setting
\begin{align*}
X(\bm{s}_j)=h\left(\frac{\bm{s}_j}{\lambda}\right) Z(\bm{s}_j),
\end{align*}
we thus obtain for fixed $\bm{k}=(k_1,k_2)^T$ with $k_i\in\{-a,\ldots,a-1\}$ for $i=1,2$:
\begin{align*}
&\phantom{=i}\sum_{(j_1,\ldots,j_4)\in\tilde{\mathcal{E}}} h\left(\frac{\bm{s}_{j_1}}{\lambda}\right)\ldots h\left(\frac{\bm{s}_{j_4}}{\lambda}\right)  Z(\bm{s}_{j_1}) \ldots Z(\bm{s}_{j_4}) \exp\big(\im (\bm{s}_{j_1}-\bm{s}_{j_2}+\bm{s}_{j_3}-\bm{s}_{j_4})^T \tilde{\bm{\omega}}_{\bm{k}}\big) \\
&=\Big (\sum_{j_1\neq j_2} X(\bm{s}_{j_1}) X(\bm{s}_{j_2}) \exp\big(\im (\bm{s}_{j_1}-\bm{s}_{j_2})^T \tilde{\bm{\omega}}_{\bm{k}}\big) \Big )^2 - \sum_{\substack{j_1\neq j_2 \\ j_2 \neq j_4}} X(\bm{s}_{j_1}) X^2(\bm{s}_{j_2}) X(\bm{s}_{j_4}) \exp\big(\im (\bm{s}_{j_1}-\bm{s}_{j_4})^T \tilde{\bm{\omega}}_{\bm{k}}\big)\\
&\phantom{=i}-\sum_{\substack{j_1\neq j_2\\ j_1 \neq j_3 \\ j_2\neq j_3}} X^2(\bm{s}_{j_1}) X(\bm{s}_{j_2}) X(\bm{s}_{j_3}) \exp\big(\im (\bm{s}_{j_3}-\bm{s}_{j_2})^T \tilde{\bm{\omega}}_{\bm{k}}\big)\\
&=\Big (\Big |\sum_{j=1}^n X(\bm{s}_j) \exp\big(\im \bm{s}_j^T \tilde{\bm{\omega}}_{\bm{k}}\big)\Big |^2 - \sum_{j=1}^n X^2(\bm{s}_j)\Big )^2 - \sum_{j_2=1}^n X^2(\bm{s}_{j_2}) \Big | \sum_{j\neq j_2} X(\bm{s}_j) \exp\big(\im \bm{s}_j^T \tilde{\bm{\omega}}_{\bm{k}}\big) \Big |^2\\
&\phantom{=i}-\sum_{j_1=1}^n X^2(\bm{s}_{j_1}) \Big (\Big |\sum_{j\neq j_1} X(\bm{s}_j) \exp\big(\im \bm{s}_j^T \tilde{\bm{\omega}}_{\bm{k}}\big)\Big |^2 - \sum_{j\neq j_1} X^2(\bm{s}_j)\Big ). 
\end{align*}
This yields the claim.

\subsection{Proof of Proposition \ref{prop_eff_D1}}

By Proposition \ref{approx_D1} it is sufficient to show
\begin{align*}
\frac{\lambda}{\tau_{1,2,\lambda,a}} \big (\hat{D}_{1,\lambda,a}^{\mathrm{eff}}- D_{1,2,\lambda,a } \big ) \dn \mathcal{N}(0,1) \qquad \text{as } a,\lambda,n\rightarrow \infty.
\end{align*}
It is easy to see that 
\begin{align*}
\lambda^{q}\, \cum_q\big(\hat{D}_{1,\lambda,a}^{\text{eff}}\big)=\begin{cases}
\Landau(\lambda), &\qquad q=1,\\
\Landau(1), &\qquad q=2,\\
o(1), &\qquad q\geq 3,
\end{cases}
\end{align*}
since for the proof of equation \eqref{cumulant_orders} in the proof of Theorem \ref{asymptotic_normality}, we do not make use of the assumption that $j_1,\ldots,j_4$ are pairwise different. Moreover, with the same argument we obtain
\begin{align*}
\lambda^2 \Var\left(\hat{D}_{1,\lambda,a}^{\mathrm{eff}}\right) = \tau_{1,\lambda,a}^2 + o(1),
\end{align*}
so it remains to show that 
$
\lambda \, \E\big[\hat{D}_{1,\lambda,a}^{\text{eff}}-
D_{1,2,\lambda,a}
\big]\rightarrow 0 
$
as $a,\lambda,n\to\infty.$
Observing the decomposition
\begin{align*} 
\hat{D}_{1,\lambda,a}^{\text{eff}}=\hat{D}_{1,2,\lambda,a}+R_{1,\lambda,a}
\end{align*}
for
\begin{align*}
R_{1,\lambda,a}&:=\frac{(2\pi\lambda)^2}{2n^4 H_2(\bm{0})^2} \sum_{\bm{k}=-a}^{a-1} \sum_{(j_1,j_2,j_3,j_4)\in\tilde{\mathcal{E}}\setminus \mathcal{E}} h\left(\frac{\bm{s}_{j_1}}{\lambda}\right) \ldots h\left(\frac{\bm{s}_{j_4}}{\lambda}\right) Z(\bm{s}_{j_1}) \ldots Z(\bm{s}_{j_4})\\
&\phantom{==================} \exp\big(\im (\bm{s}_{j_1}-\bm{s}_{j_2}+\bm{s}_{j_3}-\bm{s}_{j_4})^T \tilde{\bm{\omega}}_{\bm{k}}\big)
\end{align*}
and using Theorem \ref{expectation_theo}, it thus suffices to prove that
\begin{align} \label{efficient_D1}
\E\big[R_{1,\lambda,a}\big]=o(\lambda^{-1}).
\end{align}

To see why \eqref{efficient_D1} holds, first note that in general,
the expected value of any term with $j_k=j_l$ for some $k\neq l$ is of order $\Landau(\lambda^2/n)$.
However, the term $R_{1,\lambda,a}$ only allows for the equalities $j_1=j_3$ or $j_2=j_4$ (or both).  
Since both $j_1$ and $j_3$ correspond to a frequency of positive sign, any set of the partition containing both $j_1$ and $j_3$ will evoke a restriction and therefore a further order change of $\lambda^{-2}$ (see Lemma \ref{order_depending_on_number_of_restr}). 
The same holds for any set containing both $j_2$ and $j_4$, since they both correspond to a frequency of negative sign. If $j_k=j_l$ with $j_k$ and $j_l$ not belonging to the same set, we immediately obtain an order of $O(1/n)$ (see the proof of Lemma \ref{order_depending_on_number_of_restr}). 
Therefore, it holds that
\begin{align*} 
\E\big[R_{1,\lambda,a}\big]=\Landau\left(\frac{1}{n}\right),
\end{align*}
and since by assumption $\lambda^2/n\to 0$, 
this shows \eqref{efficient_D1} and completes the proof of Proposition \ref{prop_eff_D1}.

\subsection{Proof of Proposition \ref{prop_eff_D2}}

By Proposition \eqref{approx_D2} it is sufficient to show
\begin{align*}
\frac{\lambda}{\tau_{2,\lambda,a}} \big (\hat{D}_{2,\lambda,a}^{\mathrm{eff}}- D_{2,2,\lambda,a}\big ) \dn \mathcal{N}(0,1) \qquad \text{as } a,\lambda,n\rightarrow \infty .
\end{align*}
Note that equation \eqref{to_show_sec_int} in the proof of Theorem \ref{asymptotic_normality_secint} is not only satisfied for the estimate $\hat{D}_{2,\lambda,a}$, but also for $\hat{D}_{2,\lambda,a}^{\text{eff}}$, since in the proof of \eqref{to_show_sec_int} we do not make use of the assumption that $j_1,\ldots,j_4$ are pairwise different. Moreover, with the same argument we obtain
\begin{align*}
\lambda^2 \Var\big (\hat{D}_{2,\lambda,a}^{\mathrm{eff}}\big )=\tau_{2,\lambda,a}^2 + o(1).
\end{align*}

What remains to be shown is 
\begin{align*} 
\lambda \, \E\big[\hat{D}_{2,\lambda,a}^{\text{eff}}-D_{2,2,\lambda,a}\big] \rightarrow 0 \qquad \text{as } \lambda,a,n\to\infty.
\end{align*}
It obviously holds that
 $
\hat{D}_{2,\lambda,a}^{\text{eff}} = \hat{D}_{2,2,\lambda,a} + R_{2,\lambda,a},
$
where
\begin{align*}
R_{2,\lambda,a}&:=\frac{1}{\lambda} \sum_{r=0}^{a-1} \tilde{\omega}_r \Big[\Big (\frac{2\pi}{\lambda}\Big  )^4 \sum_{\bm{k}=-a}^{a-1} \sum_{\bm{\ell}=-a}^{a-1} \frac{\lambda^4}{n^4 H_2(\bm{0})^2} \sum_{(j_1,j_2,j_3,j_4)\in \tilde{\tilde{\mathcal{E}}}\setminus \mathcal{E}} h\Big (\frac{\bm{s}_{j_1}}{\lambda}\Big )\ldots h\Big (\frac{\bm{s}_{j_4}}{\lambda}\Big  ) \\
&Z(\bm{s}_{j_1}) \ldots Z(\bm{s}_{j_4}) \exp\big(\im (\bm{s}_{j_1}-\bm{s}_{j_2})^T \tilde{\bm{\omega}}_{\bm{k}}\big) \exp\big(\im(\bm{s}_{j_3}-\bm{s}_{j_4})^T \tilde{\bm{\omega}}_{\bm{\ell}}\big) J_0(\tilde{\omega}_r \|\tilde{\bm{\omega}}_{\bm{k}}\|) J_0(\tilde{\omega}_r \|\tilde{\bm{\omega}}_{\bm{\ell}}\|) \Big].
\end{align*}
By Theorem \ref{expect_theo_sec_int}, it thus suffices to prove
\begin{align}  \label{efficient_D2}
\E\big[R_{2,\lambda,a}\big]=o(\lambda^{-1}),
\end{align}
which follows from similar arguments as \eqref{efficient_D1}: First recall that $\E[\hat{D}_{2,2,\lambda,a}]=\Landau(1)$, and by definition of the set $\mathcal{E}$ in the estimator $\hat{D}_{2,2,\lambda,a}$, this order corresponds to terms where all $j_1,\ldots,j_4$ are pairwise different. In general, decreasing the number of different elements in $(j_1,\ldots,j_4)$ by $1$ leads to an order change of $\lambda^2/n$, which follows with the same arguments as in Lemma \ref{order_depending_on_number_of_restr}. 
However, the estimator $\hat{D}_{2,2,\lambda,a}^{\text{eff}}$ only allows for combinations of $j_1,\ldots,j_4$ with $j_1\neq j_2$ and $j_3\neq j_4$, so the term $R_{2,\lambda,a}$ requires $j_1\neq j_2$ and $j_3\neq j_4$, but at least one of the equalities $j_1=j_3$, $j_1=j_4$, $j_2=j_3$, or $j_2=j_4$.
If appearing within the same set, each of these constraints necessarily leads to a restriction between the frequencies $\tilde{\bm{\omega}}_{\bm{k}}$ and $\tilde{\bm{\omega}}_{\bm{\ell}}$, which then yields a further order change of $a^2/\lambda^4$ (compare to the proof of Theorem \ref{expect_theo_sec_int}). Otherwise, we directly obtain an order of $\Landau(1/n)$. We thus obtain
\begin{align*}
\E\big[R_{2,\lambda,a}\big]=\Landau\left(\frac{\lambda^2}{n}\times \frac{a^2}{\lambda^4}\right),
\end{align*}
which is of order $o(\lambda^{-1})$ since we assumed $a^2/\lambda^3\to 0$ and $\lambda^2/n\to 0$. This shows \eqref{efficient_D2} and thus the statement of Proposition \ref{prop_eff_D2}.

\setcounter{equation}{0}
\section{Auxiliary results}


\begin{lemma} \label{Riemann_sum}
Assume that $f:\R^d\rightarrow \R$ is twice differentiable and satisfies 
\begin{align} \label{second_deriv_of_f}
|f(\bm{\omega})|\leq \beta_{1+\delta}(\bm{\omega}), \qquad \left|\frac{\partial f(\bm{\omega})}{\partial \omega_i}\right|\leq \beta_{1+\delta}(\bm{\omega}), \quad \text{and} \quad  \left|\frac{\partial^2 f(\bm{\omega})}{\partial^2 \omega_i}\right|\leq \beta_{1+\delta}(\bm{\omega}),\qquad i=1,\ldots,d,
\end{align}
for some $\delta>0$, where the function $\beta_{1+\delta}$ is defined in \eqref{beta_fct}, and let $a,\lambda\to\infty$ with $a/\lambda\to\infty$. 
Then, for arbitrary $d\geq 1$ and uniformly over $\bm{u}=(u_1,\ldots,u_d)^T\in\R^d$, $\bm{v}=(v_1,\ldots,v_d)^T\in\R^d$ and $a\in\N$, we have
\begin{align*}
&\phantom{i=}\Bigg|\left(\frac{2\pi}{\lambda}\right)^d \sum_{k_1,\ldots,k_d=-a}^{a-1} f\left(\tilde{\omega}_{k_1}-u_1,\ldots,\tilde{\omega}_{k_d}-u_d\right) f(\tilde{\omega}_{k_1}-v_1,\ldots,\tilde{\omega}_{k_d}-v_d)\\
&\phantom{=====}- \int_{[-2\pi a/\lambda,2\pi a/\lambda]^d} f(\omega_1-u_1,\ldots,\omega_d-u_d) f(\omega_1-v_1,\ldots,\omega_d-v_d)\, d\omega_1\ldots d\omega_d \Bigg| \lesssim \frac{1}{\lambda^2},
\end{align*} 
where the vector of shifted Fourier frequencies $(\tilde{\omega}_{k_1},\ldots,\tilde{\omega}_{k_d})^T=(\tilde{\omega}_{k_1,\lambda},\ldots,\tilde{\omega}_{k_d,\lambda})^T$ is defined in \eqref{shifted_Four_freq}. 
\end{lemma}

\begin{proof}
For the ease of exposition, we only prove the result for the case $d=1$, since the other cases can be proven similarly.
Set
\begin{align*}
A(u,v,\lambda,a):=\frac{2\pi}{\lambda} \sum_{k=-a}^{a-1} f\left(\tilde{\omega}_k-u\right) f\left(\tilde{\omega}_k-v\right)- \int_{-2\pi a/\lambda}^{2\pi a/\lambda} f(\omega-u) f(\omega-v)\, d\omega
\end{align*}
and note that
\begin{align*}
\int_{-2\pi a/\lambda}^{2\pi a/\lambda} f(\omega-u) f(\omega-v)\, d\omega=\sum_{k=-a}^{a-1} \int_0^{2\pi/\lambda} f\left(\omega+\frac{2\pi k}{\lambda}-u\right)f\left(\omega+\frac{2\pi k}{\lambda}-v\right)\, d\omega.
\end{align*}
We thus obtain
\begin{align}  \label{Riemann_approx}
&A(u,v,\lambda,a)
=\sum_{k=-a}^{a-1} \int_0^{2\pi/\lambda} \left[h_{k,\lambda,u,v}\left(\frac{\pi}{\lambda}\right)-h_{k,\lambda,u,v}(\omega)\right] d\omega,
\end{align}
where $h_{k,\lambda,u,v}(\omega):=f\left(\omega+\frac{2\pi k}{\lambda}-u\right)f\left(\omega+\frac{2\pi k}{\lambda}-v\right)$. 
Using a Taylor expansion and assumption \eqref{second_deriv_of_f}, this gives
\begin{align*}
|A(u,v,\lambda,a)|&\lesssim \sum_{k=-a}^{a-1} \int_0^{2\pi/\lambda} \left|h_{k,\lambda,u,v}''(\xi) \right| \left(\omega-\frac{\pi}{\lambda}\right)^2 \, d\omega\\
&\lesssim \sum_{k=-a}^{a-1} \int_{0}^{2\pi/\lambda} \beta_{1+\delta}\left(\xi + \frac{2\pi k}{\lambda}-u \right)  \left(\omega-\frac{\pi}{\lambda}\right)^2 d\omega
\end{align*}
for some intermediate point $\xi$ between $\omega$ and $\pi/\lambda$.
Due to the monotonicity and integrability of the function $\beta_{1+\delta}$, the expression
\begin{align*}
\frac{1}{\lambda} \sum_{k=-a}^{a-1} \beta_{1+\delta} \left(\xi+\frac{2\pi k}{\lambda}-u\right) 
\end{align*}
is uniformly bounded, which immediately yields $|A(u,v,\lambda,a)|\lesssim 1/\lambda^2$.
\end{proof}
\medskip

Note that the proof of Lemma \ref{Riemann_sum} works analogously when considering a product of four terms: 

\begin{corollary} \label{cor_Riemann_sum}
Under the same assumptions as in Lemma \ref{Riemann_sum}, the following holds: 
For arbitrary $d\geq 1$ and uniformly over $\bm{u}=(u_1,\ldots,u_d)^T\in\R^d$, $\bm{v}=(v_1,\ldots,v_d)^T\in\R^d$, $\bm{w}=(w_1,\ldots,w_d)^T\in\R^d$, $\bm{x}=(x_1,\ldots,x_d)^T\in\R^d$ and $a\in\N$, we have
\begin{align*}
&\phantom{i=}\Bigg|\left(\frac{2\pi}{\lambda}\right)^d \sum_{k_1,\ldots,k_d=-a}^{a-1} f\left(\tilde{\omega}_{k_1}-u_1,\ldots,\tilde{\omega}_{k_d}-u_d\right) f(\tilde{\omega}_{k_1}-v_1,\ldots,\tilde{\omega}_{k_d}-v_d)\\
&\phantom{===i==i===}\times f\left(\tilde{\omega}_{k_1}-w_1,\ldots,\tilde{\omega}_{k_d}-w_d\right) f(\tilde{\omega}_{k_1}-x_1,\ldots,\tilde{\omega}_{k_d}-x_d)\\
&\phantom{=====}- \int_{[-2\pi a/\lambda,2\pi a/\lambda]^d} f(\omega_1-u_1,\ldots,\omega_d-u_d) f(\omega_1-v_1,\ldots,\omega_d-v_d)\\
&\phantom{======i===ii=}\times f(\omega_1-w_1,\ldots,\omega_d-w_d) f(\omega_1-x_1,\ldots,\omega_d-x_d)\, d\omega_1\ldots d\omega_d \Bigg| \lesssim \frac{1}{\lambda^2}.
\end{align*} 
\end{corollary}


\begin{lemma} \label{Riemann_radius}
Let $\bm{k}, \bm{\ell}\in\Z^2$ be arbitrary and assume $a>\lambda\geq 1$.
It then holds that
\begin{align*}
&\phantom{=i}\left|\frac{2\pi}{\lambda} \sum_{r=0}^{a-1} \tilde{\omega}_r\, J_0(\tilde{\omega}_r  \|\tilde{\bm{\omega}}_{\bm{k}}\|) J_0(\tilde{\omega}_r  \|\tilde{\bm{\omega}}_{\bm{\ell}}\|) -\int_{0}^{2\pi a /\lambda} r\, J_0(r \|\tilde{\bm{\omega}}_{\bm{k}}\|) J_0(r  \|\tilde{\bm{\omega}}_{\bm{\ell}}\|)\, dr\right|\\
& \lesssim \frac{a^{2}}{\lambda^{4}} \left(\|\tilde{\bm{\omega}}_{\bm{k}}\|^2+\|\tilde{\bm{\omega}}_{\bm{\ell}}\|^2\right)+\frac{a}{\lambda^{3}} \left(\|\tilde{\bm{\omega}}_{\bm{k}}\|_1+\|\tilde{\bm{\omega}}_{\bm{\ell}}\|_1\right).
\end{align*}
\end{lemma}

\begin{proof}
Note that
\begin{align*}
B(\bm{k},\bm{\ell},\lambda,a)&:=\frac{2\pi}{\lambda} \sum_{r=0}^{a-1} \tilde{\omega}_r\, J_0(\tilde{\omega}_r  \|\tilde{\bm{\omega}}_{\bm{k}}\|) J_0(\tilde{\omega}_r  \|\tilde{\bm{\omega}}_{\bm{\ell}}\|) -\int_{0}^{2\pi a /\lambda} r\, J_0(r \|\tilde{\bm{\omega}}_{\bm{k}}\|) J_0(r  \|\tilde{\bm{\omega}}_{\bm{\ell}}\|)\, dr\\
&= \sum_{r=0}^{a-1} \int_{0}^{2\pi/\lambda} \left[h_{r,\lambda,\bm{k},\bm{\ell}}\left(\frac{\pi}{\lambda}\right)-h_{r,\lambda,\bm{k},\bm{\ell}}(x)\right] dx,
\end{align*}
where
\begin{align*}
h_{r,\lambda,\bm{k},\bm{\ell}}(x):= \left(x+\frac{2\pi r}{\lambda}\right)  J_0\left(\left(x+\frac{2\pi r}{\lambda}\right) \|\tilde{\bm{\omega}}_{\bm{k}}\|\right)  J_0\left(\left(x+\frac{2\pi r}{\lambda}\right) \|\tilde{\bm{\omega}}_{\bm{\ell}}\|\right).
\end{align*}
Using a Taylor expansion, it follows that
\begin{align} \label{Riemann_radius_eq1}
\left|B(\bm{k},\bm{\ell},\lambda,a) \right|\lesssim \sum_{r=0}^{a-1} \int_{0}^{2\pi/\lambda} \left|h_{r,\lambda,\bm{k},\bm{\ell}}^{''}(\xi)\right| \left(x-\frac{\pi}{\lambda}\right)^2 \, dx
\end{align}
for some intermediate point $\xi$ between $x$ and $\pi/\lambda$. By definition of the Bessel function, it is easy to see that
\begin{align*}
\left|h_{r,\lambda,\bm{k},\bm{\ell}}^{''}(\xi)\right| &\lesssim \big|\tilde{\omega}_{k_1}\big|+\big|\tilde{\omega}_{k_2}\big|+\big|\tilde{\omega}_{\ell_1}\big|+\big|\tilde{\omega}_{\ell_2}\big|+\left|\xi+\frac{2\pi r}{\lambda}\right| \big(\tilde{\omega}_{k_1}^2+\tilde{\omega}_{k_2}^2+ \tilde{\omega}_{\ell_1}^2 + \tilde{\omega}_{\ell_2}^2\big)\\
&\leq \|\tilde{\bm{\omega}}_{\bm{k}}\|_1+\|\tilde{\bm{\omega}}_{\bm{\ell}}\|_1+\frac{2\pi (r+1)}{\lambda} \big(\|\tilde{\bm{\omega}}_{\bm{k}}\|^2+\|\tilde{\bm{\omega}}_{\bm{\ell}}\|^2\big),
\end{align*}
which yields
\begin{align*}
\left|B(\bm{k},\bm{\ell},\lambda,a) \right|\lesssim \frac{1}{\lambda^3} \sum_{r=0}^{a-1} \left[ \|\tilde{\bm{\omega}}_{\bm{k}}\|_1+\|\tilde{\bm{\omega}}_{\bm{\ell}}\|_1+ \frac{2\pi (r+1)}{\lambda} \left(\|\tilde{\bm{\omega}}_{\bm{k}}\|^2+\|\tilde{\bm{\omega}}_{\bm{\ell}}\|^2\right)  \right].
\end{align*}
\end{proof}

\begin{lemma} \label{Riemann_f_sec_int}
Let $\bm{y}\in\R^2$ be arbitrary. 
Under the same assumptions as in Lemma \ref{Riemann_sum} for $d=2$, 
we then have
\begin{align*}
&\phantom{\lesssim i} \left|\left(\frac{2\pi}{\lambda}\right)^2 \sum_{\bm{k}=-a}^{a-1} f(\tilde{\bm{\omega}}_{\bm{k}}-\bm{u}) \exp(\im \bm{y}^T \tilde{\bm{\omega}}_{\bm{k}})-\int_{[-2\pi a/\lambda,2\pi a/\lambda]^2} f(\bm{\omega}-\bm{u}) \exp(\im \bm{y}^T \bm{\omega})\, d\bm{\omega}\right|\\
 & \lesssim \frac{1+\|\bm{y}\|_1+\|\bm{y}\|^2}{\lambda^2},
\end{align*}
uniformly over $\bm{u}\in\R^2$ and $a\in\N$. 
\end{lemma}

\begin{proof}
For arbitrary $\bm{u}=(u_1,u_2)^T\in\R^2$, we obtain 
\begin{align} \label{two_dim_lemma_6.3}
&\phantom{==}\left|\left(\frac{2\pi}{\lambda}\right)^2 \sum_{\bm{k}=-a}^{a-1} f(\tilde{\bm{\omega}}_{\bm{k}}-\bm{u})\exp(\im \bm{y}^T \tilde{\bm{\omega}}_{\bm{k}}) - \int_{[-2\pi a/\lambda,2\pi a/\lambda]^2} f(\bm{\omega}-\bm{u}) \exp(\im \bm{y}^T \bm{\omega})\, d\bm{\omega}\right|\\
&\leq \frac{2\pi}{\lambda} \sum_{k_1=-a}^{a-1} A_{\lambda,u_1,u_2,y_2}(k_1)+\int_{-2\pi a/\lambda}^{2\pi a/\lambda} B_{\lambda,u_1,u_2,y_1}(\omega_2)\, d\omega_2\nonumber,
\end{align}
where
\begin{align*}
A_{\lambda,u_1,u_2,y_2}(k_1)&:=\bigg|\frac{2\pi}{\lambda} \sum_{k_2=-a}^{a-1} f\left(\tilde{\omega}_{k_1}-u_1,\tilde{\omega}_{k_2}-u_2\right) \exp(\im y_2 \tilde{\omega}_{k_2})\nonumber\\
&\phantom{===========}-\int_{-2\pi a/\lambda}^{2\pi a/\lambda} f\left(\tilde{\omega}_{k_1}-u_1,\omega_2-u_2\right) \exp(\im y_2 \omega_2)\, d\omega_2 \bigg|
\end{align*}
and
\begin{align*}
B_{\lambda,u_1,u_2,y_1}(\omega_2)&:=\bigg|\frac{2\pi}{\lambda} \sum_{k_1=-a}^{a-1} f\left(\tilde{\omega}_{k_1}-u_1,\omega_2-u_2\right) \exp(\im y_1 \tilde{\omega}_{k_1})\nonumber\\
&\phantom{===========}-\int_{-2\pi a/\lambda}^{2\pi a/\lambda} f\left(\omega_1-u_1,\omega_2-u_2\right) \exp(\im y_1 \omega_1)\, d\omega_1 \bigg|.
\end{align*}
Setting
\begin{align*}
h_{k_1,\lambda,u_1,u_2,y_2}(\omega):=f\left(\tilde{\omega}_{k_1}-u_1,\omega+\frac{2\pi k_2}{\lambda}-u_2\right) \exp\left(\im y_2 \left(\omega+\frac{2\pi k_2}{\lambda}\right)\right)
\end{align*}
and using a Taylor expansion, we obtain
\begin{align*}
A_{\lambda,u_1,u_2,y_2}(k_1)
 &=\Bigg|\sum_{k_2=-a}^{a-1} \int_{0}^{2\pi/\lambda} \left[h_{k_1,\lambda,u_1,u_2,y_2}\left(\frac{\pi}{\lambda}\right)-h_{k_1,\lambda,u_1,u_2,y_2}(\omega)\right]\, d\omega \Bigg|\\
&\lesssim (1+|y_2|+y_2^2) \sum_{k_2=-a}^{a-1} \int_{0}^{2\pi/\lambda} \beta_{1+\delta}\left(\tilde{\omega}_{k_1}-u_1,\xi+\frac{2\pi k_2}{\lambda}-u_2\right) \left(\omega_2-\frac{\pi}{\lambda}\right)^2 d\omega_2\\
&\lesssim \frac{1+|y_2|+y_2^2}{\lambda^2} \, \beta_{1+\delta} (\tilde{\omega}_{k_1}-u_1)
\end{align*}
for some intermediate point $\xi$ between $\omega_2$ and $\pi/\lambda$.
The same argument yields
\begin{align*}
B_{\lambda,u_1,u_2,y_1}(\omega_2)\lesssim \frac{1+|y_1|+y_1^2}{\lambda^2}\, \beta_{1+\delta}(\omega_2-u_2),
\end{align*}
and observing \eqref{two_dim_lemma_6.3}, this yields the claim of the lemma.
\end{proof}

\begin{lemma} \label{bound_variance}
Let the taper function $h:\R^d\to\R$ satisfy Assumption \ref{ass_on_h}, part (ii). Moreover, let $f:\R^d \rightarrow \R$ be a positive and bounded function with 
$\int_{\R^d} f^3(\bm{x})\, d\bm{x} <\infty$, and let the partial derivatives of $f$ exist and be bounded. Then, uniformly over $a\in\N$, it holds that
\begin{align*}
\sum_{\bm{m}=-2a+1}^{2a-1} H_d(\bm{m})^2 \int_{2\pi \max(-a,-a-\bm{m})/\lambda}^{2\pi \min(a,a-\bm{m})/\lambda} f^3(\bm{x}) \left|f(\bm{\omega_m}+\bm{x})-f(\bm{x})\right|\, d\bm{x}
\lesssim \frac{1}{(\log \lambda)^3}
\end{align*}
for $\lambda$ sufficiently large, where $\bm{\omega_m}=\bm{\omega}_{\bm{m},\lambda}$ and $H_d(\bm{m})=H_{d,h}(\bm{m})$ are defined in \eqref{Four_freq} and \eqref{H(m)}, respectively. 
\end{lemma}

\begin{proof}
For $d=1$, the expression of interest is bounded by $R_1(a,\lambda)+R_2(a,\lambda)$,
where
\begin{align*}
R_1(a,\lambda):=\sum_{|m|\leq \log(\lambda)} H_1(m)^2 \int_{-2\pi a/\lambda}^{2\pi a/\lambda} f^3(x) \left|f(\omega_m+x)-f(x)\right|\, dx
\end{align*}
and
\begin{align*}
R_2(a,\lambda):=\sum_{|m|>\log(\lambda)} H_1(m)^2 \int_{-2\pi a/\lambda}^{2\pi a/\lambda} f^3(x) \left|f(\omega_m+x)-f(x)\right|\, dx.
\end{align*}
Using the mean value theorem and the fact that $H_1(m)$ is uniformly bounded gives
\begin{align*}
R_1(a,\lambda)\lesssim \sum_{|m|\leq \log(\lambda)} \frac{2\pi |m|}{\lambda}\int_{-2\pi a/\lambda}^{2\pi a/\lambda} f^3(x) \, dx \lesssim \frac{(\log \lambda)^2}{\lambda}.
\end{align*}
Concerning the term $R_2(a,\lambda)$, note that integration by parts for $m\neq 0$ yields
\begin{align} \label{abschaetzung_fuer_H_1}
|H_1(m)|
&=\left| \frac{1}{2\pi^2 m^2} \int_{\frac{1}{2}}^{\frac{1}{2}} [h^{\prime\prime}(s) h(s) + h^{\prime}(s) h^{\prime}(s)] \exp(-2\pi \im m s)\, ds\right|\lesssim \frac{1}{m^2},
\end{align}
which gives
\begin{align*}
R_2(a,\lambda)
&\lesssim \sum_{|m|>\log(\lambda)} \frac{1}{m^4}\int_{-2\pi a/\lambda}^{2\pi a/\lambda} f^3(x) \left|f(\omega_m+x)-f(x)\right|\, dx \lesssim \frac{1}{(\log\lambda)^3}.
\end{align*}
The argumentation for $d>1$ works similarly, noting that
\begin{align*} 
&\phantom{==}\sum_{\bm{m}=-2a+1}^{2a-1} H_d(\bm{m})^2 \int_{[-2\pi a/\lambda,2\pi a/\lambda]^d} f^3(\bm{x}) \big|f(\bm{\omega_m}+\bm{x})-f(\bm{x})\big|\, d\bm{x}\nonumber\\
&\leq \sum_{i=1}^d \Bigg\{\sum_{\substack{m_j=-2a+1,\\ j\in\{1,\ldots,d\}\setminus i}}^{2a-1}
 H_{d-1}(m_1,\ldots,m_{i-1},m_{i+1},\ldots,m_d)^2  \Bigg[\sum_{m_i=-2a+1}^{2a-1} H_1(m_i)^2\nonumber\\
 &\phantom{====}\int_{[-2\pi a/\lambda,2\pi a/\lambda]^{d}}  f^3(x_1,\ldots,x_d) \big|f(x_1,\ldots,x_{i-1},\omega_{m_i}+x_i,\ldots,\omega_{m_d}+x_d)\nonumber\\
 &\phantom{==============}-f(x_1,\ldots,x_i,\omega_{m_{i+1}}+x_{i+1},\ldots,\omega_{m_d}+x_d)\big|\, dx_1 \ldots dx_d\Bigg] \Bigg\}.
\end{align*}
\end{proof}

\begin{lemma} \label{bound_variance_secint}
Let the taper function $h:\R^2\to\R$ satisfy Assumption \ref{ass_on_h}, part (ii). Moreover, let $f:\R^2 \rightarrow \R$ be a positive and bounded function with 
$\int_{\R^2} f(\bm{x})\, d\bm{x} <\infty$, and let the partial derivatives of $f$ exist and be bounded. Then, for $\bm{x}\in \R^2$ with $\|\bm{x}\|\leq 2\pi a/\lambda$ and $\lambda$ sufficiently large, it follows that
\begin{align*}
&\phantom{=i}\sum_{\bm{m}=-2a+1}^{2a-1} H_2(\bm{m})^2  \left(\int_{[-2\pi a/\lambda,2\pi a/\lambda]^2} f(\bm{z})\, \Big|\, f(\bm{z}+\bm{\omega_m})\, \exp(\im \bm{x}^T (\bm{z}+\bm{\omega_m}))-f(\bm{z}) \exp(\im \bm{x}^T \bm{z})\Big|\, d\bm{z}\right) \\
&\lesssim \frac{(\log \lambda)^2 \, a}{\lambda^{2}} +\frac{1}{(\log \lambda)^3}.
\end{align*}
\end{lemma}

\begin{proof}
We proceed similarly as in the proof of Lemma \ref{bound_variance}. In the one-dimensional case, let $|x|\leq 2\pi a/\lambda$ be arbitrary.
The expression of interest can then be bounded by
$R_1(a,\lambda,x)+R_2(a,\lambda,x)$,
where
\begin{align*}
R_1(a,\lambda,x):=\sum_{|m|\leq \log (\lambda)} H_1(m)^2 \int_{-2\pi a/\lambda}^{2\pi a/\lambda} f(z) 
\Big|f(z+\omega_m) \exp(\im x (z+\omega_m))-f(z) \exp(\im   x z)\Big|\, dz
\end{align*}
and
\begin{align*}
R_2(a,\lambda,x):=\sum_{|m|> \log (\lambda)}  H_1(m)^2 \int_{-2\pi a/\lambda}^{2\pi a/\lambda} f(z) 
\Big|f(z+\omega_m) \exp(\im x (z+\omega_m))-f(z) \exp(\im   x z)\Big|\, dz.
\end{align*}
By the mean value theorem, for some intermediate point $\xi$ between $z$ and $z+\omega_m$, we have
\begin{align*}
R_1(a,\lambda,x)&\lesssim \sup_{x\in\R:\, |x|\leq 2\pi a/\lambda} \sum_{|m|\leq \log(\lambda)} \frac{2\pi |m|}{\lambda} \int_{-2\pi a/\lambda}^{2\pi a/\lambda} f(z)\, \left|f^{'}(\xi) \exp(\im x \xi) +\im\, x f(\xi)  \exp(\im x \xi)\right| dz\\
&\lesssim \frac{a}{\lambda} \sum_{|m|\leq \log(\lambda)} \frac{2\pi |m|}{\lambda} \int_{-2\pi a/\lambda}^{2\pi a/\lambda} f(z)\, dz \lesssim \frac{(\log \lambda)^2 \, a}{\lambda^{2}},
\end{align*}
Using \eqref{abschaetzung_fuer_H_1}, it moreover follows that
\begin{align*}
R_2(a,\lambda,x) &\lesssim \sum_{|m|>\log(\lambda)} \frac{1}{m^4} \int_{-2\pi a/\lambda}^{2\pi a/\lambda} f(z) \big|f(z+\omega_m) \exp(\im x (z+\omega_m))-f(z) \exp(\im x z)\big|\, dz\lesssim \frac{1}{(\log\lambda)^3}.
\end{align*}
The two-dimensional case can be proven analogously and is obmitted for the sake of brevity. 
\end{proof}

The following result shows that the quantity $D_{2,2,\lambda,a}$ from \eqref{D_22lambda,a} is finite as long as the decay parameter $\delta$ from Assumption \ref{assumption_on_Z} is sufficiently large.

\begin{lemma} \label{D_2_finite}
Assume that the covariance function $c:\R^2\to\R$ and the spectral density $f:\R^2\to\R$ satisfy
\begin{align*}
\int_{\R^2} |c(\bm{h})|^2\, d\bm{h} <\infty \qquad \text{and} \qquad  f(\bm{\omega}) \leq \beta_{1+\delta}(\bm{\omega})
\end{align*}
for some $\delta>0$ and assume that $a,\lambda\rightarrow\infty$ with $a/\lambda\rightarrow\infty$. We then have
\begin{align*}
D_{2,2,\lambda,a}=\Landau\left(1+\frac{a^{2-2\delta}}{\lambda^{2-2\delta}} \right) \qquad \text{as } \lambda,a\rightarrow \infty.
\end{align*}
\end{lemma}

\begin{proof}
Note that by the Cauchy-Schwarz inequality and by definition of the covariance function, we have
\begin{align*}
\left|D_{2,2,\lambda,a}\right|
&\lesssim \int_{0}^{2\pi a/\lambda} \int_{0}^{2\pi} \left|\int_{[-2\pi a/\lambda,2\pi a/\lambda]^2} f(\bm{x}) \exp\left(\im \, r \begin{pmatrix}
\cos \theta\\
\sin \theta
\end{pmatrix}^T \bm{x} \right) d\bm{x} \right|^2 r \, d\theta \,  dr\\
&\lesssim \int_{\|\bm{y}\|\leq 2\pi a/\lambda} |c(\bm{y})|^2\, d\bm{y} + \int_{\|\bm{y}\|\leq 2\pi a/\lambda} \left|\int_{([-2\pi a/\lambda,2\pi a/\lambda]^2)^c} f(\bm{x}) \exp(\im  \bm{y}^T \bm{x})\, d\bm{x}\right|^2\, d\bm{y}.
\end{align*}
By assumption, the first term is uniformly bounded. Concerning the second term, we have
\begin{align*}
\sup_{\|\bm{y}\|\leq 2\pi a/\lambda} \left|\int_{([-2\pi a/\lambda,2\pi a/\lambda]^2)^c} f(\bm{x}) \exp(\im  \bm{y}^T \bm{x})\, d\bm{x}\right| 
& = \Landau\left(\int_{2\pi a/\lambda}^{\infty} \beta_{1+\delta}(x)\, dx\right) = \Landau\left(\frac{\lambda^\delta}{a^\delta}\right),
\end{align*}
which yields
\begin{align*}
|D_{2,2,\lambda,a}| =\Landau\left( 1+\frac{a^{2}}{\lambda^{2}} \left(\frac{\lambda^\delta}{a^\delta}\right)^2\right) = \Landau\left(1+\frac{a^{2-2\delta}}{\lambda^{2-2\delta}}\right).
\end{align*}
\end{proof}

\begin{corollary} \label{cor_of_D2_finite}
Under the same assumptions as in Lemma \ref{D_2_finite}, we have
\begin{align*}
\int_{\|\bm{y}\|\leq 2\pi a/\lambda} \Bigg|\int_{[-2\pi a/\lambda,2\pi a/\lambda]^2} f(\bm{x})\, \exp(\im \bm{y}^T \bm{x} )\, d\bm{x}\, \Bigg|\, d\bm{y} = \Landau\left(1+\frac{a^{2-\delta}}{\lambda^{2-\delta}} \right) \qquad \text{as } \lambda,a\rightarrow \infty.
\end{align*}
In particular, in order for this expression to be bounded we need to assume that $\delta\geq 2$. 
\end{corollary}

\begin{proof} 
Using the same arguments as in the proof of Lemma \ref{D_2_finite}, we obtain
\begin{align*}
&\phantom{==}\int_{\|\bm{y}\|\leq 2\pi a/\lambda} \Bigg|\int_{[-2\pi a/\lambda,2\pi a/\lambda]^2} f(\bm{x})\, \exp(\im \bm{y}^T \bm{x} )\, d\bm{x}\, \Bigg|\, d\bm{y}\\
& \lesssim \int_{\|\bm{y}\|\leq 2\pi a/\lambda} \left|c(\bm{y})\right|\, d\bm{y} + \int_{\|\bm{y}\|\leq 2\pi a/\lambda} \Bigg|\int_{([-2\pi a/\lambda,2\pi a/\lambda]^2)^c} f(\bm{x})\, \exp(\im \bm{y}^T \bm{x} )\, d\bm{x}\Bigg|\, d\bm{y} \\
&=\Landau\left( 1+\frac{a^{2}}{\lambda^{2}} \times \frac{\lambda^{\delta}}{a^{\delta}}\right) = \Landau\left(1+\frac{a^{2-\delta}}{\lambda^{2-\delta}}\right).
\end{align*}
\end{proof}

\begin{lemma} \label{cov_fct_integrable}
Assume that the covariance function $c:\R^2\to\R$ satisfies
\begin{align*}
\int_{\R^2} |c(\bm{h})|\, d\bm{h} <\infty 
\end{align*}
and let $a,\lambda\rightarrow\infty$ with $a/\lambda\rightarrow\infty$. Assume that the taper function $h:\R^2\to\R$ satisfies Assumption \ref{ass_on_h}, part (i), and define the frequency window $B(\bm{u})=B_{\lambda,2,h}(\bm{u})$ as in \eqref{Four_trafo_of_h}.
\begin{enumerate}
\item[(i)]
Under assumption \eqref{second_deriv_of_f} for $d=2$ and some $\delta>0$, we have
\begin{align*}
&\phantom{=i}\frac{1}{\lambda^2} \int_{\R^2} B(\bm{u})^2 \left[\int_{\|\bm{y}\|\leq 2\pi a/\lambda}  \left|\left(\frac{2\pi}{\lambda}\right)^2 \sum_{\bm{k}=-a}^{a-1} f(\bm{u}-\tilde{\bm{\omega}}_{\bm{k}}) \exp(\im  \bm{y}^T \tilde{\bm{\omega}}_{\bm{k}}) \right|\, d\bm{y}\, \right] \, d\bm{u}\\
& =\Landau\left(1+\frac{a^{2-\delta}}{\lambda^{2-\delta}}+\frac{a}{\lambda^{2}}+\frac{a^{4}}{\lambda^{6}}\right)  \qquad \text{as } \lambda,a\rightarrow\infty.
\end{align*}
\item[(ii)] If $f(\bm{\omega})\leq \beta_{1+\delta}(\bm{\omega})$ for some $\delta>0$, then
\begin{align*}
&\phantom{=i}\frac{1}{\lambda^2} \int_{\R^2} B(\bm{u})^2 \left[\int_{\|\bm{y}\|\leq 2\pi a/\lambda}  \left|\int_{[-2\pi a/\lambda,2\pi a/\lambda]^2} f(\bm{u}-\bm{x}) \exp(\im \bm{y}^T \bm{x}) \, d\bm{x}\right|\, d\bm{y}\, \right] \, d\bm{u}\\
& =\Landau\left(1+\frac{a^{2-\delta}}{\lambda^{2-\delta}}+\frac{a}{\lambda^{2}}\right)  \qquad \text{as } \lambda,a\rightarrow\infty.
\end{align*}
\end{enumerate}
\end{lemma}

\begin{proof}
We first prove part (i) and note that
\begin{align*}
\int_{\|\bm{y}\|\leq 2\pi a/\lambda} \left|\left(\frac{2\pi}{\lambda}\right)^2 \sum_{\bm{k}=-a}^{a-1} f(\bm{u}-\tilde{\bm{\omega}}_{\bm{k}}) \exp(\im  \bm{y}^T \tilde{\bm{\omega}}_{\bm{k}}) \right|\, d\bm{y} \leq E_1(a,\lambda,\bm{u})+E_2(a,\lambda,\bm{u})+E_3(a,\lambda,\bm{u}),
\end{align*}
where
\begin{align*}
E_1(a,\lambda,\bm{u}):=\int_{\|\bm{y}\|\leq 2\pi a/\lambda}  \left|\int_{\R^2} f(\bm{u}-\bm{x}) \exp(\im \bm{y}^T \bm{x}) \, d\bm{x}\right|\, d\bm{y},
\end{align*}
\begin{align*}
E_2(a,\lambda,\bm{u}):=\int_{\|\bm{y}\|\leq 2\pi a/\lambda}  \left|\int_{([-2\pi a/\lambda,2\pi a/\lambda]^2)^c} f(\bm{u}-\bm{x}) \exp(\im \bm{y}^T \bm{x}) \, d\bm{x}\right|\, d\bm{y},
\end{align*}
and
\begin{align*}
E_3(a,\lambda,\bm{u})&:=\int_{\|\bm{y}\|\leq 2\pi a/\lambda}  \Bigg| \left(\frac{2\pi}{\lambda}\right)^2 \sum_{\bm{k}=-a}^{a-1} f(\bm{u}-\tilde{\bm{\omega}}_{\bm{k}}) \exp(\im  \bm{y}^T \tilde{\bm{\omega}}_{\bm{k}}) \\
&\phantom{=============}- \int_{[-2\pi a/\lambda,2\pi a/\lambda]^2} f(\bm{u}-\bm{x}) \exp(\im \bm{y}^T \bm{x}) \, d\bm{x} \Bigg|\, d\bm{y}.
\end{align*}
We start with the term $E_1(a,\lambda,\bm{u})$ and immediately get from Lemma \ref{orders_of_B}, part (i), 
\begin{align*}
\frac{1}{\lambda^2} \int_{\R^2} B(\bm{u})^2\, E_1(a,\lambda,\bm{u})\, d\bm{u} 
&= \left(\frac{1}{\lambda^2} \int_{\R^2} B(\bm{u})^2\, d\bm{u}\right) \left( \int_{\|\bm{y}\|\leq 2\pi a/\lambda}  \left|\int_{\R^2} f(\bm{\omega}) \exp(-\im \bm{y}^T \bm{\omega}) \, d\bm{\omega}\right|\, d\bm{y} \right)\\
&\eqsim \left(\frac{1}{\lambda^2} \int_{\R^2} B(\bm{u})^2 \, d\bm{u}\right) \left( \int_{\|\bm{y}\|\leq 2\pi a/\lambda}   \left|c(\bm{y})\right|\, d\bm{y}\right)= \Landau(1).
\end{align*} 
Concerning the term $E_2(a,\lambda,\bm{u})$, note that
\begin{align*} 
\frac{1}{\lambda^2}\int_{\R^2} B(\bm{u})^2 \, E_2(a,\lambda,\bm{u})\, d\bm{u} 
&\lesssim  \left(\frac{a}{\lambda}\right)^{2} B(a,\lambda,\delta),
\end{align*}
where
\begin{align*}
B(a,\lambda,\delta)&:= \frac{1}{\lambda^2} \int_{\R^2} B(\bm{u})^2\left(\int_{([-2\pi a/\lambda,2\pi a/\lambda]^2)^c} \beta_{1+\delta}(\bm{x}-\bm{u})\, d\bm{x}  \right) d\bm{u}\\
&\lesssim  \frac{1}{\lambda} \int_{\R} B(u)^2 \left(\int_{|x|>2\pi a/\lambda} \beta_{1+\delta}(x-u)\, dx\right) du=\frac{2}{\lambda} \int_{\R} B(u)^2 \left(\int_{2\pi a/\lambda - u}^{\infty} \beta_{1+\delta}(\omega)\, d\omega\right) du.
\end{align*}
We broadly use the idea that if $|u|$ is small, then the inner integral will be small, whereas if $|u|$ is large, then $B(u)^2$ will be small. More precisely, recall that due to Lemma \ref{bounds_for_B}, part (i), $|B(u)|\lesssim 1/|u|$ for $|u|>1/\lambda$. This gives
\begin{align*}
B(a,\lambda,\delta)
&\lesssim \int_{2\pi a/\lambda }^{\infty} \beta_{1+\delta}(\omega)\, d\omega + \frac{1}{\lambda} \int_{|u|>\pi a/\lambda} B(u)^2\, du= \Landau\left(\frac{\lambda^\delta}{a^\delta}+ \frac{1}{a}\right),
\end{align*}
and therefore
\begin{align*}
\frac{1}{\lambda^2}\int_{\R^2} B(\bm{u})^2 \, E_2(a,\lambda,\bm{u})\, d\bm{u} = \Landau\left(\frac{a^{2}}{\lambda^{2}} \left[\frac{\lambda^\delta}{a^\delta}+\frac{1}{a}\right]\right) = \Landau\left(\frac{a^{2-\delta}}{\lambda^{2-\delta}}+ \frac{a}{\lambda^{2}}\right).
\end{align*}
We finally deal with the term $E_3(a,\lambda,\bm{u})$ and note that due to Lemma \ref{Riemann_f_sec_int}, it holds that
\begin{align*}
&\sup_{\|\bm{y}\|\leq 2\pi a/\lambda} \left| \left(\frac{2\pi}{\lambda}\right)^2 \sum_{\bm{k}=-a}^{a-1} f(\bm{u}-\tilde{\bm{\omega}}_{\bm{k}}) \exp(\im  \bm{y}^T \tilde{\bm{\omega}}_{\bm{k}}) - \int_{[-2\pi a/\lambda,2\pi a/\lambda]^2} f(\bm{u}-\bm{x}) \exp(\im \bm{y}^T \bm{x}) \, d\bm{x} \right| = \Landau\left(\frac{a^{2}}{\lambda^{4}}\right),
\end{align*}
uniformly over $\bm{u}\in\R^2$. This yields the claim of part (i), since
\begin{align*}
\frac{1}{\lambda^2} \int_{\R^2} B(\bm{u})^2\, E_3(a,\lambda,\bm{u}) \, d\bm{u} = \Landau\left(\frac{a^{2}}{\lambda^{2}}\times \frac{a^{2}}{\lambda^{4}}\right) = \Landau\left(\frac{a^{4}}{\lambda^{6}}\right).
\end{align*}
The proof of part (ii) works with exactly the same arguments as used for the proof of part (i).
\end{proof}

\begin{lemma} \label{f_times_x^2}
Let the assumption in \eqref{second_deriv_of_f} for $d=2$ and some $\delta>2$ be satisfied and assume that $a,\lambda\rightarrow\infty$ with $a/\lambda\rightarrow\infty$. Furthermore, let the taper function $h:\R^2\to\R$ satisfy Assumption \ref{ass_on_h}, part (ii). 
We then have for $i=1,2$
\begin{enumerate}
\item[(i)] 
\begin{align*}
\frac{1}{\lambda^2} \int_{\R^2} B(\bm{u})^2 \left[\left(\frac{2\pi}{\lambda}\right)^2 \sum_{\bm{k}=-a}^{a-1} f(\bm{u}-\tilde{\bm{\omega}}_{\bm{k}}) \, \tilde{\omega}_{k_i}^2 \right]  d\bm{u} = \Landau\left(1+\frac{a^2}{\lambda^4}\right) \qquad \text{as } \lambda,a\rightarrow \infty,
\end{align*}
\item[(ii)]
\begin{align*}
&\phantom{=}\frac{1}{\lambda^4} \sum_{\bm{m}=-2a+1}^{2a-1} \int_{\R^4} \big|B(\bm{u}) B(\bm{u}+\bm{\omega_m}) B(\bm{v}) B(\bm{v}-\bm{\omega_m})\big| \left[\left(\frac{2\pi}{\lambda}\right)^2 \sum_{\bm{k}=-a}^{a-1} f(\bm{u}-\tilde{\bm{\omega}}_{\bm{k}}) \, \tilde{\omega}_{k_i}^2 \right]  d\bm{u} d\bm{v}\\
&=\Landau\left(1+\frac{a^2}{\lambda^4} + \frac{(\log a)^4}{\lambda^2}\right) \qquad \text{as } \lambda,a\rightarrow \infty.
\end{align*}
\end{enumerate}
\end{lemma}

\begin{proof}
Assume without loss of generality that $i=1$. 
Using analogous arguments as in the proof of Lemma \ref{Riemann_f_sec_int}, we obtain
\begin{align} \label{riemann_lemma6.9}
&\phantom{\lesssim} \left|\left(\frac{2\pi}{\lambda}\right)^2 \sum_{\bm{k}=-a}^{a-1} f(\bm{u}-\tilde{\bm{\omega}}_{\bm{k}}) \, \tilde{\omega}_{k_1}^2 - \int_{[-2\pi a/\lambda,2\pi a/\lambda]^2} f(\bm{u}-\bm{x}) \, x_1^2 \, d\bm{x} \right|\nonumber\\
&\lesssim \frac{a^2}{\lambda^4}\left(\frac{2\pi}{\lambda} \sum_{k_2=-a}^{a-1} \beta_{1+\delta}(u_2-\tilde{\omega}_{k_2}) + \int_{-2\pi a/\lambda}^{2\pi a/\lambda} \beta_{1+\delta}(u_1-x_1)\, dx_1\right)
\lesssim \frac{a^2}{\lambda^4}
\end{align}
for $a>\lambda\geq 1$, uniformly over $\bm{u}\in\R^2$. 
We now prove statement (i) of the lemma and note that \eqref{riemann_lemma6.9} gives
\begin{align*}
\frac{1}{\lambda^2} \int_{\R^2} B(\bm{u})^2 \left[\left(\frac{2\pi}{\lambda}\right)^2 \sum_{\bm{k}=-a}^{a-1} f(\bm{u}-\tilde{\bm{\omega}}_{\bm{k}}) \, \tilde{\omega}_{k_1}^2 \right]  d\bm{u} = \Landau\Bigg(\frac{1}{\lambda} \int_{\R} B(u)^2 \Bigg[\int_{\R} \beta_{1+\delta}(x-u)\, x^2\, dx \Bigg] du + \frac{a^2}{\lambda^4}\Bigg).
\end{align*}
For $u\in\R$ arbitrary, note that
\begin{align*}
\int_{\R} \beta_{1+\delta} (x-u)\, x^2\, dx 
&\lesssim \int_{\R} \beta_{1+\delta}(x)\,x^2\, dx + |u| \int_{\R} \beta_{1+\delta}(x) \,|x|\, dx + u^2 \int_{\R} \beta_{1+\delta}(x) \, dx.
\end{align*}
By definition of the function $\beta_{1+\delta}$ and since we assumed $\delta>2$, we obtain
\begin{align*}
\int_{\R} \beta_{1+\delta}(x)\, x^2\, dx \lesssim \int_{|x|\leq 1} \beta_{1+\delta}(x)\, dx + \int_{|x|>1} \frac{1}{|x|^{\delta-1}} \, dx \leq C_1
\end{align*}
and in the same way 
$\int_{\R} \beta_{1+\delta}(x)\, |x|\, dx \leq C_2$
for some generic constants $C_1,C_2>0$. 
Lemma \ref{orders_of_B} then gives
\begin{align*}
\frac{1}{\lambda} \int_{\R} B(u)^2 \left[\int_{\R} \beta_{1+\delta} (x-u)\, x^{2}\, dx\right] du &\lesssim \frac{1}{\lambda} \int_{\R} B(u)^2 \left[1+|u|+u^2\right]\, du =\Landau(1),
\end{align*}
which yields the claim in (i). For the proof of (ii), note that with the same arguments as in the proof of part (i) and by Lemma \ref{bounds_for_B}, part (iii), 
\begin{align} \label{allerletzte_gleichung}
&\phantom{=i}\frac{1}{\lambda^4} \sum_{\bm{m}=-2a+1}^{2a-1} \int_{\R^4} \big|B(\bm{u}) B(\bm{u}+\bm{\omega_m}) B(\bm{v}) B(\bm{v}-\bm{\omega_m})\big| \left[\left(\frac{2\pi}{\lambda}\right)^2 \sum_{\bm{k}=-a}^{a-1} f(\bm{u}-\tilde{\bm{\omega}}_{\bm{k}}) \, \tilde{\omega}_{k_1}^2 \right] \, d\bm{u} d\bm{v}\nonumber\\
&=\Landau\Bigg(1+\frac{1}{\lambda^4} \sum_{\bm{m}=-2a+1}^{2a-1} \int_{\R^4} \big|B(\bm{u}) B(\bm{u}+\bm{\omega_m}) B(\bm{v}) B(\bm{v}-\bm{\omega_m})\big| \big[|u_1|+u_1^2 \big] \, d\bm{u} d\bm{v}+\frac{a^2}{\lambda^4}\Bigg).
\end{align}
From Lemma \ref{orders_of_B}, part (iv), we obtain
\begin{align*}
&\phantom{=i}  \int_{\R} \big|B(u_1) B(u_1+\omega_{m_1})\big|  \left[\big|u_1\big|+u_1^2 \right] du_1\\
&\lesssim \int_{|u_1|\leq 1} \big|B(u_1) B(u_1+\omega_{m_1})\big|\, du_1 + \int_{|u_1|>1} \big|B(u_1) B(u_1+\omega_{m_1})\big|\, u_1^2\, du_1\\
&\lesssim\int_{\R} \big|B(u_1) B(u_1+\omega_{m_1})\big|\, du_1 + \frac{1}{\lambda}.
\end{align*}
Therefore, the expression in \eqref{allerletzte_gleichung} is of order
\begin{align*}
&\Landau\Bigg(1+\frac{1}{\lambda^4} \sum_{\bm{m}=-2a+1}^{2a-1} \int_{\R^4} \big|B(\bm{u}) B(\bm{u}+\bm{\omega_m}) B(\bm{v}) B(\bm{v}-\bm{\omega_m})\big|\, d\bm{u} d\bm{v} \\
&\phantom{======}+ \frac{1}{\lambda^5} \sum_{\bm{m}=-2a+1}^{2a-1} \int_{\R^3} \big|B(u_2) B(u_2+\omega_{m_2}) B(\bm{v}) B(\bm{v}-\bm{\omega_m})\big| du_2\, d\bm{v}\Bigg)=\Landau\left(1+\frac{(\log a)^4}{\lambda^2}\right),
\end{align*}
where we used Lemma \ref{bounds_for_B}, parts (ii) and (iii), and Lemma \ref{orders_of_B}, part (i), in the last step.
\end{proof}

\begin{lemma} \label{lemma_for_cumulants}
Let the spectral density $f:\R^2\to\R$
satisfy the assumption in \eqref{second_deriv_of_f} for $d=2$ and some $\delta>0$ and
assume that the covariance function $c:\R^2\to\R$ is uniformly bounded and satisfies 
\begin{align*}
c(\bm{h})=\Landau\left(\|\bm{h}\|^{-(2+\varepsilon)}\right), \qquad \text{as } \|\bm{h}\|\rightarrow\infty,
\end{align*} 
for some $\varepsilon>0$. Let $a,\lambda\rightarrow\infty$ with $a/\lambda\to\infty$ and let the taper function $h:\R^2\to\R$ satisfy Assumption \ref{ass_on_h}, part (i).
Then, we have
\begin{align} \label{formula_lemmaD11}
&\phantom{=i}\frac{1}{\lambda^2} \int_{\R^2} B(\bm{u})^2 \times \frac{2\pi}{\lambda} \sum_{r=0}^{a-1} \tilde{\omega}_r \int_{0}^{2\pi} \left|\left(\frac{2\pi}{\lambda}\right)^2 \sum_{\bm{k}=-a}^{a-1}  f(\bm{u}-\tilde{\bm{\omega}}_{\bm{k}}) \exp\bigg(\im\, \tilde{\omega}_r \begin{pmatrix}
\cos \theta\\
\sin \theta
\end{pmatrix}^T \tilde{\bm{\omega}}_{\bm{k}}\bigg)\right|d\theta\, d\bm{u}\\
&=\Landau\left(1+\frac{a^{2-\delta}}{\lambda^{2-\delta}} + \frac{a}{\lambda^{2}} + \frac{a^{4}}{\lambda^{6}}\right) \qquad \text{as } \lambda,a\to\infty.\nonumber
\end{align}
\end{lemma}

\begin{proof}
Note that 
\begin{align*}
&\phantom{iii}\left|\left(\frac{2\pi}{\lambda}\right)^2 \sum_{\bm{k}=-a}^{a-1}  f(\bm{u}-\tilde{\bm{\omega}}_{\bm{k}}) \exp\bigg(\im\, \tilde{\omega}_r \begin{pmatrix}
\cos \theta\\
\sin \theta
\end{pmatrix}^T \tilde{\bm{\omega}}_{\bm{k}}\bigg)\right|\leq D_1(r,\theta,\bm{u})+D_2(r,\theta,\bm{u})+D_3(r,\theta,\bm{u}),
\end{align*}
where
\begin{align*}
D_1(r,\theta,\bm{u}):=\left|\int_{\R^2} f(\bm{u}-\bm{x}) \exp\bigg(\im\, \tilde{\omega}_r \begin{pmatrix}
\cos \theta\\
\sin \theta
\end{pmatrix}^T \bm{x} \bigg)\,d\bm{x}\right|,
\end{align*}
\begin{align*}
D_2(r,\theta,\bm{u}):=\left|\int_{([-2\pi a/\lambda,2\pi a/\lambda]^2)^c} f(\bm{u}-\bm{x}) \exp\bigg(\im\, \tilde{\omega}_r \begin{pmatrix}
\cos \theta\\
\sin \theta
\end{pmatrix}^T \bm{x} \bigg)\,d\bm{x}\right|,
\end{align*}
and
\begin{align*}
D_3(r,\theta,\bm{u})&:=\Bigg|\left(\frac{2\pi}{\lambda}\right)^2 \sum_{\bm{k}=-a}^{a-1}  f(\bm{u}-\tilde{\bm{\omega}}_{\bm{k}}) \exp\bigg(\im\, \tilde{\omega}_r \begin{pmatrix}
\cos \theta\\
\sin \theta
\end{pmatrix}^T \tilde{\bm{\omega}}_{\bm{k}}\bigg) \\
&\phantom{===========}- \int_{[-2\pi a/\lambda,2\pi a/\lambda]^2} f(\bm{u}-\bm{x}) \exp\bigg(\im\, \tilde{\omega}_r \begin{pmatrix}
\cos \theta\\
\sin \theta
\end{pmatrix}^T \bm{x} \bigg)\,d\bm{x}\Bigg|.
\end{align*}
It obviously holds that
\begin{align*}
D_1(r,\theta,\bm{u})=(2\pi)^2\, c\left(\tilde{\omega}_r \cos\theta, \tilde{\omega}_r \sin \theta\right),
\end{align*}
and using Lemma \ref{Riemann_f_sec_int}, we furthermore obtain
\begin{align*}
\sup_{0\leq r \leq a-1} \sup_{\theta\in [0,2\pi)} \sup_{\bm{u}\in\R^2}\, D_3(r,\theta,\bm{u}) \lesssim \sup_{0\leq r \leq a-1} \frac{1+|\tilde{\omega}_r|+\tilde{\omega}_r^2}{\lambda^2} = \Landau\left(\frac{a^{2}}{\lambda^{4}}\right).
\end{align*}
Therefore, the left hand side of \eqref{formula_lemmaD11} is of order
\begin{align*}
&\phantom{=i}\Landau\bigg(\frac{1}{\lambda^2} \int_{\R^2} B(\bm{u})^2\, d\bm{u} \times \frac{2\pi}{\lambda} \sum_{r=0}^{a-1} \tilde{\omega}_r \int_{0}^{2\pi} \left|c\left(\tilde{\omega}_r \begin{pmatrix}
\cos \theta\\
\sin \theta
\end{pmatrix}\right)\right|\, d\theta\\
&\phantom{i}+\frac{1}{\lambda^2} \int_{\R^2} B(\bm{u})^2 \times \frac{2\pi}{\lambda} \sum_{r=0}^{a-1} \tilde{\omega}_r \left[ \int_{([-2\pi a/\lambda,2\pi a/\lambda]^2)^c} f(\bm{u}-\bm{x})\, d\bm{x}\right]\, d\bm{u} + \frac{a^2}{\lambda^6} \int_{\R^2} B(\bm{u})^2  \, d\bm{u}\times \frac{2\pi}{\lambda} \sum_{r=0}^{a-1} \tilde{\omega}_r \bigg)\\
&= \Landau\bigg(\sup_{\theta\in[0,2\pi)} \frac{2\pi}{\lambda} \sum_{r=0}^{a-1} \tilde{\omega}_r \left|c\left(\tilde{\omega}_r \begin{pmatrix}
\cos \theta\\
\sin \theta
\end{pmatrix}\right)\right|+ \frac{a^{2-\delta}}{\lambda^{2-\delta}} + \frac{a}{\lambda^{2}} + \frac{a^{4}}{\lambda^{6}}\bigg),
\end{align*}
where we furthermore used that 
\begin{align*}
\frac{1}{\lambda^2} \int_{\R^2} B(\bm{u})^2 \left[\int_{([-2\pi a/\lambda,2\pi a/\lambda]^2)^c} f(\bm{u}-\bm{x})\, d\bm{x}\right] d\bm{u} = \Landau\left(\frac{\lambda^\delta}{a^\delta} + \frac{1}{a}\right)
\end{align*} 
(see the proof of Lemma \ref{cov_fct_integrable}). In order to finalize the proof, note that by assumption
\begin{align} \label{Landau_cond}
\exists\, C>0 \quad \exists \,h_0>0 \quad \forall \, \bm{h}\in\R^2 \text{ with } \|\bm{h}\|\geq h_0: |c(\bm{h})|\leq C \|\bm{h}\|^{-(2+\varepsilon)}.
\end{align}
Define $K:=h_0/2\pi$. Note that for $r\leq K\lambda$, we have
$\tilde{\omega}_r \leq h_0 + \frac{\pi}{\lambda}$,
while for $r>K\lambda$, it holds that
$\tilde{\omega}_r >  h_0$.
By \eqref{Landau_cond} and since $c$ is uniformly bounded, this yields
\begin{align*}
\frac{1}{\lambda} \sum_{r=0}^{a-1} \tilde{\omega}_r \left|c\left(\tilde{\omega}_r \begin{pmatrix}
\cos \theta\\
\sin \theta
\end{pmatrix}\right)\right| &\leq \frac{1}{\lambda} \sum_{r=0}^{\lfloor K \lambda \rfloor}  \tilde{\omega}_r \left|c\left(\tilde{\omega}_r \begin{pmatrix}
\cos \theta\\
\sin \theta
\end{pmatrix}\right)\right| + \frac{1}{\lambda} \sum_{r=\lceil K\lambda \rceil}^{\infty} \tilde{\omega}_r \left|c\left(\tilde{\omega}_r \begin{pmatrix}
\cos \theta\\
\sin \theta
\end{pmatrix}\right)\right|\\
&\lesssim \frac{1}{\lambda} \sum_{r=0}^{\lfloor K \lambda \rfloor} \left(h_0 + \frac{\pi}{\lambda}\right) + \frac{1}{\lambda} \sum_{r=\lceil K \lambda \rceil}^{\infty} \left(\tilde{\omega}_r\right)^{-(1+\varepsilon)}\\
&=\Landau\left(1+\lambda^{\varepsilon} \sum_{r=\lceil K \lambda \rceil}^{\infty} \left(\frac{1}{r}\right)^{1+\varepsilon}\right)=\Landau(1).
\end{align*}
\end{proof}

\cleardoublepage
\end{document}